\documentclass[11pt]{amsart}

\usepackage{mathtools}
\usepackage{amsmath, amsthm, amssymb, amsfonts, enumerate}
\usepackage{dsfont}
\usepackage{color}
\usepackage{geometry}
\usepackage{todonotes}
\usepackage{graphicx}
\usepackage{caption}
\usepackage{float}
\usepackage{soul, comment}
\usepackage{tensor}
\usepackage[colorlinks=true,linkcolor=blue,urlcolor=blue, citecolor=blue]{hyperref}

\mathtoolsset{showonlyrefs}

\def \cA{\mathcal{A}}
\def \cAcirc{\mathcal{A}^{\raise1pt\hbox{${\scriptstyle\circ}$}}}
\def \cB{\mathcal{B}}
\def \cC{\mathcal{C}}

\def \cE{\mathcal{E}}
\def \cF{\mathcal{F}}
\def \cG{\mathcal{G}}
\def \cH{\mathcal{H}}

\def \cJ{\mathcal{J}}
\def \cK{\mathcal{K}}
\def \cL{\mathcal{L}}
\def \cM{\mathcal{M}}

\def \cP{\mathcal P}

\def \cR{\mathcal{R}}
\def \cS{\mathcal{S}}
\def \cU{\mathcal{U}}
\def \cT{\mathcal{T}}
\def \P{\mathsf P}

\def \E{\mathsf E}

\def \N{\mathbb{N}}
\def \H{\mathbb{H}}
\def \R{\mathbb{R}}
\def \F{\mathbb F}
\def \G{\mathbb G}

\def \ud{\mathrm{d}}

\newcommand{\roundbrackets}[1]{(#1)}
\newcommand\optional[1]{\tensor*[^o]{{\kern-1pt #1}}{}}
\newcommand\previsible[1]{\tensor*[^p]{{\kern-1pt #1}}{}}

\makeatletter 
\def\punctuationSpace{\@ifnextchar.{}{\@ifnextchar,{}{\@ifnextchar;{}{ }}}}
\makeatother

\newcommand{\cadlag}{c\`adl\`ag\punctuationSpace}
\newcommand{\caglad}{c\`agl\`ad\punctuationSpace}
\newcommand\as{\mbox{-a.s.}\punctuationSpace}

\newcommand\Gam{\Gamma}

\newcommand{\eps}{\varepsilon}
\newcommand{\ind}{\mathbf{1}}
\newcommand\indd[1]{\ind_{\{#1\}}}

\newtheorem{theorem}{Theorem}[section]
\newtheorem{lemma}[theorem]{Lemma}
\newtheorem{corollary}[theorem]{Corollary}
\newtheorem{proposition}[theorem]{Proposition}

\newtheorem{definition}[theorem]{Definition}
\newtheorem{remark}[theorem]{Remark}
\newtheorem{assumption}[theorem]{Assumption}
\newtheorem{convention}[theorem]{Convention}
\newtheorem{notation}[theorem]{Notation}

\theoremstyle{definition}

\DeclareMathOperator*{\esssup}{ess\,sup}
\DeclareMathOperator*{\essinf}{ess\,inf}

\definecolor{ballblue}{rgb}{0.13, 0.67, 0.8}

\allowdisplaybreaks

\makeatletter
\@namedef{subjclassname@2020}{\textup{2020}Mathematics Subject Classification}
\makeatother

\title[Dynkin games with asymmetric information]{Martingale theory \\ for Dynkin games with asymmetric information}
\thanks{{\bf Acknowledgements}: T.\ De Angelis was partially supported by EU -- Next Generation EU -- PRIN2022 (2022BEMMLZ) CUP: D53D23005780006 and PRIN-PNRR2022 (P20224TM7Z) CUP: D53D23018780001.}
\author[De Angelis]{Tiziano De Angelis}
\author[Palczewski]{Jan Palczewski}
\author[Smith]{Jacob Smith}
\address{T.\ De Angelis: School of Management and Economics, Dept.\ ESOMAS, University of Torino, Corso Unione Sovietica, 218 Bis, 10134, Torino, Italy; Collegio Carlo Alberto, Piazza Arbarello 8, 10122, Torino, Italy.}
\email{\href{mailto:tiziano.deangelis@unito.it}{tiziano.deangelis@unito.it}}
\address{J.\ Palczewski: Faculty of Mathematics, Wroc\l{}aw University of Science and Technology, Wybrze\.{z}e Wyspia\'{n}skiego 27,
50-370, Wroc\l{}aw, Poland.}
\email{\href{mailto:jan.palczewski@pwr.edu.pl}{jan.palczewski@pwr.edu.pl}}
\address{J.\ Smith: School of Mathematics, University of Leeds, Woodhouse Lane, LS2 9JT Leeds, UK.}
\email{\href{mailto:mm15js@leeds.ac.uk}{mm15js@leeds.ac.uk}}

\date{\today}
\numberwithin{equation}{section}

\begin{document}

\begin{abstract}
This paper provides necessary and sufficient conditions for a pair of randomised stopping times to form a saddle point of a zero-sum Dynkin game with partial and/or asymmetric information across players. The framework is non-Markovian and covers essentially any information structure. Our methodology relies on the identification of suitable super and submartingales involving players' equilibrium payoffs. Saddle point strategies are characterised in terms of the dynamics of those equilibrium payoffs and are related to their Doob-Meyer decompositions.
\end{abstract}

\maketitle
\tableofcontents

\newpage

\section{Introduction}
Zero-sum optimal stopping games (Dynkin games) in which players have access to different filtrations are an emerging field of continuous-time stochastic game theory. Recent results by \cite{de2022value} have shown that such games admit a saddle point in randomised stopping times in a general non-Markovian setting. Players' filtrations are arbitrary and only need to satisfy the usual conditions. The underlying payoff processes are bounded in expectation, \cadlag, measurable but not necessarily adapted to either of the players' filtrations.

The results in \cite{de2022value} concern the existence of saddle points but do not offer a dynamic picture which is necessary for characterisation and construction of players' optimal strategies. The present paper addresses this shortcoming by offering three main contributions: (i) \emph{necessary} and \emph{sufficient} conditions for a pair of randomised stopping times to form a saddle point (equilibrium) of the game, (ii) a dynamic characterisation of the game, including super/submartingale conditions for players' equilibrium value processes\footnote{In line with game-theoretic terminology for repeated games, these are also referred to as continuation value processes, because they represent the optimal value that a player can obtain by continuing to play the game from a given instant of time.} and a representation of optimal strategies, (iii) an application of the abstract theory to two classes of games with asymmetric information. From the methodological perspective we break with the classical approach to zero-sum games which is based on the study of the game's value. The generality of the information structure leads us naturally to consider zero-sum games through the lens of nonzero-sum games in which each player's equilibrium value process has dynamics adapted to the player's own filtration. The zero-sum feature of the game is recovered thanks to an equivalence in expectation between the two players' payoffs, which yields a quantity that can be thought of as the {\em ex-ante} value of the game (i.e., the game's value before players have access to the pieces of information that create asymmetry in the game).

To obtain necessary and sufficient conditions for a saddle point, we build on the general theory of stochastic processes in the spirit of El Karoui's seminal work \cite{elkaroui1981} which was brought in touch with Dynkin games by Lepeltier and Maingueneau \cite{lepeltier1984} in a full information setup. Unlike those contributions, games with asymmetric information {\em cannot} be characterised by a value process common to both players. Consequenty, their stopping strategies {\em cannot} be written purely in terms of the coincidence of the value process with the respective payoff process. Instead, starting from general families of random variables and proceeding with aggregation results, we develop a martingale theory for the players' equilibrium value processes and optimal randomised strategies. 

We show that equilibrium value processes can be described by optional semi-martingales (with respect to each player's own filtration) which evolve above/below suitable optional projections of players' stopping payoffs (depending on whether a player is a maximiser or a minimiser). Differently from the full information game, the aggregation step and the super/submartingale properties require a dynamic representation of the game involving non-decreasing processes that ``generate'' the players' optimal randomised stopping times; these generating processes are conditional cumulative distribution functions of the players' 
randomised stopping times; due to the asymmetry of information, each player must project their opponent's stopping rule (e.g., optimal generating process) onto the observed filtration. We also obtain exact formulae for jumps in the dynamics of the equilibrium value processes and determine the support of the generating processes for the equilibrium pair of randomised stopping times. Finally, we deduce a relationship between the Doob-Meyer decompositions of the equilibrium value processes and the generating processes of the corresponding equilibrium stopping times. That relationship can be used to construct optimal randomised stopping times in concrete applications of the theory. 

Our paper provides theoretical foundations for the study of stopping games with asymmetric information in general non-Markovian setups, where the information structure is essentially arbitrary. We believe this to be the first comprehensive treatment of such games in their generality. The literature on Dynkin games with asymmetric information is patchy and problems have been solved on a case-by-case basis with ad-hoc methods (cf.\ next section for details). In this work we focus on the dynamic aspects of the game and devise a methodology that allows a more systematic approach to the construction of saddle points in randomised strategies. We remark that randomised strategies are indeed necessary for the construction of saddle points, because simple counterexamples with no saddle point in pure strategies are provided in, e.g., \cite{Grun2013} and \cite[Sec.\ 6]{de2022value}. 

Beyond the abstract theory, we demonstrate the scope of our study by specialising to two natural classes of games with asymmetric information and showing that our methods lead to computable expressions. In the first class of games, motivated by the seminal paper \cite{Grun2013}, there are finitely many payoff regimes. One player is fully informed and learns the realisation of the regime at the beginning of the game. The other player is only aware of its probability distribution. A fully worked out example of a game from this class is available in \cite[Ch.\ 6]{smith2024martingale}. In the second class of games, the asymmetry of information is analogous albeit the regime affects only the drift of the underlying observable diffusion which determines the payoff; this generalises a specific problem solved in \cite{DEG2020}. The difference between the first and the second class of games lies in the available learning opportunities for the less-informed player. In the first class, this player can only infer information about the regime from the lack of actions of the other player. In the second class, an additional source of information comes in the form of the observations of the underlying diffusion, which can be utilised through the application of the stochastic filtering theory. 

\subsection{Related literature}
The study of zero-sum Dynkin games, formulated by Dynkin \cite{dynkin1969}, dates back to the 1970s in the classical set-up where players have full and symmetric information. Although in this introduction we focus on the continuous-time version of the problem, we acknowledge the existence of a vast body of work in the discrete-time setting, see, e.g., \cite{domansky2002}, \cite{kiefer1971optimal}, \cite[Ch.\ VI]{neveu}, \cite{rosenberg2001}, \cite{yasuda1985}. 

The theory of zero-sum Dynkin games has been developed over the past five decades with methods ranging from PDEs to BSDEs and the theory of Markov processes. Early contributions were due to Bensoussan and Friedman \cite{bensoussan1974}, Bismut \cite{bismut1977} and Stettner \cite{stettner1982}, among others. A milestone in the general (non-Markovian) theory was the work of Lepeltier and Maingueneau \cite{lepeltier1984}, who used tools from the theory of stochastic processes to prove the existence and obtain a characterisation of the game's value and of the equilibrium stopping times.

Following this, the research on Dynkin games slowed for nearly two decades before being revived in the early 2000s. Hamad\`ene and Lepeltier \cite{hamadene2000reflected} applied BSDE methods to solve mixed control-stopping games, while Kifer \cite{kifer2000} introduced an application of Dynkin games in mathematical finance. Relaxation of classical conditions on the payoffs sparked a new wave of general results on the existence of a value in randomised stopping times in a non-Markovian setting, see, e.g., Touzi and Vieille \cite{TouziVieille2002} and Laraki and Solan \cite{LarakiSolan2005}. In parallel, Ekstr\"om and Peskir developed the theory for general Markovian games in \cite{ekstrom2008}, while Ekstr\"om and Villeneuve \cite{ekstrom2006} solved the problem for one-dimensional diffusions under relaxed integrability conditions. More recently, there has been a new drive towards the development of a general theory for Markovian equilibria in randomised stopping times for Dynkin games with full information. The field is growing rapidly with the most recent contributions by D\'ecamps, Gensbittel and Mariotti \cite{decamps2024mixed}, \cite{decamps2022mixed},  Christensen and Lindens\"jo \cite{christensen2024general}, and Christensen and Schultz \cite{christensen2024existence}.

While the full-information case has been extensively studied, 
the literature on partial and asymmetric information games is more recent and less developed. We now review the main contributions in this direction. The first work was due to Gr\"un \cite{Grun2013} who formulated a class of Markovian zero-sum Dynkin games with asymmetric information and diffusive underlying dynamics. The asymmetry arises from a discrete, finite-valued random variable that enters into the payoff functions along with the underlying diffusion. The player who observes this variable at the outset of the game is strictly more informed about the payoff processes than the opponent who only learns it when the game ends. Gr\"un studied the game with methods based on viscosity solutions of variational inequalities, inspired by the work of Cardaliaguet and Rainer \cite{cardaliaguet2009} on stochastic differential games with asymmetric information. In \cite{Grun2013}, the author showed the existence of the game's value and of an optimal strategy for the more informed player. Numerical methods for the solution of the associated variational problem were studied many years later by Ba{\v{n}}as et al.\ \cite{bavnas2025numerical}. Using similar methods as \cite{Grun2013}, Gensbittel and Gr\"un \cite{GenGrun2019} proved the existence of a saddle point in a game where each player observes only their private finite-state continuous-time Markov chain, while the payoff is a function of both Markov chains. 

In both \cite{Grun2013} and \cite{GenGrun2019}, a player's learning is restricted to inference from the opponent’s actions (or inaction). A different model is considered by De Angelis, Ekström, and Glover \cite{DEG2020}, who analysed a bi-valued regime influencing the drift of the diffusion observed by both players. The realisation of the regime is revealed to the informed player at the outset. The uninformed player must infer it from the diffusion path and, as in previously discussed papers, from the opponent's behaviour. The authors constructed a saddle point based on an educated guess followed by a verification theorem. Their approach is difficult to generalise. The variational problem obtained in \cite{DEG2020} is conceptually distinct from those in \cite{Grun2013, GenGrun2019} and closer in spirit to the theory that we develop here. 

As we already mentioned in the previous section, the recent paper \cite{de2022value} provided minimal sufficient conditions for the existence of a value and a saddle point in non-Markovian Dynkin games with partial and asymmetric information. The present paper continues that line of research by characterising the dynamics of equilibrium value processes and optimal strategies through necessary and sufficient conditions. 

Our focus here is exclusively on asymmetric-information games. We therefore do not review the literature on games with partial but \emph{symmetric} information (see, e.g., \cite{DGV2017}), where players are equally uninformed. Such models can typically be reformulated within the frameworks of \cite{lepeltier1984} or \cite{ekstrom2008} by means of Bayesian filtering.

Finally, we mention related but distinct works on nonzero-sum partial-information games. For instance, a recent work by Kwon and Palczewski \cite{kwon2023exit} considered a nonzero-sum stopping game modelling the exit problem from a duopoly, where both players' random exit values are private and unobservable by the opponent. The problem falls outside our zero-sum framework, but its structure shares similarities with the games that we analyse here. Although a Nash equilibrium is found in pure stopping times, mathematical tools used to represent the opponent's behaviour (due to the private information) are akin to randomised strategies. In this context, both players are equally uninformed and they learn solely from the actions of their opponents. 

\subsection{Structure of the paper}

The problem is formulated in Section \ref{sec:setting}. Sections \ref{sec:neccond} and \ref{sec:suffcond} contain necessary and sufficient conditions for a saddle point, respectively. In particular, Theorem \ref{thm:value} aggregates the players' equilibrium value processes into optional semi-martingales and it quantifies the size of their jumps. Proposition \ref{prop:link} establishes the connection between the two players' equilibrium value processes via an averaging procedure that can be interpreted as evaluating the game's {\em ex-ante} value (cf.\ also Corollary \ref{corr:link}). Proposition \ref{prop:support} and Corollary \ref{cor:support} establish further properties that characterise the equilibrium strategies and equilibrium value processes. Theorem \ref{thm:suff_saddle} is the main result of Section \ref{sec:suffcond}. It provides sufficient conditions for a saddle point in terms of a suitable quadruple of stochastic processes. That result is refined and reinterpreted in various ways in the subsequent statements of the section in order to provide a transparent formulation of all the key ingredients. In Section \ref{sec:examples} we apply our results to two particularly representative classes of problems. In Section \ref{sec:partial} we generalise Gr\"un's original problem from  \cite{Grun2013} to a non-Markovian setting. In Section \ref{sec:partdyn} we apply our methods to the class of games with partially observed dynamics motivated by \cite{DEG2020}. Finally, Section \ref{sec:PDE} provides an informal, yet detailed connection between our martingale methods and PDE characterisation of equilibrium value functions. This provides a theoretical justification for the variational problem used by \cite{DEG2020} in their guess-and-verify approach and a comparison with Gr\"un's variational approach in \cite{Grun2013}. The paper concludes with several technical results collected in Appendices.


\section{Setting and preliminaries}\label{sec:setting}

In this section we formally introduce the problem in a framework very similar to the one in \cite{de2022value}.  
Let $T\in(0,\infty)$. All definitions that follow remain valid in the case $T=\infty$ with a one-point compactification of $[0,\infty)$; then $[0, \infty]$ is mapped homeomorfically into $[0,1]$ and the game is considered on the latter, see also \cite{de2022value}.\footnote{When $T=\infty$, the space $\mathcal{S}$ in \cite[Section 5]{de2022value} should have been defined with a weighted Lebesgue measure such that the measure of $[0, \infty]$ is finite; all remaining arguments in the paper remain valid. Alternatively, one can map homemorfically $[0, \infty]$ into $[0,1]$ and note that all assumptions remain valid. It is the latter approach that we follow in this paper.} Let $(\Omega,\cF,\P)$ be a complete probability space, equipped with a right-continuous filtration $\F:=(\cF_t)_{t\in[0,T]}$ completed with all $\P$-null sets. We emphasise that $\cF_T\subsetneq \cF$ because we need $\cF$ to be sufficiently rich to accommodate so-called randomisation devices, according to Definition \ref{def:rst}. In the paper, all equations and inequalities between random variables are meant $\P$\as unless stated otherwise. Let $\cL_b(\P)$ be the Banach space of c\`adl\`ag $\cB([0,T])\times\cF$-measurable processes $X=(X_t)_{t\in[0,T]}$ with finite norm
\[
\|X\|_{\cL_b(\P)}:=\E\Big[\sup_{0\le t\le T}|X_t|\Big]<\infty.
\] 

There are two players in the game and they have access to the information contained in two different filtrations. In particular, we consider right-continuous filtrations $\F^1:=(\cF^1_t)_{t\in[0,T]}$ and $\F^2:=(\cF^2_t)_{t\in[0,T]}$, which are both contained in $\F$ (i.e., $\cF^1_t\vee\cF^2_t\subseteq\cF_t$ for all $t\in[0,T]$) and completed with $\P$-null sets. We assume that the first player (Player 1) has access to $\F^1$ and the second player (Player 2) has access to $\F^2$. Player 1 chooses a random time $\tau$ and Player 2 chooses a random time $\sigma$. At time $\tau\wedge\sigma$ Player 1 delivers a payoff $\cP(\tau,\sigma)$ to Player 2. Therefore, Player 1 (the $\tau$-player) is a minimiser while Player 2 (the $\sigma$-player) is a maximiser.
The payoff exchanged between players is of the following form
\[
\cP(\tau,\sigma)\coloneqq f_\tau \ind_{\{\tau<\sigma\}}+h_\tau \ind_{\{\tau=\sigma\}}+g_\sigma \ind_{\{\sigma<\tau\}},
\] 
where $f$, $g$, $h$ are stochastic processes satisfying assumptions specified below.
\begin{assumption}\label{ass:payoff}
The payoff processes in the game are $f,g,h \in \cL_b(\P)$ adapted to the filtration $\F$ and such that $f_t\ge h_t\ge g_t$ for all $t\in[0,T]$, $\P$-a.s. 
\end{assumption}
We notice that, in general, the payoff processes are not adapted to the players' filtrations and therefore they are not fully observable by either player. Players choose random times $\sigma$ and $\tau$ from the class of randomised stopping times with respect to their observed filtrations. In particular, we will say that Player 1 chooses $\tau$ as an $\F^1$-randomised stopping time and Player 2 chooses $\sigma$ as an $\F^2$-randomised stopping time, according to the next definition (that generalises slightly the one used in \cite{de2022value}). From now on, and unless otherwise specified, we denote $\G:=(\cG_t)_{t\in[0,T]}\subseteq \F$ a generic right-continuous filtration, completed with $\P$-null sets, which we use for the statement of some general definitions and results.

\begin{definition}\label{def:rst}
Given a $\G$-stopping time $\theta$, we introduce a class of processes
\begin{align*}
\cAcirc_\theta(\G):=\{\rho:&\,\text{$\rho$ is $\G$-adapted with $t\mapsto \rho_t(\omega)$ c\`adl\`ag and non-decreasing,}\\
&\text{$\rho_{\theta-}(\omega)=0$ and $\rho_T(\omega)=1$ for all $\omega\in\Omega$}\}.
\end{align*}
We say that $\eta$ is a $\G$-randomised stopping time {\em after} time $\theta$, with {\em generating process} $\rho$ and {\em randomisation device} $Z$, if
\[
\eta=\inf\{t\in[0,T]:\rho_t>Z\},
\]
where $\rho\in\cAcirc_\theta(\G)$ and $Z\sim U([0,1])$ is a random variable independent of $\cF_T$. 
 
Clearly $\P(\theta\le \eta\le T)=1$ and  the class of such randomised stopping times is denoted by $\cT^R_\theta(\G)$.
In the context of the game, the randomisation device for Player 1 is independent from the randomisation device for Player 2.
\end{definition}

For $\eta\in\cT^R_\theta(\G)$, we will often use the notation 
\begin{align}\label{eq:etau}
\eta(\rho,z)=\inf\{t\in[0,T]:\rho_t>z\},\quad\text{for $z\in[0,1]$},
\end{align}
and even write $\eta(z)$, when the underlying generating process is clear from the context.
Given a $\G$-stopping time $\theta\in[0,T]$ we also use the notation 
\[
\cT_\theta(\G):=\big\{\eta:\text{$\eta$ is $\G$-stopping time with $\eta\in[\theta,T]$, $\P$-a.s.}\big\},
\]
for the class of $\G$-stopping times {\em after} time $\theta$.
It is clear that $\cT_\theta(\G)\subset\cT^R_\theta(\G)$.

Given a pair $(\tau,\sigma)\in\cT^R_0(\F^1)\times \cT^R_0(\F^2)$ and a $\sigma$-algebra $\cH_0\subset\cF^1_0\cap\cF^2_0$, representing shared information at time zero\footnote{Since $\cF^i_0$ is not necessarily trivial, we can cover examples like $\cH_0=\cF^X_\lambda$, where $\F^X=(\cF^X_t)_{t\in[0,T]}$ is the filtration generated by a stochastic process $(X_t)_{t\in[0,T]}$ and $\lambda$ is an $\F^X$-stopping time.}, the players evaluate the associated expected payoff of the game by \[\E[\cP(\tau,\sigma)|\cH_0].\] The upper and lower value of the game {\em at time zero} read
\begin{align}\label{eq:V}
\overline V\coloneqq \essinf_{\tau\in\cT^R_0(\F^1)}\esssup_{\sigma\in\cT^R_0(\F^2)}\E[\cP(\tau,\sigma)|\cH_0]\quad\text{and}\quad \underline V\coloneqq\esssup_{\sigma\in\cT^R_0(\F^2)}\essinf_{\tau\in\cT^R_0(\F^1)}\E[\cP(\tau,\sigma)|\cH_0].
\end{align}
In \cite{de2022value} it is shown that under Assumption \ref{ass:payoff} and mild technical assumptions\footnote{Here we do not impose these assumptions because we are interested in a characterisation of optimal strategies rather than in their existence. However, for the reader's convenience we recall them: ${}^pf_t \le f_{t-}$ and ${}^pg_t \ge g_{t-}$ for $t \in (0, T)$, and ${}^pf_T = f_{T-}$ and ${}^pg_T = g_{T-}$, where ${}^p(\cdot)$ denotes the $\F$-previsible projection.} they indeed coincide, i.e., $\underline V=\overline V\eqqcolon V$, and $V$ is called the value of the game at time zero. It is also shown that both players have an optimal strategy, that is, the game admits a saddle point $(\tau_*,\sigma_*)\in\cT^R_0(\F^1)\times\cT^R_0(\F^2)$ such that
\[
\E[\cP(\tau_*,\sigma)|\cH_0]\le\E[\cP(\tau_*,\sigma_*)|\cH_0]\le \E[\cP(\tau,\sigma_*)|\cH_0] 
\]
for all other pairs $(\tau,\sigma)\in \cT^R_0(\F^1)\times\cT^R_0(\F^2)$. An important first step in the derivation of those results is the following observation (cf.\ \cite[Prop.\ 4.4]{de2022value}): for a given pair $(\tau,\sigma)\in\cT^R_0(\F^1)\times\cT^R_0(\F^2)$ with generating processes $(\xi,\zeta)\in\cA^\circ_0(\F^1)\times\cA^\circ_0(\F^2)$, it holds
\begin{align}\label{eq:payoffxizeta}
\E\big[\cP(\tau,\sigma)\big|\cH_0\big]=\E\Big[\int_{[0,T)}f_t(1-\zeta_t)\ud \xi_t+\int_{[0,T)}g_t(1-\xi_t)\ud \zeta_t+\sum_{t\in[0,T]}h_t\Delta\zeta_t\Delta\xi_t\Big|\cH_0\Big].
\end{align}
The argument is performed pathwise and it uses the well-known change of variable of integration from \cite[Prop.\ 0.4.9]{revuzyor}.
We will use this formula and suitable variants thereof several times throughout the paper.
\begin{notation}\label{not:equil}
From now on we denote $(\xi^*,\zeta^*)\in\cAcirc_{0}(\F^1)\times \cAcirc_{0}(\F^2)$ the pair of generating processes of a saddle point $(\tau_*,\sigma_*)\in \cT^R_0(\F^1)\times\cT^R_0(\F^2)$ for the game starting at time zero as in \eqref{eq:payoffxizeta}. 
\end{notation}

A dynamic formulation of the game requires a definition of upper and lower value at any random time $\theta$, much in the spirit of the formulation provided by \cite{lepeltier1984} for stopping games with full information. However, such approach also requires a definition of {\em conditional} expected payoff at time $\theta$ and, due to the different filtrations available to the two players, there is no unique way to choose the conditioning $\sigma$-algebra. Likewise, it is not clear what random times $\theta$ are to be considered. One possible approach would be to work with the common filtration shared by both players, 
$\F^{1,2}=(\cF^{1,2}_t)_{t\in[0,T]}$ with $\cF^{1,2}_t\coloneqq \cF^1_t\cap\cF^2_t$. Then, for $\theta\in\cT^R_0(\F^{1,2})$ we could consider the conditional upper and lower values
\begin{align}\label{eq:Vtheta}
\overline{V}(\theta)\coloneqq
\essinf_{\tau\in\cT^R_\theta(\F^1)}\esssup_{\sigma\in\cT^R_\theta(\F^2)}\E\big[\cP(\tau,\sigma)\big|\cF^{1,2}_\theta\big]
\quad\text{and}\quad \underline{V}(\theta)\coloneqq\esssup_{\sigma\in\cT^R_\theta(\F^2)}\essinf_{\tau\in\cT^R_\theta(\F^1)}\E\big[\cP(\tau,\sigma)|\cF^{1,2}_\theta\big].
\end{align}
Existence of a value $V(\theta)=\overline{V}(\theta)=\underline{V}(\theta)$ and of a saddle point can be deduced by the argument of proof of \cite[Thm.\ 2.6]{de2022value}. Methods as those that we will illustrate in detail in later sections (cf.\ Corollary \ref{corr:link}) allow us to aggregate the family $\{V(\theta),\,\theta\in\cT^R_0(\F^{1,2})\}$ into a $\F^{1,2}$-adapted stochastic value process $(V_t)_{t\in[0,T]}$. However, this formulation is subject to a number of drawbacks. We mention two important ones: (i) since $\F^{1,2}\subset\F^i$, $i=1,2$, the conditioning is not properly keeping track of players' updates of their information over time (this becomes particularly apparent when $\cF^1_t\cap\cF^2_t=\{\varnothing,\Omega\}$ for all $t\in[0,T]$); (ii) the dynamics of the value process $(V_t)_{t\in[0,T]}$ does not, in general, reveal optimality conditions for the saddle point exactly because of the insufficient information content of the filtration $\F^{1,2}$. For these reasons we take a different approach, considering each player's expected payoff, for a fixed strategy of their opponent. By doing that, we formally recast our zero-sum game as a nonzero-sum one. We proceed to illustrate the details in the next section.

\subsection{Players' subjective views and equilibrium values as families of random variables}
As the game proceeds, players acquire more information via their filtration and, crucially, via actions of their opponent (or rather the lack of action, in the sense that they learn for as long as their opponent has not stopped). This naturally leads to a notion of players' dynamic {\em subjective view} of the game. In order to capture that idea we need to work under dynamically changing probability measures, which motivates the next definition. 
\begin{definition}
Given a $\sigma$-algebra $\cG \subseteq \cF$, let
\[
\cR(\cG)\coloneqq\{\Pi:\text{$\Pi$ is $\cF$-measurable with $\Pi\ge 0$ and $\E[\Pi|\cG]=1$\}.}
\]
For $\Pi\in\cR(\cG)$ we denote by $\P^\Pi$ the probability measure defined by
\[
\P^\Pi(A)=\E[ 1_A\Pi],\quad\text{for $A\in\cF$}.
\]
\end{definition}
The condition $\E[\Pi|\cG]=1$ satisfied by $\Pi\in\cR(\cG)$ is stronger than the usual condition $\E[\Pi]=1$, needed for the change of measure. 
Indeed we observe that for $\Pi\in\cR(\cG)$
\[
\P^\Pi(A|\cG)=\E^\Pi[1_A|\cG]=\frac{\E[\Pi 1_A|\cG]}{\E[\Pi|\cG]}=\E[\Pi 1_A|\cG],
\]
and therefore we have the interpretation
\begin{align}\label{eq:condRNd}
\Pi=\frac{\ud \P^\Pi(\,\cdot\,|\cG)}{\ud \P(\,\cdot\,|\cG)}.
\end{align}

Recall the notation for an equilibrium pair $(\xi^*,\zeta^*)$ from Notation \ref{not:equil}. For our purposes, we are particularly interested in the dynamic changes of measure induced by the families of random variables 
\begin{align}\label{eq:Pi}
\begin{split}
&\Pi^{*,1}_\theta\coloneqq\frac{1-\zeta^*_{\theta-}}{\E[1-\zeta^*_{\theta-}|\cF^1_\theta]},\quad\widehat\Pi^{*,1}_\theta\coloneqq\frac{1-\zeta^*_{\theta-}}{\E[1-\zeta^*_{\theta-}|\cF^1_{\theta-}]},\quad\theta\in\cT_0(\F^1),\\
&\Pi^{*,2}_\gamma\coloneqq\frac{1-\xi^*_{\gamma-}}{\E[1-\xi^*_{\gamma-}|\cF^2_\gamma]},\quad\widehat\Pi^{*,2}_\gamma\coloneqq\frac{1-\xi^*_{\gamma-}}{\E[1-\xi^*_{\gamma-}|\cF^2_{\gamma-}]},\quad\gamma\in\cT_0(\F^2),
\end{split}
\end{align}
where the ratio is defined to be $1$ whenever the denominator is $0$. We emphasise that the random variables in \eqref{eq:Pi} are defined $\P$-a.s.\ as the conditional expectations are defined $\P$-a.s.\ and the set of zero probability depends on $\theta$ and $\gamma$, respectively.  We notice that $\Pi^{*,1}_\theta\in\cR(\cF^1_\theta)$ and $\Pi^{*,2}_\gamma\in\cR(\cF^2_\gamma)$ whereas $\widehat\Pi^{*,1}_\theta\in\cR(\cF^1_{\theta-})$ and $\widehat\Pi^{*,2}_\gamma\in\cR(\cF^2_{\gamma-})$, by construction. Moreover, when the filtrations $\F^1$ and $\F^2$ are continuous we have $(\widehat\Pi^{*,1}_\theta,\widehat\Pi^{*,2}_\gamma)=(\Pi^{*,1}_\theta,\Pi^{*,2}_\gamma)$.

To link random variables $\Pi^{*, 1}_\theta$ and $\Pi^{*,2}_\gamma$ to players' views on the remainder of the game, we notice that, for $\theta \in \cT_0(\F^1)$, 
\[
\P(\sigma_* \ge \theta | \cF_T) = \P(\zeta^*_{\theta-} \le Z_\zeta | \cF_T) = 1 - \zeta^*_{\theta-}, 
\]
where $Z_\zeta$ is the randomisation device of Player 2. Upon conditioning on $\cF^1_\theta$, we further have $\P(\sigma_* \ge \theta | \cF^1_\theta) = \E[1 - \zeta^*_{\theta-} | \cF^1_\theta]$. This offers a new perspective on $\Pi^{*,1}_\theta$:
\begin{equation}\label{eqn:belief2}
\Pi^{*,1}_\theta = \frac{\P(\sigma_* \ge \theta | \cF_T)}{\P(\sigma_* \ge \theta | \cF^1_\theta)}
= \frac{\E[\indd{\sigma_* \ge \theta} | \cF_T]}{\E[\indd{\sigma_* \ge \theta} | \cF^1_\theta]}.
\end{equation}
The expectation of an integrable $\cF_T$-measurable random variable $X$ given the information at time $\theta$ and conditional on Player 2 not having terminated the game prior to $\theta$ is given by (see \cite[Exercise~10, p.~94]{kallenberg2002})
\begin{equation}\label{eqn:belief1}
\E[X | \cF^1_\theta, \sigma_* \ge \theta] = \frac{\E[X \indd{\sigma_* \ge \theta} | \cF^1_\theta]}{\P(\sigma_* \ge \theta|\cF^1_\theta)}, \qquad \text{on the event $\{\P(\sigma_* \ge \theta|\cF^1_\theta)>0\}$.}
\end{equation}
Here, by the conditional expectation $\E[X | \cF^1_\theta, \sigma_* \ge \theta]$ we mean the conditional expectation given $\cF^1_\theta \vee \sigma(\{\sigma_* \ge \theta\})$. Definition \eqref{eqn:belief2} and the tower property of conditional expectation yield
\[
\E[X | \cF^1_\theta, \sigma_* \ge \theta]
=
\E \Big[ \frac{\E[\indd{\sigma_* \ge \theta}|\cF_T]}{\P(\sigma_* \ge \theta|\cF^1_\theta)} X  \Big| \cF^1_\theta \Big]
=
\E[ \Pi^{*, 1}_\theta X | \cF^1_\theta ] = \E^{\Pi^{*, 1}_\theta}[ X | \cF^1_\theta ].
\]
Hence, the measure $\Pi^{*, 1}_\theta$ encapsulates conditioning on the event $\{\sigma_* \ge \theta\}$. The fact that a change of measure describes the conditioning is natural as Player 1 reassesses her perception of the world given that the opponent has not acted yet. This motivates the following terminology. 

\begin{definition}
The family of random variables $\{\Pi^{*,1}_\theta:\ \theta \in \cT_0(\F^1)\}$ is called \emph{Player 1's subjective view}, while the family $\{\Pi^{*,2}_\gamma:\ \gamma \in \cT_0(\F^2)\}$ is called \emph{Player 2's subjective view}. 
\end{definition}
It is worth noting that processes $(\Pi^{*,i}_t)_{t\in[0,T]}$, $i=1,2$, are not adapted to players' filtrations $\F^i$, $i=1,2$. However, in the concrete examples of Section \ref{sec:examples} we show how the players’ subjective views can be linked to {\em belief processes} which are adapted to the players’ filtrations.

It will emerge that assigning one to the ratios in \eqref{eq:Pi} when the denominators are zero is convenient for the interpretation of $\Pi^{*,1}_\theta$ and $\Pi^{*,2}_\gamma$ as changes of probability measure. The freedom to choose this convention follows from the implication: for $\P$-a.e. $\omega \in \Omega$,
\begin{align}\label{eq:impl}
\E[1-\zeta^*_{\theta-}|\cF^1_\theta](\omega)=0 \implies (1-\zeta^*_{\theta-})(\omega)=0,
\end{align}
and analogously for $\xi^*$. For the proof, denote $A = \{\omega \in \Omega:\, \E[1-\zeta^*_{\theta-}|\cF^1_\theta](\omega)=0 \}$. We have $A \in \cF^1_\theta$ and
\[
\E \big[ 1_A (1-\zeta^*_{\theta-}) \big] = \E \big[ 1_A \E[1-\zeta^*_{\theta-}|\cF^1_\theta]\big] = 0,
\]
which, by the non-negativity of $1-\zeta^*_{\theta-}$ implies that $1-\zeta^*_{\theta-}(\omega) = 0$ for $\P$-a.e. $\omega \in A$.

\begin{remark}
If regular conditional probabilities $\P(\,\cdot\,|\cF^1_\theta)$ and $\P(\,\cdot\,|\cF^1_{\theta-})$ exist, then the ratios in the first line of \eqref{eq:Pi} can be defined $\omega$ by $\omega$ with the convention that $0/0 = 1$. An analogous statement can be made for the second line of \eqref{eq:Pi}. However, we want to avoid using regular conditional probabilities as their existence requires additional assumptions on the probability space and filtrations.
\end{remark}

We now focus on the concept of the {\em player's equilibrium value}, representing the player's perception of the future play in the game, 
based on the information gathered until a stopping time. For this, we need notions of \emph{dynamic payoff} associated to a filtration $\G$ and \emph{truncated controls}.

\begin{definition}
Given $\theta\in\cT_0(\G)$, $\Pi\in\cR(\cG_\theta)$ and a pair $(\tau,\sigma)\in\cT^R_{\theta}(\F)\times \cT^R_{\theta}(\F)$, the \emph{dynamic payoff} is defined as
\[
J^\Pi(\tau,\sigma|\cG_\theta)\coloneqq\E^\Pi[\cP(\tau,\sigma)|\cG_\theta]=\E[\Pi\cP(\tau,\sigma)|\cG_\theta],
\] 
where the final equality holds because $\Pi\in\cR(\cG_\theta)$. 
 \end{definition}
As in \cite[Sec.\ 4]{de2022value}, using the definition of $\cT^R_{\theta}(\F)$ we can derive the $\cG_\theta$-conditional analogue of \eqref{eq:payoffxizeta}. That is, it is not hard to verify that
\begin{align}\label{eq:J}
\begin{aligned}
&J^\Pi(\tau,\sigma|\cG_\theta)=J^\Pi(\xi,\zeta|\cG_\theta)\\
&=\E^\Pi\Big[\int_{[\theta,T)}f_t(1-\zeta_t)\ud \xi_t+\int_{[\theta,T)}g_t(1-\xi_t)\ud \zeta_t+\sum_{t\in[\theta,T]}h_t\Delta\zeta_t\Delta\xi_t\Big|\cG_\theta \Big]\\
&=\E\Big[\Pi\Big(\int_{[\theta,T)}f_t(1-\zeta_t)\ud \xi_t+\int_{[\theta,T)}g_t(1-\xi_t)\ud \zeta_t+\sum_{t\in[\theta,T]}h_t\Delta\zeta_t\Delta\xi_t\Big)\Big|\cG_\theta \Big],
\end{aligned}
\end{align}
where $(\xi,\zeta)\in\cAcirc_\theta(\F)\times \cAcirc_\theta(\F)$ are the generating processes of the pair $(\tau,\sigma)$. This notation of generating processes will be upheld throughout the paper. With a slight abuse of notation, we will also write
\[
\cP(\xi, \zeta) \coloneqq \int_{[0,T)}f_t(1-\zeta_t)\ud \xi_t+\int_{[0,T)}g_t(1-\xi_t)\ud \zeta_t+\sum_{t\in[0,T]}h_t\Delta\zeta_t\Delta\xi_t.
\]
Using this notation, \eqref{eq:J} reads $J^\Pi(\xi,\zeta|\cG_\theta) = \E[\Pi\, \cP(\xi, \zeta) | \cG_\theta]$ as the integrals and the sum over $[0, \theta)$ are zero because $(\xi, \zeta) \in \cAcirc_\theta(\F^1) \times \cAcirc_\theta(\F^2)$. When $\theta=0$ and $\Pi = 1$, with $\cG_0\supseteq \cH_0$, we recover the expected payoff of the game $\E[J^\Pi(\xi, \zeta|\cG_0)|\cH_0] = \E[\cP(\xi, \zeta)|\cH_0]$.

\begin{definition}\label{def:trunc}
Given a filtration $\G\subset\F$, a stopping time $\eta\in\cT_0(\G)$ and a generating pair $(\xi,\zeta)\in\cAcirc_0(\F)\times \cAcirc_0(\F)$ of randomised stopping times, we call {\em truncated controls} the following processes:
\begin{align}\label{eq:trunc}
\xi^{\eta}_t\coloneqq\frac{\xi_t-\xi_{\eta-}}{1-\xi_{\eta-}}\ind_{\{t\ge \eta\}}\quad\text{and}\quad
\zeta^{\eta}_t\coloneqq\frac{\zeta_t-\zeta_{\eta-}}{1-\zeta_{\eta-}}\ind_{\{t\ge \eta\}},\quad\text{for $t\in[0,T]$},
\end{align}
with the convention $0/0=1$. By construction $\xi^{\eta}_{\eta-} = \zeta^{\eta}_{\eta-} = 0$ and $\xi^{\eta}_{T} = \zeta^{\eta}_{T} = 1$.
\end{definition}

The dynamics of \emph{equilibrium values}\footnote{In the theory of non-zero sum games, players' equilibrium values are often called \emph{equilibrium payoffs},  \emph{expected payoffs} or \emph{continuation values}. However, we decided to use the terms `equilibrium value' or `player's value' to emphasise parallels with the full information theory of Dynkin games and to minimise confusion with payoff processes $f, g, h$ defining the game.} (often referred to, in short, as \emph{players' values}) for the two players are modelled via two families of random variables
\begin{align}\label{eq:fV}
\big\{V^{*,1}(\theta),\theta\in\cT_0(\F^1)\big\}\quad\text{and}\quad \big\{V^{*,2}(\gamma),\gamma\in\cT_0(\F^2)\big\},
\end{align}
defined by 
\begin{align}\label{eq:V1V2}
V^{*,1}(\theta)\coloneqq\essinf_{\xi\in\cA^\circ_\theta(\F^1)}J^{\Pi^{*,1}_\theta}\big(\xi,\zeta^{*;\theta}\big|\cF^1_\theta\big)\quad\text{ and }\quad V^{*,2}(\gamma)\coloneqq\esssup_{\zeta\in\cA^\circ_\gamma(\F^2)}J^{\Pi^{*,2}_\gamma}\big(\xi^{*;\gamma},\zeta\big|\cF^2_\gamma\big).
\end{align}

We remark that the value of $V^{*,1}(\theta)$ on the event $\{\E[\zeta^*_{\theta-}|\cF^1_\theta] = 1\}\subseteq\{\zeta^*_{\theta-}=1\}$ (cf.\ \eqref{eq:impl}) is trivially equal to $\E[g_\theta| \cF^1_{\theta}]$. Indeed, on this event, $\Pi^{*, 1}_\theta=1$, $\zeta^{*;\theta}_{\theta-}=0$ and $\zeta^{*;\theta}_t=1$ for $t\in[\theta,T]$; thus, $J^{\Pi^{*,1}}_\theta(\xi,\zeta^{*;\theta}|\cF^1_\theta)=\E[g_\theta (1-\xi_{\theta})+h_\theta \Delta\xi_{\theta}| \cF^1_{\theta}]$ for any $\xi\in\cAcirc_\theta(\F^1)$ and, since $g_\theta\le h_\theta$, it is optimal to choose $\xi_\theta=0$. 
Same comments apply to $V^{*,2}(\gamma)=\E[f_\gamma|\cF^2_\gamma]$ on the event $\{\E[\xi^*_{\gamma-}|\cF^2_{\gamma}]=1\}$. 
We note that the equilibrium values $V^{*,1}(\theta)$ and $V^{*,2}(\gamma)$ on the events $\{\E[\zeta^*_{\theta-}|\cF^1_\theta] = 1\}$ and $\{\E[\xi^*_{\gamma-}|\cF^2_{\gamma}]=1\}$, respectively, do not play any significant role in the context of the game: 
in most expressions $V^{*,1}(\theta)$ and $V^{*,2}(\gamma)$ are preceded by $\E[1-\zeta^*_{\theta-}|\cF^1_\theta]$ and $\E[1-\xi^*_{\gamma-}|\cF^2_{\gamma}]$, respectively (see, e.g., Theorem \ref{thm:value}).

The interpretation of $V^{*,1}(\theta)$ (and analogously for $V^{*,2}(\gamma)$) is as follows: at time zero the game starts and the players pick an optimal pair $(\xi^*,\zeta^*)$; at time $\theta$, if the game has not ended, Player 1 calculates $V^{*,1}(\theta)$ as the smallest payoff they can attain with a best response to the remainder of strategy $\zeta^*$ on the interval $[\theta,T]$. We are going to show in Proposition \ref{thm:aggr2} that the truncated control $\xi^{*;\theta}$ attains the infimum in $V^{*,1}(\theta)$ and, analogously, $\zeta^{*;\gamma}$ attains the infimum in $V^{*,2}(\gamma)$. Thus, for any $\beta\in\cT^R(\F^{1,2})$ the pair of truncated controls $(\xi^{*;\beta},\zeta^{*;\beta})$ is a saddle point for the game started at $\beta$ with payoff $\E[\cP(\xi,\zeta)|\cF^{1,2}_\beta]$ for $(\xi,\zeta)\in\cA^\circ_\beta(\F^1)\times\cA^\circ_\beta(\F^2)$ (cf.\ Proposition \ref{prop:link} and Corollary \ref{corr:link} for further details). 

\begin{remark}
In the theory of zero-sum games with full information, values $V^{*,1}(\beta)$ and $V^{*,2}(\beta)$ coincide for any stopping time $\beta$ with respect to the common filtration. They are then termed the value of the game and play a pivotal role in determining players' optimal strategies. Here, the values of players are distinct for multiple reasons: (i) they are defined for different familites of random times, (ii) they condition on the information available to the player at that time, and (iii) they include the updated perception of the future probabilities via $\Pi^{*, j}$ arising from learning from the opponent's inaction.
\end{remark}

In view of the above remark, it should come as no surprise that the families of random variables \eqref{eq:fV} are the main object of interest throughout the paper. We will establish a link between the two via the so-called {\em ex-ante} value of the game in Corollary \ref{corr:link}.

\begin{convention}
For the ease of exposition, in the next sections we assume $\cH_0=\{\Omega,\varnothing\}$. This comes with no loss of generality because all results continue to hold with generic $\cH_0\subseteq\cF^1_0\cap\cF^2_0$ upon replacing everywhere the unconditional expectation $\E[\cdot]$ with conditional one $\E[\cdot|\cH_0]$. 
\end{convention}

\subsection{Roadmap}\label{subsec:roadmap}
To facilitate reading of the upcoming technical content, we provide intuitions and a quick sketch of results to come. Section \ref{sec:neccond} is devoted to the derivation of necessary conditions, i.e., of conditions that must be satisfied by players' equilibrium values and their optimal strategies. Unlike the classical theory of Dynkin games, where the equilibrium value can be defined without knowing players' optimal strategies, here those strategies play a pivotal role: they reveal additional information that shapes players' subjective views about the remainder of the game and the probability measures ($\P^{\Pi^{*,1}_\theta}$ and $\P^{\Pi^{*, 2}_\gamma}$) under which they assess future payoffs. A subset of those necessary conditions is shown in Section \ref{sec:suffcond} to be sufficient for a saddle point. Hence, we will concentrate on explaining our ideas guiding developments in Section \ref{sec:neccond}.

We start by recalling the classical theory of Dynkin games with full information. Denoting by $V_t$ the value process and by $(\hat\tau, \hat \sigma)$ a saddle point, they satisfy the super- and sub-martingale conditions:
\begin{align*}
t \mapsto V_{t \wedge \hat\tau} \text{ is a supermartingale and}\ t\mapsto V_{t \wedge \hat\sigma} \text{ is a submartingale.}
\end{align*}

In our framework, due to involvement of randomised strategies and learning, these martingale conditions take a much more complicated form. For $\theta\in\cT_0(\F^1)$ and $\gamma\in\cT_0(\F^2)$, let
\begin{align*}
M^0(\theta)&= \E \Big[ \int_{[0,\theta)}\!\!g_t \ud \zeta^*_t\! \Big| \cF^1_\theta \Big] + \E[1\!-\!\zeta^*_{\theta-}|\cF^1_\theta]V^{*,1}(\theta),\\
N^0(\gamma)&= \E \Big[\! \int_{[0,\gamma)}\!\! f_t \ud \xi^*_t  \Big| \cF^2_\gamma \Big] + \E[1\!-\!\xi^*_{\gamma-}|\cF^2_\gamma] V^{*,2}(\gamma);
\end{align*}
such familities of random variables will be defined in greater generality in Section \ref{subsec:aux-mart}. We then show that the family $\{M^0(\theta):\ \theta \in \cT_0(\F^1)\}$ can be aggregared into a \cadlag $\F^1$-submartingale $(M^0_t)_{t\in[0,T]}$ and the family $\{N^0(\gamma):\ \gamma \in \cT_0(\F^2)\}$ can be aggregared into a \cadlag $\F^2$-supermartingale $(N^0_t)_{t\in[0,T]}$. An analogue of the martingale condition for $t\mapsto V_{t\wedge\hat\tau\wedge\hat\sigma}$ requires new families of random variables $\{M^*(\theta):\ \theta \in \cT_0(\F^1)\}$ and $\{N^*(\gamma):\ \gamma \in \cT_0(\F^2)\}$ involving both $\xi^*$ and $\zeta^*$, which can be aggregated into $\F^1$ and $\F^2$-martingales, respectively.

Differently from the full information setting, where the aggregation of the value process is done directly, in our setting players' equilibrium values are aggregated indirectly through the families $\{M^0(\theta):\theta\in\cT_0(\F^1)\}$ and $\{N^0(\gamma):\gamma\in\cT_0(\F^2)\}$. To be more precise, the families 
\[
\{\E[1\!-\!\zeta^*_{\theta-}|\cF^1_\theta]V^{*,1}(\theta):\theta \in \cT_0(\F^1)\}\quad\text{and}\quad\{\E[1\!-\!\xi^*_{\gamma-}|\cF^2_\gamma] V^{*,2}(\gamma):\gamma \in \cT_0(\F^2)\}
\] 
are aggregated into optional semi-martingales. The prefactor in front of the equilibrium value is needed because a player's value is only defined as long as the game has not been terminated by the opponent; it should be stressed that its role is purely technical: it is non-zero if and only if there is a possibility of the opponent to be still in the game given the information available to the player; however, players observe when the game finishes so the game will never be active when the prefactor is $0$, see also \eqref{eq:impl}. 

Our final major aim is to identify the support of optimal strategies of players, understood as the set of times when players may stop optimally. In the classical theory, that is the coincidence set of the value process with respective payoffs. We show in Subsection \ref{subsec:structure} that a close analogue of this behaviour holds in the present much more complex setting. To simplify this informal presentation, assume that players do not act simultaneously, i.e., $\Delta \xi^*_t \Delta \zeta^*_t = 0$ for all $t \in [0, T]$. Define
\[
Y^1_t \coloneqq \hat V^{*,1}_t -\optional{\big(}f_\cdot(1-\zeta^*_\cdot)\big)_t^{\F^1}\quad\text{and}\quad Y^2_t\coloneqq\hat V^{*,2}_t- \optional{\big(}g_\cdot(1-\xi^*_\cdot)\big)^{\F^2}_t,
\]
where $\optional{(}\cdot)^\G_t$ denotes the optional projection with respect to the filtration $\G$. We take the perspective of Player 1. The quantity $\optional{\big(}f_\cdot(1-\zeta^*_\cdot)\big)_t^{\F^1}$ is the perception of the payoff if Player 1 stops the game at time $t$ -- the optional projection is needed as neither the process $f$ nor the opponent's strategy $\zeta^*$ have to be $\F^1$-adapted. In analogy to the full information game, we will show that $Y^1_t \le 0$ for any $t \in [0, T]$, i.e., the equilibrium value of Player 1 is dominated by the optional projection of the payoff -- a natural requirement for a minimiser. Furthermore, we will show that Player 1 acts only when $Y^1_t = 0$, which is formally written as $\int_{[0,T]} Y^1_t \ud \xi^*_t = 0$. Analogous conditions for Player 2 are $Y^2_t \ge 0$ and $\int_{[0,T]} Y^2_t \ud\zeta^*_t = 0$.

The above conditions allow, in specific settings, to formulate variational inequalities for equilibrium value functions and to postulate players' action sets -- the sets on which players are allowed to increase their generating processes. Such examples are discussed in Section \ref{sec:examples}.

\section{Necessary conditions for a saddle point}\label{sec:neccond}

In this section we obtain properties of the equilibrium values of the two players and of their optimal strategies. The analysis is performed in a dynamic setting. 
We will later show in Section \ref{sec:suffcond} that such properties are indeed sufficient to characterise any equilibrium in the game. 
Since the section is quite rich of technical materials it is worth summarising here the results that, taken together, provide the desired necessary conditions. In particular, we are going to prove:
\begin{itemize}
\item Aggregation into optional semi-martingales of players' equilibrium values (Theorem \ref{thm:value}), 
\item Martingale characterisation of players' equilibrium values (Propositions \ref{thm:aggr1} and \ref{thm:aggr2}, and Corollary \ref{cor:M_zero_mart}), 
\item Link between players' equilibrium values and {\em ex-ante} value of the game (Proposition \ref{prop:link} and Corollary \ref{corr:link}), 
\item Properties of equilibrium strategies and their link to the dynamics of equilibrium values (Proposition \ref{prop:support} and Corollary \ref{cor:support}).
\end{itemize} 

\subsection{Aggregation of the equilibrium dynamics}\label{subsec:aggregation}
The first step in our analysis is to aggregate the families \eqref{eq:fV} into stochastic processes. For the convenience of the reader, we collect aggregation results from the general theory of stochastic processes in Appendix \ref{app:aggr}.
We start by showing upward/downward-directed properties (cf.\ Appendix \ref{subsec:updownd}) of the payoffs in \eqref{eq:V1V2} which allow replacing essential suprema/infima with monotone limits. 

\begin{lemma}\label{lem:ud}
Given $\theta\in\cT_0(\F^1)$, 
the family $\{J^{\Pi^{*,1}_\theta}(\xi,\zeta^{*;\theta}|\cF^1_\theta),\ \xi\in\cAcirc_\theta(\F^1)\}$ is downward-directed. Therefore, there is a sequence $(\xi^n)_{n\in\N}\subset\cAcirc_\theta(\F^1)$ such that
\[
V^{*,1}(\theta)=\lim_{n\to\infty}J^{\Pi^{*,1}_\theta}(\xi^n,\zeta^{*;\theta}|\cF^1_\theta),
\]
where the limit is monotone from above.

Analogously, given $\gamma\in\cT_0(\F^2)$, 
the family $\{J^{\Pi^{*,2}_\gamma}(\xi^{*;\gamma},\zeta|\cF^2_\gamma),\ \zeta\in\cAcirc_\gamma(\F^2)\}$ is upward-directed. Therefore, there is a sequence $(\zeta^n)_{n\in\N}\subset\cAcirc_\gamma(\F^2)$ such that
\[
V^{*,2}(\gamma)=\lim_{n\to\infty}J^{\Pi^{*,2}_\gamma}(\xi^{*;\gamma},\zeta^n|\cF^2_\gamma),
\]
where the limit is monotone from below.
\end{lemma}

The proof is completely standard and it is provided in Appendix \ref{app:lemud}. The next lemma states that each player's equilibrium value is attained in pure strategies once the other player's strategy is fixed. Notice, however, that equilibria in pure strategies do not exist in the generality of our setting (see various counterexamples in, e.g., \cite[Sec.\ 6]{de2022value}).
 
\begin{lemma}\label{lem:pure}
For $(\theta,\gamma)\in\cT_0(\F^1)\times\cT_0(\F^2)$, let $\tau_*^\gamma$ and $\sigma_*^\theta$ be the randomised stopping times generated by the truncated controls $\xi^{*;\gamma}$ and $\zeta^{*;\theta}$, respectively. Then
\begin{align}\label{eq:pure}
V^{*,1}(\theta)=\essinf_{\tau\in\cT_\theta(\F^1)}J^{\Pi^{*,1}_\theta}(\tau,\sigma_*^\theta|\cF^1_\theta)
\quad\text{and}\quad 
V^{*,2}(\gamma)=\esssup_{\sigma\in\cT_\gamma(\F^2)}J^{\Pi^{*,2}_\gamma}(\tau_\gamma^*,\sigma|\cF^2_\gamma).
\end{align}
\end{lemma}
\begin{proof}
We only prove the claim for $V^{*,1}(\theta)$ as the one for $V^{*,2}(\gamma)$ can be proven analogously.

By Lemma \ref{lem:ud}, there is a sequence $(\xi_n) \subset \cAcirc_\theta(\F_1)$ such that
\[
V^{*,1}(\theta)=\lim_{n\to\infty}J^{\Pi^{*,1}_\theta}(\xi^n,\zeta^{*;\theta}|\cF^1_\theta),\quad\P-a.s.
\]
We take expectation on both sides and apply the monotone convergence theorem to obtain
\[
\E\big[V^{*,1}(\theta)\big]=\lim_{n\to\infty}\E\big[ \Pi^{*,1}_\theta \cP(\xi^n, \zeta^{*;\theta}) \big].
\]
Combining the above equality with the following upper bound
\[
\E\big[V^{*,1}(\theta)\big] = \E\big[\essinf_{\xi\in\cAcirc_\theta(\F^1)}J^{\Pi^{*,1}_\theta}(\xi,\zeta^{*;\theta}|\cF^1_\theta) \big]
\le \inf_{\xi\in\cAcirc_\theta(\F^1)} \E\big[J^{\Pi^{*,1}_\theta}(\xi,\zeta^{*;\theta}|\cF^1_\theta) \big]
= \inf_{\xi\in\cAcirc_\theta(\F^1)} \E\big[\Pi^{*,1}_\theta \cP(\xi,\zeta^{*;\theta}) \big]
\]
we obtain
\begin{align}\label{eqn:z1}
\begin{aligned}
\E\big[V^{*,1}(\theta)\big] &= \E\big[\essinf_{\xi\in\cAcirc_\theta(\F^1)}J^{\Pi^{*,1}_\theta}(\xi,\zeta^{*;\theta}|\cF^1_\theta) \big]\\
&= \inf_{\xi\in\cAcirc_\theta(\F^1)} \E\big[\Pi^{*,1}_\theta \cP(\xi,\zeta^{*;\theta}) \big] = \inf_{\tau\in\cT^R_\theta(\F^1)} \E\big[\Pi^{*,1}_\theta \cP(\tau,\sigma_*^\theta) \big],
\end{aligned}
\end{align}
where the last equality is due to the relationship between randomised stopping times and their generating processes. For $\tau\in\cT^R_\theta(\cF^1_\theta)$, recalling the notation $\tau(z)=\tau(\xi,z)$, $z\in[0,1]$, from \eqref{eq:etau}, we have
\begin{align*}
\inf_{\tau\in\cT^R_\theta(\F^1)} \E\big[\Pi^{*,1}_\theta \cP(\tau,\sigma_*^\theta )\big]
&=
\inf_{\tau\in\cT^R_\theta(\F^1)} \int_0^1 \E\big[\Pi^{*,1}_\theta \cP(\tau(z),\sigma_*^\theta )\big] dz\\
&\ge
\inf_{\tau\in\cT^R_\theta(\F^1)} \int_0^1 \inf_{\bar \tau \in \cT_\theta(\F_1)} \E\big[\Pi^{*,1}_\theta \cP(\bar\tau,\sigma_*^\theta) \big] dz\\
&=
\inf_{\tau\in\cT_\theta(\F^1)} \E\big[\Pi^{*,1}_\theta \cP(\tau,\sigma_*^\theta) \big]
= \inf_{\tau\in\cT_\theta(\F^1)} \E\big[J^{\Pi^{*,1}_\theta}(\tau,\sigma_*^\theta|\cF^1_\theta) \big],
\end{align*}
where in the first inequality we integrate with respect to the distribution of Player 1's randomisation device and the inequality holds because $\tau(z)\in\cT_\theta(\F^1)$ for each $z\in[0,1]$. We insert this estimate into the equality \eqref{eqn:z1} to notice
\[
\E\big[\essinf_{\tau\in\cT^R_\theta(\F^1)}J^{\Pi^{*,1}_\theta}(\tau,\sigma_*^\theta|\cF^1_\theta) \big]
\ge
\inf_{\tau\in\cT_\theta(\F^1)} \E\big[J^{\Pi^{*,1}_\theta}(\tau,\sigma_*^\theta|\cF^1_\theta) \big]
\ge
\E\big[\essinf_{\tau\in\cT_\theta(\F^1)} J^{\Pi^{*,1}_\theta}(\tau,\sigma_*^\theta|\cF^1_\theta) \big].
\]
Since trivially $\essinf_{\tau\in\cT^R_\theta(\F^1)}J^{\Pi^{*,1}_\theta}(\tau,\sigma_*^\theta|\cF^1_\theta)
 \le \essinf_{\tau\in\cT_\theta(\F^1)} J^{\Pi^{*,1}_\theta}(\tau,\sigma_*^\theta|\cF^1_\theta)$, we have
\[
\essinf_{\tau\in\cT^R_\theta(\F^1)}J^{\Pi^{*,1}_\theta}(\tau,\sigma_*^\theta|\cF^1_\theta) = \essinf_{\tau\in\cT_\theta(\F^1)} J^{\Pi^{*,1}_\theta}(\tau,\sigma_*^\theta|\cF^1_\theta),
\]
which concludes the proof.
\end{proof}

It should be noted that the randomised stopping times $\tau_*^\gamma$ and $\sigma_*^\theta$ appearing in the above lemma may not belong to $\cT_\gamma^R(\F^1)$ and $\cT_\theta^R(\F^2)$, respectively, because $\gamma$ and $\theta$ are stopping times with respect to the opponent's filtration. This fact causes no difficulty in the statements and proofs above and we recall that $\tau_*^\gamma$ can be expressed in terms of the truncated control $\xi^{*;\gamma}$, representing the remainder of Player 1's stopping after time $\gamma$ (we can argue analogously for $\sigma^\theta_*$ and $\zeta^{*;\theta}$). Finally, we recall that the smallest filtration under which $\tau_*^\gamma$ is a randomised stopping time is $\cF^1_t \vee \sigma(\gamma\wedge t)$, $t \in [0, T]$ (or equivalently $\cF^1_t \vee \sigma(\gamma\wedge s,\ s\le t)$, $t\in[0,T]$). Analogous considerations hold for $\sigma_*^\theta$.

An analogue of Lemma \ref{lem:ud} holds for pure stopping times. The proof of Lemma \ref{lem:ud} can be repeated nearly verbatim or one can use the classical optimal stopping theory recalled in Appendix \ref{subsec:updownd}. The result is formulated rigorously in the corollary below.
\begin{corollary}\label{cor:pure_ud}
Given $\theta\in\cT_0(\F^1)$, 
the family $\{J^{\Pi^{*,1}_\theta}(\tau,\sigma_*^\theta|\cF^1_\theta),\ \tau\in\cT_\theta(\F^1)\}$ is downward-directed. Therefore, there is a sequence $(\tau^n)_{n\in\N}\subset\cT_\theta(\F^1)$ such that
\[
V^{*,1}(\theta)=\lim_{n\to\infty}J^{\Pi^{*,1}_\theta}(\tau^n,\sigma_*^\theta|\cF^1_\theta),
\]
where the limit is monotone from above.

Analogously, given $\gamma\in\cT_0(\F^2)$, 
the family $\{J^{\Pi^{*,2}_\gamma}(\tau_\gamma^*,\sigma|\cF^2_\gamma),\ \sigma\in\cT_\gamma(\F^2)\}$ is upward-directed. Therefore, there is a sequence $(\sigma^n)_{n\in\N}\subset\cT_\gamma(\F^2)$ such that
\[
V^{*,2}(\gamma)=\lim_{n\to\infty}J^{\Pi^{*,2}_\gamma}(\tau_\gamma^*,\sigma^n|\cF^2_\gamma),
\]
where the limit is monotone from below.
\end{corollary}

We state here the main aggregation result concerning players' values $V^{*, i}$, $i=1,2$. 
Its proof 
is formally presented in Section \ref{sec:proofthm}. 
Before stating the theorem, we introduce some notation. Given a process $(X_t)_{t\in[0,T]}$ we define its left and right limit as 
\[
X_{t_0+} \coloneqq \lim_{ \stackrel{s\to t_0}{s > t_0}} X_s=\lim_{s\downarrow t_0} X_s\quad\text{ and }\quad X_{t_0-} \coloneqq \lim_{\stackrel{s\to t_0}{s < t_0}} X_s=\lim_{s\uparrow t_0} X_s,
\]
whenever they exist. In order to emphasise that a process $(X_t)_{t\in[0,T]}$ is adapted/optional/previsible with respect to a filtration $\G$, we use $(X_t,\G)_{t\in[0,T]}$ or, occasionally, $(X_t,\G,\P)_{t\in[0,T]}$. 

\begin{theorem}\label{thm:value}
Given an optimal pair $(\xi^*,\zeta^*)\in\cAcirc_0(\F^1)\times\cAcirc_0(\F^2)$, the families 
\begin{align}\label{eq:families}
\begin{aligned}
&{\bf V}^{*,1}\coloneqq\big\{\E[1-\zeta^*_{\theta-}|\cF^1_\theta]\, V^{*,1}(\theta),\,\theta\in\cT_0(\F^1)\big\},\\
&{\bf V}^{*,2}\coloneqq\big\{\E[1-\xi^*_{\gamma-}|\cF^2_\gamma]\, V^{*,2}(\gamma),\,\gamma\in\cT_0(\F^2)\big\},
\end{aligned}
\end{align}
can be aggregated into optional semi-martingale processes of class $(D)$
\begin{align*}
&\big(\hat V^{*,1}_t,\,\F^1\big)_{t\in[0,T]}\quad\text{and}\quad \big(\hat V^{*,2}_t,\,\F^2\big)_{t\in[0,T]},
\end{align*}
where, for $\theta\in\cT_0(\F^1)$ and $\gamma\in\cT_0(\F^2)$,
\begin{align}\label{eq:hatv12}
\begin{aligned}
\hat V^{*,1}_\theta&=\essinf_{\tau\in\cT_\theta(\F^1)}\E\Big[f_\tau(1-\zeta^{*}_{\tau})+\int_{[\theta,\tau)}g_u\ud \zeta^{*}_u+h_\tau\Delta\zeta^{*}_\tau\Big|\cF^1_\theta\Big],\\
\hat V^{*,2}_\gamma&=\esssup_{\sigma\in\cT_\gamma(\F^2)}\E\Big[g_\sigma(1-\xi^{*}_{\sigma})+\int_{[\gamma,\sigma)}f_u\ud \xi^{*}_u+h_\sigma\Delta\xi^{*}_\sigma\Big|\cF^2_\gamma\Big].
\end{aligned}
\end{align}
Moreover, for $\theta\in\cT_0(\F^1)$ and $\gamma\in\cT_0(\F^2)$ the limits below hold $\P$-a.s. 
\begin{align}\label{eq:rlim}
\begin{aligned}
&\hat V^{*,1}_{\theta+} = \hat V^{*,1}_{\theta}- \E[1-\zeta^*_{\theta -}|\cF^1_{\theta}]\, \E^{\Pi^{*,1}_{\theta}}[g_{\theta}\Delta\zeta^{*;\theta}_{\theta}|\cF^1_{\theta}] = \hat V^{*,1}_{\theta}- \E[g_{\theta}\Delta\zeta^{*}_{\theta}|\cF^1_{\theta}],\\
&\hat V^{*,2}_{\gamma+} = \hat V^{*,2}_{\gamma}- \E[1-\xi^*_{\gamma -}|\cF^2_{\gamma}]\, \E^{\Pi^{*,2}_{\gamma}}[f_{\gamma}\Delta\xi^{*;\gamma}_{\gamma}|\cF^2_{\gamma}] = \hat V^{*,2}_{\gamma}- \E[f_{\gamma}\Delta\xi^{*}_{\gamma}|\cF^2_{\gamma}].
\end{aligned}
\end{align}
Finally, for any previsible $\theta\in\cT_0(\F^1)$ and $\gamma\in\cT_0(\F^2)$,
\begin{align}\label{eq:llim}
\begin{aligned}
\hat V^{*,1}_{\theta-} &\le \E[1-\zeta^*_{\theta-}|\cF^1_{\theta-}]\essinf_{\xi\in\cAcirc_{\theta}(\F^1)}J^{\widehat \Pi^{*,1}_{\theta}}\big(\xi,\zeta^{*;\theta}|\cF^1_{\theta-}),\\
\hat V^{*,2}_{\gamma-} 
&\ge \E[1-\xi^*_{\gamma-}|\cF^2_{\gamma-}]\esssup_{\zeta\in\cAcirc_{\gamma}(\F^2)}J^{\widehat \Pi^{*,2}_{\gamma}}\big(\xi^{*;\gamma},\zeta|\cF^2_{\gamma-}),
\end{aligned}
\end{align}
with equality on the events $\{\xi^*_{\theta-} < 1\}$ and $\{\zeta^*_{\gamma-} < 1\}$, respectively, and $\widehat\Pi^{*,1}_\theta$ and $\widehat\Pi^{*,2}_\gamma$ defined in \eqref{eq:Pi}. 
Thus, if the filtrations $\F^1$ and $\F^2$ are continuous, the processes $(\hat V^{*,1}_t)_{t\in[0,T]}$, $(\hat V^{*,2}_t)_{t\in[0,T]}$ are \caglad as long as $\xi^*_{t-} < 1$ and $\zeta^*_{t-} < 1$, respectively.
\end{theorem}
It is important to notice that in the formulae \eqref{eq:hatv12} for the processes $(\hat V^{*,i}_t)_{t\in[0,T]}$, $i=1,2$, the optimisation runs over stopping times for the players' respective filtrations. This is a result that we will essentially derive from the Lemma \ref{lem:pure}.

From the definition of $\widehat \Pi^{*,1}_{\theta}$ in \eqref{eq:Pi} it is easy to verify that, for any previsible $\theta\in\cT_0(\F^1)$ and $\gamma\in\cT_0(\F^2)$, the right-hand sides of inequalities in \eqref{eq:llim} can be equivalently written as
\begin{equation}\label{eq:llim_def}
\begin{aligned}
&\E[1\!-\!\zeta^*_{\theta-}|\cF^1_{\theta-}]\essinf_{\xi\in\cAcirc_{\theta}(\F^1)}J^{\widehat \Pi^{*,1}_{\theta}}\big(\xi,\zeta^{*;\theta}|\cF^1_{\theta-})
= \essinf_{\tau\in\cT_\theta(\F^1)}\E\Big[f_\tau(1\!-\!\zeta^{*}_{\tau})\!+\!\int_{[\theta,\tau)}\!\!g_u\ud \zeta^{*}_u\!+\!h_\tau\Delta\zeta^{*}_\tau\Big|\cF^1_{\theta-}\Big],\\
&\E[1\!-\!\xi^*_{\gamma-}|\cF^2_{\gamma-}]\esssup_{\zeta\in\cAcirc_{\gamma}(\F^2)}J^{\widehat \Pi^{*,2}_{\gamma}}\big(\xi^{*;\gamma},\zeta|\cF^2_{\gamma-})=\esssup_{\sigma\in\cT_\gamma(\F^2)}\E\Big[g_\sigma(1\!-\!\xi^{*}_{\sigma})\!+\!\int_{[\gamma,\sigma)}\!\!f_u\ud \xi^{*}_u\!+\!h_\sigma\Delta\xi^{*}_\sigma\Big|\cF^2_{\gamma-}\Big].
\end{aligned}
\end{equation}
This observation allows us to derive a corollary that refines Theorem \ref{thm:value}.
\begin{corollary}\label{cor:value_ineq}
When the first inequality in \eqref{eq:llim} is strict, we have $\hat V^{*,1}_{\theta-} = \E\big[f_{\theta-} (1-\zeta^*_{\theta-})\big|\cF^1_{\theta-}\big]$. When the second inequality in \eqref{eq:llim} is strict, we have $\hat V^{*,2}_{\gamma-} = \E[g_{\gamma-} (1-\xi^*_{\gamma-})|\cF^2_{\gamma-}]$.
Therefore,
\begin{equation*}
\begin{aligned}
\hat V^{*,1}_{\theta-}&=\min\Big(\E\big[f_{\theta-} (1-\zeta^*_{\theta-})\big|\cF^1_{\theta-}\big],\essinf_{\tau\in\cT_\theta(\F^1)}\E\Big[f_\tau(1\!-\!\zeta^{*}_{\tau})\!+\!\int_{[\theta,\tau)}\!\!g_u\ud \zeta^{*}_u\!+\!h_\tau\Delta\zeta^{*}_\tau\Big|\cF^1_{\theta-}\Big]\Big),\\
\hat V^{*,2}_{\gamma-}&=\max\Big(\E[g_{\gamma-} (1-\xi^*_{\gamma-})|\cF^2_{\gamma-}],\esssup_{\sigma\in\cT_\gamma(\F^2)}\E\Big[g_\sigma(1\!-\!\xi^{*}_{\sigma})\!+\!\int_{[\gamma,\sigma)}\!\!f_u\ud \xi^{*}_u\!+\!h_\sigma\Delta\xi^{*}_\sigma\Big|\cF^2_{\gamma-}\Big]\Big).
\end{aligned}
\end{equation*}
\end{corollary}
This corollary is justified after the proof of Theorem \ref{thm:value} in Section \ref{sec:proofthm}.

\subsection{Auxiliary super/sub-martingale systems}\label{subsec:aux-mart}
Let us now prepare the ground for the proof of Theorem \ref{thm:value} and for the martingale characterisation of players' values by first introducing two auxiliary families of random variables.
Let $(\xi,\zeta)\in\cAcirc_0(\F^1)\times\cAcirc_0(\F^2)$ be an arbitrary pair and recall that $(\xi^*,\zeta^*)\in \cAcirc_0(\F^1)\times\cAcirc_0(\F^2)$ is an optimal pair for the game started at zero in \eqref{eq:V} (cf.\ Notation \ref{not:equil}). Let ${\bf M}^\xi\coloneqq\{M^\xi(\theta),\theta\in\cT_0(\F^1)\}$ and ${\bf N}^\zeta\coloneqq\{N^\zeta(\gamma),\gamma\in\cT_0(\F^2)\}$ be defined as
\begin{align}
M^\xi(\theta)&= \E \Big[\! \int_{[0,\theta)}\!\! f_t (1\!-\!\zeta^*_t)  \ud \xi_t\! +\! \int_{[0,\theta)}\!\!g_t (1\!-\!\xi_t) \ud \zeta^*_t\! +\!\! \sum_{t \in [0,\theta)}\!\!h_t \Delta \xi_t \Delta \zeta^*_t \Big| \cF^1_\theta \Big]\label{eq:Mxi}\\
&\quad+\!(1\!-\!\xi_{\theta-}) \E[1\!-\!\zeta^*_{\theta-}|\cF^1_\theta]V^{*,1}(\theta),\notag\\
N^\zeta(\gamma)&= \E \Big[\! \int_{[0,\gamma)}\!\! f_t (1\!-\!\zeta_t)  \ud \xi^*_t\! +\! \int_{[0,\gamma)}\!\!g_t (1\!-\!\xi^*_t) \ud \zeta_t\! +\! \!\sum_{t \in [0,\gamma)}\!\! h_t \Delta \xi^*_t \Delta \zeta_t \Big| \cF^2_\gamma \Big]\label{eq:Nzi}\\
&\quad+\!(1\!-\!\zeta_{\gamma-})\E[1\!-\!\xi^*_{\gamma-}|\cF^2_\gamma] V^{*,2}(\gamma).\notag
\end{align}
Two choices of $\xi$ and $\zeta$ will be of particular interest in the paper: (i) $\xi\equiv \xi^*$ and $\zeta\equiv \zeta^*$, yielding ${\bf M}^*={\bf M}^{\xi^*}$ and ${\bf N}^*={\bf N}^{\zeta^*}$, respectively, and (ii) $\xi\equiv 0$ and $\zeta\equiv 0$, yielding ${\bf M}^0$ and ${\bf N}^0$, respectively. The families ${\bf M}^*$ and ${\bf N}^*$, where both players are acting optimally, will be shown to form martingale systems -- an analogue of the martingale condition for the value process $V_{t\wedge\hat\tau\wedge\hat\sigma}$ in the full-information game (cf.\ Subsection \ref{subsec:roadmap}). 
Related super- and submartingale conditions will be formulated for the families ${\bf M}^0$ and ${\bf N}^0$ as mentioned in Subsection \ref{subsec:roadmap}.

The pair $(\xi^*,\zeta^*)$ generates the randomised stopping times $(\tau_*,\sigma_*)=(\tau_*(\xi^*,Z_1),\sigma_*(\zeta^*,Z_2))$ (Definition \ref{def:rst}), where $Z_1$ and $Z_2$ are uniformly distributed on $[0,1]$, independent of $\cF_T$ and also mutually independent. For $z\in[0,1)$ we denote, 
\begin{align}\label{eq:taubar}
\bar\tau_*(z)=\inf\{t\in[0,T]:\xi^*_t>z\}\quad\text{and}\quad \bar\sigma_*(z)=\inf\{t\in[0,T]:\zeta^*_t>z\},
\end{align}
so that $\tau_*=\bar \tau_*(Z_1)$ and $\sigma_*=\bar \sigma_*(Z_2)$. Since $z\mapsto\bar\tau_*(z)$ and $z\mapsto\bar\sigma_*(z)$ are increasing, we can define the pair of largest optimal stopping times 
\begin{align}\label{eq:finaltime}
\bar\tau_*(1)=\inf\{t\in[0,T]:\xi^*_t=1\}\quad\text{and}\quad \bar\sigma_*(1)=\inf\{t\in[0,T]:\zeta^*_t=1\}.
\end{align}
The stopping time $\bar\tau_*(1)\wedge\bar\sigma_*(1)$ is the latest time at which the game -- started at time zero -- ends, in equilibrium.

\begin{remark}\label{rem:tau_z}
Observe that $\bar \tau_*(z) \in \cT(\F^1)$ for each $z \in (0,1)$. Moreover, it follows from the proof of Lemma \ref{lem:pure} that $V^{*,1}(0) = J(\bar\tau_*(z), \sigma_*|\cF^1_0)$ for a.e.~$z \in (0,1)$. Hence, the optimal strategy $\tau_*$ can be interpreted as a randomisation over \emph{optimal pure} stopping times $\bar \tau_*(z)$ for Player 1. An analogous conclusion holds for the pair $\bar \sigma_*(z)$ and $\sigma_*$, concerning Player 2. Further properties of the optimal strategies are presented in Subsection \ref{subsec:structure}.
\end{remark}

In the next proposition we show (super/sub)martingale properties of the families ${\bf M}^0$, ${\bf M}^*$, ${\bf N}^0$ and ${\bf N}^*$ that lead to aggregation results. A reader unfamiliar with notions of super/sub-martingale systems may refer to Definition \ref{def:supsubsys} in Appendix \ref{app:aggr}.
\begin{proposition}\label{thm:aggr1}
The family ${\bf M}^0$ is a $\cT_0(\F^1)$-submartingale system and the family ${\bf N}^0$ is a $\cT_0(\F^2)$-super\-mar\-tingale system. Both families are right-continuous in expectation and of class $(D)$. Thus, they can be aggregated (uniquely up to indistinguishability) into a \cadlag  submartingale $(M^0_t,\F^1,\P)_{t\in[0,T]}$ and a \cadlag supermartingale $(N^0_t,\F^2,\P)_{t\in[0,T]}$, respectively.
\end{proposition} 
\begin{proof}
From the definition of $V^{*,1}(\theta)$ and $V^{*,2}(\gamma)$ (cf.~Eq.~\eqref{eq:fV}) it is not difficult to verify that the families ${\bf M}^\xi$ and ${\bf N}^\zeta$ are a $\cT_0(\F^1)$-system and a $\cT_0(\F^2)$-system, respectively, for any choice of $\xi\in\cAcirc_0(\F^1)$ and $\zeta\in\cAcirc_0(\F^2)$. It remains to verify the claimed (super/sub)martingale properties. We only present a proof for ${\bf M}^0$ as analogous arguments apply to ${\bf N}^0$.

Since $f,g\in\cL_b(\P)$, it is easy to verify that ${\bf M}^0$ satisfies $\E[\esssup_{\theta\in\cT_0(\F^1)}|M^0(\theta)|]<\infty$. Hence the family is of class $(D)$. 
The submartingale property of ${\bf M}^0$ is equivalent to $\E[M^0(\tau)]\ge \E[M^0(\sigma)]$ for every $\tau,\sigma\in\cT_0(\F^1)$, $\sigma\le \tau$ (cf.\ Lemma \ref{lem:M}), which we are about to prove. 

Take $\tau,\sigma\in\cT_0(\F^1)$, $\sigma \le \tau$. We will argue first on the event $\{\sigma < T\}$ as on the event $\{\sigma=T\}$ we trivially have $M^0(\sigma)=M^0(\tau)=M^0(T)$. 
By the definition of $M^0$ we have
\begin{equation}\label{eq:subm00}
M^0(\sigma) =\E\Big[\int_{[0,\sigma)}g_t\ud\zeta^*_t \Big| \cF^1_\sigma\Big] +\E[1-\zeta^*_{\sigma-}|\cF^1_\sigma]\,V^{*,1}(\sigma).
\end{equation}
We recall from \eqref{eq:V1V2} the definition of $V^{*,1}(\sigma)$ and use Lemma \ref{lem:pure} to obtain
\begin{equation}
\begin{aligned}\label{eq:subm0}
&\E[1-\zeta^*_{\sigma-}|\cF^1_\sigma]V^{*,1}(\sigma)\\
&= 
\E[1-\zeta^*_{\sigma-}|\cF^1_\sigma] \essinf_{\theta \in \cT_\sigma(\F^1)} \E\Big[\Pi^{*,1}_\sigma \Big(f_\theta(1-\zeta^{*;\sigma}_\theta)+\int_{[\sigma,\theta)}g_t\ud \zeta^{*;\sigma}_t+h_\theta\Delta\zeta^{*;\sigma}_\theta\Big) \Big|\cF^1_\sigma\Big]\\
&= 
\essinf_{\theta \in \cT_\sigma(\F^1)} \E\Big[f_\theta(1-\zeta^{*}_\theta)+\int_{[\sigma,\theta)}g_t\ud \zeta^{*}_t+h_\theta\Delta\zeta^{*}_\theta \Big|\cF^1_\sigma\Big],
\end{aligned}
\end{equation}
where for the second equality we used $\E[1-\zeta^*_{\sigma-}|\cF^1_\sigma]\Pi^{*,1}_\theta=1-\zeta^*_{\sigma-}$ by the definition of $\Pi^{*,1}_\theta$ in \eqref{eq:Pi} and the definition of truncated controls (cf.\ \eqref{eq:trunc}). Hence, taking $\eta \in \cT_\tau(\F^1)\subset\cT_\sigma(\F^1)$ we have
\begin{equation}\label{eq:subm01}
\begin{aligned}
&\E[1-\zeta^*_{\sigma-}|\cF^1_\sigma]V^{*,1}(\sigma)\le\E\Big[f_\eta(1-\zeta^*_\eta)+\int_{[\sigma,\eta)}g_t\ud \zeta^*_t+h_\eta\Delta\zeta^*_\eta\Big|\cF^1_\sigma\Big]\\
&=\E\Big[f_\eta(1-\zeta^*_\eta)+\int_{[\sigma,\tau)}g_t\ud \zeta^*_t+\int_{[\tau,\eta)}g_t\ud \zeta^*_t+h_\eta\Delta\zeta^*_\eta\Big|\cF^1_\sigma\Big]\\
&=\E\bigg[\int_{[\sigma,\tau)}g_t\ud \zeta^*_t+\E\Big[f_{\eta}(1-\zeta^*_{\eta})+\int_{[\tau,\eta)}g_t\ud \zeta^*_t+h_{\eta}\Delta\zeta^*_{\eta}\Big|\cF^1_\tau\Big]\bigg|\cF^1_\sigma\bigg]\\
&=
\E\bigg[\int_{[\sigma,\tau)}g_t\ud \zeta^*_t+\E[1-\zeta^*_{\tau-}|\cF^1_\tau]\E\Big[\Pi^{*, 1}_{\tau} \Big( f_{\eta}(1-\zeta^{*;\tau}_{\eta})+\int_{[\tau,\eta)}g_t\ud \zeta^{*;\tau}_t+h_{\eta}\Delta\zeta^{*;\tau}_{\eta} \Big)\Big|\cF^1_\tau\Big]\bigg|\cF^1_\sigma\bigg]\\
&=
\E\Big[\int_{[\sigma,\tau)}g_t\ud \zeta^*_t+\E[1-\zeta^*_{\tau-}|\cF^1_\tau]J^{\Pi^{*, 1}_{\tau}}(\eta,\sigma_*^\tau|\cF^1_\tau)\Big|\cF^1_\sigma\Big],
\end{aligned}
\end{equation}
where in the final expression $\sigma_*^\tau$ is generated by the truncated control $\zeta^{*;\tau}$ (cf.\ notation in Lemma \ref{lem:pure}). Thanks to Lemma \ref{lem:pure} and since the family 
\(\big\{J^{\Pi^{*, 1}_{\tau}}(\eta,\sigma_*^\tau|\cF^1_\tau),\,\eta\in\cT_\tau(\F^1)\big\}\)
is downward-directed (Corollary \ref{cor:pure_ud}), we can take a sequence $(\eta_n) \subset \cT_\tau(\F^1)$ 
such that
\begin{equation}\label{eqn:eta_n_limit}
\lim_{n \to \infty}J^{\Pi^{*, 1}_{\tau}}(\eta_n,\sigma_*^\tau|\cF^1_\tau)
= V^{*,1}(\tau)
\end{equation}
and the limit is monotone from above. Equations \eqref{eq:subm0} and \eqref{eq:subm01} yield
\begin{align*}
&\E[1-\zeta^*_{\sigma-}|\cF^1_\sigma]V^{*,1}(\sigma)\\
&\le \E\Big[\int_{[\sigma,\tau)}g_t\ud \zeta^*_t+\E[1-\zeta^*_{\tau-}|\cF^1_\tau]\,J^{\Pi^{*, 1}_{\tau}}(\eta_n,\sigma_*^\tau|\cF^1_\tau)\Big|\cF^1_\sigma\Big]\\
&\xrightarrow{n \to \infty}
\E\Big[\int_{[\sigma,\tau)}g_t\ud \zeta^*_t+\E[1-\zeta^*_{\tau-}|\cF^1_\tau]\,V^{*,1}(\tau)\Big|\cF^1_\sigma\Big],
\end{align*}
where the limit is by the monotone convergence theorem and \eqref{eqn:eta_n_limit}. Substituting into \eqref{eq:subm00} and adding $M^0(\sigma)$ on the event $\{\sigma = T\}$ yields
\begin{align*}
\E[M^0(\sigma)]&= \E[\indd{\sigma < T} M^0(\sigma) + \indd{\sigma = T} M^0(\sigma)]\\
&\le \E\Big[\indd{\sigma < T}\Big( \int_{[0,\tau)}g_t\ud \zeta^*_t+\E[1-\zeta^*_{\tau-}|\cF^1_\tau]\,V^{*,1}(\tau)\Big) + \indd{\sigma = T} M^0(\sigma) \Big]\\
&= \E \big[\indd{\sigma < T} M^0(\tau) + \indd{\sigma = T} M^0(\tau) \big]=\E[M^0(\tau)],
\end{align*}
where we used that $\{\sigma < T \} \in \cF^1_{\sigma}$ combined with the tower property, and $M^0(\sigma) = M^0(T) = M^0(\tau)$ on $\{\sigma = T\}$. This is the required inequality for the submartingale property of the family.

In order to show the right-continuity in expectation let us consider a sequence $(\tau_n)_{n\in\N}\subset\cT_0(\F^1)$ such that $\tau_n\downarrow \tau\in\cT_0(\F^1)$. Arguing as in \eqref{eq:subm0} with $\tau_n$ in place of $\sigma$ we have the first equality below. The second one follows by the monotone convergence theorem and Corollary \ref{cor:pure_ud} (cf.\ \eqref{eq:commute}):
\begin{equation}\label{eqn:M0_sigman}
\begin{aligned}
\E[M^0(\tau_n)]&=\E\Big[\int_{[0,\tau_n)}g_t\ud \zeta^*_t  +\essinf_{\eta\in\cT_{\tau_n}(\F^1)}\E\Big[f_\eta(1-\zeta^*_\eta)+\int_{[\tau_n,\eta)}g_t\ud \zeta^*_t+h_\eta\Delta\zeta^*_\eta\Big|\cF^1_{\tau_n}\Big] \Big]\\
&=\inf_{\eta\in\cT_{\tau_n}(\F^1)}\E\Big[f_\eta(1-\zeta^*_\eta)+\int_{[0,\eta)}g_t\ud \zeta^*_t+h_\eta\Delta\zeta^*_\eta\Big].
\end{aligned}
\end{equation}
We claim that
\begin{align}\label{eq:limM0}
\begin{aligned}
&\lim_{n \to \infty} \inf_{\eta\in\cT_{\tau_n}(\F^1)}\E\Big[f_\eta(1-\zeta^*_\eta)+\int_{[0,\eta)}g_t\ud \zeta^*_t+h_\eta\Delta\zeta^*_\eta\Big] 
\\
&=
\inf_{\eta \in \cT_\tau(\F^1)} \E\Big[f_{\eta}(1-\zeta^*_{\eta}) + \int_{[0,\eta)}g_t\ud\zeta^*_t + h_\eta \Delta \zeta^*_\eta\Big].
\end{aligned}
\end{align}
Deferring for a moment the proof of \eqref{eq:limM0}, we observe that the latter and \eqref{eqn:M0_sigman} yield
\begin{align*}
\begin{aligned}
\lim_{n\to\infty}\E[M^0(\tau_n)]&=\inf_{\eta \in \cT_\tau(\F^1)} \E\Big[f_{\eta}(1-\zeta^*_{\eta}) + \int_{[0,\eta)}g_t\ud\zeta^*_t + h_\eta \Delta \zeta^*_\eta\Big]=\E[M^0(\tau)], 
\end{aligned}
\end{align*}
where for the final equality we applied an analogue of \eqref{eqn:M0_sigman} with $\tau$ instead of $\tau_n$. This 
completes the proof of the right-continuity of $M^0$ in expectation.

In order to prove \eqref{eq:limM0} we first observe that the limit exists because $\cT_{\tau_n}(\F^1)\subset\cT_{\tau_{n+1}}(\F^1)$ and so the associated infima form a decreasing sequence. Moreover, $\cT_{\tau_n}(\F^1)\subset\cT_\tau(\F^1)$ trivially implies that \eqref{eq:limM0} holds with ``$\ge$'' instead of equality. It remains to show the opposite inequality. Let us fix $\theta \in \cT_\tau(\F^1)$. Let $\theta_n = \theta \vee \tau_n$ for $n\in\N$ and notice that 
\[
A = \{ \theta_n > \theta,\ \forall n\in\N \} = \bigcap_{n \ge 1} \{ \theta_n > \theta \} = \bigcap_{n \ge 1} \{ \tau_n > \theta \} \subset \{\theta = \tau \},
\] 
where the last inclusion is deduced from $\tau_n \downarrow \tau$. On the set $A^c$, the sequence $\theta_n$ stabilises, i.e., $\theta_n = \theta$ for all $n > N(\omega)$ and some $N(\omega)\in\N$. We will therefore argue differently on the set $A$ and on its complement $A^c$. 

Since $f \ge h$ (cf.\ Assumption \ref{ass:payoff}), 
\[
f_{\theta_n}(1-\zeta^*_{\theta_n})+h_{\theta_n}\Delta\zeta^*_{\theta_n}=f_{\theta_n}(1-\zeta^*_{\theta_n-})+(h_{\theta_n}-f_{\theta_n})\Delta\zeta^*_{\theta_n}\le f_{\theta_n}(1-\zeta^*_{\theta_n-}).
\]
Using this inequality we first write
\begin{align*}
&\E\Big[f_{\theta_n}(1-\zeta^*_{\theta_n}) + \int_{[\tau,\theta_n)}g_t\ud\zeta^*_t + h_{\theta_n}\Delta\zeta^*_{\theta_n} \Big]\\
&=
\E\Big[1_A \Big(f_{\theta_n}(1-\zeta^*_{\theta_n}) + \int_{[\tau,\theta_n)}g_t\ud\zeta^*_t + h_{\theta_n}\Delta\zeta^*_{\theta_n}\Big) + \ind_{A^c} \Big(f_{\theta_n}(1-\zeta^*_{\theta_n}) + \int_{[\tau,\theta_n)}g_t\ud\zeta^*_t + h_{\theta_n}\Delta\zeta^*_{\theta_n}\Big) \Big]\\
&\le
\E\Big[1_A \Big(f_{\theta_n}(1-\zeta^*_{\theta_n-}) + \int_{[\tau,\theta_n)}g_t\ud\zeta^*_t \Big) 
+ \ind_{A^c} \Big(f_{\theta_n}(1-\zeta^*_{\theta_n}) + \int_{[\tau,\theta_n)}g_t\ud\zeta^*_t + h_{\theta_n}\Delta\zeta^*_{\theta_n}\Big) \Big].
\end{align*}
Now, given that $\theta_n\in\cT_{\tau_n}(\F^1)$ for all $n\in\N$, we have
\begin{align}\label{eq:limM00}
\begin{aligned}
&\lim_{n \to \infty} \inf_{\eta\in\cT_{\tau_n}(\F^1)}\E\Big[f_\eta(1-\zeta^*_\eta)+\int_{[0,\eta)}g_t\ud \zeta^*_t+h_\eta\Delta\zeta^*_\eta\Big]\\
&\le
\lim_{n \to \infty} \E\Big[\ind_{A} \Big(f_{\theta_n}(1\!-\!\zeta^*_{\theta_n-})\! +\! \int_{[\sigma,\theta_n)}\!g_t\ud\zeta^*_t \Big)\! +\! 
\ind_{A^c} \Big(f_{\theta_n}(1\!-\!\zeta^*_{\theta_n})\! +\! \int_{[\sigma,\theta_n)}\!g_t\ud\zeta^*_t \!+\! h_{\theta_n}\Delta\zeta^*_{\theta_n}\Big) \Big]\\
&=
\E\Big[\ind_{A} \Big(f_{\theta}(1-\zeta^*_{\theta}) + \int_{[\sigma,\theta]}g_t\ud\zeta^*_t \Big) + 
\ind_{A^c} \Big(f_{\theta}(1-\zeta^*_{\theta}) + \int_{[\sigma,\theta)}g_t\ud\zeta^*_t + h_{\theta}\Delta\zeta^*_{\theta}\Big) \Big],
\end{aligned}
\end{align}
by the dominated convergence theorem and the fact that on $A^c$ we have $\theta_n = \theta$ for all sufficiently large $n \ge N(\omega)$. Next we use $g_\theta \le h_\theta$ to write
\[
1_A\int_{[\sigma,\theta]}g_t\ud\zeta^*_t=1_A\Big(\int_{[\sigma,\theta)}g_t\ud\zeta^*_t+g_\theta\Delta\zeta^*_\theta\Big)\le1_A\Big(\int_{[\sigma,\theta)}g_t\ud\zeta^*_t+h_\theta\Delta\zeta^*_\theta\Big). 
\]
Inserting this inequality into \eqref{eq:limM00} and recombining the indicator functions yield
\begin{align*}
\lim_{n \to \infty} \inf_{\eta\in\cT_{\tau_n}(\F^1)}\E\Big[f_\eta(1-\zeta^*_\eta)+\int_{[0,\eta)}g_t\ud \zeta^*_t+h_\eta\Delta\zeta^*_\eta\Big]\le
\E\Big[f_{\theta}(1-\zeta^*_{\theta}) + \int_{[\sigma,\theta)}g_t\ud\zeta^*_t + h_\theta \Delta \zeta^*_\theta\Big]. 
\end{align*}
From the arbitrariness of $\theta \in \cT_\tau(\F^1)$, we conclude that \eqref{eq:limM0} holds.

Properties of ${\bf N}^0$ are shown in an analogous way so we omit their proof.
It remains to invoke Proposition \ref{prop:aggr} to conclude that the families ${\bf M}^0$ and ${\bf N}^0$ can be aggregated into a \cadlag  submartingale $(M^0_t,\F^1,\P)_{t\in[0,T]}$ and a \cadlag supermartingale $(N^0_t,\F^2,\P)_{t\in[0,T]}$, respectively.
\end{proof}

We will later refine the above result by showing in Corollary \ref{cor:M_zero_mart} that the processes $(M^0_t)_{t\in[0,T]}$ and $(N^0_t)_{t\in[0,T]}$ are martingales up to the `last optimal stopping time' for Player 1 and Player 2, respectively. Next, we aggregate the families ${\bf M}^*={\bf M}^{\xi^*}$ and ${\bf N}^*={\bf N}^{\zeta^*}$.

\begin{proposition}\label{thm:aggr2}
The family ${\bf M}^*$ is a $\cT_0(\F^1)$-martingale system and the family ${\bf N}^*$ is a $\cT_0(\F^2)$-martingale system. Both are of class $(D)$. Hence, they can be uniquely aggregated into \cadlag martingales $(M^*_t,\F^1,\P)_{t\in[0,T]}$ and $(N^*_t,\F^2,\P)_{t\in[0,T]}$ (up to indistinguishability).

Moreover, 
the \emph{truncated controls} remain optimal at every time (prior to the end of the game) in the sense that
\begin{equation}\label{eqn:trunc_optim}
\begin{aligned}
&\ind_{\Gam^1_\theta} V^{*,1}(\theta)
=\ind_{\Gam^1_\theta} J^{\Pi^{*,1}_\theta}(\xi^{*;\theta},\zeta^{*;\theta}|\cF^1_\theta),\quad\theta\in\cT_0(\F^1),\\
&\ind_{\Gam^2_\gamma}V^{*,2}(\gamma)
=\ind_{\Gam^2_\gamma} J^{\Pi^{*,2}_\gamma}(\xi^{*;\gamma},\zeta^{*;\gamma}|\cF^2_\gamma),\quad\gamma\in\cT_0(\F^2),
\end{aligned}
\end{equation}
where 
\begin{align*}
&\Gam^1_\theta = \big\{\omega\in\Omega: (1-\xi^*_{\theta-}(\omega))\E[1-\zeta^*_{\theta-}|\cF^1_\theta](\omega) > 0\big\} \in \cF^1_\theta,\\
&\Gam^2_\gamma = \big\{\omega\in\Omega: (1-\zeta^*_{\gamma-}(\omega))\E[1-\xi^*_{\gamma-}|\cF^2_\gamma](\omega) > 0 \big\} \in \cF^2_\gamma.
\end{align*}
\end{proposition}
\begin{proof}
The fact that ${\bf M}^*$ and ${\bf N}^*$ are $\cT_0(\F^1)$- and $\cT_0(\F^2)$-systems and their integrability is argued as in the proof of Proposition \ref{thm:aggr1}. Next we show the martingale property. We only consider the question for ${\bf M}^*$ as analogous arguments apply to ${\bf N}^*$.

In suffices to show that $\E[M^*(\theta)]=\E[M^*(0)]$ for any $\theta\in\cT_0(\F^1)$ in order to establish the martingale property of ${\bf M}^*$ (cf.\ Lemma \ref{lem:M}). Fix $\theta\in\cT_0(\F^1)$ and define 
\begin{align}\label{eq:xitilde}
\tilde\xi_t\coloneqq\xi^*_t \ind_{[0,\theta)}(t)+\big[\xi^*_{\theta-}+(1-\xi^*_{\theta-})\xi_t\big]\ind_{[\theta,T]}(t)
\end{align}
for some arbitrary $\xi\in\cAcirc_\theta(\F^1)$. Then it is easy to check that $\tilde\xi\in\cAcirc_0(\F^1)$.
From the definition of $M^*$, noticing that $\Pi^{*,1}_0=1$, using tower property and sub-optimality of $\tilde\xi$ we have 
\begin{align}\label{eq:inM0}
\begin{split}
\E[M^*(0)]& = \E[V^{*,1}(0)]\\
&\le \E\Big[\int_{[0,T)}f_t(1-\zeta^*_t)\ud \tilde\xi_t+\int_{[0,T)}g_t(1-\tilde\xi_t)\ud \zeta^*_t+\sum_{t\in[0,T]}h_t\Delta\zeta^*_t\Delta\tilde\xi_t\Big]\\
&=\E\Big[\int_{[0,\theta)}f_t(1-\zeta^*_t)\ud \xi^*_t+\int_{[0,\theta)}g_t(1-\xi^*_t)\ud \zeta^*_t+\sum_{t\in[0,\theta)}h_t\Delta\zeta^*_t\Delta\xi^*_t\Big]\\
&\quad+\E\Big[(1-\xi^*_{\theta-})\Big(\int_{[\theta,T)}f_t(1-\zeta^*_t)\ud \xi_t+\int_{[\theta,T)}g_t(1-\xi_t)\ud \zeta^*_t+\sum_{t\in[\theta,T]}h_t\Delta\zeta^*_t\Delta\xi_t\Big)\Big].
\end{split}
\end{align}
Notice that the random variable inside the final expectation is different from zero only on the event $\{\xi^*_{\theta-}\vee\zeta^*_{\theta-}<1\}$.
Setting $\zeta^{*;\theta}_t$ as in \eqref{eq:trunc} and recalling the convention $\frac00=1$, we can write
\begin{align*}
&\E\Big[(1-\xi^*_{\theta-})\Big(\int_{[\theta,T)}f_t(1-\zeta^*_t)\ud \xi_t+\int_{[\theta,T)}g_t(1-\xi_t)\ud \zeta^*_t+\sum_{t\in[\theta,T]}h_t\Delta\zeta^*_t\Delta\xi_t\Big)\Big|\cF^1_\theta\Big]\\
&=(1-\xi^*_{\theta-})\E\Big[(1-\zeta^*_{\theta-})\Big(\int_{[\theta,T)}f_t(1-\zeta^{*;\theta}_t)\ud \xi_t+\int_{[\theta,T)}g_t(1-\xi_t)\ud \zeta^{*;\theta}_t+\sum_{t\in[\theta,T]}h_t\Delta\zeta^{*;\theta}_t\Delta\xi_t\Big)\Big|\cF^1_\theta\Big]\\
&=(1-\xi^*_{\theta-})\E[1-\zeta^*_{\theta-}|\cF^1_\theta]J^{\Pi^{*,1}_\theta}(\xi,\zeta^{*;\theta}|\cF^1_\theta).
\end{align*}
Combining this with \eqref{eq:inM0} yields for any $\xi\in\cAcirc_\theta(\F^1)$
\begin{equation}\label{eq:Msub}
\begin{aligned}
\E[M^*(0)]&\le \E\Big[\int_{[0,\theta)}f_t(1-\zeta^*_t)\ud \xi^*_t+\int_{[0,\theta)}g_t(1-\xi^*_t)\ud \zeta^*_t+\sum_{t\in[0,\theta)}h_t\Delta\zeta^*_t\Delta\xi^*_t\\
&\qquad+(1-\xi^*_{\theta-})\E[1-\zeta^*_{\theta-}|\cF^1_\theta]\,J^{\Pi^{*,1}_\theta}(\xi,\zeta^{*;\theta}|\cF^1_\theta)\Big].
\end{aligned}
\end{equation}
Thanks to Lemma \ref{lem:ud} we can select a sequence $(\xi^n)\subset\cAcirc_\theta(\F^1)$ such that 
\begin{align}\label{eq:max-xin}
V^{*,1}(\theta)=\lim_{n\to\infty}J^{\Pi^{*,1}_\theta}(\xi^n,\zeta^{*;\theta}|\cF^1_\theta),
\end{align}
where the limit is monotone from above. We write the inequality \eqref{eq:Msub} with $\xi = \xi^n$ and let $n\to\infty$. Invoking the monotone convergence theorem, we arrive at
\begin{equation}\label{eq:Msubopt}
\begin{aligned}
\E[M^*(0)]&\le \E\Big[\int_{[0,\theta)}f_t(1-\zeta^*_t)\ud \xi^*_t+\int_{[0,\theta)}g_t(1-\xi^*_t)\ud \zeta^*_t+\sum_{t\in[0,\theta)}h_t\Delta\zeta^*_t\Delta\xi^*_t\Big]\\
&\quad+\lim_{n\to\infty}\E\Big[(1-\xi^*_{\theta-})\E[1-\zeta^*_{\theta-}|\cF^1_\theta]J^{\Pi^{*,1}_\theta}(\xi^n,\zeta^{*;\theta}|\cF^1_\theta)\Big]\\
&= \E\Big[\int_{[0,\theta)}f_t(1-\zeta^*_t)\ud \xi^*_t+\int_{[0,\theta)}g_t(1-\xi^*_t)\ud \zeta^*_t+\sum_{t\in[0,\theta)}h_t\Delta\zeta^*_t\Delta\xi^*_t\Big]\\
&\quad+\E\Big[(1-\xi^*_{\theta-})\E[1-\zeta^*_{\theta-}|\cF^1_\theta]V^{*, 1}(\theta)\Big]
=\E[M^*(\theta)].
\end{aligned}
\end{equation}

The inequality in \eqref{eq:inM0} becomes an equality if we replace $\tilde\xi$ with $\xi^*$. So \eqref{eq:Msub} becomes
\begin{align}\label{eq:Mopt}
\begin{split}
\E[M^*(0)]
&= \E\Big[\int_{[0,\theta)}f_t(1-\zeta^*_t)\ud \xi^*_t+\int_{[0,\theta)}g_t(1-\xi^*_t)\ud \zeta^*_t+\sum_{t\in[0,\theta)}h_t\Delta\zeta^*_t\Delta\xi^*_t\\
&\qquad+(1-\xi^*_{\theta-})\E[1-\zeta^*_{\theta-}|\cF^1_\theta]J^{\Pi^{*,1}_\theta}(\xi^{*;\theta},\zeta^{*;\theta}|\cF^1_\theta)\Big]\ge\E[M^*(\theta)],
\end{split}
\end{align}
where the inequality is due to the fact that {\em a priori} the truncated control $\xi^{*;\theta}\in\cAcirc_\theta(\F^1)$ may not be optimal for $J^{\Pi^{*,1}_\theta}(\cdot,\zeta^{*;\theta}|\cF^1_\theta)$.

Combining \eqref{eq:Msubopt} and \eqref{eq:Mopt} yields the desired result, i.e., ${\bf M}^*$ is a martingale system. By Corollary \ref{cor:aggr}, it can be uniquely aggregated into a \cadlag martingale $(M^*_t,\F^1,\P)_{t\in[0,T]}$ (up to indistinguishability). 
The inequalities \eqref{eq:Msubopt} and \eqref{eq:Mopt} also show that 
\[
(1-\xi^*_{\theta-})\E[1-\zeta^*_{\theta-}|\cF^1_\theta]V^{*,1}(\theta)=(1-\xi^*_{\theta-})\E[1-\zeta^*_{\theta-}|\cF^1_\theta]J^{\Pi^{*,1}_\theta}\big(\xi^{*;\theta},\zeta^{*;\theta}\big|\cF^1_\theta\big),
\]
from which we deduce \eqref{eqn:trunc_optim} and the optimality of the truncated strategy $\xi^{*; \theta}$. 
\end{proof}

\subsection{Proof of Theorem \ref{thm:value} and some further results}\label{sec:proofthm}

Thanks to Proposition \ref{thm:aggr1} we are able to obtain an aggregation of the systems of equilibrium values into optional processes. Moreover, we compute the right and left limits of such optional processes, thus providing also a formula for their jumps. 

\begin{proof}[{\bf Proof of Theorem \ref{thm:value}}]
We prove all the claims for $\hat V^{*,1}_t$ as the ones for $\hat V^{*,2}_t$ follow by analogous arguments.
By the definition of the submartingale $(M^0_t)_{t\in[0,T]}$ obtained in Proposition \ref{thm:aggr1} we have, for any $\tau \in \cT_0(\F^1)$,
\begin{align}\label{eq:hatV00}
\E[1-\zeta^*_{\tau-}|\cF^1_\tau]V^{*,1} (\tau) = M^0_\tau - \E\Big[ \int_{[0, \tau)} g_s \ud \zeta^*_s \Big| \cF^1_\tau \Big]=M^0_\tau-S^1_\tau,
\end{align}
where 
$(S^1_t)_{t\in[0,T]}$ is the $\F^1$-optional projection of $\big(\int_{[0, t)} g_s \ud \zeta^*_s\big)_{t \in [0, T]}$. This shows that ${\bf V}^{*,1}$ is aggregated into an optional process $\hat V^{*, 1} \coloneqq M^0 - S^1$. The process $S^1$ is the $\F^1$-optional projection of a bounded variation process, hence a difference of two submartingales and, in particular, a semi-martingale. Therefore, $\hat V^{*,1}$ is also a semi-martingale, as claimed. The explicit expression \eqref{eq:hatv12} for $\hat V^{*,1}$ is easily deduced from the one for $V^{*,1}(\theta)$ in \eqref{eq:V1V2}, upon noticing that 
\begin{equation}\label{eqn:re_theta}
\E[1-\zeta^*_{\theta-}|\cF^1_\theta]J^{\Pi^{*,1}_\theta}(\xi,\zeta^{*;\theta}|\cF^1_\theta)=\E\Big[f_\tau(1-\zeta^{*}_{\tau})+\int_{[\theta,\tau)}g_u\ud \zeta^{*}_u+h_\tau\Delta\zeta^{*}_\tau\Big|\cF^1_\theta\Big],
\end{equation}
for $\tau\in\cT^R_\theta(\F^1)$ generated by $\xi\in\cA^\circ_\theta(\F^1)$, and then applying Lemma \ref{lem:pure} to restrict the optimisation to stopping times $\tau\in\cT_\theta(\F^1)$ (cf.\ \eqref{eq:subm0} for the same argument). The class $(D)$ property easily follows because $f,g,h\in\cL_b(\P)$.

The process $M^0$ has \cadlag paths. Moreover, there is a set $\Omega_*\in\cF$ of probability one such that for all $\omega\in\Omega_*$ the process $t\mapsto S^1_t(\omega)$ has right and left limits at all $t\in(0,T)$, thanks to \cite[Prop.\ I.3.14]{karatzas1998brownian}, because $S^1$ is a difference of two submartingales. Then, outside of a universal null set, all paths of the process $\hat V^{*,1}$ have right and left limits $\hat V^{*,1}_{t+}$ and $\hat V^{*,1}_{t-}$ at all points $t\in(0,T)$.
Moreover, we notice that by the right continuity of the filtration, the process $(\hat V^{*,1}_{t+})_{t\in[0,T)}$ is $\F^1$-adapted.

Fix $\theta\in\cT_0(\F^1)$, $\theta<T$. From \eqref{eq:hatv12} with $\theta$ therein replaced by $\theta_n\in\cT_0(\F^1)$, $\theta_n>\theta$, we get 
\begin{equation}\label{eq:v1}
\hat V^{*,1}_{\theta_n}
=\essinf_{\tau\in\cT_{\theta_n}(\F^1)}\E\Big[f_\tau(1-\zeta^{*}_{\tau})+\int_{[\theta_n,\tau)}g_u\ud \zeta^{*}_u+h_\tau\Delta\zeta^{*}_\tau\Big|\cF^1_{\theta_n}\Big].
\end{equation}
We apply monotone convergence and Corollary \ref{cor:pure_ud} to the above equality to obtain (cf.\ \eqref{eq:commute})
\begin{equation} \label{eqn:v_cond}
\E [\hat V^{*,1}_{\theta_n} | \cF^1_{\theta} ] = \essinf_{\tau\in\cT_{\theta_n}(\F^1)}\E\Big[f_\tau(1-\zeta^{*}_{\tau})+\int_{[\theta_n,\tau)}g_u\ud \zeta^{*}_u+h_\tau\Delta\zeta^{*}_\tau\Big|\cF^1_{\theta}\Big].
\end{equation}
Then,
\[
\E [\hat V^{*,1}_{\theta_n} | \cF^1_{\theta} ] \ge \essinf_{\tau\in\cT_{>\theta}(\F^1)}\E\Big[f_\tau(1-\zeta^{*}_{\tau})+\int_{[\theta_n,\tau)}g_u\ud \zeta^{*}_u+h_\tau\Delta\zeta^{*}_\tau\Big|\cF^1_{\theta}\Big],
\]
where $\cT_{>\theta}(\F^1)\coloneqq\{\tau\in\cT_{\theta}(\F^1): \P(\tau>\theta)=1\}$. We let $\theta_n \downarrow \theta$ as $n\to\infty$ (i.e., $\theta_n\to \theta$, $\theta_n>\theta$). Taking limits in the expression above yields
\begin{align}\label{eq:V+0}
\begin{aligned}
&\hat V^{*,1}_{\theta+} 
= \E [\hat V^{*,1}_{\theta+} | \cF^1_{\theta} ]
= \lim_{n\to \infty} \E [\hat V^{*,1}_{\theta_n} | \cF^1_{\theta} ]\\
& \ge \essinf_{\tau\in\cT_{>\theta}(\F^1)}\E\Big[f_\tau(1-\zeta^{*}_{\tau})+\int_{(\theta,\tau)}g_u\ud \zeta^{*}_u+h_\tau\Delta\zeta^{*}_\tau\Big|\cF^1_{\theta}\Big],
\end{aligned}
\end{align}
where we refer to the right continuity of the filtration $\F^1$ to justify the first equality, to the dominated convergence theorem for the second equality and, for the convergence of the right-hand side, to
\begin{equation}\label{eqn:g_to_0}
0 \le \E\Big[ \int_{(\theta, \theta_n)} |g_u| \ud \zeta^*_{u} \Big] \le
\E \big[ (\zeta^*_{\theta_n-} - \zeta^*_{\theta}) \sup_{u \in [0, T]} |g_u|  \big] \to 0,
\end{equation}
which holds by the dominated convergence as $\theta_n \downarrow \theta$.

For the opposite inequality, we fix $\tau \in \cT_{>\theta}(\F^1)$, $\theta_n \in (\theta, T)$ and notice that $\tau\vee \theta_n\in\cT_{\theta_n}(\F^1)$. It then follows from \eqref{eqn:v_cond} that
\[
\E [\hat V^{*,1}_{\theta_n} | \cF^1_{\theta} ]\le \E\Big[f_{\tau \vee \theta_n}(1-\zeta^{*}_{\tau \vee \theta_n})+\int_{[\theta_n,\tau \vee \theta_n)}g_u\ud \zeta^{*}_u+h_{\tau \vee \theta_n}\Delta\zeta^{*}_{\tau \vee \theta_n} \Big|\cF^1_{\theta}\Big].
\]
Letting $\theta_n\downarrow \theta$ as $n\to\infty$, we observe that $[\theta_n, \tau \vee \theta_n)\to(\theta,\tau)$, because $\tau>\theta$, and likewise
\begin{align*}
\lim_{n\to\infty}\Big(f_{\tau \vee \theta_n}(1\!-\!\zeta^{*}_{\tau \vee \theta_n})\!+\!\int_{[\theta_n,\tau \vee \theta_n)}\!\!g_u\ud \zeta^{*}_u\!+\!h_{\tau \vee \theta_n}\Delta\zeta^{*}_{\tau \vee \theta_n}\Big)\!=\!f_{\tau}(1\!-\!\zeta^{*}_{\tau})\!+\!\int_{(\theta,\tau)}\!\!g_u\ud \zeta^{*}_u\!+\!h_{\tau }\Delta\zeta^{*}_{\tau}.
\end{align*}
Therefore, the dominated convergence theorem yields (cf.\ \eqref{eqn:g_to_0} for a similar argument)
\[
\hat V^{*,1}_{\theta+}\le \E\Big[f_\tau(1-\zeta^{*}_{\tau})+\int_{(\theta,\tau)}g_u\ud \zeta^{*}_u+h_\tau\Delta\zeta^{*}_\tau\Big|\cF^1_{\theta}\Big].
\]
Thus, combining this with \eqref{eq:V+0}
we obtain
\begin{align*}
\hat V^{*,1}_{\theta+}
&=\essinf_{\tau\in\cT_{>\theta}(\F^1)}\E\Big[f_\tau(1-\zeta^{*}_{\tau})+\int_{(\theta,\tau)}g_u\ud \zeta^{*}_u+h_\tau\Delta\zeta^{*}_\tau\Big|\cF^1_{\theta}\Big]\\
&=\essinf_{\tau\in\cT_{>\theta}(\F^1)}\E\Big[f_\tau(1-\zeta^{*}_{\tau})+\int_{[\theta,\tau)}g_u\ud \zeta^{*}_u+h_\tau\Delta\zeta^{*}_\tau\Big|\cF^1_{\theta}\Big]-\E\big[g_{\theta}\Delta\zeta^{*}_{\theta}\big|\cF^1_{\theta}\big].
\end{align*}

To prove \eqref{eq:rlim} it remains to show that, although $\cT_{>\theta}(\F^1)\subsetneq\cT_{\theta}(\F^1)$, it still holds
\begin{align}\label{eq:vhat>}
\hat V^{*,1}_{\theta}=\essinf_{\tau\in\cT_{>\theta}(\F^1)}\E\Big[f_\tau(1-\zeta^{*}_{\tau})+\int_{[\theta,\tau)}g_u\ud \zeta^{*}_u+h_\tau\Delta\zeta^{*}_\tau\Big|\cF^1_{\theta}\Big].
\end{align}
First, we notice that for any $\eta\in\cT_{\theta}(\F^1)$
\begin{align*}
&\E\Big[f_\eta(1-\zeta^*_\eta)+\int_{[\theta,\eta)}g_t\ud\zeta^*_t+h_\eta\Delta\zeta^*_{\eta}\Big|\cF^1_{\theta}\Big]\\
&=\ind_{\{\eta=\theta\}}\E\Big[f_{\theta}(1-\zeta^*_{\theta})+h_{\theta}\Delta\zeta^*_{\theta}\Big|\cF^1_{\theta}\Big]+\ind_{\{\eta>\theta\}}\E\Big[f_\eta(1-\zeta^*_\eta)+\int_{[\theta,\eta)}g_t\ud\zeta^*_t+h_\eta\Delta\zeta^*_{\eta}\Big|\cF^1_{\theta}\Big]\\
&\ge\ind_{\{\eta=\theta\}}\E\Big[f_{\theta}(1-\zeta^*_{\theta})+h_{\theta}\Delta\zeta^*_{\theta}\Big|\cF^1_{\theta}\Big]\\
&\quad+\ind_{\{\eta>\theta\}}\essinf_{\tau\in\cT_{>\theta}(\cF^1_{t})}\E\Big[f_\tau(1-\zeta^*_\tau)+\int_{[\theta,\tau)}g_t\ud\zeta^*_t+h_\tau\Delta\zeta^*_{\tau}\Big|\cF^1_{\theta}\Big].
\end{align*}
It then follows that
\begin{align}\label{eq:V+01}
\begin{aligned}
&\essinf_{\eta\in\cT_{\theta}(\F^1)}\E\Big[f_\eta(1-\zeta^*_\eta)+\int_{[\theta,\eta)}g_t\ud\zeta^*_t+h_\eta\Delta\zeta^*_{\eta}\Big|\cF^1_{\theta}\Big]\\
&=\E[f_{\theta}(1-\zeta^*_{\theta})+h_{\theta}\Delta\zeta^*_{\theta}|\cF^1_{\theta}]\wedge\essinf_{\tau\in\cT_{>{\theta}}(\F^1)}\E\Big[f_\tau(1-\zeta^*_\tau)+\int_{[\theta,\tau)}g_t\ud\zeta^*_t+h_\tau\Delta\zeta^*_\tau\Big|\cF^1_{\theta}\Big].
\end{aligned}
\end{align}
Take a sequence $(\tau_n)_{n\in\N} \subset \cT_{>\theta}(\F^1)$ such that $\tau_n\downarrow \theta$ (i.e., $\tau_n>\theta$ for all $n\in\N$). We have 
\begin{align*}
&\essinf_{\tau\in\cT_{>\theta}(\F^1)}\E\Big[f_\tau(1-\zeta^*_\tau)+\int_{[\theta,\tau)}g_t\ud\zeta^*_t+h_\tau\Delta\zeta^*_\tau\Big|\cF^1_{\theta}\Big]\\
&= \essinf_{\tau\in\cT_{>\theta}(\F^1)}\E\Big[f_\tau(1-\zeta^*_{\tau-})+\int_{[\theta,\tau)}g_t\ud\zeta^*_t+(h_{\tau}-f_\tau)\Delta\zeta^*_\tau\Big|\cF^1_{\theta}\Big]\\
&\le \lim_{n\to\infty}\E\Big[f_{\tau_n}(1-\zeta^*_{\tau_n-})+\int_{[\theta,\tau_n)}g_t\ud\zeta^*_t\Big|\cF^1_{\theta}\Big]
=\E\Big[f_{\theta}(1-\zeta^*_{\theta})+g_{\theta}\Delta\zeta^*_{\theta}\Big|\cF^1_{\theta}\Big]\\
&\le \E[f_{\theta}(1-\zeta^*_{\theta})+h_{\theta}\Delta\zeta^*_{\theta}|\cF^1_{\theta}],
\end{align*}
where we used 
$f \ge h \ge g$ in both inequalities, and we passed the limit inside the expectation by the dominated convergence theorem. Combining the above inequality and \eqref{eq:V+01} yields \eqref{eq:vhat>} as needed.
Then, the first equation in \eqref{eq:rlim} holds. Analogous arguments prove the aggregation for the family ${\bf V}^{*,2}$ and the second equation in \eqref{eq:rlim}.

We proceed now with the left-limit. For $\tau \in \cT^R_\theta(\F^1)$ and any $\upsilon\in\cT^R_0(\F^1)$ with $\upsilon\le \tau$, denote
\[
\cP_{\upsilon}(\tau,\zeta^*)\coloneqq f_{\tau}(1-\zeta^*_{\tau})+\int_{[\upsilon,\tau)}g_t\ud\zeta^*_t+h_{\tau}\Delta\zeta^*_{\tau}
\]
and notice that, by the definition of $\widehat \Pi^{*,1}_{\theta}$ in \eqref{eq:Pi},
\begin{equation}\label{eq:equiv}
\begin{aligned}
&\E[1-\zeta^*_{\theta-}|\cF^1_{\theta-}] J^{\widehat \Pi^{*,1}_{\theta}}\big(\xi,\zeta^{*;\theta}|\cF^1_{\theta-}) = 
\E[\cP_{\theta}(\tau, \zeta^*) | \cF^1_{\theta-}],\\
&\E[1-\zeta^*_{\theta-}|\cF^1_{\theta-}] J^{\widehat \Pi^{*,1}_{\theta}}\big(\xi,\zeta^{*;\theta}|\cF^1_{\theta}) = 
\E[\cP_{\theta}(\tau, \zeta^*) | \cF^1_{\theta}],
\end{aligned}
\end{equation}
where $\tau$ is the randomised stopping time generated by $\xi \in \cAcirc_{\theta}(\F^1)$. Let $(\tau_k)_{k\in\N} \subset \cT_\theta(\F^1)$ be a minimising sequence $\lim_{k \to \infty} \E[\cP_{\theta}(\tau_k, \zeta^*) | \cF^1_{\theta}]=\essinf_{\tau\in\cT_{\theta}(\F^1)} \E[\cP_{\theta}(\tau, \zeta^*) | \cF^1_{\theta}]$, which exists thanks to Corollary \ref{cor:pure_ud} and \eqref{eq:equiv}. Then, 
\begin{equation*}
\begin{aligned}
&\essinf_{\tau\in\cT_{\theta}(\F^1)} \E[\cP_{\theta}(\tau, \zeta^*) | \cF^1_{\theta-}]
\ge
\essinf_{\tau\in\cT^R_{\theta}(\F^1)} \E[\cP_{\theta}(\tau, \zeta^*) | \cF^1_{\theta-}]\\
&\ge
\E \Big[ \essinf_{\tau\in\cT^R_{\theta}(\F^1)} \E[\cP_{\theta}(\tau, \zeta^*) | \cF^1_{\theta}] \Big| \cF^1_{\theta-} \Big]
=
\E \Big[ \lim_{k \to \infty} \E[\cP_{\theta}(\tau_k, \zeta^*) | \cF^1_{\theta}] \Big| \cF^1_{\theta-} \Big]\\
&=
\lim_{k \to \infty} \E[\cP_{\theta}(\tau_k, \zeta^*) | \cF^1_{\theta-}]
\ge
\essinf_{\tau\in\cT_{\theta}(\F^1)} \E[\cP_{\theta}(\tau, \zeta^*) | \cF^1_{\theta-}],
\end{aligned}
\end{equation*}
where we used the dominated convergence theorem in the second equality. Hence,
\begin{equation}\label{eq:Sigma}
\Sigma_{\theta-}\coloneqq\essinf_{\tau\in\cT_{\theta}(\F^1)} \E[\cP_{\theta}(\tau, \zeta^*) | \cF^1_{\theta-}]
=
\essinf_{\tau\in\cT^R_{\theta}(\F^1)} \E[\cP_{\theta}(\tau, \zeta^*) | \cF^1_{\theta-}],
\end{equation}
and $(\tau_k)_{k\in\N}$ is also a minimising sequence for $\Sigma_{\theta-}$. The first inequality in \eqref{eq:llim} is equivalent to showing that $\Delta\coloneqq\hat V^{*,1}_{\theta-} - \Sigma_{\theta-} \le 0$, which we set out to do next.

Letting $(\theta_n)_{n\in\N}\subset\cT_0(\F^1)$ be an announcing sequence for $\theta$ (i.e., $\theta_n<\theta_{n+1}<\theta$ and $\theta_n\uparrow \theta$), we define $\Delta_n \coloneqq \hat V^{*,1}_{\theta_n}-\Sigma_{\theta-}$ and note that $\Delta = \lim_{n \to \infty} \Delta_n$, recalling the existence of the left limit of $\hat V^{*,1}$ at all times.  We proceed to derive an upper bound on $\Delta_n$. Let $(\tau_k)_{k\in\N}$ be the minimising sequence for $\Sigma_{\theta-}$ and set $U_k \coloneqq \E\big[\cP_{\theta}(\tau_k,\zeta^*)\big|\cF^1_{\theta-}\big] - \Sigma_{\theta-}$, so that $(U_k)_{k\in\N}$ is a non-negative sequence and it converges to zero $\P$\as as $k \to \infty$. Using that $\cT_{\theta_n}(\F^1)\supset\cT_{\theta}(\F^1)$ and \eqref{eq:v1} we have
\begin{align*}
\begin{aligned}
\Delta_n&\le \E\big[\cP_{\theta_n}(\tau_k,\zeta^*)\big|\cF^1_{\theta_n}\big]
-\E\big[\cP_{\theta}(\tau_k,\zeta^*)\big|\cF^1_{\theta-}\big] + U_k\\
&=\E\big[\cP_{\theta}(\tau_k,\zeta^*)\big|\cF^1_{\theta_n}\big]-\E\big[\cP_{\theta}(\tau_k,\zeta^*)\big|\cF^1_{\theta-}\big]+ U_k + \E\Big[\int_{(\theta_n,\theta)}g_t\ud \zeta^*_t\Big|\cF^1_{\theta_n}\Big]\\
&\le\big|\E\big[\cP_{\theta}(\tau_k,\zeta^*)\big|\cF^1_{\theta_n}\big]-\E\big[\cP_{\theta}(\tau_k,\zeta^*)\big|\cF^1_{\theta-}\big]\big|+U_k + \E\Big[\sup_{0\le t\le T}|g_t|\big(\zeta^*_{\theta-}-\zeta^*_{\theta_n}\big)\Big|\cF^1_{\theta_n}\Big].
\end{aligned}
\end{align*}
For fixed $k\in\N$ we let $n\to\infty$. For the first term in the last line, we define a martingale $\Lambda^k_t\coloneqq\E\big[\cP_{\theta}(\tau_k,\zeta^*)\big|\cF^1_t\big]$, $t\ge 0$, which is \cadlag thanks to Proposition \ref{prop:aggr}. We have   
\[
\lim_{n\to\infty}\E\big[\cP_{\theta}(\tau_k,\zeta^*)\big|\cF^1_{\theta_n}\big]=\lim_{n\to\infty}\Lambda^k_{\theta_n}=\Lambda^k_{\theta-}=\E\big[\cP_{\theta}(\tau_k,\zeta^*)\big|\cF^1_{\theta-}\big],
\]
where the third equality holds thanks to \cite[Thm.\ VI.14]{DellacherieMeyer} (recall that $\theta$ is previsible). The last term in the upper bound for $\Delta_n$ is positive and, by the Markov inequality, for any $\eps>0$,
\[
\P\Big(\E\Big[\sup_{0\le t\le T}|g_t|\big(\zeta^*_{\theta-}-\zeta^*_{\theta_n}\big)\Big|\cF^1_{\theta_n}\Big]>\eps\Big)\le \eps^{-1}\E\Big[\sup_{0\le t\le T}|g_t|\big(\zeta^*_{\theta-}-\zeta^*_{\theta_n}\big)\Big].
\]
Since $\lim_{n \to \infty} \zeta^*_{\theta_n} = \zeta^*_{\theta-}$, the right-hand side converges to $0$ by the dominated convergence theorem. Hence, $\E[\sup_{0\le t\le T}|g_t|(\zeta^*_{\theta-}-\zeta^*_{\theta_n})|\cF^1_{\theta_n}]\to 0$ in probability as $n\to\infty$, and it converges a.s.\ along a subsequence. In conclusion, for each $k$, we have $\Delta \le U_k$. Letting $k\to\infty$ yields $\Delta\le 0$, as needed.

In order to conclude the proof of the theorem it remains to show $\indd{\xi^*_{\theta-}<1}\Delta=0$. We proceed with the lower bound on $\Delta_n$ on the set $\{\xi^*_{\theta-} < 1\}$. We start from the observation that 
\begin{align*}
\begin{aligned}
&\indd{\xi^*_{\theta-} < 1}\hat V^{*,1}_{\theta_n}\\
&=\indd{\xi^*_{\theta-} < 1}\E\Big[\int_{[\theta_n,T)}f_t\big(1\!-\!\zeta^*_t\big)\ud \xi^{*;\theta_n}_t\!+\!\int_{[\theta_n,T)}g_t\big(1\!-\!\xi^{*;\theta_n}_t\big)\ud \zeta^*_t\!+\!\sum_{t\in[\theta_n,T]}h_t\Delta\xi^{*;\theta_n}_t\Delta\zeta^*_t\Big|\cF^1_{\theta_n}\Big],
\end{aligned}
\end{align*}
where we invoked the optimality of the truncated controls shown in Proposition \ref{thm:aggr2}, which holds only on the set $\{\xi^*_{\theta_n-} < 1\} \supset \{\xi^*_{\theta_-} < 1\}$. Splitting the integration at $\theta$ yields
\begin{equation}\label{eqn:theta_n_split}
\begin{aligned}
&\indd{\xi^*_{\theta-} < 1}\hat V^{*,1}_{\theta_n}\\
&=\indd{\xi^*_{\theta-} < 1}\E\Big[\int_{[\theta_n,\theta)}f_t\big(1-\zeta^*_t\big)\ud \xi^{*;\theta_n}_t+\int_{[\theta_n,\theta)}g_t\big(1-\xi^{*;\theta_n}_t\big)\ud \zeta^*_t+\sum_{t\in[\theta_n,\theta)}h_t\Delta\xi^{*;\theta_n}_t\Delta\zeta^*_t\Big|\cF^1_{\theta_n}\Big]\\
&\quad+\indd{\xi^*_{\theta-} < 1}\E\Big[\int_{[\theta,T)}f_t\big(1-\zeta^*_t\big)\ud \xi^{*;\theta_n}_t+\int_{[\theta,T)}g_t\big(1-\xi^{*;\theta_n}_t\big)\ud \zeta^*_t+\sum_{t\in[\theta,T]}h_t\Delta\xi^{*;\theta_n}_t\Delta\zeta^*_t\Big|\cF^1_{\theta_n}\Big].
\end{aligned}
\end{equation}
For the first term on the right-hand side, denoting for simplicity $Z\coloneqq \sup_{0\le t\le T}\big(|f_t|+|g_t|+|h_t|\big)$, we have 
\begin{align*}
&\indd{\xi^*_{\theta-} < 1}\E\Big[\int_{[\theta_n,\theta)}f_t\big(1-\zeta^*_t\big)\ud \xi^{*;\theta_n}_t+\int_{[\theta_n,\theta)}g_t\big(1-\xi^{*;\theta_n}_t\big)\ud \zeta^*_t+\sum_{t\in[\theta_n,\theta)}h_t\Delta\xi^{*;\theta_n}_t\Delta\zeta^*_t\Big|\cF^1_{\theta_n}\Big]\\
&\ge 
-\ind_{\{\xi^*_{\theta-}<1\}} \E\Big[Z\Big(\frac{\xi^*_{\theta-}-\xi^*_{\theta_n-}}{1-\xi^*_{\theta_n-}}+\zeta^*_{\theta-}-\zeta^*_{\theta_n-}\Big)\Big|\cF^1_{\theta_n}\Big]\\
&\ge -\ind_{\{\xi^*_{\theta-}<1\}}\E\Big[Z\big(\ind_{\{\xi^*_{\theta-}<1\}}+\ind_{\{\xi^*_{\theta-}=1\}}\big)\Big(\frac{\xi^*_{\theta-}-\xi^*_{\theta_n-}}{1-\xi^*_{\theta_n-}}+\zeta^*_{\theta-}-\zeta^*_{\theta_n-}\Big)\Big|\cF^1_{\theta_n}\Big]\\
&\ge-\ind_{\{\xi^*_{\theta-}<1\}}\Big(\E\Big[Z\ind_{\{\xi^*_{\theta-}<1\}}\Big(\frac{\xi^*_{\theta-}-\xi^*_{\theta_n-}}{1-\xi^*_{\theta_n-}}+\zeta^*_{\theta-}-\zeta^*_{\theta_n-}\Big)\Big|\cF^1_{\theta_n}\Big]+\E\big[2Z\ind_{\{\xi^*_{\theta-}=1\}}\big|\cF^1_{\theta_n}\big]\Big). 
\end{align*}

For the second term of \eqref{eqn:theta_n_split} we have,
\begin{align*}
&\indd{\xi^*_{\theta-} < 1}\E\Big[\int_{[\theta,T)}f_t\big(1-\zeta^*_t\big)\ud \xi^{*;\theta_n}_t+\int_{[\theta,T)}g_t\big(1-\xi^{*;\theta_n}_t\big)\ud \zeta^*_t+\	\sum_{t\in[\theta,T]}h_t\Delta\xi^{*;\theta_n}_t\Delta\zeta^*_t\Big|\cF^1_{\theta_n}\Big]\\
&=\indd{\xi^*_{\theta-} < 1}\E\Big[\ind_{\{\xi^*_{\theta-}<1\}}\Big(\int_{[\theta,T)}\!\!f_t\big(1\!-\!\zeta^*_t\big)\ud \xi^{*;\theta_n}_t\!+\!\int_{[\theta,T)}\!\!g_t\big(1\!-\!\xi^{*;\theta_n}_t\big)\ud \zeta^*_t\!+\!\sum_{t\in[\theta,T]}\!h_t\Delta\xi^{*;\theta_n}_t\Delta\zeta^*_t\Big)\Big|\cF^1_{\theta_n}\Big],
\end{align*}
because the expression under the expectation equals $0$ on $\{\xi^*_{\theta-} = 1\}$ as $\xi^{*;\theta_n}_t=1$ for $t\in[\theta,T]$ (recall \eqref{eq:trunc}).
By the tower property and skipping the indicator $\indd{\xi^*_{\theta-} < 1}$ outside the expectation for brevity,  
\begin{align*}
&\E\Big[\ind_{\{\xi^*_{\theta-}<1\}}\Big(\int_{[\theta,T)}f_t\big(1-\zeta^*_t\big)\ud \xi^{*;\theta_n}_t+\int_{[\theta,T)}g_t\big(1-\xi^{*;\theta_n}_t\big)\ud \zeta^*_t+\	\sum_{t\in[\theta,T]}h_t\Delta\xi^{*;\theta_n}_t\Delta\zeta^*_t\Big)\Big|\cF^1_{\theta_n}\Big]\\
&=\E\Big[\ind_{\{\xi^*_{\theta-}<1\}}\frac{1\!-\!\xi^*_{\theta-}}{1\!-\!\xi^*_{\theta_n-}}\E\Big[\int_{[\theta,T)}\!\!f_t\big(1\!-\!\zeta^*_t\big)\ud \xi^{*;\theta}_t\!+\!\int_{[\theta,T)}\!\!g_t\big(1\!-\!\xi^{*;\theta}_t\big)\ud \zeta^*_t\!+\!\sum_{t\in[\theta,T]}\!h_t\Delta\xi^{*;\theta}_t\Delta\zeta^*_t\Big|\cF^1_{\theta-}\Big]\Big|\cF^1_{\theta_n}\Big]\\
&\ge \E\Big[\ind_{\{\xi^*_{\theta-}<1\}}\frac{1-\xi^*_{\theta-}}{1-\xi^*_{\theta_n-}}\essinf_{\tau\in\cT^R_\theta(\F^1)}\E\Big[f_\tau\big(1-\zeta^*_\tau\big)+\int_{[\theta,\tau)}g_t\ud \zeta^*_t+h_\tau\Delta\zeta^*_\tau\Big|\cF^1_{\theta-}\Big]\Big|\cF^1_{\theta_n}\Big]\\
&= \E\Big[\ind_{\{\xi^*_{\theta-}<1\}}\frac{1-\xi^*_{\theta-}}{1-\xi^*_{\theta_n-}}\essinf_{\tau\in\cT_\theta(\F^1)}\E\Big[f_\tau\big(1-\zeta^*_\tau\big)+\int_{[\theta,\tau)}g_t\ud \zeta^*_t+h_\tau\Delta\zeta^*_\tau\Big|\cF^1_{\theta-}\Big]\Big|\cF^1_{\theta_n}\Big]\\
&=\E\big[\ind_{\{\xi^*_{\theta-}<1\}}\frac{1-\xi^*_{\theta-}}{1-\xi^*_{\theta_n-}}\Sigma_{\theta-}\big|\cF^1_{\theta_n}\big],
\end{align*}
where in the penultimate equality we substitute $\cT^R_\theta(\F^1)$ with $\cT_\theta(\F^1)$ by an analogous argument as in the proof of Lemma \ref{lem:pure}. Notice that $|\Sigma_{\theta-}|\le \E[Z|\cF^1_{\theta-}]$ and therefore we can further continue the lower bound as
\begin{align*}
\E\big[\ind_{\{\xi^*_{\theta-}<1\}}\frac{1-\xi^*_{\theta-}}{1-\xi^*_{\theta_n-}}\Sigma_{\theta-}\big|\cF^1_{\theta_n}\big]
&=\E\big[\ind_{\{\xi^*_{\theta-}<1\}}\Sigma_{\theta-}\big|\cF^1_{\theta_n}\big]-\E\big[\ind_{\{\xi^*_{\theta-}<1\}}\frac{\xi^*_{\theta-}-\xi^*_{\theta_n-}}{1-\xi^*_{\theta_n-}}\Sigma_{\theta-}\big|\cF^1_{\theta_n}\big]\\
&\ge \E\big[\ind_{\{\xi^*_{\theta-}<1\}}\Sigma_{\theta-}\big|\cF^1_{\theta_n}\big]-\E\big[\ind_{\{\xi^*_{\theta-}<1\}}\frac{\xi^*_{\theta-}-\xi^*_{\theta_n-}}{1-\xi^*_{\theta_n-}}\E[Z|\cF^1_{\theta-}]\big|\cF^1_{\theta_n}\big]\\
&= \E\big[\ind_{\{\xi^*_{\theta-}<1\}}\Sigma_{\theta-}\big|\cF^1_{\theta_n}\big]-\E\big[\ind_{\{\xi^*_{\theta-}<1\}}Z\frac{\xi^*_{\theta-}-\xi^*_{\theta_n-}}{1-\xi^*_{\theta_n-}}\big|\cF^1_{\theta_n}\big],
\end{align*}
where the final equality is by the tower property.

Combining the estimates above, the lower bound on $\Delta_n$ takes the form 
\begin{align}\label{eq:delta0}
\begin{aligned}
\indd{\xi^*_{\theta-} < 1} \Delta_n &\ge-\!\ind_{\{\xi^*_{\theta-}<1\}}\Big(\E\Big[Z\ind_{\{\xi^*_{\theta-}<1\}}\Big(2\frac{\xi^*_{\theta-}\!-\!\xi^*_{\theta_n-}}{1-\xi^*_{\theta_n-}}\!+\!\zeta^*_{\theta-}\!-\!\zeta^*_{\theta_n-}\Big)\Big|\cF^1_{\theta_n}\Big] - 2 \E\big[Z\ind_{\{\xi^*_{\theta-}=1\}}\big|\cF^1_{\theta_n}\big]\Big)\\
&\quad+\indd{\xi^*_{\theta-} < 1}\Big(\E\big[\ind_{\{\xi^*_{\theta-}<1\}}\Sigma_{\theta-}\big|\cF^1_{\theta_n}\big]-\Sigma_{\theta-}\Big).
\end{aligned}
\end{align}
Using again \cite[Thm.\ VI.14]{DellacherieMeyer} we have $\P$-a.s.
\begin{equation}
\begin{aligned}
\lim_{n\to\infty}\ind_{\{\xi^*_{\theta-}<1\}}\E\big[Z\ind_{\{\xi^*_{\theta-}=1\}}\big|\cF^1_{\theta_n}\big]&=\ind_{\{\xi^*_{\theta-}<1\}}\E\big[Z\ind_{\{\xi^*_{\theta-}=1\}}\big|\cF^1_{\theta-}\big]\\
&=\ind_{\{\xi^*_{\theta-}<1\}}\ind_{\{\xi^*_{\theta-}=1\}}\E[Z|\cF^1_{\theta-}]=0,
\end{aligned}
\end{equation}
and
\begin{equation*}
\begin{aligned}
\lim_{n\to\infty}\indd{\xi^*_{\theta-} < 1}\Big(\E\big[\ind_{\{\xi^*_{\theta-}<1\}}\Sigma_{\theta-}\big|\cF^1_{\theta_n}\big]-\Sigma_{\theta-}\Big)=\indd{\xi^*_{\theta-} < 1}\Big(\E\big[\ind_{\{\xi^*_{\theta-}<1\}}\Sigma_{\theta-}\big|\cF^1_{\theta-}\big]-\Sigma_{\theta-}\Big)=0.
\end{aligned}
\end{equation*}
By the Markov inequality, for any $\eps>0$, 
\begin{equation*}
\begin{aligned}
&\lim_{n\to\infty}\P\Big(\E\Big[Z\ind_{\{\xi^*_{\theta-}<1\}}\Big(2\frac{\xi^*_{\theta-}-\xi^*_{\theta_n-}}{1-\xi^*_{\theta_n-}}+\zeta^*_{\theta-}-\zeta^*_{\theta_n-}\Big)\Big|\cF^1_{\theta_n}\Big]>\eps\Big)\\
&\le \eps^{-1}\lim_{n\to\infty}\E\Big[Z\ind_{\{\xi^*_{\theta-}<1\}}\Big(2\frac{\xi^*_{\theta-}-\xi^*_{\theta_n-}}{1-\xi^*_{\theta_n-}}+\zeta^*_{\theta-}-\zeta^*_{\theta_n-}\Big)\Big]=0,
\end{aligned}
\end{equation*}
where the equality holds by the monotone convergence, because $\xi^*_{\theta_n-}\uparrow \xi^*_{\theta-}$, $\zeta^*_{\theta_n-}\uparrow \zeta^*_{\theta-}$, 
and the mapping $x \mapsto (\xi^*_{\theta-} - x)/(1-x)$ is decreasing for $x \in [0,\xi^*_{\theta-}]$.

Thus, up to selecting a subsequence, all terms on the right-hand side of \eqref{eq:delta0} vanish $\P$-a.s.\ when $n\to\infty$.
This concludes the proof of $\indd{\xi^*_{\theta-} < 1} \Delta = \lim_{n \to \infty} \indd{\xi^*_{\theta-} < 1} \Delta_n \ge 0$. Combining with the upper bound $\Delta \le 0$ demonstrates that the inequality in \eqref{eq:llim} becomes an equality on the set $\{\xi^*_{\theta-} < 1\}$.
\end{proof}

\begin{proof}[Proof of Corollary \ref{cor:value_ineq}]
We recall \eqref{eq:v1}, which implies 
\begin{equation}\label{eq:Jan0}
\begin{aligned}
\hat V^{*,1}_{\theta_n} 
&\le \E[f_{\theta_n} (1-\zeta^*_{\theta_n}) + h_{\theta_n} \Delta \zeta^*_{\theta_n}| \cF^1_{\theta_n}]
\le \E[f_{\theta_n} (1-\zeta^*_{\theta_n-}) | \cF^1_{\theta_n}] \\
&=\E[f_{\theta_n} (\zeta^*_{\theta-}-\zeta^*_{\theta_n-}) | \cF^1_{\theta_n}]+\E[(f_{\theta_n}-f_{\theta-}) (1-\zeta^*_{\theta-}) | \cF^1_{\theta_n}]+\E[f_{\theta-} (1-\zeta^*_{\theta-}) | \cF^1_{\theta_n}]\\
&\le \E[|f_{\theta_n}| (\zeta^*_{\theta-}-\zeta^*_{\theta_n-}) | \cF^1_{\theta_n}]+\E[|f_{\theta_n}-f_{\theta-}| (1-\zeta^*_{\theta-}) | \cF^1_{\theta_n}]+\E[f_{\theta-} (1-\zeta^*_{\theta-}) | \cF^1_{\theta_n}].
\end{aligned}
\end{equation}
By the Markov inequality, for any $\eps>0$ we have
\begin{equation*}
\begin{aligned}
&\P\big(\E[|f_{\theta_n}| (\zeta^*_{\theta-}-\zeta^*_{\theta_n-}) | \cF^1_{\theta_n}]>\eps\big)\le \tfrac1\eps\E\big[\sup_{0\le t\le T}|f_{t}| (\zeta^*_{\theta-}-\zeta^*_{\theta_n-}) \big],\\
&\P\big(\E[|f_{\theta_n}-f_{\theta-}| (1-\zeta^*_{\theta-}) | \cF^1_{\theta_n}]>\eps\big)\le \tfrac1\eps\E\big[|f_{\theta_n}-f_{\theta-}| (1-\zeta^*_{\theta-}) \big].
\end{aligned}
\end{equation*}
Letting $n\to\infty$, and applying the dominated convergence on the right-hand side of the above inequalities we obtain convergence in probability 
\[
\E[|f_{\theta_n}| (\zeta^*_{\theta-}-\zeta^*_{\theta_n-})|\cF_{\theta_n}]\to 0\quad\text{and}\quad\E[|f_{\theta_n}-f_{\theta-}| (1-\zeta^*_{\theta-}) | \cF^1_{\theta_n}]\to 0.
\]
Moreover, $\E[f_{\theta-} (1-\zeta^*_{\theta-}) | \cF^1_{\theta_n}]\to \E[f_{\theta-} (1-\zeta^*_{\theta-}) | \cF^1_{\theta-}]$, $\P$-a.s., by \cite[Thm.\ VI.14]{DellacherieMeyer}. In conclusion, passing to the limit along a subsequence in \eqref{eq:Jan0} we obtain
\[
\hat V^{*,1}_{\theta-} \le  \E\big[f_{\theta-} (1-\zeta^*_{\theta-})\big|\cF^1_{\theta-}\big].
\]

Recalling \eqref{eq:v1} and the notation $\Sigma_{\theta-}$ from \eqref{eq:Sigma}, the tower property yields
\begin{align*}
\hat V^{*,1}_{\theta_n} 
&\ge \essinf_{\tau \in \cT_{\theta_n}(\F^1)} \E \Big[ \indd{\tau < \theta} \inf_{s \in [\theta_n, \theta)} \Big(f_s (1-\zeta^*_s) + \int_{[\theta_n, s)} g_u \ud \zeta^*_u + h_s \Delta \zeta^*_s\Big)\\
&\hspace{65pt} + \indd{\tau \ge \theta} \Big(\int_{[\theta_n, \theta)} g_u \ud \zeta^*_u +  \Sigma_{\theta-} \Big) \Big| \cF^1_{\theta_n} \Big].
\end{align*}
Denoting $Z_n \coloneqq \E\big[\sup_{u \in [0, T]} |g_u| (\zeta^*_{\theta-} - \zeta^*_{\theta_n-})\big| \cF^1_{\theta_n}\big]$, 
we get
\begin{align*}
\hat V^{*,1}_{\theta_n} 
&\ge \essinf_{\tau \in \cT_{\theta_n}(\F^1)} \E \Big[ \indd{\tau < \theta} \E\Big[\inf_{s \in [\theta_n, \theta)} \Big(f_s (1-\zeta^*_s) + h_s \Delta \zeta^*_s\Big)\Big|\cF^1_{\theta-}\Big]+ \indd{\tau \ge \theta} \Sigma_{\theta-} \Big| \cF^1_{\theta_n} \Big]-Z_n\\
&\ge \E \Big[ \min\Big( \E\big[\inf_{s \in [\theta_n, \theta)} \big(f_s (1-\zeta^*_s) + h_s \Delta \zeta^*_s\big) \big| \cF^1_{\theta-}\big],\Sigma_{\theta-} \Big) \Big| \cF^1_{\theta_n} \Big] - Z_n,
\end{align*}
where in the first line we also used the tower property, the fact that $\Sigma_{\theta-}$ is $\cF^1_{\theta-}$-measurable and $\{\tau < \theta\} \in \cF^1_{\theta-}$ because $\theta$ is previsible.
Setting $Y_n=\E[\sup_{u \in [0, T]}|h_u| (\zeta^*_{\theta-}-\zeta^*_{\theta_n-})|\cF_{\theta_n}]$ and continuing from the above inequalities we get
\begin{equation}\label{eq:jan2}
\begin{aligned}
\hat V^{*,1}_{\theta_n} &\ge \E \Big[ \min\Big( \E\big[\inf_{s \in [\theta_n, \theta)} f_s (1-\zeta^*_s) \big| \cF^1_{\theta-}\big],\Sigma_{\theta-} \Big) \Big| \cF^1_{\theta_n} \Big]-Y_n - Z_n.
\end{aligned}
\end{equation}
By analogous arguments to those employed above, using the Markov inequality, we can show that $Y_n\to0$ and $Z_n\to 0$ in probability as $n\to\infty$. Moreover, letting $n\to\infty$
\[
\E\big[\inf_{s \in [\theta_n, \theta)} f_s (1-\zeta^*_s) \big| \cF^1_{\theta-}\big]\to \E\big[f_{\theta-} (1-\zeta^*_{\theta-}) \big| \cF^1_{\theta-}\big],
\]
$\P$-a.s.\ by the dominated convergence theorem for conditional expectation \cite[Thm.~34.2]{Billingsley}.

To simplify presentation, let 
\[
W_n\coloneqq \min\Big( \E\big[\inf_{s \in [\theta_n, \theta)} f_s (1-\zeta^*_s) \big| \cF^1_{\theta-}\big],\Sigma_{\theta-} \Big)\ \text{and}\ W\coloneqq \min\Big( \E\big[f_{\theta-} (1-\zeta^*_{\theta-}) \big| \cF^1_{\theta-}\big],\Sigma_{\theta-} \Big).
\] 
Since we have shown that $W_{n}\to W$, $\P$-a.s., it is not difficult to show 
\[
\E[W_n|\cF_{\theta_n}]-\E[W|\cF_{\theta-}]=\E[W_n-W|\cF_{\theta_n}]+\E[W|\cF_{\theta_n}]-\E[W|\cF_{\theta-}]\to 0,
\] 
in probability as $n\to\infty$, using the Markov inequality and \cite[Thm.\ VI.14]{DellacherieMeyer} as before.

Finally, we can select a subsequence $(n_k)_{n\in\N}$ along which all limits above hold $\P$-a.s., and passing to the limit along such subsequence in \eqref{eq:jan2} we get
\[
\hat V^{*,1}_{\theta-} \ge 
\min\big( \E\big[f_{\theta-} (1-\zeta^*_{\theta-})\big|\cF^1_{\theta-}\big], \Sigma_{\theta-} \big).
\]
Since we have also shown that $\hat V^{*,1}_{\theta-} \le  \E\big[f_{\theta-} (1-\zeta^*_{\theta-})\big|\cF^1_{\theta-}\big]$, we must have $\hat V^{*,1}_{\theta-} = \E\big[f_{\theta-} (1-\zeta^*_{\theta-})\big|\cF^1_{\theta-}\big]$ on the set $\{ \hat V^{*,1}_{\theta-} < \Sigma_{\theta-} \}$.
\end{proof}

The next result is a refinement of Propositions \ref{thm:aggr1} and \ref{thm:aggr2}. In the definition of the families ${\bf M}^\xi$ and ${\bf N}^\zeta$, for arbitrary $(\xi,\zeta)\in\cAcirc_0(\F^1)\times\cAcirc_0(\F^2)$, the only terms that require an aggregation step are the families $\bf V^{*,1}$ and $\bf V^{*,2}$ appearing therein (the rest is an optional projection of an $\F$-adapted process). The latter have been aggregated into optional processes in Theorem \ref{thm:value}. Then, the families ${\bf M}^\xi$ and ${\bf N}^\zeta$ can also be aggregated into optional processes. The next proposition shows that the resulting processes are respectively super- and sub-martingales as well.
\begin{proposition}\label{prop:subsupmg}
For any $(\xi,\zeta)\in\cAcirc_0(\F^1)\times\cAcirc_0(\F^2)$ the families ${\bf M}^\xi$ and ${\bf N}^\zeta$ are of class $(D)$ and can be aggregated into an optional submartingale process $(M^\xi_t,\F^1,\P)_{t\in[0,T]}$ and an optional supermartingale process $(N^\zeta_t,\F^2,\P)_{t\in[0,T]}$, respectively. 
\end{proposition}
\begin{proof}
We only need to prove the sub/super-martingale property of the families. As usual, we provide a proof only for $M^\xi$ because the arguments for $N^\zeta$ are analogous. We argue in a similar way as in the proof of Proposition \ref{thm:aggr1}.

Since $f,g\in\cL_b(\P)$, it is easy to verify that ${\bf M}^\xi$ satisfies $\E[\esssup_{\theta\in\cT_0(\F^1)}|M^\xi(\theta)|]<\infty$. We now want to verify that $\E[M^\xi(\tau)]\ge \E[M^\xi(\sigma)]$ for every $\tau,\sigma\in\cT_0(\F^1)$, $\sigma\le \tau$ so that the submartingale property can be deduced by Lemma \ref{lem:M}.

First we argue on the set $\{\sigma < T \}$. By the definition of $M^\xi$ and noticing that 
\begin{equation*}
\int_{[0,\sigma)}(1-\zeta^*_t)f_t\ud\xi_t+\sum_{t\in[0,\sigma)}h_t\Delta\zeta^*_t\Delta\xi_t=\int_{[0,\sigma)}\big[(1-\zeta^*_t)f_t+h_t\Delta\zeta^*_t\big]\ud\xi_t,
\end{equation*}
we have
\begin{equation}\label{eq:subm00a}
\begin{aligned}
M^\xi(\sigma)&=\E\Big[\int_{[0,\sigma)}\big[(1\!-\!\zeta^*_t)f_t\!+\!h_t\Delta\zeta^*_t\big]\ud\xi_t +\!\int_{[0,\sigma)}\!(1\!-\!\xi_t)g_t\ud\zeta^*_t \Big| \cF^1_\sigma \Big]\\
&\quad 
+\!(1\!-\!\xi_{\sigma-})\E[1\!-\!\zeta^*_{\sigma-}|\cF^1_\sigma]V^{*,1}(\sigma).
\end{aligned}
\end{equation}
We recall from \eqref{eq:V1V2} the definition of $V^{*,1}(\sigma)$ and obtain
\begin{equation}
\begin{aligned}\label{eq:subm0a}
&\E[1-\zeta^*_{\sigma-}|\cF^1_\sigma]V^{*,1}(\sigma)\\
&= 
\E[1-\zeta^*_{\sigma-}|\cF^1_\sigma] \essinf_{\theta \in \cT^R_\sigma(\F^1)} \E\Big[\Pi^{*,1}_\sigma \Big(f_\theta(1-\zeta^{*;\sigma}_\theta)+\int_{[\sigma,\theta)}g_t\ud \zeta^{*;\sigma}_t+h_\theta\Delta\zeta^{*;\sigma}_\theta\Big) \Big|\cF^1_\sigma\Big]\\
&= 
\essinf_{\theta \in \cT^R_\sigma(\F^1)} \E\Big[f_\theta(1-\zeta^{*}_\theta)+\int_{[\sigma,\theta)}g_t\ud \zeta^{*}_t+h_\theta\Delta\zeta^{*}_\theta \Big|\cF^1_\sigma\Big]\\
&\le
\E\Big[f_{\bar\theta}(1-\zeta^*_{\bar\theta})+\int_{[\sigma,{\bar\theta})}g_t\ud \zeta^*_t+h_{\bar\theta}\Delta\zeta^*_{\bar\theta}\Big|\cF^1_\sigma\Big]\quad\text{for any $\bar\theta\in\cT^R_\sigma(\F^1)$,}
\end{aligned}
\end{equation}
where for the second equality we use that $\E[1-\zeta^*_{\sigma-}|\cF^1_\sigma]\Pi^{*,1}_\theta=1-\zeta^*_{\sigma-}$ by the definition of $\Pi^{*,1}_\theta$ in \eqref{eq:Pi}. In particular, we choose $\bar\theta\in\cT^R_\sigma(\F^1)$ generated by a process $\bar \xi\in\cAcirc_\sigma(\F^1)$ of the form
\[
\bar\xi_t=\xi^{\sigma}_t \ind_{\{t\in[\sigma,\tau)\}}+[\xi^{\sigma}_{\tau-}+(1-\xi^\sigma_{\tau-})\ind_{\{t\ge \eta\}}]\ind_{\{t\ge \tau \}},
\]
for an arbitrary $\eta\in\cT_\tau(\F^1)$ and where $\xi^\sigma$ is the truncated control $\xi$ at time $\sigma$. The increasing process $\bar\xi$ follows the truncated control $\xi^\sigma$ between time $\sigma$ and time $\tau$ and then it has a single jump to $1$ at time $\eta$.
Such choice of $\bar \theta\in\cT^R_\sigma(\F^1)$ in \eqref{eq:subm0a} yields 
\begin{equation}
\begin{aligned}\label{eq:subm01a}
&\E[1-\zeta^*_{\sigma-}|\cF^1_\sigma]V^{*,1}(\sigma)\\
&\le\E\Big[\int_{[\sigma,\tau)}(1-\zeta^*_t)f_t\ud \xi^\sigma_t+\int_{[\sigma,\tau)}(1-\xi^\sigma_t)g_t\ud \zeta^*_t+\sum_{s\in[\sigma,\tau)}h_s\Delta\zeta^*_s\Delta\xi^\sigma_s\\
&\qquad+(1-\xi^\sigma_{\tau-})\E\Big[f_\eta(1-\zeta^*_\eta)+\int_{[\tau,\eta)}g_t\ud \zeta^*_t+h_\eta\Delta\zeta^*_\eta\Big|\cF^1_\tau\Big]\Big|\cF^1_\sigma\Big]\\
&=
\E\Big[\int_{[\sigma,\tau)}(1-\zeta^*_t)f_t\ud \xi^\sigma_t+\int_{[\sigma,\tau)}(1-\xi^\sigma_t)g_t\ud \zeta^*_t+\sum_{s\in[\sigma,\tau)}h_s\Delta\zeta^*_s\Delta\xi^\sigma_s\\
&\qquad+(1-\xi^\sigma_{\tau-})\E[1-\zeta^*_{\tau-}|\cF^1_\tau]\E\Big[\Pi^{*, 1}_{\tau} \Big( f_{\eta}(1-\zeta^{*;\tau}_{\eta})+\int_{[\tau,\eta)}g_t\ud \zeta^{*;\tau}_t+h_{\eta}\Delta\zeta^{*;\tau}_{\eta} \Big)\Big|\cF^1_\tau\Big]\Big|\cF^1_\sigma\Big]\\
&=
\E\Big[\int_{[\sigma,\tau)}(1-\zeta^*_t)f_t\ud \xi^\sigma_t+\int_{[\sigma,\tau)}(1-\xi^\sigma_t)g_t\ud \zeta^*_t+\sum_{s\in[\sigma,\tau)}h_s\Delta\zeta^*_s\Delta\xi^\sigma_s\\
&\qquad+(1-\xi^\sigma_{\tau-})\E[1-\zeta^*_{\tau-}|\cF^1_\tau]J^{\Pi^{*,1}_\tau}(\eta,\sigma_*^\tau)|\cF^1_\tau)\Big|\cF^1_\sigma\Big],
\end{aligned}
\end{equation}
where in the final expression we use the notation $\sigma_*^\tau$ for the randomised stopping time generated by $\zeta^{*;\tau}$ (cf.\ Lemma \ref{lem:pure}). By Corollary \ref{cor:pure_ud} we can take a sequence $(\eta_n)_{n\in\N} \subset \cT_\tau(\F^1)$ such that $\P$-a.s.
\begin{equation}\label{eqn:eta_n_limita}
\lim_{n \to \infty}J^{\Pi^{*,1}_\tau}(\eta_n,\sigma(\zeta^{*;\tau})|\cF^1_\tau), 
= V^{*,1}(\tau)
\end{equation}
and the limit is monotone from above (although this is a feature which does not play a role in the arguments below). Equation \eqref{eq:subm01a} with $\eta_n$ instead of $\eta$ yields
\begin{align*}
\E[1-\zeta^*_{\sigma-}|\cF^1_\sigma]V^{*,1}(\sigma)&\le \E\Big[\int_{[\sigma,\tau)}(1-\zeta^*_t)f_t\ud \xi^\sigma_t+\int_{[\sigma,\tau)}(1-\xi^\sigma_t)g_t\ud \zeta^*_t+\sum_{s\in[\sigma,\tau)}h_s\Delta\zeta^*_s\Delta\xi^\sigma_s\\
&\qquad+(1-\xi^\sigma_{\tau-})\E[1-\zeta^*_{\tau-}|\cF^1_\tau]J^{\Pi^{*,1}_\tau}(\eta_n,\sigma_*^\tau)|\cF^1_\tau)\Big|\cF^1_\sigma\Big]\\
&\hspace{-21pt}\xrightarrow{n \to \infty}
\E\Big[\int_{[\sigma,\tau)}(1-\zeta^*_t)f_t\ud \xi^\sigma_t+\int_{[\sigma,\tau)}(1-\xi^\sigma_t)g_t\ud \zeta^*_t+\sum_{s\in[\sigma,\tau)}h_s\Delta\zeta^*_s\Delta\xi^\sigma_s\\
&\qquad+(1-\xi^\sigma_{\tau-})\E[1-\zeta^*_{\tau-}|\cF^1_\tau]V^{*,1}(\tau)\Big|\cF^1_\sigma\Big],
\end{align*}
where the limit holds by the dominated convergence theorem and \eqref{eqn:eta_n_limita}. Substituting into \eqref{eq:subm00a}, adding the trivial equality $M^\xi(\sigma) = M^\xi(\tau) = M^\xi(T)$ on $\{\sigma = T\}$ and taking expectation we obtain
\begin{equation*}
\begin{aligned}
\E[M^\xi(\sigma)]&= \E[ \indd{\sigma < T} M^\xi(\sigma) + \indd{\sigma = T} M^\xi(\sigma) ]\\
&\le \E\Big[\indd{\sigma < T}\Big( \int_{[0,\tau)}\![(1\!-\!\zeta^*_t)f_t\!+\!h_t\Delta\zeta^*_t]\ud\xi_t\!+\!\int_{[0,\tau)}(1\!-\!\xi_t)g_t\ud\zeta^*_t\!+\!(1\!-\!\xi_{\tau-})\E[1\!-\!\zeta^*_{\tau-}|\cF^1_\tau]V^{*,1}(\tau)\Big)\\
&\qquad+ \indd{\sigma=T} M^\xi(T)\Big]\\
&=\E[M^\xi(\tau)],
\end{aligned}
\end{equation*}
where we used that for $t\in[\sigma,\tau]$ the identities hold  
$(1-\xi_{\sigma-})(1-\xi^\sigma_{t})=(1-\xi_t)$ and $(1-\xi_{\sigma-})\ud \xi^\sigma_t =\ud\xi_t$. 
The above is the required inequality for the submartingale property of the family.
\end{proof}

Finally, we obtain a link between the equilibrium values of the two players. Informally, we say that such link is obtained using the information available to both players, in the sense that we consider conditional expectations with respect to the common filtration. Recall the notation $\cF^{1,2}_t=\cF^1_t\cap\cF^2_t$, $t \in [0, T]$, and $\F^{1,2}= (\cF^{1,2}_t)_{t \in [0, T]}$.
\begin{proposition}\label{prop:link}
Let $(\xi^*,\zeta^*)\in\cAcirc_0(\F^1)\times \cAcirc_0(\F^2)$ be an optimal pair and recall the families ${\bf V}^{*,1}$ and ${\bf V}^{*,2}$ from \eqref{eq:families}. Then, for any $\lambda\in\cT_0(\F^{1,2})$ it holds
\begin{align}\label{eq:link}
\begin{aligned}
&\E\big[(1-\xi^*_{\lambda-})\E[1-\zeta^*_{\lambda-}|\cF^1_\lambda]V^{*,1}(\lambda)|\cF^{1,2}_\lambda\big]\\
&=\E\big[(1-\zeta^*_{\lambda-})\E[1-\xi^*_{\lambda-}|\cF^2_\lambda]V^{*,2}(\lambda)|\cF^{1,2}_\lambda\big]\\
&=\E\Big[\int_{[\lambda,T)}(1-\zeta^*_t)f_t\ud \xi^*_t+\int_{[\lambda,T)}\big[(1-\xi^*_t)g_t+h_t\Delta\xi^*_t\big]\ud \zeta^*_t+h_T\Delta\zeta^*_T\Delta\xi^*_T\Big|\cF^{1,2}_\lambda\Big].
\end{aligned}
\end{align}
\end{proposition}
\begin{proof}
From the definition of $V^{*,1}$ and the assumed optimality of the pair $(\xi^*,\zeta^*)$ we have
\begin{align}\label{eq:merge1}
\begin{aligned}
&\E\big[(1-\xi^*_{\lambda-})\E[1-\zeta^*_{\lambda-}|\cF^1_\lambda]V^{*,1}(\lambda)|\cF^{1,2}_\lambda\big]\\
&=\E\big[(1-\xi^*_{\lambda-})\E[1-\zeta^*_{\lambda-}|\cF^1_\lambda]\essinf_{\xi\in\cAcirc_{\lambda}(\F^1)}J^{\Pi^{*,1}_\lambda}\big(\xi,\zeta^{*;\lambda}\big|\cF^1_\lambda\big)|\cF^{1,2}_\lambda\big]\\
&=\E\big[(1-\xi^*_{\lambda-})\E[1-\zeta^*_{\lambda-}|\cF^1_\lambda]J^{\Pi^{*,1}_\lambda}\big(\xi^{*;\lambda},\zeta^{*;\lambda}\big|\cF^1_\lambda\big)|\cF^{1,2}_\lambda\big],
\end{aligned}
\end{align}
where the final equality uses the optimality of $\xi^{*;\lambda}$ (cf.\ Proposition \ref{thm:aggr2}). Then, the tower property and the definition of $\Pi^{*,1}_\lambda$ yield
\begin{align*}
\begin{aligned}
&\E\big[(1-\xi^*_{\lambda-})\E[1-\zeta^*_{\lambda-}|\cF^1_\lambda]V^{*,1}(\lambda)|\cF^{1,2}_\lambda\big]\\
&=\E\Big[\int_{[\lambda,T)}\big[(1-\zeta^*_t)f_t+h_t\Delta\zeta^*_t\big]\ud \xi^*_t+\int_{[\lambda,T)}(1-\xi^*_t)g_t\ud \zeta^*_t+h_T\Delta\zeta^*_T\Delta\xi^*_T\Big|\cF^{1,2}_\lambda\Big].
\end{aligned}
\end{align*}
By analogous arguments, we obtain
\begin{align}\label{eq:merge2}
\begin{aligned}
&\E\big[(1-\zeta^*_{\lambda-})\E[1-\xi^*_{\lambda-}|\cF^2_\lambda]V^{*,2}(\lambda)|\cF^{1,2}_\lambda\big]\\
&=\E\big[(1-\zeta^*_{\lambda-})\E[1-\xi^*_{\lambda-}|\cF^2_\lambda]\esssup_{\zeta\in\cAcirc_{\lambda}}J^{\Pi^{*,2}_\lambda}\big(\xi^{*;\lambda},\zeta\big|\cF^2_\lambda\big)|\cF^{1,2}_\lambda\big]\\
&=\E\Big[\int_{[\lambda,T)}(1-\zeta^*_t)f_t\ud \xi^*_t+\int_{[\lambda,T)}\big[(1-\xi^*_t)g_t+h_t\Delta\xi^*_t\big]\ud \zeta^*_t+h_T\Delta\zeta^*_T\Delta\xi^*_T\Big|\cF^{1,2}_\lambda\Big].
\end{aligned}
\end{align}
This concludes the proof upon noticing $\int_{[\lambda,T)}h_t\Delta\zeta^*_t\ud \xi^*_t=\int_{[\lambda,T)}h_t\Delta\xi^*_t\ud \zeta^*_t=\sum_{t\in[\lambda,T)}h_t\Delta\xi^*_t\Delta\zeta^*_t$.
\end{proof}
\begin{remark}
Notice that the first two lines of \eqref{eq:link} equivalently read
\[
\begin{aligned}
&\E\big[(1-\xi^*_{\lambda-})(1-\zeta^*_{\lambda-}) V^{*,1}(\lambda)|\cF^{1,2}_\lambda\big]
=\E\big[(1-\zeta^*_{\lambda-})(1-\xi^*_{\lambda-})V^{*,2}(\lambda)|\cF^{1,2}_\lambda\big],
\end{aligned}
\]
because $(1-\xi^*_{\lambda-}) V^{*,1}(\lambda)$ is $\cF^1_\lambda$-measurable and $(1-\zeta^*_{\lambda-}) V^{*,2}(\lambda)$ is $\cF^2_\lambda$-measurable, so that the tower property yields the required transformation of the first two expressions in \eqref{eq:link}. We formulated \eqref{eq:link} with the additional conditional expectation due to the special role played by the quantities $\E[1-\zeta^*_{\lambda-}|\cF^1_\lambda]V^{*,1}(\lambda)$ and $\E[1-\xi^*_{\lambda-}|\cF^2_\lambda]V^{*,2}(\lambda)$ in Theorem \ref{thm:value}.
\end{remark}

\begin{remark}\label{rem:3.13}
When $\F^{1,2}=\{\Omega,\varnothing\}$, \eqref{eq:link} reduces to 
\begin{align*}
&\E\big[(1-\xi^*_{\lambda-})\E[1-\zeta^*_{\lambda-}|\cF^1_\lambda]V^{*,1}(\lambda)\big]=\E\big[(1-\zeta^*_{\lambda-})\E[1-\xi^*_{\lambda-}|\cF^2_\lambda]V^{*,2}(\lambda)\big],
\end{align*}
for deterministic times $\lambda\in[0,T]$. 

When $\F^1\supset\F^2$, then \eqref{eq:link} holds for any $\lambda\in\cT_0(\F^2)$ in a more explicit form:
\begin{align*}
&(1-\zeta^*_{\lambda-})\E\big[(1-\xi^*_{\lambda-})V^{*,1}(\lambda)|\cF^2_\lambda\big]=\E\big[1-\xi^*_{\lambda-}|\cF^2_\lambda\big](1-\zeta^*_{\lambda-})V^{*,2}(\lambda).
\end{align*}
\end{remark}

The final expression in \eqref{eq:link} can be related to the {\em ex-ante value} of the game. To make the statement rigorous, we introduce the family ${\bf \bar V}\coloneqq\{\bar V(\lambda),\,\lambda\in\cT_0(\F^{1,2})\}$, where
\begin{align*}
\bar V(\lambda)&\coloneqq \E\Big[\int_{[\lambda,T)}(1-\zeta^*_t)f_t\ud \xi^*_t+\int_{[\lambda,T)}\big[(1-\xi^*_t)g_t+h_t\Delta\xi^*_t\big]\ud \zeta^*_t+h_T\Delta\zeta^*_T\Delta\xi^*_T\Big|\cF^{1,2}_\lambda\Big]\\
&=\E\Big[\int_{[\lambda,T)}(1-\zeta^*_t)\big[f_t+h_t\Delta\zeta^*_t\big]\ud \xi^*_t+\int_{[\lambda,T)}(1-\xi^*_t)g_t\ud \zeta^*_t+h_T\Delta\zeta^*_T\Delta\xi^*_T\Big|\cF^{1,2}_\lambda\Big],
\end{align*}
where the second equality is simply eliciting the symmetry of the expected payoffs. 
The family ${\bf \bar V}$ can be aggregated into an $\F^{1,2}$-optional process $(\bar V_t)_{t \in [0, T]}$, thanks to the following observation:
\[
\bar V(\lambda) 
= \E\big[ (1-\xi^*_{\lambda-}) (1-\zeta^*_{\lambda-}) V^{*,1}(\lambda) \big| \cF^{1,2}_\lambda \big]
= \E\big[ (1-\xi^*_{\lambda-}) \hat V^{*,1}_\lambda \big| \cF^{1,2}_\lambda \big],
\]
where the first equality is by \eqref{eq:link} and the second stems from the fact that $(\hat V_t)_{t \in [0, T]}$ aggregates the family $\{\E[1-\zeta^*_{\lambda-}|\cF^1_\theta] V^{*,1}(\theta):\ \theta \in \cT_0(\F^1) \}$ and $\cT_0(\F^{1,2}) \subset \cT_0(\F^1)$. Hence, the process $(\bar V_t)_{t \in [0, T]}$ that aggregates the family ${\bf \bar V}$ is the $\F^{1,2}$-optional projection of the process $((1-\xi^*_{t-}) \hat V^{*,1}_t)_{t \in [0, T]}$ or equivalently of the process $((1-\zeta^*_{t-})\hat V^{*,2}_t)_{t\in[0,T]}$.

\begin{corollary}\label{corr:link}
Let $(\xi^*,\zeta^*)\in\cAcirc_0(\F^1)\times \cAcirc_0(\F^2)$ be an optimal pair. For any $\lambda\in\cT_0(\F^{1,2})$, set 
\[
\Pi^*_\lambda\coloneqq\frac{(1-\xi^*_{\lambda-})(1-\zeta^*_{\lambda-})}{\E\big[(1-\xi^*_{\lambda-})(1-\zeta^*_{\lambda-})\big|\cF^{1,2}_\lambda\big]}\in\cR(\cF^{1,2}_\lambda),
\]
with the convention adopted in \eqref{eq:Pi} that $0/0=1$ (cf.\ \eqref{eq:impl} for a justification of this choice).

Define $V(\lambda) \coloneqq J^{\Pi^*_\lambda}(\xi^{*;\lambda},\zeta^{*;\lambda}|\cF^{1,2}_\lambda)$. Then, 
we have
\begin{align}\label{eq:exanteV}
\begin{aligned}
V(\lambda) =\essinf_{\xi\in \cAcirc_\lambda(\F^1)}\esssup_{\zeta\in \cAcirc_\lambda(\F^2)}J^{\Pi^*_\lambda}(\xi,\zeta|\cF^{1,2}_\lambda)=\esssup_{\zeta\in \cAcirc_\lambda(\F^2)}\essinf_{\xi\in \cAcirc_\lambda(\F^1)}J^{\Pi^*_\lambda}(\xi,\zeta|\cF^{1,2}_\lambda),
\end{aligned}
\end{align}
on the event $\{\E[(1-\xi^*_{\lambda-})(1-\zeta^*_{\lambda-}) |\cF^{1,2}_\lambda] > 0 \}$.
Moreover, it holds
\begin{align}\label{eq:VbarV}
\bar V(\lambda)= \E\big[(1-\xi^*_{\lambda-})(1-\zeta^*_{\lambda-})\big|\cF^{1,2}_\lambda\big]V(\lambda).
\end{align}
\end{corollary}
\begin{proof}
From the second line of \eqref{eq:merge1}, using the minimising sequence from Lemma \ref{lem:ud} and the monotone convergence we get the second equality below (cf.\ \eqref{eq:commute})
\begin{align}\label{eq:barVl}
\begin{aligned}
\bar V(\lambda)&=\E\big[(1-\xi^*_{\lambda-})\E[1-\zeta^*_{\lambda-}|\cF^1_\lambda]V^{*,1}(\lambda)\,\big|\cF^{1,2}_\lambda\big]\\
&=\E\big[(1-\xi^*_{\lambda-})\E[1-\zeta^*_{\lambda-}|\cF^1_\lambda]\essinf_{\xi\in\cAcirc_{\lambda}(\F^1)}J^{\Pi^{*,1}_\lambda}(\xi,\zeta^{*;\lambda}|\cF^1_\lambda)\,\big|\cF^{1,2}_\lambda\big]\\
&=\essinf_{\xi\in\cAcirc_{\lambda}(\F^1)}\E\big[(1-\xi^*_{\lambda-})\E[1-\zeta^*_{\lambda-}|\cF^1_\lambda]J^{\Pi^{*,1}_\lambda}(\xi,\zeta^{*;\lambda}|\cF^1_\lambda)\,\big|\cF^{1,2}_\lambda\big].
\end{aligned}
\end{align}
Using the expression for $J^{\Pi^{*,1}_\lambda}\big(\xi,\zeta^{*;\lambda}\big|\cF^1_\lambda\big)$ and the tower property
\begin{align*}
\begin{aligned}
&\E\big[(1-\xi^*_{\lambda-})\E[1-\zeta^*_{\lambda-}|\cF^1_\lambda]J^{\Pi^{*,1}_\lambda}\big(\xi,\zeta^{*;\lambda}\big|\cF^1_\lambda\big)|\cF^{1,2}_\lambda\big]\\
&=\E\Big[(1-\xi^*_{\lambda-})(1-\zeta^*_{\lambda-})\Big(\int_{[\lambda,T)}\!\!\big[(1\!-\!\zeta^{*;\lambda}_T)f_t\!+\!h_t\Delta\zeta^{*;\lambda}_t\big]\ud \xi_t\\
&\hspace{120pt}+\!\int_{[\lambda,T)}\!\!(1\!-\!\xi_t)g_t\ud \zeta^{*;\lambda}_t\!+\!h_T\Delta\xi_T\Delta\zeta^{*;\lambda}_T \Big)\Big|\cF^{1,2}_\lambda\Big]\\
&=\E\big[(1-\xi^*_{\lambda-})(1-\zeta^*_{\lambda-})\big|\cF^{1,2}_\lambda\big]J^{\Pi^*_\lambda}\big(\xi,\zeta^{*;\lambda}\big|\cF^{1,2}_\lambda\big).
\end{aligned}
\end{align*}
Combining the two expressions above we deduce
\begin{align}\label{eq:ex1}
\begin{aligned}
\bar V(\lambda)&=\E\big[(1-\xi^*_{\lambda-})(1-\zeta^*_{\lambda-})\big|\cF^{1,2}_\lambda\big]\essinf_{\xi\in\cAcirc_{\lambda}(\F^1)}J^{\Pi^*_\lambda}\big(\xi,\zeta^{*;\lambda}\big|\cF^{1,2}_\lambda\big)\\
&\le \E\big[(1-\xi^*_{\lambda-})(1-\zeta^*_{\lambda-})\big|\cF^{1,2}_\lambda\big]\esssup_{\zeta\in\cAcirc_{\lambda}(\F^2)}\essinf_{\xi\in\cAcirc_{\lambda}(\F^1)}J^{\Pi^*_\lambda}\big(\xi,\zeta\big|\cF^{1,2}_\lambda\big).
\end{aligned}
\end{align}

For the reverse inequality, we start from the second line in \eqref{eq:merge2} and follow the same steps as above. That yields 
\begin{align}\label{eq:ex2}
\begin{aligned}
\bar V(\lambda)&=\E\big[(1-\zeta^*_{\lambda-})\E[1-\xi^*_{\lambda-}|\cF^2_\lambda]V^{*,2}(\lambda)\,\big|\cF^{1,2}_\lambda\big]\\
&=\esssup_{\zeta\in\cAcirc_{\lambda}(\F^2)}\E\big[(1-\zeta^*_{\lambda-})\E[1-\xi^*_{\lambda-}|\cF^2_\lambda]J^{\Pi^{*,2}_\lambda}(\xi^{*;\lambda},\zeta|\cF^2_\lambda)\,\big|\cF^{1,2}_\lambda\big]\\
&=\E\big[(1-\xi^*_{\lambda-})(1-\zeta^*_{\lambda-})\big|\cF^{1,2}_\lambda\big]\esssup_{\zeta\in\cAcirc_{\lambda}(\F^2)}J^{\Pi^*_\lambda}(\xi^{*;\lambda},\zeta|\cF^{1,2}_\lambda)\\
&\ge \E\big[(1-\xi^*_{\lambda-})(1-\zeta^*_{\lambda-})\big|\cF^{1,2}_\lambda\big]\essinf_{\xi\in\cAcirc_{\lambda}(\F^1)}\esssup_{\zeta\in\cAcirc_{\lambda}(\F^2)}J^{\Pi^*_\lambda}(\xi,\zeta|\cF^{1,2}_\lambda),
\end{aligned}
\end{align}
where the third equality is by the use of the tower property and the definition of $J^{\Pi^{*,2}_\lambda}\big(\xi^{*;\lambda},\zeta\big|\cF^2_\lambda\big)$.

Since 
\[
\essinf_{\xi\in\cAcirc_{\lambda}(\F^1)}\esssup_{\zeta\in\cAcirc_{\lambda}(\F^2)}J^{\Pi^*_\lambda}\big(\xi^{*;\lambda},\zeta\big|\cF^{1,2}_\lambda\big)\ge \esssup_{\zeta\in\cAcirc_{\lambda}(\F^2)}\essinf_{\xi\in\cAcirc_{\lambda}(\F^1)}J^{\Pi^*_\lambda}\big(\xi^{*;\lambda},\zeta\big|\cF^{1,2}_\lambda\big),
\]
the inequalities \eqref{eq:ex1} and \eqref{eq:ex2} yield \eqref{eq:VbarV} and the second equality in \eqref{eq:exanteV}. The first equality in \eqref{eq:exanteV} is easily deduced from the first line of \eqref{eq:barVl} and \eqref{eq:ex2}, using the optimality of truncated controls (cf.\ Proposition \ref{thm:aggr2}). 
\end{proof}

The random variable $V(\lambda)$ is a dynamic analogue of the ex-ante value of the game starting at time $\lambda$. In the expression for $V(\lambda)$, players optimise the expected payoff conditional on the jointly available information at time $\lambda$ (i.e., conditioning on $\cF^{1,2}_\lambda$). However, they still pick their randomised stopping times making use of their individual filtrations $\F^1$ and $\F^2$. That is to say, players know that they will have access to the full content of their individual filtrations after time $\lambda$.

\subsection{Structure of optimal strategies}\label{subsec:structure}
In this section we look into some structural properties of optimal strategies. 
The main results in this section are Proposition \ref{prop:support} and Corollary \ref{cor:support} but they require some preparation which we do in Propositions \ref{prop:meas}, \ref{prop:mnz} and Corollary \ref{cor:M_zero_mart}. Throughout the section we maintain Notation \ref{not:equil} for the optimal pairs $(\xi^*,\zeta^*)$ and $(\tau_*,\sigma_*)$.

Combining the expressions of the aggregated \cadlag processes $M^*$ and $N^*$ from Proposition \ref{thm:aggr2} and of the optional processes $\hat V^{*,1}$ and $\hat V^{*,2}$ from Theorem \ref{thm:value} we have, for any $(\theta, \gamma) \in \cT_0(\F^1) \times \cT_0(\F^2)$,
\begin{align*}
M^*_\theta&=\E\Big[\int_{[0,\theta)}\Big(f_s(1-\zeta^*_s)+h_s\Delta\zeta^*_s\Big)\ud \xi^*_s+\int_{[0,\theta)}(1-\xi^*_s)g_s\ud \zeta^*_s\Big|\cF^1_\theta\Big]+(1-\xi^*_{\theta-})\hat V^{*,1}_\theta\\
&=\E\Big[\ind_{\{\tau_*<\theta\}}\Big(f_{\tau_*}\ind_{\{\tau_*<\sigma_*\}}+h_{\tau_*}\ind_{\{\tau_*=\sigma_*\}}\Big)+\ind_{\{\sigma_*<\theta\}}\ind_{\{\sigma_*<\tau_*\}}g_{\sigma_*}\Big|\cF^1_\theta\Big]+\E[\ind_{\{\theta\le \tau_*\}}|\cF^1_\theta]\hat V^{*,1}_\theta,\\
N^*_\gamma&=\E\Big[\int_{[0,\gamma)}f_s(1-\zeta^*_s)\ud \xi^*_s+\int_{[0,\gamma)}\Big((1-\xi^*_s)g_s+h_s\Delta\xi^*_s\Big)\ud \zeta^*_s\Big|\cF^2_\gamma\Big]+(1-\zeta^*_{\gamma-})\hat V^{*,2}_\gamma\\
&=\E\Big[\ind_{\{\tau_*<\gamma\}}\ind_{\{\tau_*<\sigma_*\}}f_{\tau_*}+\ind_{\{\sigma_*<\gamma\}}\Big(\ind_{\{\sigma_*<\tau_*\}}g_{\sigma_*}+h_{\sigma_*}\ind_{\{\tau_*=\sigma_*\}}\Big)\Big|\cF^2_\gamma\Big]+\E[\ind_{\{\gamma\le \sigma_*\}}|\cF^2_\gamma]\hat V^{*,2}_\gamma.
\end{align*}
We will explicitly write the integration of players' randomisation devices and in doing so identify stopping times over which the players randomise.
Recall $(\bar \tau_*(z),\bar\sigma_*(z))\in\cT_0(\F^1)\times\cT_0(\F^2)$ from \eqref{eq:taubar}, which we denote here $(\tau_*(z),\sigma_*(z))$ for the simplicity of notation. 

The next characterisation of $M^*_\theta$ and $N^*_\gamma$ proves useful.
\begin{proposition}\label{prop:meas}
There are mappings $(z,\omega)\mapsto m^*(\theta; z)(\omega)$ and $(z,\omega)\mapsto n^*(\gamma; z)(\omega)$ which are measurable for the $\sigma$-algebras $\cB([0,1])\times\cF^1_\theta$ and $\cB([0,1])\times\cF^2_\gamma$, respectively, and such that:
\begin{itemize}
\item[(i)] it holds
\begin{equation}\label{eqn:mnstar}
M^*_\theta=\int_0^1 m^*(\theta; z)\ud z\quad\text{and}\quad N^*_\gamma=\int_0^1 n^*(\gamma; z)\ud z;
\end{equation}
\item[(ii)] for each $z\in[0,1]$, $\P$-a.s.,
\begin{align*}
\begin{aligned}
m^*(\theta; z)&=\E\Big[\ind_{\{\tau_*(z)<\theta\}}\Big(f_{\tau_*(z)}\ind_{\{\tau_*(z)<\sigma_*\}}\!+\!h_{\tau_*(z)}\ind_{\{\tau_*(z)=\sigma_*\}}\Big)\Big|\cF^1_\theta\Big]\\
&\quad+\E\Big[\ind_{\{\sigma_*<\theta\}}\ind_{\{\sigma_*<\tau_*(z)\}}g_{\sigma_*}\Big|\cF^1_\theta\Big]\!+\!\ind_{\{\theta\le \tau_*(z)\}}\hat V^{*,1}_\theta,
\end{aligned}
\end{align*}
and
\begin{align*}
n^*(\gamma; z)&=
\E\Big[\ind_{\{\sigma_*(z)<\gamma\}}\Big(\ind_{\{\sigma_*(z)<\tau_*\}}g_{\sigma_*(z)}\!+\!h_{\sigma_*(z)}\ind_{\{\tau_*=\sigma_*(z)\}}\Big)\Big|\cF^2_\gamma\Big]\\
&\quad+\E\Big[\ind_{\{\tau_*<\gamma\}}\ind_{\{\tau_*<\sigma_*(z)\}}f_{\tau_*}\Big|\cF^2_\gamma\Big]\!+\!\ind_{\{\gamma\le \sigma_*(z)\}}\hat V^{*,2}_\gamma.
\end{align*}
\end{itemize}
\end{proposition}

The result can be interpreted as an application of Fubini's theorem, although some care is needed because the conditional expectations on the right-hand side of the equations in (ii) are not necessarily jointly measurable in $(z,\omega)$ and therefore we need to consider suitable modifications $m^*(\theta;z)(\omega)$ and $n^*(\gamma;z)(\omega)$. Although the result is not surprising, its proof is slightly technical and we provide it in Appendix \ref{app:propmeas}. 

\begin{proposition}\label{prop:mnz}
Families $\{m^*(\theta; z), \theta \in \cT_0(\F^1) \}$ and $\{n^*(\gamma; z), \gamma \in \cT_0(\F^2) \}$ are of class $(D)$ and can be aggregated into optional processes $(m^*_t(z))_{t \in [0, T]}$ and $(n^*_t(z))_{t \in [0, T]}$, respectively, for every $z\in[0,1]$. 

Furthermore, $(m^*_t(z))_{t\in[0,T]}$ is a \cadlag $\F^1$-martingale and $(n^*_t(z))_{t\in[0,T]}$ is a \cadlag $\F^2$-martingale for almost every $z \in [0, 1]$.
We also have
\begin{equation}\label{eqn:final_statement}
M^*_\theta=\int_0^1 m^*_\theta (z)\ud z, \qquad N^*_\gamma=\int_0^1 n^*_\gamma (z)\ud z,
\end{equation}
for any $(\theta, \gamma) \in \cT_0(\F^1) \times \cT_0(\F^2)$.
\end{proposition}
\begin{proof}
The class $(D)$ property is immediate because $f,g,h\in\cL_b(\P)$. We recall that the family ${\bf M}^\xi$ can be aggregated into an optional submartingale process $(M^\xi_t,\F^1,\P)_{t\in[0,T]}$ (cf.\ Proposition \ref{prop:subsupmg}) for any $\xi\in\cAcirc_0(\F^1)$. Then we observe that taking $\xi_t=\ind_{\{\tau_*(z)\le t\}}$ in the definition of $M^\xi$ we obtain $M^\xi_\theta=m^*(\theta;z)$ for any $\theta \in \cT_0(\F^1)$. Hence, for every $z\in[0,1]$ and for $\xi_t=\ind_{\{\tau_*(z)\le t\}}$
\[
(m^*_t(z))_{t\in[0,T]} \coloneqq (M^\xi_t)_{t \in [0, T]},
\] 
is an $\F^1$-optional submartingale that aggregates the family $\{m^*(\theta; z), \theta \in \cT_0(\F^1) \}$. 

From \eqref{eqn:mnstar} and from the fact that $m^*_t(z) = m^*(t; z)$, $\P$-a.s. for any $z \in [0, 1]$ we obtain
\begin{equation}\label{eqn:Mm}
\E [M^*_T - M^*_0] = \int_0^1 \E \big[m^*(T; z) - m^*(0; z)\big] \ud z
= \int_0^1 \E [m^*_T(z) - m^*_0(z)] \ud z.
\end{equation}
We can only guarantee the measurability of the map $z \mapsto \E [m^*_T(z) - m^*_0(z)]$ by referring to the measurability of $(z,\omega)\mapsto [m^*(T; z)(\omega) - m^*(0; z)(\omega)]$ (cf.\ Proposition \ref{prop:meas}) and to the fact that $m^*_T(z) = m^*(T; z)$ and $m^*_0(z) = m^*(0; z)$, $\P$\as
However, the map $z \mapsto m^*_T(z)$ may not be measurable, hence, the order of integration on the right-hand side of \eqref{eqn:Mm} cannot be interchanged.
 
Recall that $(m^*_t(z))_{t \in [0, T]}$ is an $\F^1$-submartingale for every $z \in [0,1]$. This implies that $\E [m^*_T(z) - m^*_0(z)] \ge 0$. We also note that $(M^*_t)_{t \in [0, T]}$ is an $\F^1$-martingale, so $\E [M^*_T - M^*_0] = 0$. These two observations and \eqref{eqn:Mm} lead to the conclusion that $\E [m^*_T(z) - m^*_0(z)] = 0$ for almost every $z \in [0, 1]$. When we combine this with the submartingale property of $(m^*_t(z))_{t \in [0, T]}$, we see that $(m^*_t(z))_{t \in [0, T]}$ is an $\F^1$-optional martingale for almost every $z \in [0,1]$. Then, by Corollary \ref{cor:aggr} it is indistinguishable from a \cadlag martingale: indeed for a.e.\ $z\in[0,1]$, the family $\{m^*_\theta(z),\,\theta\in\cT_0(\F^1)\}$ can be aggregated into a \cadlag martingale $(\bar m^z_t)_{t\in[0,T]}$ (i.e., $\P(m^*_\theta(z)=\bar m^z_\theta)=1$ for any $\theta \in \cT_0(\F^1)$). Since $m^*_t(z)$ and $\bar m^z_t$ are optional processes that coincide on stopping times, they are indistinguishable \cite[Lemma VI.5.2]{RogersWilliams}. Thus, $m^*_t(z)$ is itself \cadlag.

The arguments for $(n^*_t(z))_{t \in [0, T]}$ are analogous and, therefore, omitted. Equalities \eqref{eqn:final_statement} are
proven by combining \eqref{eqn:mnstar} and the aggregation result.
\end{proof}

A corollary links the above result to properties of processes $M^0$ and $N^0$ from Proposition \ref{thm:aggr1}. 
\begin{corollary}\label{cor:M_zero_mart}
The processes $(M^0_{t\wedge\tau_*(z)})_{t\in[0,T]}$ and $(N^0_{t\wedge\sigma_*(z)})_{t\in[0,T]}$ are a \cadlag $\F^1$-martingale and a \cadlag $\F^2$-martingale, respectively, for any $z\in[0,1)$.
\end{corollary}
\begin{proof}
Take any $z\in[0,1]$ from the full measure set on which $(m^*_t(z))_{t \in [0, T]}$ constructed in Proposition \ref{prop:mnz} is a \cadlag $\F^1$-martingale. For any $\theta\in\cT_0(\F^1)$, by the definition of $m^*_t(z)$ we have
\begin{equation}\label{eqn:M0}
\begin{aligned}
m^*_{\theta\wedge\tau_*(z)}(z)&=\E\big[\ind_{\{\sigma_*<\theta\wedge\tau_*(z)\}}g_{\sigma_*}\big|\cF^1_{\theta\wedge\tau_*(z)}\big]+\hat V^{*,1}_{\theta\wedge\tau_*(z)}\\
&=\E\Big[\int_{[0,\theta\wedge\tau_*(z))}g_s\ud \zeta^*_s\Big|\cF^1_{\theta\wedge\tau_*(z)}\Big]+\hat V^{*,1}_{\theta\wedge\tau_*(z)}=M^0_{\theta\wedge\tau_*(z)},
\end{aligned}
\end{equation}
where the final equality holds by the expression \eqref{eq:subm00} for $M^0$.

By Proposition \ref{thm:aggr1}, $(M^0_t)_{t\in[0,T]}$ is a \cadlag $\F^1$-submartingale. Hence, $(M^0_{t\wedge\tau_*(z)})_{t\in[0,T]}$ is actually a \cadlag $\F^1$-martingale, indistinguishable from $(m^*_{t\wedge\tau_*(z)}(z))_{t\in[0,T]}$ thanks to \cite[Lemma VI.5.2]{RogersWilliams}. This proves that $(M^0_{t\wedge\tau_*(z)})_{t\in[0,T]}$ is an $\F^1$-martingale for {\em almost every} $z \in [0, 1]$, but the result extends to {\em all} $z\in[0,1)$ as follows: for any $z < 1$, there is $\hat z \in [z,1]$ for which the martingale condition holds; due to the monotonicity of $z \mapsto \tau_*(z)$, we have $\tau_*(z) \le \tau_*(\hat z)$, so the martingale condition holds until time $\tau_*(z)$ too. 

An analogous argument leads to the proof of the claim for $N^0$. 
\end{proof}

In the next proposition we determine the support of the generating processes for an optimal randomised pair $(\tau_*,\sigma_*)$. We introduce a notation that is needed for the next result and is used extensively in the Appendix: given a process $X\in \cL_b(\P)$ and a filtration $\G$, we denote by $\optional{X}^{\G}=(\optional{X}^{\G}_t)_{t\in[0,T]}$ its $\G$-optional projection under the measure $\P$.
\begin{proposition}\label{prop:support}
Let $(\xi^*,\zeta^*)\in\cAcirc_0(\F^1)\times\cAcirc_0(\F^2)$ be an optimal pair and recall the optional semimartingale processes $(\hat V^{*,1}_t)_{t\in[0,T]}$ and $(\hat V^{*,2}_t)_{t\in[0,T]}$ from Theorem \ref{thm:value}. Let us define 
\begin{align*}
Y^1_t\coloneqq\hat V^{*,1}_t-\optional{\big(}f_\cdot(1-\zeta^*_\cdot)\big)_t^{\F^1}-\optional{\big(}h_\cdot\Delta\zeta^*_\cdot\big)_t^{\F^1}\quad\text{and}\quad Y^2_t\coloneqq\hat V^{*,2}_t- \optional{\big(}g_\cdot(1-\xi^*_\cdot)\big)^{\F^2}_t-\optional{\big(}h_\cdot\Delta\xi^*_\cdot\big)^{\F^2}_t.
\end{align*}
Then, the following properties hold:
\begin{itemize}
\item[(i)] $\P\big(Y^1_t\le 0 \text{ and } Y^2_t\ge 0 \text{ for all $t\in[0,T]$}\big) = 1$.
\item[(ii)] 
We have
\begin{align}\label{eq:supp0}
\int_{[0,T]}Y^1_t\ud \xi^*_t=0\quad\text{and}\quad\int_{[0,T]}Y^2_t\ud \zeta^*_t=0,
\end{align}
or, equivalently,
\begin{align}\label{eq:supp1}
\int_{[0,T]}\ind_{\{Y^1_t<0\}}\ud \xi^*_t=0\quad\text{and}\quad\int_{[0,T]} \ind_{\{Y^2_t>0\}}\ud \zeta^*_t=0.
\end{align}
\end{itemize}
\end{proposition}
\begin{proof}
Taking $\tau=\theta$ and $\sigma=\gamma$ on the right-hand side of \eqref{eq:hatv12}, we have $Y^1_\theta\le 0$ and $Y^2_\gamma\ge 0$ for arbitrary stopping times $\theta\in\cT_0(\F^1)$ and $\gamma\in\cT_0(\F^2)$. Since the process $Y^1$ and $Y^2$ are optional with respect to $\F^1$ and $\F^2$, respectively, then (i) holds by repeating arguments in the proof of \cite[Lemma VI.5.2]{RogersWilliams} with the set $F$ therein replaced by $\{(t, \omega):\ Y^1_t(\omega) > 0\}$ and $\{(t, \omega):\ Y^2_t(\omega) < 0\}$, respectively.

To prove (ii), let $(z_n)_{n\in\N}\subset[0,1]$ with $z_n\uparrow 1$ be such that $(m^*_t(z_n))_{t\in[0,T]}$ is a \cadlag $\F^1$-martingale and $(n^*_t(z_n))_{t\in[0,T]}$ is a \cadlag $\F^2$-martingale for each $n\in\N$. By definition of $\sigma_*(z)$ and $\tau_*(z)$ we have $\sigma_*(z_n)\le \sigma_*(z_{n+1})$ and $\tau_*(z_n)\le \tau_*(z_{n+1})$, with $\sigma_*(z_n),\tau_*(z_n)\uparrow \sigma_*(1),\tau_*(1)$ as defined in \eqref{eq:finaltime} (recall that we skip bars in the notation for the sake of readability).

Now we focus on the support of $\xi^*$ as the result for $\zeta^*$ can be obtained analogously. 
By optional sampling $\E[m^*_0(z_n)]=\E[m^*_{\tau_*(u)}(z_n)]$ for any $u\in[0,1]$. Hence, integrating over $[0,z_n]$ yields
\begin{align}\label{eq:supp}
\E\big[m^*_0(z_n)\big]=\frac{1}{z_n}\int_0^{z_n}\E\big[m^*_{\tau_*(u)}(z_n)\big]\ud u.
\end{align}

From the definition of $m^*_t(z)$ we have 
\begin{align}\label{eq:Em0}
\begin{aligned}
&\E[m^*_0(z_n)]=\E[\hat V^{*,1}_0]=\E[V^{*,1}(0)]\\
&=\E\Big[\int_{[0,T)}\big(f_t(1-\zeta^*_t)+h_t\Delta\zeta^*_t\big)\ud \xi^*_t+\int_{[0,T)}g_t(1-\xi^*_t)\ud \zeta^*_t+h_T\Delta\zeta^*_T\Delta\xi^*_T\Big].
\end{aligned}
\end{align}
It is important to notice that the expression is independent of $z_n$. 
For $u\le z_n$, using that $\tau_*(u)\le \tau_*(z_n)$ we obtain 
\begin{align*}
\begin{aligned}
\E[m^*_{\tau_*(u)}(z_n)]&=\E\big[\ind_{\{\sigma_*<\tau_*(u)\}}g_{\sigma_*}+\hat V^{*,1}_{\tau_*(u)}\big]\\
&=\E\Big[\int_{[0,\tau_*(u))}g_{t}\ud \zeta^*_t+\hat V^{*,1}_{\tau_*(u)}\ind_{\{\tau_*(u)<T\}}+h_{T}\Delta\zeta^*_T \ind_{\{\tau_*(u)=T\}}\Big],
\end{aligned}
\end{align*}
where for the second equality we integrated out the randomisation device of $\sigma_*$ and we used that $\hat V^{*,1}_T=\E[h_T\Delta\zeta^*_T|\cF^1_T]$, along with tower property. Integrating over $u\in[0,z_n]$ yields
\begin{align}\label{eq:Emtau}
\begin{aligned}
&\int_0^{z_n}\E[m^*_{\tau_*(u)}(z_n)]\ud u\\
&=\E\Big[\int_{[0,T)}(z_n-\xi^*_t)^+g_t\ud\zeta^*_t+\int_{[0,T\wedge\tau_*(z_n))}\hat V^{*,1}_{t}\ud \xi^*_t+h_T\Delta\zeta^*_T(z_n-\xi^*_{T-})^+\Big],
\end{aligned}
\end{align}
where we used the following formulae:
\begin{align*}
&\int_0^{z_n}\Big(\int_{[0,T)}\ind_{\{t<\tau_*(u)\}}g_t\ud \zeta^*_t\Big)\ud u=\int_{[0,T)}\Big(\int_0^{z_n}\ind_{\{\xi^*_t<u\}}\ud u\Big) g_t\ud \zeta^*_t=\int_{[0,T)}(z_n-\xi^*_t)^+g_t\ud\zeta^*_t,\\
&\int_0^{z_n}\ind_{\{\tau_*(u)=T\}}\ud u=\int_0^{z_n}\ind_{\{\xi^*_{T-}<u\}}\ud u=(z_n-\xi^*_{T-})^+.
\end{align*}

Finally, inserting \eqref{eq:Em0} and \eqref{eq:Emtau} into \eqref{eq:supp} and letting $n\to\infty$ yields
\begin{align}\label{eq:supp1a}
\begin{aligned}
&\E\Big[\int_{[0,T)}\big(f_t(1-\zeta^*_t)+h_t\Delta\zeta^*_t\big)\ud \xi^*_t+\int_{[0,T)}g_t(1-\xi^*_t)\ud \zeta^*_t+h_T\Delta\zeta^*_T\Delta\xi^*_T\Big]\\
&=\E\Big[\int_{[0,T)}(1-\xi^*_t)g_t\ud\zeta^*_t+\int_{[0,T)}\hat V^{*,1}_{t}\ud \xi^*_t+h_T\Delta\zeta^*_T\Delta\xi^*_{T}\Big],
\end{aligned}
\end{align}
where we used that 
\[
\int_{[0,T\wedge\tau_*(z_n))}\hat V^{*,1}_{t}\ud \xi^*_t\xrightarrow{n\to\infty}\int_{[0,T\wedge\tau_*(1))}\hat V^{*,1}_{t}\ud \xi^*_t=\int_{[0,T)}\hat V^{*,1}_{t}\ud \xi^*_t,
\]
and the final equality above holds because $\xi^*_t(\omega)=1$ for $t\ge \tau_*(1)(\omega)$.

From \eqref{eq:supp1a} and using optional projections and \cite[Thm.~VI.57]{DellacherieMeyer} we obtain
\begin{align*}
0&=\E\Big[\int_{[0,T)} \big(\hat V^{*,1}_{t}-f_t(1-\zeta^*_t)-h_t\Delta\zeta^*_t\big)\ud \xi^*_t\Big]\\
&=\E\Big[\int_{[0,T)} \big(\hat V^{*,1}_{t}-\big[\optional{\big(}f_\cdot(1-\zeta^*_\cdot)\big)_t^{\F^1}+\optional{\big(}h_\cdot\Delta\zeta^*_\cdot\big)_t^{\F^1}\big]\ud \xi^*_t\Big]=\E\Big[\int_{[0,T)} Y^1_t\ud \xi^*_t\Big].
\end{align*}
Since we know that $Y^1_t\le 0$ for all $t\in[0,T]$, $\P$-a.s.\ (item (i) above) then we conclude that the first equation in \eqref{eq:supp0} must hold. The other one can be obtained by analogous methods and using $Y^2_t\ge 0$ for all $t\in[0,T]$, $\P$-a.s.
\end{proof}

\begin{remark}
We can rephrase Proposition \ref{prop:support} using that the pair $(\tau_*,\sigma_*)$ is generated by $(\xi^*,\zeta^*)$ and therefore writing the processes $Y^1$ and $Y^2$ as
\begin{align*}
Y^1_t=\hat V^{*,1}_t- \optional{\big(} f_\cdot \ind_{\{\sigma_*>\,\cdot\}}\big)_t^{\F^1}- \optional{\big(} h_\cdot \ind_{\{\sigma_*=\,\cdot\}}\big)_t^{\F^1}\quad\text{and}\quad Y^2_t=\hat V^{*,2}_t- \optional{\big(}g_\cdot \ind_{\{\tau_*>\,\cdot\}}\big)^{\F^2}_t- \optional{\big(}h_\cdot \ind_{\{\tau_*=\,\cdot\}}\big)^{\F^2}_t.
\end{align*}

Moreover, adopting a terminology for local-times of \cadlag semimartingales (cf.\ \cite[Thm.\ IV.7.69]{protter2005stochastic}) we can informally state (ii) in the above proposition by saying that the measure associated to the process $t\mapsto\xi^*_t(\omega)$ is carried by the set $\{t\in[0,T]:Y^1_t(\omega)=0\}$, for $\P$-a.e.\ $\omega\in\Omega$. Analogously, the measure associated to the process $t\mapsto\zeta^*_t(\omega)$ is carried by the set $\{t\in[0,T]:Y^2_t(\omega)=0\}$ for $\P$-a.e.\ $\omega\in\Omega$. 

\end{remark}

Proposition \ref{prop:support} yields a characterisation of ``when'' the optimal randomised strategies should ``activate'' but it does not quantify the ``amount'' of stopping which is required (i.e., with what speed should the generating processes increase). This is a very difficult issue to address in general and the answer depends on the particular structure of the game. We address the question in the next corollary which should hold, in spirit, in a broader generality than its actual statement. The corollary was used in \cite[Ch.\ 6]{smith2024martingale} to construct equilibrium strategies in the game proposed by \cite[Sec.\ 6.2]{de2022value} and in the working paper \cite{DEAHP}. Moreover, it was implicitly used in the construction of the equilibrium strategies of \cite{DEG2020}. 

Let us make a preliminary observation that $\hat V^{*,1}_0=M^0_0$ and $\hat V^{*,1}_0=N^0_0$. 
\begin{corollary}\label{cor:support}
Assume that the processes
\begin{align}\label{eq:BVass}
\begin{aligned}
&t\mapsto \optional{\Big(}\int_{[0,\,\cdot)}g_s\ud \zeta^*_s\Big)^{\F^1}_{t\wedge\tau_*(z)} \quad\text{and}\quad t\mapsto \optional{\Big(}\int_{[0,\,\cdot)}f_s\ud \xi^*_s\Big)^{\F^2}_{t\wedge\sigma_*(z)}
\end{aligned}
\end{align}  
are of bounded variation and continuous for $z\in[0,1]$. Then, for $z\in[0,1)$ the processes $t\mapsto\hat V^{*,1}_{t\wedge \tau_*(z)}$ and $t\mapsto\hat V^{*,2}_{t\wedge \sigma_*(z)}$ are \cadlag with semimartingale decompositions given by 
\[
\hat V^{*,1}_{t\wedge\tau_*(z)} =\hat V^{*,1}_0+A^1_{t}+B^1_{t}\quad\text{and}\quad\hat V^{*,2}_{t\wedge\sigma_*(z)} =\hat V^{*,2}_0+A^2_{t}+B^2_{t},
\] 
where $A^{i}$ is the previsible bounded variation part and $B^i$ the martingale part, with 
\begin{align}
\begin{aligned}
&A^1_{t}=-\optional{\Big(}\int_{[0,\,\cdot)}g_s\ud \zeta^*_s\Big)^{\F^1}_{t\wedge\tau_*(z)} \quad\text{and}\quad B^1_{t}=M^0_{t\wedge\tau_*(z)}-M^0_{0},\\
&A^2_{t}=-\optional{\Big(}\int_{[0,\,\cdot)}f_s\ud \xi^*_s\Big)^{\F^2}_{t\wedge\sigma_*(z)} \quad\text{and}\quad B^2_{t}=N^0_{t\wedge\sigma_*(z)}-N^0_{0}.
\end{aligned}
\end{align}  
\end{corollary}
\begin{proof}
Let us start by noticing that for any $\theta\in\cT_0(\F^1)$ it holds
\[
\optional{\Big(}\int_{[0,\,\cdot)}g_s\ud \zeta^*_s\Big)^{\F^1}_{\theta}=\E\Big[\int_{[0,\theta)}g_s\ud \zeta^*_s\Big|\cF^1_{\theta}\Big]=\E\big[\ind_{\{\sigma_*<\theta\}}g_{\sigma_*}\big|\cF^1_{\theta}\big].
\]
By Corollary \ref{cor:M_zero_mart}, for any $z \in [0, 1)$, the process $(M^0_{t \wedge \tau_*(z)})_{t \in [0, T]}$ is a \cadlag $\F^1$-martingale and (see Eq. \eqref{eqn:M0})
\[
\begin{aligned}
M^0_{t\wedge\tau_*(z)}
=\optional{\Big(}\int_{[0,\,\cdot)}g_s\ud \zeta^*_s\Big)^{\F^1}_{t\wedge\tau_*(z)}+\hat V^{*,1}_{t\wedge\tau_*(z)}.
\end{aligned}
\]
We know from Theorem \ref{thm:value} that $t\mapsto \hat V^{*,1}_{t\wedge\tau_*(z)}$ is an $\F^1$-optional semi-martingale. The above equality proves that it is actually \cadlag, by the assumed continuity of the processes in \eqref{eq:BVass}. Rearranging terms we express $\hat V^{*,1}_{t\wedge\tau_*(z)}$ as
\[
\hat V^{*,1}_{t\wedge\tau_*(z)}=M^0_{0}+\big[M^0_{t\wedge\tau_*(z)}-M^0_{0}\big]-\optional{\Big(}\int_{[0,\,\cdot)}g_s\ud \zeta^*_s\Big)^{\F^1}_{t\wedge\tau_*(z)}.
\]
Since $\optional{\Big(}\int_{[0,\,\cdot)}g_s\ud \zeta^*_s\Big)^{\F^1}_{t\wedge\tau_*(z)}$ is previsible and of bounded variation, then the above formula determines uniquely the Doob-Meyer's decomposition of $\hat V^{*,1}_{t\wedge\tau_*(z)}$ by, e.g., \cite[Thm.\ III.1.2]{protter2005stochastic}.
The argument for $\hat V^{*,2}$ is analogous.
\end{proof}

In a nutshell, the corollary establishes a link between the optimal control $\zeta^*_t$ (resp.\ $\xi^*_t$) and the bounded variation part of the Doob-Meyer decomposition of $\hat V^{*,2}_t$ (resp.\ $\hat V^{*,1}_t$). When the latter is known (e.g., by PDE results in a Markovian framework) one may use the corollary to identify the optimal speed of increase for the generating process $\zeta^*$ (resp.\ $\xi^*$).

The reader is referred to Section \ref{sec:examples} to appreciate how the necessary conditions presented in this section enable us to make significant insights into two classes of games with asymmetric information. These insights pave the way for the solution of specific games by identifying objects to be modelled (player values, beliefs) as well as the structure of optimal strategies.

\section{Sufficient conditions for a saddle point}\label{sec:suffcond}

In this section we will formulate a verification result, i.e., a set of sufficient conditions for a saddle point and the value of the game. These conditions closely resemble the necessary conditions derived earlier in this paper. Indeed, 
upon close inspection, results of the previous section show that the sufficient conditions formulated below are also necessary. 
To facilitate such comparisons, we employ notations aligned with those used for the necessary conditions.  

The fact that our verification theorems provide conditions which are also necessary shows that these conditions are optimal, i.e., cannot be relaxed any further. 
This emphasises the completeness of our derivation of necessary conditions and the strength of the sufficient conditions. We present various equivalent sets of sufficient conditions, 
so that the reader may choose the most convenient one for their own specific application. 

\begin{remark} 
The assumptions for this section can be relaxed compared to the rest of the paper and we no longer require the ordering of the payoff processes $f\ge h\ge g$. 
\end{remark}

We start with a definition of $\cT_0$ sub- and supermartingale systems, analogues of ${\bf M}^\xi$ and ${\bf N}^\zeta$ from Section \ref{subsec:aggregation}. Fix a pair $(\hat\xi,\hat \zeta)\in \cAcirc_0(\F^1)\times\cAcirc_0(\F^2)$, an $\F^1$-progressively measurable process $(\cU^1_t)_{t \in [0, T]}$ and an $\F^2$-progressively measurable process $(\cU^2_t)_{t \in [0, T]}$. 
For any $(\xi,\zeta)\in \cAcirc_0(\F^1)\times\cAcirc_0(\F^2)$, consider the families $\hat{\bf M}^\xi\coloneqq\{\hat M^\xi(\theta),\theta\in\cT_0(\F^1)\}$ and $\hat{\bf N}^\zeta\coloneqq\{\hat N^\zeta(\gamma),\gamma\in\cT_0(\F^2)\}$ defined by
\begin{align}
\hat M^\xi(\theta) &= \E \Big[\! \int_{[0,\theta)}\!\! f_t (1\!-\!\hat \zeta_t)  \ud \xi_t\! +\! \int_{[0,\theta)}\!\!g_t (1\!-\!\xi_t) \ud \hat \zeta_t\! +\!\! \sum_{t \in [0,\theta)}\!\!h_t \Delta \xi_t \Delta \hat\zeta_t \Big| \cF^1_\theta \Big]\label{eqn:M_xi}\\
&\quad +\!(1\!-\!\xi_{\theta-}) \E[1\!-\!\hat \zeta_{\theta-}|\cF^1_\theta]\cU^1_\theta,\notag\\
\hat N^\zeta(\gamma)&= \E \Big[\! \int_{[0,\gamma)}\!\! f_t (1\!-\!\zeta_t)  \ud \hat\xi_t\! +\! \int_{[0,\gamma)}\!\!g_t (1\!-\!\hat\xi_t) \ud \zeta_t\! +\! \!\sum_{t \in [0,\gamma)}\!\! h_t \Delta \hat\xi_t \Delta \zeta_t \Big| \cF^2_\gamma \Big]\label{eqn:N_zeta}\\
&\quad +\!\E[1\!-\!\hat\xi_{\gamma-}|\cF^2_\gamma](1\!-\!\zeta_{\gamma-})\cU^2_\gamma.\notag
\end{align}

We start with a sequence of results which bear the strongest similarity to the necessary conditions developed in the previous section. We will follow those with stronger results that allow identification of the saddle point of the game -- the main aim is to replace arbitrary generating processes $(\xi, \zeta)$ with stopping times.

\begin{theorem}\label{thm:suff_saddle}
Let $(\cU^1_t)_{t \in [0, T]}$ and $(\cU^2_t)_{t \in [0, T]}$ be $\F^1$- and $\F^2$-progressively measurable process, respectively, and let $(\hat\xi,\hat \zeta)\in \cAcirc_0(\F^1)\times\cAcirc_0(\F^2)$. Assume that
\begin{enumerate}[(i)]
 \item $\hat{\bf M}^\xi$ is a $\cT_0(\F^1)$-submartingale system for any $\xi \in \cAcirc_0(\F^1)$,
 \item $\hat{\bf N}^\zeta$ is a $\cT_0(\F^2)$-supermartingale system for any $\zeta \in \cAcirc_0(\F^2)$,
 \item $\E[\Delta\hat\zeta_T | \cF^1_T]\, \cU^1_T = \E[\Delta\hat\zeta_T h_T | \cF^1_T]$ and $\E[\Delta\hat\xi_T | \cF^2_T]\, \cU^2_T = \E[\Delta\hat\xi_T h_T | \cF^2_T]$,
  \item $\E[\cU^1_0] = \E[\cU^2_0]$.
\end{enumerate}
Then the game has a value, i.e., $\overline V = \underline V = \E[\cU^1_0] = \E[\cU^2_0]$, see \eqref{eq:V}, and the randomised stopping times $(\hat \tau, \hat \sigma) \in \cT^R_0(\F^1) \times \cT^R_0(\F^2)$ generated by $(\hat \xi, \hat \zeta)$ form a saddle point of the game.
\end{theorem}
\begin{proof}
Let $\xi \in \cAcirc_0(\F^1)$. By the submartingale property of $\hat{\bf M}^\xi$, we have $\E[\hat M^\xi(T) | \cF^1_0] \ge \hat M^\xi(0)$. Expanding the left- and right-hand sides from the definition \eqref{eqn:M_xi}, we have
\begin{equation}\label{eqn:subm_M}
\E \Big[\! \int_{[0,T)}\!\! f_t (1\!-\!\hat \zeta_t)  \ud \xi_t\! +\! \int_{[0,T)}\!\!g_t (1\!-\!\xi_t) \ud \hat \zeta_t\! +\!\! \sum_{t \in [0,T)}\!\!h_t \Delta \xi_t \Delta \hat\zeta_t \! +\!(1\!-\!\xi_{T-}) \E[1\!-\!\hat \zeta_{T-}|\cF^1_T] \cU^1_T \Big| \cF^1_0 \Big] \ge \cU^1_0.
\end{equation}
We note that $1- \xi_{T-} = \Delta \xi_T$ and $1-\zeta_{T-} = \Delta \zeta_T$, and use assumption (iii) of the theorem to simplify the last term on the left:
\[
\E \big[(1\!-\!\xi_{T-}) \E[1\!-\!\hat \zeta_{T-}|\cF^1_T]\, \cU^1_T \big| \cF^1_0 \big]
=
\E \big[\Delta \xi_T\, \E[\Delta  \hat \zeta_T | \cF^1_T]\, \cU^1_T \big| \cF^1_0 \big]
=
\E \big[\Delta \xi_T \Delta \hat \zeta_T h_T \big| \cF^1_0 \big].
\]
Inserting this into \eqref{eqn:subm_M} yields
\begin{equation}\label{eqn:U1}
\E[\cP(\xi, \hat\zeta) | \cF^1_0] \ge \cU^1_0, \qquad \text{for any $\xi \in \cAcirc_0(\F^1)$.}
\end{equation}
In an analogous way, we show
\begin{equation}\label{eqn:U2}
\E[\cP(\hat\xi, \zeta) | \cF^2_0] \le \cU^2_0, \qquad \text{for any $\zeta \in \cAcirc_0(\F^2)$.}
\end{equation}
Taking expectation on both sides of \eqref{eqn:U1}-\eqref{eqn:U2} with $(\xi, \zeta) =(\hat \xi, \hat \zeta)$, and applying condition (iv) of the theorem yields
\[
\E[\cP(\hat\xi, \hat\zeta)] \ge \E[\cU^1_0] = \E[\cU^2_0] \ge \E[\cP(\hat\xi, \hat \zeta)],
\]
so that $\E[\cP(\hat \xi, \hat \zeta)] = \E[\cU^1_0] = \E[\cU^2_0]$. For any $(\xi,\zeta)\in \cAcirc_0(\F^1)\times\cAcirc_0(\F^2)$, the inequalities \eqref{eqn:U1}-\eqref{eqn:U2} give
\[
\E[\cP(\hat \xi, \zeta)] \le \E[\cP(\hat \xi, \hat \zeta)] \le \E[\cP(\xi, \hat \zeta)],
\]
which demonstrates that the pair $(\hat \xi, \hat \zeta)$ generates a saddle point of the game. Consequently, the game has a value which equals $\E[\cP(\hat \xi, \hat \zeta)]$.
\end{proof}

As a corollary of the above proof we obtain the following result.

\begin{corollary}\label{lem:mart_syst}
Under the assumptions of Theorem \ref{thm:suff_saddle}, $\hat{\bf M}^{\hat \xi}$ is a $\cT_0(\F^1)$-martingale system and $\hat{\bf N}^{\hat\zeta}$ is a $\cT_0(\F^2)$-martingale system.
\end{corollary}
\begin{proof}
It transpires from the proof of Theorem \ref{thm:suff_saddle} that \eqref{eqn:U1} holds with equality for $\xi = \hat \xi$. This consequently means that \eqref{eqn:subm_M} holds with equality which translates into $\E[\hat M^{\hat\xi}(T)|\cF^1_0] = \hat M^{\hat \xi}(0)$. When we recall that $\hat {\bf M}^{\hat \xi}$ is a $\cT_0(\F^1)$-submartingale system (by assumption (ii) of Theorm \ref{thm:suff_saddle}), we immediately obtain that it is a $\cT_0(\F^1)$-martingale system. A similar argument applies to $\hat {\bf N}^{\hat \zeta}$.
\end{proof}

\begin{remark}[\bf Necessity of sufficient conditions]
Notice first that nowhere in the proof of Theorem \ref{thm:suff_saddle} we require that $(\cU^1_t)_{t\in[0,T]}$ and $(\cU^2_t)_{t\in[0,T]}$ be stochastic processes. They can be replaced by $\cT_0(\F^1)$ and $\cT_0(\F^2)$ systems, and, indeed, by the families $V^{*,1}$ and $V^{*,2}$ in \eqref{eq:fV}. Then conditions (i) and (ii) are satisfied by the families ${\bf M}^\xi$ and ${\bf N}^\zeta$ introduced in \eqref{eq:Mxi} and \eqref{eq:Nzi} (cf.\ Proposition \ref{prop:subsupmg}). Condition (iii) is satisfied by Theorem \ref{thm:value}, see \eqref{eq:hatv12}.  Condition (iv) follows from Proposition \ref{prop:link} with $\lambda=0$ therein. This shows the necessity of the sufficient conditions in Theorem \ref{thm:suff_saddle}.
\end{remark}

We now study in more detail the analogy between processes $(\cU^1_t)_{t\in[0,T]}$ and $(\cU^2_t)_{t\in[0,T]}$ from Theorem \ref{thm:suff_saddle} and processes $(V^{*,1}_t)_{t\in[0,T]}$ and $(V^{*,2}_t)_{t\in[0,T]}$ from Section \ref{sec:neccond}. Notations and arguments will follow parallel tracks but there are sufficient differences to merit complete statement of results and their proofs. 

From now on, the pair $(\hat \xi,\hat \zeta)\in\cA_0^\circ(\F^1)\times \cA_0^\circ(\F^2)$ and processes $(\cU^1_t)_{t\in[0,T]}$ and $(\cU^2_t)_{t\in[0,T]}$ are those introduced in Theorem \ref{thm:suff_saddle}.
Let us start by defining sets
\begin{align*}
\hat\Gamma^1_\theta &= \{ \omega \in \Omega: (1-\hat \xi_{\theta-}(\omega))\, \E[1-\hat\zeta_{\theta-} | \cF^1_{\theta}] (\omega) > 0 \}, \qquad \theta \in \cT_0(\F^1),\\
\hat\Gamma^2_\gamma &= \{ \omega \in \Omega: (1-\hat\zeta_{\gamma-}(\omega))\, \E[1-\hat \xi_{\gamma-}|\cF^2_{\gamma}](\omega) > 0 \}, \qquad \gamma \in \cT_0(\F^2).
\end{align*}
In analogy to \eqref{eq:Pi}, for $\theta \in\cT_0(\F^1)$ and $\gamma\in\cT_0(\F^2)$ we define
\begin{equation}\label{eq:hat_Pi}
\begin{aligned}
\hat\Pi^{1}_\theta&\coloneqq\frac{1-\hat\zeta_{\theta-}}{\E[1-\hat\zeta_{\theta-}|\cF^1_\theta]},\qquad
\hat\Pi^{2}_\gamma&\coloneqq\frac{1-\hat\xi_{\gamma-}}{\E[1-\hat\xi_{\gamma-}|\cF^2_\gamma]}, 
\end{aligned}
\end{equation}
with the ratio equal to $1$ when the denominator is zero (see the explanations after \eqref{eq:Pi} concerning this convention). The next lemma is an analogue of Proposition \ref{thm:aggr2} for $(\hat \xi,\hat\zeta)$, $\cU^1$ and $\cU^2$.

\begin{lemma}\label{lem:U_value}
Under the assumptions of Theorem \ref{thm:suff_saddle}, we have
\begin{equation}\label{eqn:U_optim}
\begin{aligned}
&\ind_{\hat\Gam^1_\theta} \cU^{1}_\theta
=\ind_{\hat\Gam^1_\theta} \essinf_{\xi\in\cAcirc_\theta(\F^1)}J^{\hat\Pi^{1}_\theta}(\xi,\hat\zeta^{\theta}|\cF^1_\theta)
=\ind_{\hat\Gam^1_\theta} J^{\hat\Pi^{1}_\theta}(\hat\xi^{\theta},\hat\zeta^{\theta}|\cF^1_\theta),\quad\theta\in\cT_0(\F^1),\\
&\ind_{\hat\Gam^2_\gamma} \cU^2_\gamma
=\ind_{\hat\Gam^2_\gamma} \esssup_{\zeta\in\cAcirc_\gamma(\F^2)}J^{\hat\Pi^{2}_\gamma}(\hat\xi^{\gamma},\zeta|\cF^2_\gamma)
=\ind_{\hat\Gam^2_\gamma} J^{\hat\Pi^{2}_\gamma}(\hat\xi^{\gamma},\hat\zeta^{\gamma}|\cF^2_\gamma),\quad\gamma\in\cT_0(\F^2),
\end{aligned}
\end{equation}
where $\hat\xi^\gamma$ and $\hat \zeta^\theta$ are truncated controls as in Definition \ref{def:trunc}.
\end{lemma}
\begin{proof}
We prove the first sequence of equalities. The second one is analogous.

Fix $\theta \in \cT_0(\F^1)$. By Corollary \ref{lem:mart_syst}, $\hat{\bf M}^{\hat \xi}$ is a $\cT_0(\F^1)$-martingale system, so 
\begin{equation}\label{eqn:u1}
\hat M^{\hat\xi}(\theta) = \E[\hat M^{\hat \xi}(T) | \cF^1_\theta]. 
\end{equation}
Arguing as in the proof of Theorem \ref{thm:suff_saddle}, we have $\E[\hat M^{\hat \xi}(T) | \cF^1_\theta]= \E[\cP(\hat\xi, \hat\zeta)|\cF^1_\theta]$. Using the definition of $\hat M^{\hat\xi}(\theta)$ and cancelling identical terms on both sides of \eqref{eqn:u1}, we obtain
\begin{equation}\label{eqn:u2}
\begin{aligned}
&(1 - \hat\xi_{\theta-}) \E[1-\hat\zeta_{\theta-}|\cF^1_\theta]\, \cU^1_\theta \\
&= 
\E \Big[\! \int_{[\theta,T)}\!\! f_t (1\!-\!\hat \zeta_t)  \ud \hat\xi_t\! +\! \int_{[\theta,T)}\!\!g_t (1\!-\!\hat\xi_t) \ud \hat \zeta_t\! +\!\! \sum_{t \in [\theta,T]}\!\!h_t \Delta \hat\xi_t \Delta \hat\zeta_t \Big| \cF^1_\theta\Big]\\
&=
\E \Big[(1 - \hat\xi_{\theta-}) (1 - \hat\zeta_{\theta-}) \Big(\! \int_{[\theta,T)}\!\! f_t (1\!-\!\hat \zeta^\theta_t)  \ud \hat\xi^\theta_t\! +\! \int_{[\theta,T)}\!\!g_t (1\!-\!\hat\xi^\theta_t) \ud \hat \zeta^\theta_t\! +\!\! \sum_{t \in [\theta,T]}\!\!h_t \Delta \hat\xi^\theta_t \Delta \hat\zeta^\theta_t \Big) \Big| \cF^1_\theta\Big]\\
&=
(1 - \hat\xi_{\theta-}) \E[1 - \hat\zeta_{\theta-} | \cF^1_\theta]\E \Big[ \hat\Pi^1_\theta \Big(\! \int_{[\theta,T)}\!\! f_t (1\!-\!\hat \zeta^\theta_t)  \ud \hat\xi^\theta_t\! +\! \int_{[\theta,T)}\!\!g_t (1\!-\!\hat\xi^\theta_t) \ud \hat \zeta^\theta_t\! +\!\! \sum_{t \in [\theta,T]}\!\!h_t \Delta \hat\xi^\theta_t \Delta \hat\zeta^\theta_t \Big) \Big| \cF^1_\theta\Big]\\
&=
(1 - \hat\xi_{\theta-}) \E[1 - \hat\zeta_{\theta-} | \cF^1_\theta]\, J^{\hat \Pi^1_\theta}(\hat\xi^\theta, \hat \zeta^\theta|\cF^1_\theta).
\end{aligned}
\end{equation}
On the set $\hat\Gam^1_\theta$, the multiplicative factor on both sides of the inequality is positive, which implies
\begin{align}\label{eq:opthat}
\ind_{\hat\Gam^1_\theta} \cU^{1}_\theta
=\ind_{\hat\Gam^1_\theta} J^{\hat\Pi^{1}_\theta}(\hat\xi^{\theta},\hat\zeta^{\theta}|\cF^1_\theta).
\end{align}

To prove the remaining assertion, take any $\xi \in \cAcirc_\theta(\F^1)$ and define (cf. \eqref{eq:xitilde})
\[
\bar \xi_t = \hat \xi_t \ind_{[0, \theta)}(t) + (\hat \xi_{\theta-} + (1-\hat\xi_{\theta-})\xi_t) \ind_{[\theta, T]}(t).
\]
By the submartingale property of the family $\hat {\bf M}^{\bar\xi}$, we have $\hat M^{\bar\xi}(\theta) \le \E[\hat M^{\bar\xi}(T) | \cF^1_\theta]$. Arguing as above, this inequality implies
\[
(1 - \bar\xi_{\theta-}) \E[1-\hat\zeta_{\theta-}|\cF^1_\theta]\, \cU^1_\theta
\le
(1 - \bar\xi_{\theta-}) \E[1 - \hat\zeta_{\theta-} | \cF^1_\theta]\, J^{\hat \Pi^1_\theta}(\bar\xi^\theta, \hat \zeta^\theta|\cF^1_\theta).
\]
By the definition of $\bar\xi$, we have $\bar\xi_{\theta-} = \hat \xi_{\theta-}$ and $\bar\xi^\theta_t = \xi_t$ for $t \in [\theta, T]$, so we can conclude that
\[
\ind_{\hat\Gam^1_\theta} \cU^{1}_\theta
\le \ind_{\hat\Gam^1_\theta} J^{\hat\Pi^{1}_\theta}(\xi,\hat\zeta^{\theta}|\cF^1_\theta).
\]
By the arbitrariness of $\xi\in \cAcirc_\theta(\F^1)$, 
\[
\ind_{\hat\Gam^1_\theta} \cU^{1}_\theta
\le \essinf_{\xi \in \cAcirc_\theta(\F^1)} \ind_{\hat\Gam^1_\theta} J^{\hat\Pi^{1}_\theta}(\xi,\hat\zeta^{\theta}|\cF^1_\theta)
=
\ind_{\hat\Gam^1_\theta} \essinf_{\xi \in \cAcirc_\theta(\F^1)}  J^{\hat\Pi^{1}_\theta}(\xi,\hat\zeta^{\theta}|\cF^1_\theta).
\]
Combining the inequality above with \eqref{eq:opthat} completes the proof.
\end{proof}
Next we obtain an analogue of Proposition \ref{prop:link}. Recall the notations $\cF^{1,2}_t=\cF^1_t\cap\cF^2_t$ and $\F^{1,2} = \F^1 \wedge \F^2$.
\begin{corollary}\label{cor:u12}
Under the assumptions of Theorem \ref{thm:suff_saddle}, for $\lambda \in \cT_0(\F^{1,2})$ we have
\begin{align*}
&\E\big[(1-\hat \xi_{\lambda-})(1-\hat\zeta_{\lambda-})\, \cU^1_{\lambda} \big| \cF^{1,2}_\lambda\big]
=
\E\big[(1-\hat \xi_{\lambda-})(1-\hat\zeta_{\lambda-})\, \cU^2_{\lambda} \big| \cF^{1,2}_\lambda \big]\\
&=\E \Big[\! \int_{[\lambda,T)}\!\! f_t (1\!-\!\hat \zeta_t)  \ud \hat\xi_t\! +\! \int_{[\lambda,T)}\!\!g_t (1\!-\!\hat\xi_t) \ud \hat \zeta_t\! +\!\! \sum_{t \in [\lambda,T]}\!\!h_t \Delta \hat\xi_t \Delta \hat\zeta_t \Big| \cF^{1,2}_\lambda\Big]\eqqcolon \bar U(\lambda).
\end{align*}
\end{corollary}
\begin{proof}
From the first equality in \eqref{eqn:u2} with $\theta = \lambda$ and an analogous derivation for $\cU^2$ we have
\begin{align*}
 &(1 - \hat\xi_{\lambda-}) \E[1 - \hat\zeta_{\lambda-} | \cF^1_\lambda]\, \cU^1_\lambda 
 = \E \Big[\! \int_{[\lambda,T)}\!\! f_t (1\!-\!\hat \zeta_t)  \ud \hat\xi_t\! +\! \int_{[\lambda,T)}\!\!g_t (1\!-\!\hat\xi_t) \ud \hat \zeta_t\! +\!\! \sum_{t \in [\lambda,T]}\!\!h_t \Delta \hat\xi_t \Delta \hat\zeta_t \Big| \cF^1_\lambda\Big],\\
 &(1 - \hat\zeta_{\lambda-}) \E[1 - \hat\xi_{\lambda-} | \cF^2_\lambda]\, \cU^2_\lambda 
 = \E \Big[\! \int_{[\lambda,T)}\!\! f_t (1\!-\!\hat \zeta_t)  \ud \hat\xi_t\! +\! \int_{[\lambda,T)}\!\!g_t (1\!-\!\hat\xi_t) \ud \hat \zeta_t\! +\!\! \sum_{t \in [\lambda,T]}\!\!h_t \Delta \hat\xi_t \Delta \hat\zeta_t \Big| \cF^2_\lambda\Big].
\end{align*}
As the expressions under the conditional expectations on the right-hand sides are identical, the right-hand sides are identical once conditioned on $\cF^{1,2}_\lambda$ by the tower property. Hence,
\[
\E \big[ (1 - \hat\xi_{\lambda-}) \E[1 - \hat\zeta_{\lambda-} | \cF^1_\lambda]\, \cU^1_\lambda \big| \cF^{1,2}_\lambda \big]
=
\E \big[ (1 - \hat\zeta_{\lambda-}) \E[1 - \hat\xi_{\lambda-} | \cF^2_\lambda]\, \cU^2_\lambda \big| \cF^{1,2}_\lambda \big].
\]
To conclude by the tower property, it is sufficient to recall that $\hat\xi_{\lambda-}$, $\cU^1_\lambda$ are $\cF^1_\lambda$-measurable and $\hat\zeta_{\lambda-}$, $\cU^2_\lambda$ are $\cF^2_\lambda$-measurable.
\end{proof}

We turn our attention to the \emph{ex-ante} value of the game and its relationship to $\cU^1$ and $\cU^2$. The next proposition is an analogue of Corollary \ref{corr:link}.

\begin{proposition}\label{prop:U_exante}
Under the assumptions of Theorem \ref{thm:suff_saddle}, for any $\lambda \in \cT_0(\F^1 \wedge \F^2)$ let
\[
\Pi_\lambda \coloneqq \frac{(1-\hat\xi_{\lambda-})(1-\hat\zeta_{\lambda-})}{\E[(1-\hat\xi_{\lambda-})(1-\hat\zeta_{\lambda-}) |\cF^{1,2}_\lambda] }\in\cR(\cF^{1,2}_\lambda),
\]
with the convention $0/0=1$ as in \eqref{eq:Pi}. 

Define $U(\lambda) \coloneqq J^{\Pi_\lambda}(\hat \xi^\lambda, \hat \zeta^\lambda|\cF^{1,2}_\lambda)$. Then
\begin{equation}\label{eqn:U_exante1}
\begin{aligned}
U(\lambda)&=\essinf_{\xi \in \cAcirc_{\lambda}(\F^1)} \esssup_{\zeta \in \cAcirc_{\lambda}(\F^2)} J^{\Pi_\lambda}(\xi, \zeta|\cF^{1,2}_\lambda)
=
\esssup_{\zeta \in \cAcirc_{\lambda}(\F^2)} \essinf_{\xi \in \cAcirc_{\lambda}(\F^1)} J^{\Pi_\lambda}(\xi, \zeta|\cF^{1,2}_\lambda),
\end{aligned}
\end{equation}
on the event $\{\E[(1-\hat\xi_{\lambda-})(1-\hat\zeta_{\lambda-}) |\cF^{1,2}_\lambda] > 0\}$ and it holds, for $i=1,2$,
\begin{align}\label{eqn:U_exante2}
\begin{aligned}
\bar U(\lambda) &=\E[(1-\hat\xi_{\lambda-})(1-\hat\zeta_{\lambda-}) |\cF^{1,2}_\lambda]U(\lambda)\\
&= \E[(1-\hat\xi_{\lambda-})(1-\hat\zeta_{\lambda-}) |\cF^{1,2}_\lambda] \E\big[\Pi_\lambda\, \cU^i_\lambda \big|  \cF^{1,2}_\lambda \big]\\
&=\E[(1-\hat\xi_{\lambda-})(1-\hat\zeta_{\lambda-}) |\cF^{1,2}_\lambda]  \E^{\Pi_\lambda}\big[ \cU^i_\lambda \big|  \cF^{1,2}_\lambda\big].
\end{aligned}
\end{align}
\end{proposition}
\begin{proof}
Fix $\lambda \in \cT_0(\F^{1,2})$. 
Arguing as in the proof of Lemma \ref{lem:U_value} (cf.\ \eqref{eqn:u2}), for any $\xi \in \cAcirc_\lambda(\F^1)$ we have
\begin{align*}
&(1 - \hat\xi_{\lambda-}) \E[1-\hat\zeta_{\lambda-}|\cF^1_\lambda]\, \cU^1_\lambda \\
&\le
\E \Big[(1 - \hat\xi_{\lambda-}) (1 - \hat\zeta_{\lambda-}) \Big(\! \int_{[\lambda,T)}\!\! f_t (1\!-\!\hat \zeta^\lambda_t)  \ud \xi_t\! +\! \int_{[\lambda,T)}\!\!g_t (1\!-\!\xi_t) \ud \hat \zeta^\lambda_t\! +\!\! \sum_{t \in [\lambda,T]}\!\!h_t \Delta \xi_t \Delta \hat\zeta^\lambda_t \Big) \Big| \cF^1_\lambda\Big],
\end{align*}
with the equality for $\xi = \hat \xi^\lambda$. On both sides we take conditional expectations with respect to $\cF^{1,2}_\lambda$ and note that the left-hand side equals $\bar U(\lambda)$. Then, for any $\xi\in\cAcirc_\lambda(\F^1)$,
\begin{align*}
\bar U(\lambda) 
&\le 
\E \Big[(1 - \hat\xi_{\lambda-}) (1 - \hat\zeta_{\lambda-}) \Big(\! \int_{[\lambda,T)}\!\! f_t (1\!-\!\hat \zeta^\lambda_t)  \ud \xi_t\! +\! \int_{[\lambda,T)}\!\!g_t (1\!-\!\xi_t) \ud \hat \zeta^\lambda_t\! +\!\! \sum_{t \in [\lambda,T]}\!\!h_t \Delta \xi_t \Delta \hat\zeta^\lambda_t \Big) \Big| \cF^{1,2}_\lambda\Big]\\
&=
\E\big[(1-\hat \xi_{\lambda-})(1-\hat\zeta_{\lambda-}) \big| \cF^{1,2}_\lambda\big] J^{\Pi_\lambda}(\xi, \hat\zeta^\lambda|\cF^{1,2}_\lambda).
\end{align*}
Since equality holds for $\xi = \hat \xi^\lambda$, we deduce
\begin{align*}
\bar U(\lambda) 
&= \E\big[(1-\hat \xi_{\lambda-})(1-\hat\zeta_{\lambda-}) \big| \cF^{1,2}_\lambda\big] \essinf_{\xi \in \cAcirc_\lambda(\F^1)} J^{\Pi_\lambda}(\xi, \hat\zeta^\lambda|\cF^{1,2}_\lambda)\\
&\le \E\big[(1-\hat \xi_{\lambda-})(1-\hat\zeta_{\lambda-}) \big| \cF^{1,2}_\lambda\big] \esssup_{\zeta \in \cAcirc_\lambda(\F^2)} \essinf_{\xi \in \cAcirc_\lambda(\F^1)} J^{\Pi_\lambda}(\xi, \zeta|\cF^{1,2}_\lambda).
\end{align*}
Analogously, we show that 
\[
\bar U(\lambda)\ge \E\big[(1-\hat \xi_{\lambda-})(1-\hat\zeta_{\lambda-}) \big| \cF^{1,2}_\lambda\big] J^{\Pi_\lambda}(\hat\xi^\lambda, \zeta|\cF^{1,2}_\lambda),
\]
for all $\zeta\in\cAcirc_\lambda(\F^2)$ and it holds with equality for $\zeta=\hat\zeta^\lambda$. It then follows
\begin{align*}
\bar U(\lambda) &= \E\big[(1-\hat \xi_{\lambda-})(1-\hat\zeta_{\lambda-}) \big| \cF^{1,2}_\lambda\big] \esssup_{\zeta \in \cAcirc_\lambda(\F^2)}  J^{\Pi_\lambda}(\hat\xi^\lambda, \zeta|\cF^{1,2}_\lambda)\\
&\ge \E\big[(1-\hat \xi_{\lambda-})(1-\hat\zeta_{\lambda-}) \big| \cF^{1,2}_\lambda\big] \essinf_{\xi\in\cA^\circ_\lambda(\F^1)}\esssup_{\zeta \in \cAcirc_\lambda(\F^2)}  J^{\Pi_\lambda}(\xi, \zeta|\cF^{1,2}_\lambda).
\end{align*}
Combining the two inequalities we prove that the order of $\esssup$ and $\essinf$ can be swapped. Moreover, using that the inequalities above hold with equality for the pair $(\hat \xi^\lambda,\hat \zeta^\lambda)$ we deduce \eqref{eqn:U_exante1} and the first equality in \eqref{eqn:U_exante2}.
The second and third equalities in \eqref{eqn:U_exante2} hold by the definition of $\Pi_\lambda$ and Corollary \ref{cor:u12}:
\begin{align*}
\bar U(\lambda)&=\E\big[(1-\hat \xi_{\lambda-})(1-\hat\zeta_{\lambda-})\, \cU^i_{\lambda} \big| \cF^{1,2}_\lambda\big]\\
&=\E\big[(1-\hat \xi_{\lambda-})(1-\hat\zeta_{\lambda-}) \big| \cF^{1,2}_\lambda\big]\E\big[\Pi_\lambda\, \cU^i_{\lambda}\big| \cF^{1,2}_\lambda \big]=\E\big[(1-\hat \xi_{\lambda-})(1-\hat\zeta_{\lambda-}) \big| \cF^{1,2}_\lambda\big]\E^{\Pi_\lambda}\big[ \cU^i_{\lambda}\big| \cF^{1,2}_\lambda \big],
\end{align*}
upon recalling the convention $0/0=1$ for $\Pi_\lambda$ and noticing that the equalities hold trivially with zero value on the event $\big\{\E\big[(1-\hat \xi_{\lambda-})(1-\hat\zeta_{\lambda-}) \big| \cF^{1,2}_\lambda \big]=0\big\}$ (cf.\ also \eqref{eq:impl}).
\end{proof}

We now develop results in the same vein as Theorem \ref{thm:suff_saddle} but under a different set of conditions.
\begin{theorem}\label{thm:suff_saddle_stop}
Let $\cU^1_0$ be $\cF^1_0$-measurable and $\cU^2_0$ be $\cF^2_0$-measurable random variables and $(\hat\xi,\hat \zeta)\in \cAcirc_0(\F^1)\times\cAcirc_0(\F^2)$. Assume that
\begin{enumerate}[(i)]
 \item for any $\tau \in \cT_0(\F^1)$, we have
\begin{equation}\label{eqn:s1}
\E\Big[f_\tau(1-\hat\zeta_\tau) + \int_{[0, \tau)} g_s \ud \hat\zeta_s + h_\tau \Delta \hat\zeta_\tau \Big| \cF^1_0 \Big] \ge \cU^1_0,
\end{equation}
 \item for any $\sigma \in \cT_0(\F^2)$, we have
\begin{equation}\label{eqn:s2}
\E\Big[\int_{[0, \sigma)} f_s \ud \hat\xi_s + g_\sigma (1-\hat\xi_\sigma) + h_\sigma \Delta \hat\xi_\sigma\Big| \cF^2_0  \Big] \le \cU^2_0,
\end{equation}
  \item $\E[\cU^1_0] = \E[\cU^2_0]$.
\end{enumerate}
Then the game has a value, i.e., $\overline V = \underline V = \E[\cU^1_0] = \E[\cU^2_0]$, and the randomised stopping times $(\hat \tau, \hat \sigma) \in \cT^R_0(\F^1) \times \cT^R_0(\F^2)$ generated by $(\hat \xi, \hat \zeta)$ form a saddle point of the game.
\end{theorem}
\begin{proof}
In order to apply arguments from the proof of Theorem \ref{thm:suff_saddle}, we need to prove \eqref{eqn:subm_M} with $\E[1-\hat \zeta_{T-}|\cF^1_T]\cU^1_T$ replaced by $h_T\Delta\hat \zeta_T$ (under expectation) and the equivalent condition arising for the supermartingale system $\hat{\bf N}^\zeta$. Those two inequalities imply \eqref{eqn:U1}-\eqref{eqn:U2}, from which it not hard to deduce all claims in the theorem. 

We will only provide details for \eqref{eqn:subm_M}, because arguments for $\hat{\bf N}^\zeta$ are analogous. 
Take any $\xi \in \cAcirc_0(\F^1)$ and let $\tau(z) = \inf\{t \ge 0: \xi_t > z\} \in \cT(\F^1)$. We have
\begin{equation}\label{eq:suff2}
\E \Big[\! \int_{[0,T)}\!\! f_t (1\!-\!\hat \zeta_t)  \ud \xi_t\! +\! \int_{[0,T)}\!\!g_t (1\!-\!\xi_t) \ud \hat \zeta_t\! +\!\! \sum_{t \in [0,T)}\!\!h_t \Delta \xi_t \Delta \hat\zeta_t \! +\! h_T \Delta\xi_{T} \Delta\hat \zeta_{T}\Big| \cF^1_0 \Big]\!=\!\int_0^1\!\! \hat m^\xi(z)\ud z,
\end{equation}
where, as in Proposition \ref{prop:meas}, $(z,\omega)\mapsto\hat m^\xi(z,\omega)$ is a $\cB([0,T])\times\cF^1_0$-measurable function such that for each $z\in[0,1]$, $\P$-a.s.
\[
\hat m^\xi(z)=\E \Big[ \indd{\tau(z) < T} \big(f_{\tau(z)}(1\!-\!\hat\zeta_{\tau(z)})\!+\! h_{\tau(z)} \Delta \hat\zeta_{\tau(z)}\big)\!+\!\int_{[0, \tau(z))} g_s \ud \hat\zeta_s\!+\!\indd{\tau(z) = T} h_T \Delta \hat \zeta_{T}  \Big| \cF^1_0 \Big].
\]
We recombine the indicator functions for the jump terms in the process $\hat \zeta$ and use $\hat\zeta_T = 1$ to drop the indicator function in the term involving $f_{\tau(z)}$ to obtain
\[
\hat m^\xi(z) = \E \Big[ f_{\tau(z)}(1-\hat\zeta_{\tau(z)}) + \int_{[0, \tau(z))} g_s \ud \hat\zeta_s + h_{\tau(z)} \Delta \hat \zeta_{\tau(z)} \Big| \cF^1_0 \Big]. 
\]
From \eqref{eqn:s1}, we further have $\hat m^\xi(z) \ge \cU^1_0$, $\P$\as Since $\cU^1_0$ does not depend on $z$, we claim that $\int_0^1\hat m^\xi(z)\ud z\ge \cU^1_0$, $\P$\as; this is not immediate as the set of measure zero in the inequality $\hat m^\xi(z) \ge \cU^1_0$ depends on $z$. However, taking $A = \{\int_0^1\hat m^\xi(z)\ud z < \cU^1_0 \}$, we have  $A \in\cF^1_0$ and
\[
\E\Big[\ind_A\int_0^1\hat m^\xi(z)\ud z\Big]=\int_0^1\E\big[\ind_A\hat m^\xi(z)\big]\ud z\ge\int_0^1\E\big[\ind_A\cU^1_0\big]\ud z=\E\big[\ind_A\cU^1_0\big], 
\]
where the first equality is by Fubini's theorem. This shows that $\P(A) = 0$ and $\int_0^1\hat m^\xi(z)\ud z\ge \cU^1_0$, $\P$-a.s. Combining the latter with \eqref{eq:suff2} we obtain \eqref{eqn:subm_M} with $\E[\Delta\hat \zeta_T|\cF^1_T]\cU^1_T$ replaced by $\E[h_T\Delta\hat \zeta_T|\cF^1_T]$ as required. Analogous arguments with the use of \eqref{eqn:s2} yield the desired result also for $\cU^2_0$.
\end{proof}
The above theorem does not employ candidate value processes $\cU^1$ and $\cU^2$, so it does not suggest an approach for finding an equilibrium. In the next theorem, we relax conditions (i) and (ii) from Theorem \ref{thm:suff_saddle} to hold only for systems $\hat{\bf M}^{0}$ and $\hat{\bf N}^{0}$ which are defined in \eqref{eqn:M_xi} and \eqref{eqn:N_zeta} with $\xi_t=\zeta_t = \indd{t = T}$. That is, we consider only
\begin{align*}
\hat M^0(\theta) &= \E \Big[\int_{[0, \theta)} g_t \ud \hat\zeta_t \Big| \cF^1_\theta \Big] + \E\big[1-\hat\zeta_{\theta-}\big| \cF^1_\theta\big] \cU^1_\theta,\\
\hat N^0(\gamma) &= \E \Big[\int_{[0, \gamma)} f_t \ud \hat\xi_t \Big| \cF^2_\gamma \Big] + \E\big[1-\hat\xi_{\gamma-}\big| \cF^2_\gamma\big] \cU^2_\gamma,
\end{align*}
for $(\theta, \gamma) \in \cT_0(\F^1) \times \cT_0(\F^2)$. The price to pay for such relaxation is to add conditions (iii) and (iv) which are the analogue in this context of the necessary condition (i) in Proposition \ref{prop:support}.

\begin{theorem}\label{thm:suff_saddle_mart}
Let $(\cU^1_t)_{t \in [0, T]}$ and $(\cU^2_t)_{t \in [0, T]}$ be $\F^1$- and $\F^2$-progressively measurable processes, respectively, and let $(\hat\xi,\hat \zeta)\in \cAcirc_0(\F^1)\times\cAcirc_0(\F^2)$. Assume that
\begin{enumerate}[(i)]
\item $\hat{\bf M}^{0}$ is a $\cT_0(\F^1)$-submartingale system,
\item $\hat{\bf N}^{0}$ is a $\cT_0(\F^2)$-supermartingale system,
\item $\optional{\big(}f_\cdot(1-\hat\zeta_\cdot)\big)_t^{\F^1}+\optional{\big(}h_\cdot\Delta\hat\zeta_\cdot\big)_t^{\F^1} \ge \optional{(}1-\hat\zeta_{\cdot-})_t^{\F^1}\, \cU^1_t$ for all $t \in [0,T]$, $\P$-a.s.,
\item $\optional{\big(}g_\cdot(1-\hat\xi_\cdot)\big)^{\F^2}_t+\optional{\big(}h_\cdot\Delta\hat\xi_\cdot\big)^{\F^2}_t \le \optional{(}1-\hat\xi_{\cdot-})_t^{\F^2}\, \cU^2_t$ for all $t \in [0,T]$, $\P$-a.s.,
\item $\E[\cU^1_0] = \E[\cU^2_0]$.
\end{enumerate}
Then the game has a value, i.e., $\overline V = \underline V = \E[\cU^1_0] = \E[\cU^2_0]$, and the randomised stopping times $(\hat \tau, \hat \sigma) \in \cT^R_0(\F^1) \times \cT^R_0(\F^2)$ generated by $(\hat \xi, \hat \zeta)$ form a saddle point of the game.
\end{theorem}
\begin{proof}
We will closely follow ideas of the proof of Theorem \ref{thm:suff_saddle}. We will prove \eqref{eqn:U1} and skip analogous arguments for \eqref{eqn:U2}. The rest of the proof follows similarly as the proof of Theorem \ref{thm:suff_saddle}.

Take any $\xi \in \cAcirc_0(\F^1)$ and set $\tau(z) = \inf\{t \ge 0: \xi_t > z\} \in \cT(\F^1)$. As in \eqref{eq:suff2} we have
\begin{equation}\label{eq:suff3}
\begin{aligned}
\E \Big[\! \int_{[0,T)}\!\! f_t (1\!-\!\hat \zeta_t)  \ud \xi_t\! +\! \int_{[0,T)}\!\!g_t (1\!-\!\xi_t) \ud \hat \zeta_t\! +\!\! \sum_{t \in [0,T)}\!\!h_t \Delta \xi_t \Delta \hat\zeta_t \! +\! h_T \Delta\xi_{T} \Delta\zeta_{T}  \Big| \cF^1_0 \Big]=\int_0^1\hat m^\xi(z)\ud z 
\end{aligned}
\end{equation}
where $(z,\omega)\mapsto \hat m^\xi(z,\omega)$ is $\cB([0,T])\times\cF^1_0$-measurable and for each $z\in[0,1]$ we have, $\P$-a.s.,
\[
\hat m(z)=\E \Big[ f_{\tau(z)}(1-\hat\zeta_{\tau(z)}) + \int_{[0, \tau(z))} g_s \ud \hat\zeta_s + h_{\tau(z)} \Delta \hat \zeta_{\tau(z)} \Big| \cF^1_0 \Big].
\]
By the definition of the optional projection and the tower property of conditional expectation, for any $z \in (0,1)$, we have 
\begin{align}\label{eq:suff4}
\begin{aligned}
&\E \Big[ f_{\tau(z)}(1-\hat\zeta_{\tau(z)}) + \int_{[0, \tau(z))} g_s \ud \hat\zeta_s + h_{\tau(z)} \Delta \hat \zeta_{\tau(z)} \Big| \cF^1_0 \Big]\\
&=
\E \Big[ \int_{[0, \tau(z))} g_s \ud \hat\zeta_s \Big| \cF^1_0 \Big] + \E \big[ \E \big[ f_{\tau(z)}(1-\hat\zeta_{\tau(z)}) + h_{\tau(z)} \Delta \hat \zeta_{\tau(z)} \big| \cF^1_{\tau(z)} \big]\big| \cF^1_0 \big]\\
&=
\E \Big[ \int_{[0, \tau(z))} g_s \ud \hat\zeta_s \Big| \cF^1_0 \Big] + \E \big[ \optional{\big(}f_\cdot(1-\hat\zeta_{\cdot-})\big)_{\tau(z)}^{\F^1}+\optional{\big(}h_\cdot\Delta\hat\zeta_\cdot\big)_{\tau(z)}^{\F^1} \big| \cF^1_0 \big]\\
&\ge
\E \Big[ \int_{[0, \tau(z))} g_s \ud \hat\zeta_s +\optional{\big(}1-\hat\zeta_{\cdot-}\big)^{\F^1}_{\tau(z)}  \cU^1_{\tau(z)}\Big| \cF^1_0\Big]\\
&=
\E \Big[ \int_{[0, \tau(z))} g_s \ud \hat\zeta_s +\E\big[1-\hat\zeta_{\tau(z)-}\big|\cF^1_{\tau(z)}\big]\,  \cU^1_{\tau(z)}\Big| \cF^1_0\Big]
\ge \cU^1_0,
\end{aligned}
\end{align}
where the first inequality is by (iii) and the last one by the submartingale property (i) and the fact that $\hat\zeta_{0-} = 0$. Analogously to the proof of Theorem \ref{thm:suff_saddle_stop} we deduce $\int_0^1\hat m^\xi(z)\ud z\ge \cU^1_0$. In summary, combining the latter with  \eqref{eq:suff3} we have proved
\[
\E \Big[\! \int_{[0,T)}\!\! f_t (1\!-\!\hat \zeta_t)  \ud \xi_t\! +\! \int_{[0,T)}\!\!g_t (1\!-\!\xi_t) \ud \hat \zeta_t\! +\!\! \sum_{t \in [0,T)}\!\!h_t \Delta \xi_t \Delta \hat\zeta_t \! +\!(1\!-\!\xi_{T-}) \E[1\!-\!\hat \zeta_{T-}|\cF^1_T] h_T \Big| \cF^1_0 \Big] \ge \cU^1_0.
\]
That is, we have \eqref{eqn:U1}. The remaining arguments in the proof of Theorem \ref{thm:suff_saddle} apply and yield identical statement of the existence of value and the pair $(\hat\xi, \hat\zeta)$ generating a saddle point of the game.
\end{proof}

\begin{remark}[{\bf A link to full information games}] 
Theorem \ref{thm:suff_saddle_mart} provides parallels between the asymmetric information framework of this paper and the classical theory of zero-sum stopping games of \cite{lepeltier1984}. In the classical setting the game has one value process $\cU$ which must satisfy $f_t \ge \cU_t \ge g_t$, which is represented by conditions (iii) and (iv). For a candidate saddle point $(\hat\tau, \hat \sigma)$, in the setting of \cite{lepeltier1984} the process $\cU_{t \wedge \hat \sigma}$ must be a supermartingale while the process $\cU_{t \wedge \hat \tau}$ must be a submartingale. Those properties are analogous to conditions (i) and (ii) in our theorem, respectively. Condition (v) is unique to the asymmetric setting, providing {\em the only} link between the candidate value processes $\cU^1$ and $\cU^2$.
\end{remark}


\section{Applications to two classes of games}\label{sec:examples}
The setting introduced above is sufficiently general to cover nearly all known examples in the literature on Dynkin games with partial and asymmetric information. Here we consider two benchmark examples that illustrate how to write more explicit formulae for the players' subjective views and players' equilibrium values introduced in the previous sections. 

In both examples, the underlying assumption is that both players know the structure of the game, in the sense that they know which processes and random variables are involved, although they may not observe their realisations. Moreover, both players know that there is an asymmetry of information and each player knows what type of information their opponent has access to.

The first example, in Section \ref{sec:partial}, is borrowed from \cite{Grun2013} which, however, only considers a Markovian setting. We discuss the non-Markovian version of the game, whose in-depth study can be found in the PhD thesis \cite{smith2024martingale}. The second example in Section \ref{sec:partdyn} is borrowed from \cite{DEG2020}, where a verification theorem is formulated and then used to solve a particular case of the general problem. Our analysis provides rigorous mathematical foundations for the verification theorem in \cite{DEG2020} which, otherwise, was the result of an educated guess.

\subsection{Partially observed scenarios}\label{sec:partial}
Let $(\Omega,\cF,\P)=(\Omega^0\times\Omega^1,\cF^0\times\cF^1,\P^0\times\P^1)$ be a product probability space. Let $(\cH_t)_{t\in[0,T]}$ be a filtration on $(\Omega^0,\cF^0)$, with $\cH_0=\{\Omega^0,\varnothing\}$, and denote by $\F^2=(\cF^2_t)_{t\in[0,T]}$ the $\P^0\times\P^1$-completion of the filtration $(\cH_t\times\{\Omega^1,\varnothing\})_{t\in[0,T]}$ (see Appendix \ref{sec:optionalproj} for more details). 
The space $(\Omega^1,\cF^1,\P^1)$ hosts a random variable $\cJ$ taking values $0$ and $1$ with $\P(\cJ=1)=\P^1(\cJ=1)=\pi$. The analysis for $\cJ$ with any finite number of values is analogous. By construction $\cJ$, considered as a r.v. on $(\Omega, \cF, \P)$, is independent of $\cF^2_T$.  Define a filtration $\F^1$ as $\cF^1_t=\cF^2_t\vee\sigma(\cJ)$ and notice that $\cF^1_T \ne \cF$ as the probability space must also carry the randomisation devices for both players. 

Let $f^0,f^1,g^0,g^1,h^0,h^1\in\cL_b(\P)$ be $\F^2$-adapted and such that $f^j_t\ge h^j_t\ge g^j_t$, for all $t\in[0,T]$, $\P$-a.s., for $j=0,1$.
We set
\begin{align}\label{eq:decomp0}
\begin{split}
&f_t=f^\cJ_t\coloneqq f^0_t1_{\{\cJ=0\}}+f^1_t1_{\{\cJ=1\}},\\
&g_t=g^\cJ_t\coloneqq g^0_t1_{\{\cJ=0\}}+g^1_t1_{\{\cJ=1\}},\\
&h_t=h^\cJ_t\coloneqq h^0_t1_{\{\cJ=0\}}+h^1_t1_{\{\cJ=1\}}.
\end{split}
\end{align} 
Since $\cF^1_t\supsetneq \cF^2_t$, we assume that Player 1 (minimiser) is {\em fully} informed whereas Player 2 (maximiser) is {\em partially} informed because she cannot observe directly $\cJ$.

Thanks to Lemma \ref{lem:suffcdual2}, any strategy $\xi\in\cAcirc_0(\F^1)$ of the fully informed player decomposes as
\begin{align}\label{eq:decomp1}
\xi_t=\xi^0_t \ind_{\{\cJ=0\}}+\xi^1_t \ind_{\{\cJ=1\}},
\end{align}
with $\xi^j\in\cAcirc_0(\F^2)$ for $j=0,1$, whereas $\zeta \in \cAcirc_0(\F^2)$. Notice that for $\theta\in\cT_0(\F^1)$ the same decomposition leads to
$\theta=\theta_0 \ind_{\{\cJ=0\}}+\theta_1 \ind_{\{\cJ=1\}}$,
with $\theta_0,\theta_1\in\cT_0(\F^2)$ (see, also, \cite[Corollary~3.2]{de2022value}). 
That motivates treating the informed player as having two types/incarnations, the \emph{incarnation $0$} and the \emph{incarnation $1$}, potentially collaborating with each other.
\smallskip

\textbf{Subjective views.} Since the above decomposition holds for an optimal $(\xi^*,\zeta^*)\in \cAcirc_0(\F^1)\times\cAcirc_0(\F^2)$, then the processes defined in \eqref{eq:Pi} read (recall the convention $0/0=1$)
\begin{align*}
&\Pi^{*,1}_\theta = 1,\quad\theta\in\cT_0(\F^1),\\
&\Pi^{*,2}_\gamma
=\frac{(1-\xi^{*,0}_{\gamma-})}{\pi(1-\xi^{*,1}_{\gamma-})+(1-\pi)(1-\xi^{*,0}_{\gamma-})}\ind_{\{\cJ=0\}}\\
&\qquad\quad+\frac{(1-\xi^{*,1}_{\gamma-})}{\pi(1-\xi^{*,1}_{\gamma-})+(1-\pi)(1-\xi^{*,0}_{\gamma-})}\ind_{\{\cJ=1\}},\quad\gamma\in\cT_0(\F^2).
\end{align*}

These processes as well as equilibrium values have an intuitive representation as long as the game is \emph{still played}, by which we mean on the events 
\[
\{\zeta^*_{\theta-} < 1\} = \{\sigma^* \ge \theta\},
\quad\text{and}\quad 
\{\xi^{*,1}_{\gamma-} \wedge\xi^{*,0}_{\gamma-}<1\} = \{\P(\tau^* \ge \gamma | \cF^2_{\gamma})>0\}.
\]
For $\gamma\in\cT_0(\F^2)$, on the event $\{\xi^{*,1}_{\gamma-} \wedge\xi^{*,0}_{\gamma-}<1\}$, the expression for $\Pi^{*,2}_\gamma$ simplifies to
\[
\Pi^{*,2}_\gamma = \frac{1-p_\gamma}{1-\pi}\ind_{\{\cJ=0\}}+\frac{p_\gamma}{\pi}\ind_{\{\cJ=1\}},
\]
where
\begin{equation}\label{eq:ppos}
 p_\gamma = \frac{\pi(1-\xi^{*,1}_{\gamma-})}{\pi(1-\xi^{*,1}_{\gamma-})+(1-\pi)(1-\xi^{*,0}_{\gamma-})}.
\end{equation}
The random variable $p_\gamma$ has a meaning of a \emph{belief} of the partially informed player: it is indeed easy to verify that
\begin{equation}\label{eqn:belief}
p_\gamma = \frac{\P(\cJ = 1, \tau_* \ge \gamma | \cF^2_\gamma)}{\P(\tau_* \ge \gamma| \cF^2_\gamma)} \quad \text{on $\{\xi^{*,1}_{\gamma-}\wedge\xi^{*,0}_{\gamma-}<1\}$},
\end{equation}
where $\tau_*$ is the randomised stopping time generated by $(\xi^*_t)_{t\in[0,T]}$. We also note that the conditional probability distribution $(1-p_\gamma, p_\gamma)$ of $\cJ$ is absolutely continuous with respect to the initial distribution $(1-\pi, \pi)$ with the Radon-Nikodym density given by $\Pi^{*,2}_\gamma$. 
\smallskip

\textbf{Equilibrium value processes.} Given a stopping time $\gamma\in\cT_0(\F^2)$ we introduce the conditional expected payoffs in each state of $\cJ$, i.e., for $i=0,1$,
\begin{align}\label{eq:Lpayoff}
L^i(\xi^i,\zeta | \cF^2_\gamma)\coloneqq\E\Big[\int_{[\gamma,T)}f^i_t(1-\zeta_t)\ud \xi^i_t+\int_{[\gamma,T)}g^i_t(1-\xi^i_t)\ud \zeta_t+\sum_{t\in[\gamma,T]}h^i_t\Delta\zeta_t\Delta\xi^i_t\Big|\cF^2_\gamma\Big].
\end{align}
Recalling that $\zeta^{*; \theta}$ and $\xi^{*,i; \gamma}$ denote the truncation of strategies $\zeta^*$ and $\xi^{*, i}$ at stopping times $(\theta,\gamma) \in \cT_0(\F^1) \times \cT_0(\F^2)$, we rewrite the formulae for the equilibrium values $V^{*,1}$ and $V^{*,2}$ (cf., \eqref{eq:V1V2}) using the above notation $L^i$. In particular, using Lemma \ref{lem:rv_decomposition} in the third equality below we obtain on $\{\zeta^*_{\theta-}<1\}$
\begin{align*}
V^{*,1}(\theta)&=\essinf_{\xi\in\cAcirc_{\theta}(\F^1)}\E\Big[\int_{[\theta,T)}f^\cJ_t(1-\zeta^{*;\theta}_t)\ud \xi_t+\int_{[\theta,T)}g^\cJ_t(1-\xi_t)\ud \zeta^{*;\theta}_t+\sum_{t\in[\theta,T]}h^\cJ_t\Delta\zeta^{*;\theta}_t\Delta\xi_t\Big|\cF^1_\theta \Big]\\
&=\essinf_{\xi\in\cAcirc_{\theta}(\F^1)}\E\Big[\sum_{i=0}^1 \indd{\cJ=i} \Big(\int_{[\theta_i,T)}f^i_t(1-\zeta^{*;\theta_i}_t)\ud \xi^i_t+\int_{[\theta_i,T)}g^i_t(1-\xi^i_t)\ud \zeta^{*;\theta_i}_t\\
&\hspace{265pt}+\sum_{t\in[\theta_i,T]}h^i_t\Delta\zeta^{*;\theta_i}_t\Delta\xi^i_t\Big)\Big|\cF^1_\theta \Big]\\
&=\ind_{\{\cJ=0\}}\essinf_{\xi^0\in\cAcirc_{\theta_0}(\F^2)}L^0(\xi^0,\zeta^{*;\theta_0}|\cF^2_{\theta_0})+\ind_{\{\cJ=1\}}\essinf_{\xi^1\in\cAcirc_{\theta_1}(\F^2)}L^1(\xi^1,\zeta^{*;\theta_1}|\cF^2_{\theta_1}).
\end{align*}
From the expression for $\Pi^{*,2}_\gamma$ we get, on $\{\xi^{*,1}_{\gamma-}\wedge\xi^{*,0}_{\gamma-}<1\}$,
\begin{align*}
V^{*,2}(\gamma)&=\esssup_{\zeta\in\cAcirc_\gamma(\F^2)}J^{\Pi^{*,2}}(\xi^{*; \gamma},\zeta|\cF^2_\gamma)
=\esssup_{\zeta\in\cAcirc_\gamma(\F^2)}\Big(p_\gamma L^1(\xi^{*,1;\gamma},\zeta|\cF^2_{\gamma})+(1-p_\gamma) L^0(\xi^{*,0;\gamma},\zeta|\cF^2_{\gamma})\Big),
\end{align*}
where the second equality holds thanks to independence of $\cF^2_T$ and $\sigma(\cJ)$.

The derived formulae for $V^{*,1}(\theta)$ and $V^{*,2}(\gamma)$ motivate the introduction of a new notation: 
\[
U^i(\theta_i) \coloneqq \essinf_{\xi^i\in\cAcirc_{\theta_i}(\F^2)}L^i(\xi^i,\zeta^{*;\theta_i}|\cF^2_{\theta_i}),\ i=0,1, \quad \text{and}\quad V(\gamma) = V^{*,2}(\gamma).
\]
Notice that these objects are well-defined on the whole $\Omega$ but hold a meaning related to the game only on the events $\{\zeta^*_{\theta_i-}<1\}$ and $\{\xi^{*,0}_{\gamma-}\wedge\xi^{*,1}_{\gamma-}<1\}$, respectively.
Thus, $U^i(\theta_i)$ is the value of the game at time $\theta_i$ for the $i$-th incarnation of the informed player, while $V(\gamma)$ is the value of the game at time $\gamma$ for the uninformed player. For simplicity, it is convenient to denote 
\[
\langle \pi, \phi \rangle = \pi \phi^1 + (1-\pi) \phi^0,\quad\text{for any $\phi \in \R^2$.}
\]
Then, by Theorem \ref{thm:value}, there are $\F^2$-optional processes $(U^0_t)_{t\in[0,T]}$, $(U^1_t)_{t\in[0,T]}$ and $(V_t)_{t\in[0,T]}$ such that 
\begin{equation}\label{eqn:ui}
(1-\zeta^*_{\theta_i-}) U^i_{\theta_i} = (1-\zeta^*_{\theta_i-}) U^i(\theta_i),
\end{equation}
for any $\F^2$-stopping time $\theta_i$, and
\begin{equation}\label{eqn:V_process}
\langle \pi, (1-\xi^{*}_{\gamma-})\rangle V_{\gamma} = \langle \pi, (1-\xi^{*}_{\gamma-})\rangle V(\gamma),
\end{equation}
for any $\F^2$-stopping time $\gamma$. 
Thanks to \eqref{eq:hatv12}, on $\{\zeta^*_{\theta_i-}<1\}$, we identify the process $(U^i_t)_{t\in[0,T]}$ with the value process of the following optimal stopping problem: 
\begin{equation}\label{eq:Uit}
U^i_{\theta_i} = \essinf_{\tau\in\cT_{\theta_i}(\F^2)}
\E\Big[f^i_{\tau}(1-\zeta^{*; \theta_i}_{\tau}) +\int_{[\theta_i,\tau)}g^i_t\ud \zeta^{*; \theta_i}_t + h^i_{\tau}\Delta\zeta^{*; \theta_i}_{\tau}\Big|\cF^2_{\theta_i} \Big]
= \essinf_{\tau\in\cT_{\theta_i}(\F^2)} L^i(\tau,\zeta^{*;\theta_i}|\cF^2_{\theta_i}),
\end{equation}
where in the final expression we slightly abuse the notation and write $L^i(\tau,\zeta^{*;\theta_i}|\cF^2_{\theta_i})$ for $L^i(\xi,\zeta^{*;\theta_i}|\cF^2_{\theta_i})$ with $\xi_t = \ind_{\{t \ge \tau\}}$. Similarly, we write on $\{\xi^{*,0}_{\gamma-}\wedge\xi^{*,1}_{\gamma-}<1\}$
\begin{align}\label{eq:Vg}
\begin{aligned}
V_\gamma 
&= \esssup_{\sigma\in\cT_\gamma(\F^2)} \E \Big[ \int_{[\gamma,\sigma)} \big\langle p_t, f_t\ud \xi^{*;\gamma}_t\big\rangle + \big\langle p_\sigma, g_\sigma(1-\xi^{*;\gamma}_\sigma) + h_\sigma \Delta \xi^{*;\gamma}_\sigma \big\rangle  \Big|\cF^2_\gamma\Big]\\
&= \esssup_{\sigma\in\cT_\gamma(\F^2)} \Big(p_\gamma L^1(\xi^{*,1;\gamma},\sigma|\cF^2_{\gamma})+(1-p_\gamma) L^0(\xi^{*,0;\gamma},\sigma|\cF^2_{\gamma})\Big),
\end{aligned}
\end{align}
where $\langle p_t, f_t\ud \xi^{*;\gamma}_t\rangle=p_t f^1_t\ud \xi^{*,1;\gamma}_t+(1-p_t)f^0_t \ud \xi^{*,1;\gamma}_t$ (and analogously for the other terms) and $L^i(\xi^{*,i;\gamma},\sigma|\cF^2_{\gamma})$ stands for $L^i(\xi^{*,i;\gamma},\zeta|\cF^2_{\gamma})$ with $\zeta= \ind_{\{t\ge\sigma\}}$. 
\smallskip

\textbf{Relationship between players' equilibrium values and the role of the belief process.} From the second statement in Remark \ref{rem:3.13}, we have the relationship
\[
(1-\zeta^*_{\gamma-})\E\big[(1-\xi^*_{\gamma-})V^{*,1}(\gamma)|\cF^2_\gamma\big]=\E\big[1-\xi^*_{\gamma-}|\cF^2_\gamma\big](1-\zeta^*_{\gamma-})V^{*,2}(\gamma)
\]
for any $\gamma \in \cT(\F^2)$. Noticing that $\E\big[1-\xi^*_{\gamma-}|\cF^2_\gamma\big] = \langle \pi, 1-\xi^*_{\gamma-} \rangle$, the right-hand side reads
\[
\langle \pi, 1-\xi^*_{\gamma-} \rangle (1-\zeta^*_{\gamma-})V^{*,2}(\gamma) = \langle \pi, 1-\xi^*_{\gamma-} \rangle (1-\zeta^*_{\gamma-})V_\gamma
\]
with the equality justified by \eqref{eqn:V_process}. For the left-hand side, we write
\begin{align*}
&(1-\zeta^*_{\gamma-})\E\big[(1-\xi^*_{\gamma-})V^{*,1}(\gamma)\big|\cF^2_\gamma\big]\\
&=
\E\big[(1-\zeta^*_{\gamma-}) \big(\indd{\cJ=1} (1-\xi^{*,1}_{\gamma-}) U^1(\gamma) + \indd{\cJ=0} (1-\xi^{*,0}_{\gamma-}) U^0(\gamma) \big) \big|\cF^2_\gamma \big]\\
&=
(1-\zeta^*_{\gamma-}) (1-\xi^{*,1}_{\gamma-}) U^1_\gamma \E\big[\indd{\cJ=1} \big|\cF^2_\gamma \big] 
+ (1-\zeta^*_{\gamma-}) (1-\xi^{*,0}_{\gamma-}) U^0_\gamma \E\big[\indd{\cJ=0} \big|\cF^2_\gamma \big]\\
&=
(1-\zeta^*_{\gamma-}) (1-\xi^{*,1}_{\gamma-}) U^1_\gamma \pi
+ (1-\zeta^*_{\gamma-}) (1-\xi^{*,0}_{\gamma-}) U^0_\gamma (1-\pi)\\
&=
\langle \pi, 1-\xi^*_{\gamma-} \rangle (1-\zeta^*_{\gamma-}) \big( p_\gamma U^1_\gamma + (1-p_\gamma)U^0_\gamma\big),
\end{align*}
where the decomposition \eqref{eq:decomp1} of $\xi^*$ was used in the first inequality, \eqref{eqn:ui} in the second equality, and the last equality follows from the definition of $p_\gamma$. In conclusion, we have
\[
\langle \pi, 1-\xi^*_{\gamma-} \rangle (1-\zeta^*_{\gamma-}) \big( p_\gamma U^1_\gamma + (1-p_\gamma)U^0_\gamma\big)
=
\langle \pi, 1-\xi^*_{\gamma-} \rangle (1-\zeta^*_{\gamma-})V_\gamma,
\]
which is best viewed as the equality
\begin{equation}\label{eqn:p_representation}
\langle p_\gamma,  U_\gamma \rangle = V_\gamma \qquad \text{on the set $\Gamma^2_\gamma$,}
\end{equation}
where $\Gamma^2_\gamma$ is defined in Proposition \ref{thm:aggr2}. This set takes an explicit form here:
\begin{align}\label{eq:gamma2expl}
\Gamma^2_\gamma = \{ \xi^{*,1}_{\gamma-} \wedge \xi^{*,0}_{\gamma-} < 1 \text{ and } \zeta^*_{\gamma-} < 1\}.
\end{align}
That lends a natural interpretation of \eqref{eqn:p_representation} as being true as long as the optimally played game has positive probability to still being played at time $\gamma$. 
\smallskip

\textbf{Martingale characterisation.} Take $\xi = (\xi^0, \xi^1) \in \cAcirc_0(\F^1)$ and $\zeta \in \cAcirc_0(\F^2)$, and recall the $\F^1$-optional submartingale $(M^\xi_t)_{t\in[0,T]}$ and the $\F^2$-optional supermartingale $(N^\zeta_t)_{t\in[0,T]}$ from Proposition \ref{prop:subsupmg}. They take the following form in this model: for any $\theta = (\theta^1, \theta^2) \in \cT_0(\F^1)$ and $\gamma \in \cT_0(\F^2)$
\begin{align*}
M^\xi_\theta &= \sum_{i=0}^1 \indd{\cJ = i} \Big( \int_{[0, \theta_i)} \big[(1-\zeta^*_t) f^i_t + \Delta\zeta^*_t h^i_t\big] \ud\xi^i_t + \int_{[0, \theta_i)} (1-\xi^i_t) g^i_t \ud\zeta^*_t + (1-\xi^i_{\theta_i-})(1-\zeta^*_{\theta_i-}) U^i_{\theta_i} \Big)\\
&=: \indd{\cJ = 0} M^{\xi^0; 0}_{\theta_0} + \indd{\cJ = 1} M^{\xi^1; 1}_{\theta_1},\\
N^\zeta_\gamma &= \sum_{i=0}^1 \pi_i \Big( \int_{[0, \gamma)} \big[(1-\zeta_t) f^i_t + \Delta\zeta_t h^i_t\big] \ud\xi^{*,i}_t + \int_{[0, \gamma)} (1-\xi^{*,i}_t) g^i_t \ud\zeta_t + (1-\xi^{*, i}_{\gamma-})(1-\zeta_{\gamma-}) V_\gamma \Big),
\end{align*}
where $(M^{\xi; i}_t)_{t\in[0,T]}$, $i=0, 1$, are $\F^2$-optional submartingales. When $\xi$ and $\zeta$ are chosen optimally, the processes $(M^{\xi^{*,0};0}_t)_{t\in[0,T]}$, $(M^{\xi^{*,1};1}_t)_{t\in[0,T]}$ and $(N^{\zeta^*}_t)_{t\in[0,T]}$ become \cadlag $\F^2$-martingales by Proposition \ref{thm:aggr2}. On the other hand, when $\xi$ and $\zeta$ are taken equal to zero, the above processes take the form: for $t \in [0, T]$,
\begin{equation}\label{eqn:6.8}
\begin{aligned}
M^{0;i}_{t} &=  \int_{[0, t)} g^i_s \ud \zeta^*_s + (1-\zeta^*_{t-}) U^i_{t}, \qquad i=0,1,\\
N^0_t &= \sum_{i=0}^1 \pi_i \Big( \int_{[0, t)} f^i_s \ud \xi^{*,i}_s + (1-\xi^{*, i}_{t-}) V_t \Big).
\end{aligned}
\end{equation}
Proposition \ref{thm:aggr1} asserts that $(M^{0; i}_t)$ is a \cadlag $\F^2$-submartingale and $(N^0_t)$ is a \cadlag $\F^2$-super\-mar\-tin\-gale. These can be shown to be martingales up to the ``last optimal stopping time'' for the respective player in the following way. We first recall the notation \eqref{eq:taubar} and adapt it to the present setting as
\[
\bar\tau_{*}^i(z)=\inf\{t\in[0,T]:\xi^{*,i}_t>z\}\quad\text{and}\quad \bar\sigma_*(z)=\inf\{t\in[0,T]:\zeta^*_t>z\},
\]
for $z \in [0,1)$. Corollary \ref{cor:M_zero_mart} then yields that $(M^{0; i}_{t \wedge \bar\tau^i_*(z)})$ is an $\F^2$-martingale for any $z \in [0, 1)$ and $(N^{0}_{t \wedge \bar\sigma_*(z)})$ is an $\F^2$-martingale for any $z \in [0, 1)$. If any of the generating processes $\xi^{*,i}$ or $\zeta^*$ has a jump at $T$, the respective processes are martingales on the full interval $[0, T]$.  
\smallskip

\textbf{Support of optimal strategies.} In line with the statement of Proposition \ref{prop:support}, given an optimal pair $(\xi^*,\zeta^*)\in\cA^\circ_0(\F^1)\times \cA^\circ_0(\F^2)$, define
\begin{align*}
Y^1_t &= \sum_{i=0}^1 \indd{\cJ=i} \big((1-\zeta^*_{t-}) U^i_t - f^i_t(1-\zeta^*_t) - h^i_t \Delta \zeta^*_t \big) =: \indd{\cJ=0} Z^0_t + \indd{\cJ=1} Z^1_t,\\
Y^2_t &=  \langle \pi, 1-\xi^{*}_{t-}\rangle V_t - \langle \pi, g_t(1-\xi^{*}_t) + h_t \Delta\xi^{*}_t \rangle,
\end{align*}
where we interpret $g_t(1-\xi^{*}_t) + h_t \Delta\xi^{*}_t$ as a vector with entries $g^i_t(1-\xi^{*,i}_t) + h^i_t \Delta\xi^{*,i}_t$, $i=0,1$.
These processes have a more convenient representation as:
\begin{align*}
Z^i_t &= (U^i_t - f^i_t)(1-\zeta^*_t) + (U^i_t - h^i_t) \Delta \zeta^*_t, \quad i=0,1,\\
Y^2_t &= \sum_{i=0}^1 \pi_i \big((V_t - g^i_t)(1-\xi^{*,i}_t) + (V_t - h^i_t) \Delta\xi^{*,i}_t \big).
\end{align*}

\begin{corollary} \label{cor:acting}
We have $Z^0_t \le 0$, $Z^1_t \le 0$ and $Y^2_t \ge 0$ for all $t \in [0, T]$, $\P$\as Moreover, 
\begin{align*}
\int_{[0,T]} Z^0_t \ud \xi^{*,0}_t + \int_{[0,T]} Z^1_t \ud \xi^{*,1}_t = 0, \quad\text{and}\quad \int_{[0,T]} Y^2_t \ud \zeta^*_t = 0.
\end{align*}
\end{corollary}
To better appreciate the conclusions of the above corollary, it is useful to assume that $(\xi^{*}_t)_{t\in[0,T]}$ and $(\zeta^*_t)_{t\in[0,T]}$ do not jump simultaneously for $t < T$. This is to be expected if $f^i_t> h^i_t > g^i_t$ for every $t \in [0, T)$: intuitively, a simultaneous jump at time $t\in[0,T)$ corresponds to the players stopping simultaneously at time $t$ with some probability; that yields a payoff $h^i_t$; each player would then prefer to delay her own jump, in order to score the more preferable payoff $f^i_t$ (for the maximiser) or $g^i_t$ (for the minimiser). 
In the absence of simultaneous jumps, the statement of Corollary \ref{cor:acting} can be rewritten in a more intuitive way as:
\begin{equation}\label{eqn:act1}
\begin{aligned}
&\int_{[0,T)} (U^i_t - f^i_t)(1-\zeta^*_t) \ud \xi^{*,i}_t = 0,\\
&(U^i_T - h^i_T)\Delta \zeta^*_T \Delta \xi^{*, i}_T = 0, 
\end{aligned}
\end{equation}
for $i=0,1$, and
\begin{equation}\label{eqn:act2}
\begin{aligned}
&\int_{[0,T)} \langle \pi, (1-\xi^{*}_{t-})\rangle \big( V_t - \langle p_{t}, g_t \rangle \big) \ud \zeta^*_t = 0,\\
&\langle \pi, (1-\xi^{*}_{T-})\rangle \big(V_T - \langle p_T, h_T \rangle\big) \Delta \zeta^*_T = 0.
\end{aligned}
\end{equation}
Equation \eqref{eqn:act2} shows that the uninformed player acts only when her value process $(V_t)_{t\in[0,T]}$ coincides with the {\em believed} payoff $\langle p_t, g_t \rangle$. Moreover, if the game has not ended by time $T$, it must be $V_T=\langle p_T, h_T \rangle$. Similarly, equation \eqref{eqn:act1} shows that the $i$-th incarnation of the informed player acts when either her value process $(U^i_t)_{t\in[0,T]}$ coincides with her payoff $(f^i_t)_{t\in[0,T]}$, or at the terminal time $T$.

We now proceed to show an interesting feature of the solution of games with asymmetric information: the equilibrium value process of the opponent determines a player's behaviour. To this end, we refer the reader to Corollary \ref{cor:support} and proceed to define objects appearing there. Let $\bar U^i_t = (1-\zeta^{*}_{t-}) U^i_t$, $i=0,1$, and $\bar V_t = \langle p_t, (1-\xi^{*}_{t-}) \rangle V_t$. The optional projections in \eqref{eq:BVass} are related to the processes
\[
t \mapsto \int_{[0, t)} g^i_s \ud \zeta^*_s, \quad i=0,1, \qquad \text{and} \qquad t \mapsto \int_{[0,t)} \langle p_t, f_s \ud \xi^{*}_s\rangle,
\]
and assumed to be continuous. According to Corollary \ref{cor:support}, the previsible bounded variation process in the semimartingale decomposition of each of the processes $(\bar U^i_{t \wedge \tau_*(z)})$ is equal to $-\int_{[0, t \wedge \tau_*(z))} g^i_s \ud \zeta^*_s$ for any $z \in [0,1)$, i.e., value processes of each incarnation of the informed player determine the optimal strategy of the uninformed player. Analogously, the previsible bounded variation process in the semimartingale decomposition of the value process $(\bar V_{t \wedge \sigma_*(z)})$ of the uninformed player equals $-\int_{[0,t \wedge \sigma_*(z))} \langle p_t, f_s \ud \xi^{*}_s \rangle$. In full generality, this may not be sufficient to find strategies of both incarnations of the informed player -- particularly because processes $U^0$, $U^1$ and $V$ are not known explicitly and, if they were, there may be multiple ways of choosing $(\xi^*,\zeta^*)$ so as to satisfy the requirements of Corollary \ref{cor:support}. However, this information can be combined with the supports of those strategies, see \eqref{eqn:act1}, to identify the incarnation/player who needs to act at a given time $t$. 
\smallskip

\textbf{Sufficient conditions.} We first rewrite assumptions of Theorem \ref{thm:suff_saddle_mart} explicitly in the form of the next corollary. \begin{corollary}\label{cor:saddle_mart}
Let $(\hat U^0_t)_{t\in[0,T]}$, $(\hat U^1_t)_{t\in[0,T]}$, $(\hat V_t)_{t\in[0,T]}$ be $\F^2$-progressively measurable processes, let $\hat \xi^0, \hat \xi^1, \hat\zeta \in \cAcirc_0(\F^2)$ and denote by $\hat p_t$ an analogue of \eqref{eq:ppos} with $\hat\xi^j$ in place of $\xi^{*,j}$. Set
\begin{align*}
\hat M^{0;i}_{t} &=  \int_{[0, t)} g^i_s \ud \hat\zeta_s + (1-\hat\zeta_{t-}) \hat U^i_{t},\\
\hat N^0_t &= \sum_{i=0}^1 \pi_i \Big( \int_{[0, t)} f^i_s \ud \hat\xi^{i}_s + (1-\hat\xi^{i}_{t-}) \hat V_t \Big) = \int_{[0, t)} \langle \pi, f_s \ud \hat\xi_s \rangle + \langle \pi, 1-\hat\xi_{t-} \rangle \hat V_t.
\end{align*}

Assume that
\begin{enumerate}[(i)]
\item the process $(\hat M^{0;i}_{t})_{t\in[0,T]}$ is an $\F^2$-submartingale for $i=0,1$,
\item the process $(\hat N^0_t)_{t\in[0,T]}$ is an $\F^2$-supermartingale,
\item for $i=0,1$, it holds $\P$-a.s.,
\[
f^i_t+(h^i_t-f^i_t)\frac{\Delta\hat\zeta_t}{1-\hat\zeta_{t-}} \ge\hat  U^i_t,\quad \text{for all $t \in [0, T]$ such that $\hat \zeta_{t-}<1$}, 
\]
\item it holds $\P$-a.s.,
\[
\langle \hat p_t, g_t\rangle +\frac{\langle \pi,(h_t-g_t)\Delta\hat\xi_t\rangle }{\langle \pi, 1-\hat\xi_{t-}\rangle } \le\hat V_t,\quad\text{for all $t \in [0, T]$ such that $\langle \pi, 1-\hat\xi_{t-}\rangle>0$ },
\] 
\item $\hat V_0 = \langle \pi, \hat U_0 \rangle$.
\end{enumerate}
Then the value of the game equals $\hat V_0$ and a saddle point is given by $((\hat\xi^0, \hat\xi^1), \hat \zeta)$.
\end{corollary}

We draw the attention of the reader to the fact that the above conditions are both \emph{necessary and sufficient}. Their sufficiency is justified by Theorem \ref{thm:suff_saddle_mart}. The necessity of (i)-(ii) is shown in Proposition \ref{thm:aggr1}. The necessity of (iii) and (iv) is by Corollary \ref{cor:acting}. Condition (v) is a consequence of \eqref{eqn:p_representation} with $\gamma = 0$ upon noticing that $p_0 = \pi$. By Corollary \ref{cor:acting}, one can see that the process $(\hat\xi^i_t)$ increases only when the inequality in (iii) is an equality. Similarly, the process $(\hat \zeta_t)$ increases only when there is equality in (iv). 

We leave the adaptation of Theorem \ref{thm:suff_saddle} to the reader and rephrase only Theorem \ref{thm:suff_saddle_stop}.
\begin{corollary}\label{cor:saddle_stop}
Let $\hat U^0_0, \hat U^1_0, \hat V_0 \in \R$ and $\hat \xi^0, \hat \xi^1, \hat\zeta \in \cAcirc_0(\F^2)$. Assume that
\begin{enumerate}[(i)]
 \item for any $i=0,1$ and any $\tau_i \in \cT(\F^2)$, we have
\begin{equation}\label{eqn:s3}
\E\Big[f^i_{\tau_i}(1-\hat\zeta_{\tau_i}) + \int_{[0, {\tau_i})} g^i_s \ud \hat\zeta_s + h^i_{\tau_i} \Delta \hat\zeta_{\tau_i} \Big] \ge \hat U^i_0,
\end{equation}
 \item for any $\sigma \in \cT(\F^2)$, we have
\begin{equation}\label{eqn:s4}
\E\Big[\int_{[0, \sigma)} \langle \pi, f_s \ud \hat\xi_s \rangle + \langle \pi, g_\sigma (1-\hat\xi_\sigma) + h_\sigma \Delta \hat\xi_\sigma\rangle \Big] \le \hat V_0,
\end{equation}
\item $\langle \pi, \hat U_0 \rangle = \hat V_0$.
\end{enumerate}
Then the value of the game is $\hat V_0$ and a saddle point is given by $((\hat\xi^0, \hat\xi^1), \hat \zeta)$.
\end{corollary}

\begin{remark}
It is not difficult to generalise this discussion to the case of two random variables $\cJ$, $\cK$, taking values in two finite sets $J$, $K$, and to payoff processes $f^{j,k},g^{j,k},h^{j,k}$, $(j,k)\in J\times K$. We could assume that all processes are adapted to a filtration $\mathbb{M}=(\cM_t)_{t\in[0,T]}$, that Player 1 has access to the filtration $\cF^1_t=\cM_t\vee\sigma(\cJ)$ and Player 2 has access to the filtration $\cF^2_t=\cM_t\vee\sigma(\cK)$. All the considerations made above would continue to hold, up to using a heavier notation.
\end{remark}

\subsection{Partially observed dynamics}\label{sec:partdyn}
Let $(\Omega,\cF,\P)$ be a probability space supporting a Brownian motion $(W_t)_{t\in[0,T]}$, players' randomisation devices and an independent random variable $\cJ\in\{0,1\}$ with $\pi=\P(\cJ=1)$. The probability space is assumed to have a product structure that allows us to apply results of Appendix \ref{sec:optionalproj}, see also Subsection \ref{sec:partial} for a similar construction. Denote by $\F^W\coloneqq(\cF^W_t)_{t\in[0,T]}$ the filtration generated by $(W_t)_{t\in[0,T]}$ augmented with $\P$-null sets. Set $\cF_t=\sigma(\cJ)\vee\cF^W_t$, so $\F$ also satisfies the usual conditions. Let $\mu_0,\mu_1$ and $\sigma$ be sufficiently regular functions so that $(X_t)_{t\in[0,T]}$ is the unique $\F$-adapted solution of the SDE on $\R$,
\[
X_t=x+\int_0^t\mu_\cJ(X_s)\ud s+\int_0^t\sigma(X_s)\ud W_s,\quad t\in[0,T].
\] 
The existence of a strong solution means that there is a measurable map $\Gamma:[0,T]\times C([0,T])\times \{0,1\}\to \R$ such that 
\begin{align}\label{eq:Xmap}
X_t(\omega)=\Gamma(t,W_{\cdot\wedge t}(\omega),\cJ(\omega)),\quad \text{for $(t,\omega)\in[0,T]\times\Omega$,}
\end{align}
see, e.g., \cite[Ch. IV, Thm. 3.2]{ikeda}.

Sometimes we denote $X=X^\cJ$ in order to emphasise the role played by $\cJ$ in determining the dynamics of $X$. In particular, for $j=0,1$, on $\{\cJ=j\}$ we have $X=X^j$, where $X^j$ is the solution of the SDE above with $\mu_\cJ$ replaced by $\mu_j$. Then, 
\begin{align}\label{eq:Xj}
X^j_t(\omega)=\Gamma(t,W_{\cdot\wedge t}(\omega),j)\eqqcolon \Gamma^j(t,W_{\cdot\wedge t}(\omega)),\quad\text{for $(t,\omega)\in[0,T]\times\Omega$.}
\end{align}
We assume that the $\P$-augmentation of the filtration $\F^{X^j}$ generated by $(X^j_t)_{t\in[0,T]}$ equals $\F^W$ -- this is guaranteed when $\sigma$ is locally uniformly non-degenerate. 

In this framework, Player 2, dubbed the \emph{uninformed player}, only observes the process $(X_t)_{t\in[0,T]}$, which translates formally into the filtration $\cF^2_t=\sigma(X_s,s\le t)$ (also augmented with $\P$-null sets). Player 1, called the \emph{informed player}, additionally observes $\cJ$, i.e., has access to the filtration $\cF^1_t=\cF^2_t\vee\sigma(\cJ) = \cF_t$, where the last equality follows from the assumption that the filtration generated by each $(X^j_t)_{t\in[0,T]}$ equals $\F^W$.

Any admissible strategy $\xi\in\cAcirc_0(\F^1)$ for the informed player decomposes as (cf.\ Lemma \ref{lem:suffcdual2}) $\xi_t=\xi^0_t \ind_{\{\cJ=0\}}+\xi^1_t \ind_{\{\cJ=1\}}$, with $\xi^j\in\cAcirc_0(\F^W)$ for $j=0,1$. Similarly, a stopping time $\theta\in\cT_0(\F^1)$ decomposes as $\theta=\theta_0 \ind_{\{\cJ=0\}}+\theta_1 \ind_{\{\cJ=1\}}$ with $\theta_0,\theta_1\in\cT_0(\F^W)$. By the right continuity of $(\xi^j_t)_{t\in[0,T]}$ there is a measurable map $\Xi^j:[0,T]\times C([0,T])\to \R$ such that $\xi^j_t(\omega)=\Xi^j(t,W_{\cdot\wedge t}(\omega))$ for $(t,\omega)\in[0,T]\times\Omega$. Due to the equality $\F^W = \F^{X^j}$, there is a measurable map $\tilde \Xi^j:[0,T]\times C([0,T])\to \R$ such that $\xi^j_t(\omega)=\tilde \Xi^j(t,X^j_{\cdot\wedge t}(\omega))$. We further have
\begin{equation}\label{eqn:tildexi_def}
\indd{\cJ=j} \xi^j_t= \indd{\cJ=j}\tilde \Xi^j(t,X^j_{\cdot\wedge t}) = \indd{\cJ=j} \tilde \Xi^j(t,X_{\cdot\wedge t}) = \indd{\cJ=j} \tilde \xi^j_t,
\end{equation}
where $\tilde \xi^j_t \coloneqq \tilde \Xi^j(t,X_{\cdot\wedge t})$, so $(\tilde \xi^j_t)_{t\in[0,T]}$ can be computed based on the observation of the process $(X_t)_{t\in[0,T]}$. In other words, $\tilde\xi^j \in \cAcirc_0(\F^2)$,
which will prove useful in the derivation of $\Pi^{*,2}$ and the value $V^{*,2}$ of Player 2. The processes $\tilde \xi^j$, $j=0,1$, can be interpreted as the result of the uninformed player's calculation, while observing the process $X$ under the {\em conviction} that $\cJ = j$. Similarly, for $j=0,1$ we can introduce $\tilde\theta_j\in\cT_0(\F^2)$ such that $\indd{\cJ=j}\theta_j=\indd{\cJ=j}\tilde\theta_j$. The construction of $\tilde \theta_j$ is similar to the construction of $\tilde \xi^j$ by considering the process $t\mapsto \indd{\theta_j \le t} \in \cAcirc_0(\F^W)$.

To derive a useful relationship between $\xi^j$ and $\tilde \xi^j$, we define probability measures $\P^0$ and $\P^1$ by
\begin{align}\label{eq:Girs1}
\frac{\ud \P^0}{\ud \P}=\frac{1}{1-\pi}\ind_{\{\cJ=0\}}\quad\text{and}\quad \frac{\ud \P^1}{\ud \P}=\frac{1}{\pi}\ind_{\{\cJ=1\}},
\end{align}
and recall the following formulae for the conditional expectation under the change of measure: for a $\P$-integrable random variable $Y$
\begin{align}\label{eq:Girs2}
\begin{aligned}
\E^0[Y|\cF^2_\gamma]&=\frac{(1-\pi)^{-1}\E[\ind_{\{\cJ=0\}}Y|\cF^2_\gamma]}{(1-\pi)^{-1}\E[\ind_{\{\cJ=0\}}|\cF^2_\gamma]}=\frac{1}{1-\psi_\gamma}\E\big[Y\ind_{\{\cJ=0\}}\big|\cF^2_\gamma\big],\\
\E^1[Y|\cF^2_\gamma]&=\frac{\pi^{-1}\E[\ind_{\{\cJ=1\}}Y|\cF^2_\gamma]}{\pi^{-1}\E[\ind_{\{\cJ=0\}}|\cF^2_\gamma]}=\frac{1}{\psi_\gamma}\E\big[Y\ind_{\{\cJ=1\}}\big|\cF^2_\gamma\big],
\end{aligned}
\end{align}
where $\psi_t \coloneqq \P(\cJ=1|\cF^2_t)=\E[\cJ|\cF^2_t]$ is the co-called posterior process, i.e., Player 2's best estimate of the value of $\cJ$ based upon the observation of the process $X$. This allows us to obtain the following identities:
\begin{equation}\label{eqn:tildexi_id}
\begin{aligned}
\psi_t \tilde \xi^1_t &=\E\big[\ind_{\{\cJ=1\}}\tilde \xi^1_t\big|\cF^2_t\big]=\E\big[\ind_{\{\cJ=1\}} \xi^1_t\big|\cF^2_t\big]=\psi_t\E^1\big[\xi^1_t\big|\cF^2_t\big],\\
(1-\psi_t) \tilde \xi^0_t &=\E\big[\ind_{\{\cJ=0\}}\tilde \xi^0_t\big|\cF^2_t\big]=\E\big[\ind_{\{\cJ=0\}} \xi^0_t\big|\cF^2_t\big]=(1-\psi_t)\E^0\big[\xi^0_t\big|\cF^2_t\big].
\end{aligned}
\end{equation}
Hence we have the relationship
\begin{equation}\label{eqn:tildexi}
\tilde \xi^j_t=\E^j\big[\xi^j_t\big|\cF^2_t\big]\quad \text{for $j=0,1$.}
\end{equation}

Before specifying the payoff functions, we make some considerations about the structure of the belief process for the uninformed player.
Much of the analysis in this section repeats steps from the study in Section \ref{sec:partial} and we therefore omit some details. 
\smallskip

\textbf{Belief process.} In this framework direct observation of the process $X$ is informative about the nature of the true drift and the belief process for the uninformed player takes a different form compared to the previous example (where each player could only learn from the actions of her opponent -- or rather the lack thereof). Given an optimal pair $(\xi^*,\zeta^*)\in \cAcirc_0(\F^1)\times\cAcirc_0(\F^2)$ and using $\indd{\cJ=j}\xi^j_t=\indd{\cJ=j}\tilde\xi^j_t$ with \eqref{eqn:tildexi}, the processes defined in \eqref{eq:Pi} read (recall the convention $0/0=1$)
\begin{align*}
\begin{aligned}
&\Pi^{*,1}_\theta = 1,\quad\theta\in\cT_0(\F^1),\\
&\Pi^{*,2}_\gamma=\frac{(1-\tilde\xi^{*,0}_{\gamma-})}{\psi_{\gamma}(1-\tilde\xi^{*,1}_{\gamma-})+(1-\psi_{\gamma})(1-\tilde\xi^{*,0}_{\gamma-})}\ind_{\{\cJ=0\}}\\
&\qquad\quad+\frac{(1-\tilde\xi^{*,1}_{\gamma-})}{\psi_{\gamma}(1-\tilde\xi^{*,1}_{\gamma-})+(1-\psi_{\gamma})(1-\tilde\xi^{*,0}_{\gamma-})}\ind_{\{\cJ=1\}}\quad\gamma\in\cT_0(\F^2).
\end{aligned}
\end{align*}
The dynamics of the process $(\psi_t)_{t\in[0,T]}$ can be computed explicitly and the two dimensional process $(X_t,\psi_t)_{t\in[0,T]}$ is ruled, under the measure $\P$, by the SDE, for $t\in[0,T]$
\begin{align}\label{eq:SDEXpsi}
\begin{aligned}
X_t&=x+\int_0^t\big(\mu_0(X_s)(1-\psi_s)+\mu_1(X_s)\psi_s\big)\ud s+\int_0^t \sigma(X_s)\ud B_s,\\
\psi_t&=\psi_0+\int_0^tw(\psi_s)\psi_s(1-\psi_s)\ud B_s,
\end{aligned}
\end{align}
where $w(z)=(\mu_1(z)-\mu_0(z))/\sigma(z)$ is the signal-to-noise ratio and $(B_t)_{t\in[0,T]}$ is a $(\F^2,\P)$-Brownian motion (the so-called innovation process) defined as:
\[
B_t\coloneqq\int_0^t\frac{\ud X_s}{\sigma(X_s)}-\int_0^t\frac{\mu_0(X_s)+(\mu_1(X_s)-\mu_0(X_s))\psi_s}{\sigma(X_s)}\ud s.
\]

In line with the previous section we now introduce a {\em belief process}, which features both the learning from observation of $X$ (via the process $\psi$) and learning from the actions of the informed player. That is, setting
\begin{align}\label{eq:pt}
p_\gamma=\frac{\psi_\gamma(1-\tilde\xi^{*,1}_{\gamma-})}{\psi_{\gamma}(1-\tilde\xi^{*,1}_{\gamma-})+(1-\psi_{\gamma})(1-\tilde\xi^{*,0}_{\gamma-})},\quad\gamma\in\cT_0(\F^2),
\end{align}
we express $\Pi^{*,2}_\gamma$ on $\{\tilde\xi^{*,0}_{\gamma-}\wedge\tilde\xi^{*,1}_{\gamma-}<1\}$ as
\[
\Pi^{*,2}_\gamma = \frac{1-p_\gamma}{1-\psi_\gamma}\ind_{\{\cJ=0\}}+\frac{p_\gamma}{\psi_\gamma}\ind_{\{\cJ=1\}},\quad\gamma\in\cT_0(\F^2).
\]
Notice that $p_0=\psi_0=\P(\cJ=1)=\pi$ and we have a clear interpretation of the belief process $(p_t)_{t\in[0,T]}$ as the simple calculations below demonstrate: letting $Z\sim U([0,1])$ be Player 1's randomisation device, we have on the event $\{\tilde\xi^{*,0}_{\gamma-}\wedge\tilde\xi^{*,1}_{\gamma-}<1\}$
\begin{align*}
\frac{\P\big(\cJ=1,\tau_*\ge \gamma\big|\cF^2_\gamma\big)}{\P\big(\tau_*\ge \gamma\big|\cF^2_\gamma\big)}
&=\frac{\P\big(\cJ=1,\xi^{*,1}_{\gamma-} \le  Z\big|\cF^2_\gamma\big)}{\P\big(\cJ=1,\xi^{*,1}_{\gamma-} \le  Z\big|\cF^2_\gamma\big)+\P\big(\cJ=0,\xi^{*,0}_{\gamma-} \le  Z\big|\cF^2_\gamma\big)}\\
&=\frac{\E\big[\ind_{\{\cJ=1\}}(1-\xi^{*,1}_{\gamma-})\big|\cF^2_\gamma\big]}{\E\big[\ind_{\{\cJ=1\}}(1-\xi^{*,1}_{\gamma-})\big|\cF^2_\gamma\big]+\E\big[\ind_{\{\cJ=0\}}(1-\xi^{*,0}_{\gamma-})\big|\cF^2_\gamma\big]}=p_\gamma,
\end{align*}
where the second equality holds because $Z$ is independent of $\cF^2_\gamma$ and the final one holds because of \eqref{eqn:tildexi}. The above shows that the process $(p_t)_{t\in[0,T]}$ is Player 2's posterior probability of $\cJ=1$ based on the observation of $(X_t)_{t\in[0,T]}$ and on the fact that the game has not ended prior to time $t\in[0,T]$. 
\medskip

\textbf{Equilibrium value processes.} We assume that the payoff processes $(f_t,g_t,h_t)_{t\in[0,T]}$ are defined as functions of the underlying process $X$. More precisely, given measurable functions $f, g, h:[0,T]\times\R^d\to\R$, with $f\ge h\ge g$, we set $f_t\coloneqq f(t,X_t)$, $g_t\coloneqq g(t,X_t)$ and $h_t\coloneqq h(t,X_t)$. Clearly the triplet $(f_t,g_t,h_t)_{t\in[0,T]}$ is $\F^2$-adapted and we assume each process to be in $\cL_b(\P)$. 

Recalling the notation $X=X^\cJ$ along with the fact that on the event $\{\cJ=j\}$ the process $X=X^j$ follows the dynamics $\ud X^j_t=\mu_j(X^j_t)\ud t+\sigma(X^j_t)\ud W_t$, we can formally cast the current problem in a similar (but not identical) fashion as the problem in Section \ref{sec:partial}. That is, we set $f^j_t\coloneqq f(t,X^j_t)$, $g^j_t\coloneqq g(t,X^j_t)$ and $h^j_t\coloneqq h(t,X^j_t)$. It follows from \eqref{eq:Xj} that $X^j$ is independent of $\cJ$, so the same holds for $f^j,g^j,h^j$. We must also notice that for $\zeta^*\in\cAcirc_0(\F^2)$ there is a measurable map $\Lambda:[0,T]\times C([0,T])\to\R$ such that
\begin{equation}\label{eq:zeta*dec}
\begin{aligned}
\zeta^*_t(\omega)&=\Lambda(t,X_{\cdot\wedge t}(\omega))=\sum_{j=0}^1\ind_{\{\cJ=j\}}(\omega)\Lambda(t,X^j_{\cdot\wedge t}(\omega))\\
&=\sum_{j=0}^1\ind_{\{\cJ=j\}}(\omega)\Lambda\big(t,\Gamma^j(t,W_{\cdot\wedge t}(\omega))\big)\eqqcolon\sum_{j=0}^1\ind_{\{\cJ=j\}}(\omega)\zeta^{*,j}_t(\omega),
\end{aligned}
\end{equation}
where clearly $\zeta^{*,0}$ and $\zeta^{*,1}$ are $\F^W$-adapted and independent of $\cJ$.

\begin{remark}
Notice that the decomposition in \eqref{eq:zeta*dec} should not be interpreted as saying that the uninformed player selects $\zeta^{*,0}$ or $\zeta^{*,1}$ depending on whether $\cJ=0$ or $\cJ=1$. The uninformed player chooses $\zeta^*$ (i.e., the map $\Lambda(\cdot,\cdot)$) but the informed player ``knows'' the value of $\cJ$. Therefore, the informed player knows the law of the realised strategy $(\zeta^{*,j}_t)_{t\in[0,T]}$, whereas the uninformed player can only estimate it as a mixture of the laws of $\zeta^{*,0}$ and $\zeta^{*,1}$. 
\end{remark}

Recalling the decompositions $\xi_t=\xi^0_t\ind_{\{\cJ=0\}}+\xi^1_t\ind_{\{\cJ=1\}}$, $\theta_t=\theta_0\ind_{\{\cJ=0\}}+\theta_1\ind_{\{\cJ=1\}}$ and that $\F^1=\F$, we have on $\{\zeta^{*,0}_{\theta_0-}\wedge\zeta^{*,1}_{\theta_1-}<1\}$
\begin{align*}
V^{*,1}(\theta)
&=\essinf_{\xi\in\cAcirc_{\theta}(\F)}\E\Big[\int_{[\theta,T)}f^\cJ_t(1-\zeta^{*;\theta}_t)\ud \xi_t+\int_{[\theta,T)}g^\cJ_t(1-\xi_t)\ud \zeta^{*;\theta}_t+\sum_{t\in[\theta,T]}h^\cJ_t\Delta\zeta^{*;\theta}_t\Delta\xi_t\Big|\cF_\theta \Big]\\
&=\essinf_{\xi\in\cAcirc_{\theta}(\F)}\sum_{j=0}^1 \indd{\cJ=j} \E\Big[\int_{[\theta_j,T)}f^j_t(1-\zeta^{*,j;\theta_j}_t)\ud \xi^j_t+\int_{[\theta_j,T)}g^j_t(1-\xi^j_t)\ud \zeta^{*,j;\theta_j}_t\\
&\hspace{265pt}+\sum_{t\in[\theta_j,T]}h^j_t\Delta\zeta^{*,j;\theta_j}_t\Delta\xi^j_t\Big|\cF_\theta \Big].
\end{align*}
Since $f^j,g^j,h^j,\xi^j,\zeta^{*,j}$ are $\F^W$-adapted, $\theta_j$ is a $\F^W$-stopping time and $\cJ$ is independent of $\cF^W_T$, we are allowed to use Lemma \ref{lem:rv_decomposition} to obtain the following representation of $V^{*,1}(\theta)$ on the event $\{\zeta^{*,0}_{\theta_0-}\wedge\zeta^{*,1}_{\theta_1-}<1\}$:
\begin{align*}
V^{*,1}(\theta)=\ind_{\{\cJ=0\}}\essinf_{\xi^0\in\cAcirc_{\theta_0}(\F^W)}L^0(\xi^0,\zeta^{*,0;\theta_0}|\cF^W_{\theta_0})+\ind_{\{\cJ=1\}}\essinf_{\xi^1\in\cAcirc_{\theta_1}(\F^W)}L^1(\xi^1,\zeta^{*,1;\theta_1}|\cF^W_{\theta_1}),
\end{align*}
 with the same notation as in \eqref{eq:Lpayoff}.

For the uninformed player calculations are slightly different. For $\gamma \in\cT_0(\F^2)$ and $\zeta\in\cAcirc_\gamma(\F^2)$, 
\begin{align*}
&J^{\Pi^{*,2}_\gamma}\big(\xi^{*;\gamma},\zeta\big|\cF^2_\gamma\big)\\
&=\E\Big[\Pi^{*,2}_\gamma\Big(\int_{[\gamma,T)}f^\cJ_t(1-\zeta_t)\ud \xi^{*;\gamma}_t+\int_{[\gamma,T)}g^\cJ_t(1-\xi^{*;\gamma}_t)\ud \zeta_t+\sum_{t\in[\theta,T]}h^\cJ_t\Delta\zeta_t\Delta\xi^{*;\gamma}_t\Big)\Big|\cF^2_\gamma\Big]\\
&=(1-p_\gamma)\E^0\Big[\int_{[\gamma,T)}f^0_t(1-\zeta_t)\ud \xi^{*,0;\gamma}_t+\int_{[\gamma,T)}g^0_t(1-\xi^{*,0;\gamma}_t)\ud \zeta_t+\sum_{t\in[\theta,T]}h^0_t\Delta\zeta_t\Delta\xi^{*,0;\gamma}_t\Big|\cF^2_\gamma\Big]\\
&\quad+p_\gamma\E^1\Big[\int_{[\gamma,T)}f^1_t(1-\zeta_t)\ud \xi^{*,1;\gamma}_t+\int_{[\gamma,T)}g^1_t(1-\xi^{*,1;\gamma}_t)\ud \zeta_t+\sum_{t\in[\theta,T]}h^1_t\Delta\zeta_t\Delta\xi^{*,1;\gamma}_t\Big|\cF^2_\gamma\Big]\\
&\eqqcolon (1-p_\gamma)\bar L^0\big(\xi^{*,0;\gamma},\zeta\big|\cF^2_\gamma\big)+p_\gamma\bar L^1\big(\xi^{*,1;\gamma},\zeta\big|\cF^2_\gamma\big),
\end{align*}
on the event $\{\tilde\xi^{*,0}_{\gamma-}\wedge\tilde\xi^{*,1}_{\gamma-}<1\}$, where $\P^0$ and $\P^1$ are probability measures defined in \eqref{eq:Girs1} and we used \eqref{eq:Girs2}. It is also worth noticing that $\bar L^j(\xi^{*,j;\gamma},\zeta|\cF^2_\gamma)=\bar L^j(\tilde\xi^{*,j;\gamma},\zeta|\cF^2_\gamma)$, $j=0,1$, due to the identity $\xi^{*,j}=\tilde\xi^{*,j}$ under $\P^j$.

In conclusion, we obtain, for $(\theta,\gamma)\in\cT_0(\F^1)\times\cT_0(\F^2)$,
\begin{align*}
V^{*,1}(\theta)&=\ind_{\{\cJ=0\}}\essinf_{\xi^0\in\cAcirc_{\theta_0}(\F^W)}L^0(\xi^0,\zeta^{*,0;\theta_0}|\cF^W_{\theta_0})+\ind_{\{\cJ=1\}}\essinf_{\xi^1\in\cAcirc_{\theta_1}(\F^W)}L^1(\xi^1,\zeta^{*,1;\theta_1}|\cF^W_{\theta_1}),\\
&\eqqcolon \ind_{\{\cJ=0\}} U^0(\theta_0)+\ind_{\{\cJ=1\}} U^1(\theta_1),\\
V^{*,2}(\gamma)
&=\esssup_{\zeta\in\cAcirc_\gamma(\F^2)}\Big((1-p_\gamma) \bar L^0\big(\xi^{*,0;\gamma},\zeta\big|\cF^2_{\gamma}\big) + p_\gamma \bar L^1\big(\xi^{*,1;\gamma},\zeta\big|\cF^2_{\gamma}\big)\Big)\eqqcolon V(\gamma),
\end{align*}
on $\{\zeta^{*,0}_{\theta_0-}\wedge\zeta^{*,1}_{\theta_1-}<1\}$ and $\{\tilde\xi^{*,0}_{\gamma-}\wedge\tilde\xi^{*,1}_{\gamma-}<1\}$, respectively.

In order to apply Theorem \ref{thm:value} in this context, it is convenient to notice that for $\gamma\in\cT_0(\F^2)$ 
\begin{align}\label{eq:projF2}
\begin{aligned}
\E\big[1-\xi^*_{\gamma-}\big|\cF^2_\gamma\big]&=\E\big[(1-\xi^{*,0}_{\gamma-})\ind_{\{\cJ=0\}}+(1-\xi^{*,1}_{\gamma-})\ind_{\{\cJ=1\}}\big|\cF^2_\gamma\big]\\
&=(1-\tilde \xi^{*,0}_{\gamma-})(1-\psi_\gamma)+(1-\tilde \xi^{*,1}_{\gamma-})\psi_\gamma,
\end{aligned}
\end{align}
by the decomposition $\xi^*=\xi^{*,0}\ind_{\{\cJ=0\}}+\xi^{*,1}\ind_{\{\cJ=1\}}$ and the identity \eqref{eqn:tildexi_id}. Similarly, using \eqref{eq:zeta*dec} we have for $\theta\in\cT_0(\F^1)$
\begin{align*}
\E\big[1-\zeta^*_{\theta-}\big|\cF^1_\theta\big]&=(1-\zeta^{*,0}_{\theta_0-})\ind_{\{\cJ=0\}}+(1-\zeta^{*,1}_{\theta_1-})\ind_{\{\cJ=1\}}.
\end{align*}
Then, invoking Theorem \ref{thm:value} there are $\F^W$-optional processes $(U^0_t)_{t\in[0,T]}$, $(U^1_t)_{t\in[0,T]}$ and an $\F^2$-optional process $(V_t)_{t\in[0,T]}$ such that for any $\F^W$-stopping time $\theta_i$ and any $\F^2$-stopping time $\gamma$
\begin{align*}
(1-\zeta^{*,i}_{\theta_i-}) U^i_{\theta_i} = (1-\zeta^{*,i}_{\theta_i-}) U^i(\theta_i)\quad \text{and}\quad \langle \psi_\gamma, (1-\tilde\xi^{*}_{\gamma-})\rangle V_{\gamma} = \langle \psi_\gamma, (1-\tilde\xi^{*}_{\gamma-})\rangle V(\gamma),
\end{align*}
where we recall that $\langle \psi_\gamma, \phi \rangle = \psi_\gamma \phi^1 + (1-\psi_\gamma) \phi^0$ for any $\phi \in \R^2$. 
Moreover, representations as optimal stopping problems similar to those in \eqref{eq:Uit} and \eqref{eq:Vg} continue to hold: that is
\begin{align}\label{eq:OSprocess}
\begin{aligned}
U^j_{\theta_j}&=\essinf_{\tau\in\cT_{\theta_j}(\F^W)}L^j\big(\tau,\zeta^{*,j;\theta_j}\big|\cF^W_{\theta_j}\big),\quad \text{on $\{\zeta^{*,j}_{\theta_j-}<1\}$,}\\
V_\gamma&=\esssup_{\sigma\in\cT_\gamma(\F^2)}\Big(p_\gamma \bar L^1\big(\xi^{*,1;\gamma},\sigma\big|\cF^2_{\gamma}\big)+(1-p_\gamma) \bar L^0\big(\xi^{*,0;\gamma},\sigma\big|\cF^2_{\gamma}\big)\Big),\quad \text{on $\{\tilde\xi^{*,0}_{\gamma-}\wedge\tilde\xi^{*,1}_{\gamma-}<1\}$},
\end{aligned}
\end{align}
where $L^j\big(\tau,\zeta^{*,j;\theta_j}\big|\cF^W_{\theta_j}\big)$ stands for $L^j\big(\xi,\zeta^{*,j;\theta_j}\big|\cF^W_{\theta_j}\big)$ with $\xi_t=\ind_{\{t\ge \tau\}}$ and $\bar L^j\big(\xi^{*,j;\gamma},\sigma\big|\cF^2_{\gamma}\big)$ stands for $\bar L^j\big(\xi^{*,j;\gamma},\zeta\big|\cF^2_{\gamma}\big)$ with $\zeta_t=\ind_{\{t\ge \sigma\}}$.

Define
\begin{equation}\label{eqn:Uj}
\widetilde U^j_t\coloneqq\E^j\big[U^j_t\big|\cF^2_t\big],\quad \text{for $j=0,1$.}
\end{equation}
Observe that
\begin{equation}\label{eq:VUUbar}
\begin{aligned}
\E\big[\ind_{\{\cJ=1\}}\widetilde U^1_t\big|\cF^2_t\big] &= \psi_t \widetilde U^1_t =\psi_t\E^1\big[U^1_t\big|\cF^2_t\big] = \E\big[\ind_{\{\cJ=1\}} U^1_t\big|\cF^2_t\big],\\
\E\big[\ind_{\{\cJ=0\}}\widetilde U^0_t\big|\cF^2_t\big] &= (1-\psi_t) \widetilde U^0_t = (1-\psi_t)\E^0\big[U^0_t\big|\cF^2_t\big] = \E\big[\ind_{\{\cJ=0\}} U^0_t\big|\cF^2_t\big],
\end{aligned}
\end{equation}
where we used \eqref{eqn:Uj} in the centre equalities. We will use these identities in deriving the relationship between values of the informed and uninformed players.
\begin{remark}
A more explicit representation of $\widetilde U^j$ can be provided when $\zeta^{*,j}_t < 1$ for all $t \in [0, T)$. Since the Brownian filtration completed with $\P$-null sets is continuous and $\F^{X^j} = \F^W$, Theorem \ref{thm:value} implies that $(U^0_t)_{t\in[0,T)}$ and $(U^1_t)_{t\in[0,T)}$ are left-continuous. Hence, 
\begin{align}\label{eq:UPsi}
U^j_t=\chi^j(t,X^j_{\cdot\wedge t}),\quad t\in[0,T],
\end{align}
for some measurable map $\chi^j:[0,T]\times C([0,T])\to\R$ which can be chosen so as to guarantee left-continuity of $[0, T) \ni t\mapsto\chi^j(t,x(\cdot\wedge t))$ for any $x(\cdot)\in C([0,T])$; the left-continuity at $T$ is not needed as this is only a single point, so the mapping can be extended to cover it in a measurable way. 

This has an important consequence allowing us to construct processes $\widetilde U^j_t$ explicitly as $\widetilde U^j_t\coloneqq \chi^j(t,X_{\cdot\wedge t})$ for $j=0,1$, instead of using the definition in \eqref{eqn:Uj}. Indeed, the processes $(\widetilde U^j_t)_{t\in[0,T]}$ are $\F^2$-adapted and left-continuous everywhere apart from $T$ (thus optional) and the following identities trivially hold
\begin{align}\label{eq:UUbar}
\ind_{\{\cJ=0\}}\widetilde U^0_t=\ind_{\{\cJ=0\}} U^0_t\quad\text{and}\quad\ind_{\{\cJ=1\}}\widetilde U^1_t=\ind_{\{\cJ=1\}} U^1_t.
\end{align}
We also observe that taking conditional expectations in \eqref{eq:UUbar} we get
\begin{align*}
\psi_t \widetilde U^1_t &=\E\big[\ind_{\{\cJ=1\}}\widetilde U^1_t\big|\cF^2_t\big]=\E\big[\ind_{\{\cJ=1\}} U^1_t\big|\cF^2_t\big]=\psi_t\E^1\big[U^1_t\big|\cF^2_t\big],\\
(1-\psi_t) \widetilde U^0_t &=\E\big[\ind_{\{\cJ=0\}}\widetilde U^0_t\big|\cF^2_t\big]=\E\big[\ind_{\{\cJ=0\}} U^0_t\big|\cF^2_t\big]=(1-\psi_t)\E^0\big[U^0_t\big|\cF^2_t\big],
\end{align*}
whence the relationship $\widetilde U^j_t=\E^j\big[U^j_t\big|\cF^2_t\big]$ for $j=0,1$ as in \eqref{eqn:Uj}.

This construction provides an intuitive interpretation of $\widetilde U^j$. This is the value process of the first optimal stopping problem in \eqref{eq:OSprocess} perceived by the uninformed player who only observes $(X_t)_{t\in[0,T]}$ but believes that $\cJ = j$ (irrespective of the true value of this random variable). 
\end{remark}

\smallskip

{\bf Relationship between the equilibrium value processes and the role of the belief process}. From the second statement in Remark \ref{rem:3.13}, we have 
\[
(1-\zeta^*_{\gamma-})\E\big[(1-\xi^*_{\gamma-})V^{*,1}(\gamma)|\cF^2_\gamma\big]=\E\big[1-\xi^*_{\gamma-}|\cF^2_\gamma\big](1-\zeta^*_{\gamma-})V^{*,2}(\gamma)
\]
for any $\gamma \in \cT(\F^2)$. Recalling \eqref{eq:projF2}, the right-hand side reads
\[
\langle \psi_\gamma, 1-\tilde \xi^*_{\gamma-} \rangle (1-\zeta^*_{\gamma-})V^{*,2}(\gamma) = \langle \psi_\gamma, 1-\tilde \xi^*_{\gamma-} \rangle (1-\zeta^*_{\gamma-})V_\gamma.
\]
For the left-hand side, recalling \eqref{eqn:tildexi_def} and \eqref{eqn:Uj} we have
\begin{equation}\label{eq:boh}
\begin{aligned}
&(1-\zeta^*_{\gamma-})\E\big[(1-\xi^{*}_{\gamma-})V^{*,1}(\gamma)\big|\cF^2_\gamma\big]\\
&=\E\big[(1-\zeta^*_{\gamma-})\big(\indd{\cJ=0} (1-\xi^{*,0}_{\gamma-})U^{0}(\gamma) + \indd{\cJ=1} (1-\xi^{*,1}_{\gamma-})U^{1}(\gamma)\big) \big|\cF^2_\gamma\big]\\
&=\E\big[(1-\zeta^*_{\gamma-})\big(\indd{\cJ=0} (1-\tilde\xi^{*,0}_{\gamma-})U^{0}_\gamma + \indd{\cJ=1} (1-\tilde\xi^{*,1}_{\gamma-})U^{1}_\gamma\big) \big|\cF^2_\gamma\big]\\
&=(1-\zeta^*_{\gamma-})(1-\tilde\xi^{*,0}_{\gamma-})\E\big[\indd{\cJ=0}U^{0}_\gamma|\cF^2_\gamma\big] + (1-\zeta^*_{\gamma-})(1-\tilde\xi^{*,1}_{\gamma-})\E\big[\indd{\cJ=1} U^{1}_\gamma |\cF^2_\gamma\big]\\
&=(1-\zeta^*_{\gamma-})(1-\tilde\xi^{*,0}_{\gamma-})(1-\psi_\gamma)\E^0\big[U^{0}_\gamma|\cF^2_\gamma\big] + (1-\zeta^*_{\gamma-})(1-\tilde\xi^{*,1}_{\gamma-})\psi_\gamma\E^1\big[U^{1}_\gamma |\cF^2_\gamma\big]\\
&=(1-\zeta^*_{\gamma-})\langle\psi_\gamma,1-\xi^*_{\gamma-}\rangle\Big((1-p_\gamma)\widetilde U^{0}_\gamma+p_\gamma\widetilde U^{1}_\gamma\Big),
\end{aligned}
\end{equation}
where in the final expression we recalled the form of $p_\gamma$ from \eqref{eq:pt}. 
Combining the above expressions we obtain
\[
(1-\zeta^*_{\gamma-})\langle\psi_\gamma,1-\tilde\xi^*_{\gamma-}\rangle\Big((1-p_{\gamma})\widetilde U^{0}_\gamma+p_{\gamma}\widetilde U^{1}_\gamma\Big)=(1-\zeta^*_{\gamma-})\langle \psi_\gamma, 1-\tilde\xi^*_{\gamma-} \rangle V_\gamma,
\]
which reads more neatly as
\[
\langle p_\gamma,\widetilde U_\gamma\rangle=V_\gamma,\quad \text{on the set $\Gamma^2_\gamma$},
\]
where $\Gamma^2_\gamma = \{ \tilde \xi^{*,1}_\gamma \wedge \tilde\xi^{*,0}_\gamma < 1  \text{ and } \zeta^*_\gamma < 1\}$, cf. \eqref{eq:gamma2expl}.
The main difference with the setting from Section \ref{sec:partial} is that therein the posterior of the uninformed player is only updated via the actions of the more informed player. Here instead, the sole observation of the underlying process $X$ already yields some posterior information $\psi$ about $\cJ$. 
\medskip

{\bf Martingale characterisation}. Take $\xi \in \cAcirc_0(\F^1)$ with decomposition $\xi=\ind_{\{\cJ=0\}}\xi^0+\ind_{\{\cJ=1\}}\xi^1$, where $(\xi^0, \xi^1)\in\cAcirc_0(\F^W)$, and $\zeta \in \cAcirc_0(\F^2)$. Recall the $\F^1$-optional submartingale $(M^\xi_t)_{t\in[0,T]}$ and the $\F^2$-optional supermartingale $(N^\zeta_t)_{t\in[0,T]}$ from Proposition \ref{prop:subsupmg}. The former reads as follows: for any $\theta = (\theta^0, \theta^1) \in \cT_0(\F^1)$ 
\begin{align*}
M^\xi_\theta &= \sum_{i=0}^1 \indd{\cJ = i} \Big( \int_{[0, \theta_i)}\! \big[(1\!-\!\zeta^{*,i}_t) f^i_t\! +\! \Delta\zeta^{*,i}_t h^i_t\big] \ud\xi^i_t\! +\! \int_{[0, \theta_i)} (1\!-\!\xi^i_t) g^i_t \ud\zeta^{*,i}_t\! +\! (1\!-\!\xi^i_{\theta_i-})(1\!-\!\zeta^{*,i}_{\theta_i-}) U^i_{\theta_i} \Big)\\
&=: \indd{\cJ = 0} M^{\xi^0; 0}_{\theta_0} + \indd{\cJ = 1} M^{\xi^1; 1}_{\theta_1},
\end{align*}
where $(M^{\xi; i}_t)_{t\in[0,T]}$, $i=0, 1$, are $\F^1$-optional submartingales. 

The derivation of the expression for $(N^\zeta_t)_{t\in[0,T]}$ deserves more detail. Notice that $f_t\coloneqq f(t,X_t)$, $g_t\coloneqq g(t,X_t)$, $h_t\coloneqq h(t,X_t)$ are $\cF^2_t$-measurable for all $t\in[0,T]$. Then for $\gamma \in \cT_0(\F^2)$
\begin{align*}
&\E \Big[\! \int_{[0,\gamma)}\!\! \big[f_t (1\!-\!\zeta_t) +h_t\Delta\zeta_t\big] \ud \xi^*_t\Big| \cF^2_\gamma \Big]\\
&=\sum_{j=0}^1\E \Big[\ind_{\{\cJ=j\}} \int_{[0,\gamma)}\!\! \big[f_t (1\!-\!\zeta_t) +h_t\Delta\zeta_t\big]  \ud \xi^{*,j}_t\Big| \cF^2_\gamma \Big]\\
&=\sum_{j=0}^1\E \Big[\ind_{\{\cJ=j\}} \Big| \cF^2_\gamma \Big]\int_{[0,\gamma)}\!\! \big[f_t (1\!-\!\zeta_t) +h_t\Delta\zeta_t\big]  \ud \tilde \xi^{*,j}_t\\
&=\psi_\gamma\int_{[0,\gamma)}\!\! \big[f_t (1\!-\!\zeta_t) +h_t\Delta\zeta_t\big]  \ud \tilde \xi^{*,1}_t+(1-\psi_\gamma)\int_{[0,\gamma)}\!\! \big[f_t (1\!-\!\zeta_t) +h_t\Delta\zeta_t\big] \ud \tilde \xi^{*,0}_t.
\end{align*}
Analogous calculations yield
\[
\E\Big[\int_{[0,\gamma)}\!\!g_t (1\!-\!\tilde\xi^*_t) \ud \zeta_t\Big| \cF^2_\gamma \Big]=\psi_\gamma\int_{[0,\gamma)}\!\!g_t (1\!-\!\tilde\xi^{*,1}_t) \ud \zeta_t+(1-\psi_\gamma)\int_{[0,\gamma)}\!\!g_t (1\!-\!\tilde\xi^{*,0}_t) \ud \zeta_t.
\]
Combining those expressions with \eqref{eq:projF2} we obtain
\begin{align*}
N^\zeta_\gamma &= (1-\psi_\gamma)\Big(\int_{[0,\gamma)}\!\! \big[f_t (1\!-\!\zeta_t) +h_t\Delta\zeta_t\big] \ud \tilde \xi^{*,0}_t+\int_{[0,\gamma)}\!\!g_t (1\!-\!\tilde\xi^{*,0}_t) \ud \zeta_t\Big)\\
&\quad+\psi_\gamma\Big(\int_{[0,\gamma)}\!\! \big[f_t (1\!-\!\zeta_t) +h_t\Delta\zeta_t\big]  \ud \tilde\xi^{*,1}_t+\int_{[0,\gamma)}\!\!g_t (1\!-\!\tilde\xi^{*,1}_t) \ud \zeta_t\Big)+\!(1\!-\!\zeta_{\gamma-})\langle\psi_\gamma,1\!-\!\tilde\xi^*_{\gamma-}\rangle V_\gamma.
\end{align*}
It is also worth noticing that by an application of It\^o's formula
\begin{align*}
N^\zeta_\gamma &= \int_{[0,\gamma)}\!\! (1-\psi_t)\big[f_t (1\!-\!\zeta_t) +h_t\Delta\zeta_t\big] \ud \tilde\xi^{*,0}_t+\int_{[0,\gamma)}\!\!(1-\psi_t)g_t (1\!-\!\tilde\xi^{*,0}_t) \ud \zeta_t\\
&\quad+\int_{[0,\gamma)}\!\! \psi_t\big[f_t (1\!-\!\zeta_t) +h_t\Delta\zeta_t\big]  \ud \tilde\xi^{*,1}_t+\int_{[0,\gamma)}\!\!\psi_t g_t (1\!-\!\tilde\xi^{*,1}_t) \ud \zeta_t\!+\!(1\!-\!\zeta_{\gamma-})\langle\psi_\gamma,1\!-\!\tilde\xi^*_{\gamma-}\rangle V_\gamma\\
&\quad+\int_0^\gamma\Big(\int_{[0,t)}\!\! \big[f_s (1\!-\!\zeta_s) +h_s\Delta\zeta_s\big] \ud(\tilde\xi^{*,1}_s\!-\! \tilde\xi^{*,0}_s)\!+\!\int_{[0,t)}\!\!g_s (\tilde\xi^{*,0}_s\!-\!\tilde\xi^{*,1}_s) \ud \zeta_s\Big)\ud \psi_t.
\end{align*}
Since the integral in the last line is an $\F^2$-martingale and $N^\zeta$ is an $\F^2$-optional supermartingale, we deduce that 
\begin{align*} 
\widetilde N^\zeta_\gamma &\coloneqq \int_{[0,\gamma)}\!\! (1-\psi_t)\big[f_t (1\!-\!\zeta_t) +h_t\Delta\zeta_t\big] \ud \tilde\xi^{*,0}_t+\int_{[0,\gamma)}\!\!(1-\psi_t)g_t (1\!-\!\tilde\xi^{*,0}_t) \ud \zeta_t\\
&\quad+\int_{[0,\gamma)}\!\! \psi_t\big[f_t (1\!-\!\zeta_t) +h_t\Delta\zeta_t\big]  \ud \tilde\xi^{*,1}_t+\int_{[0,\gamma)}\!\!\psi_t g_t (1\!-\!\tilde\xi^{*,1}_t) \ud \zeta_t\!+\!(1\!-\!\zeta_{\gamma-})\langle\psi_\gamma,1\!-\!\tilde\xi^*_{\gamma-}\rangle V_\gamma
\end{align*}
is an $\F^2$-optional supermartingale. 

When $\xi$ and $\zeta$ are chosen optimally, the processes $(M^{\xi^{*,0};0}_t)_{t\in[0,T]}$, $(M^{\xi^{*,1};1}_t)_{t\in[0,T]}$ become \cadlag $\F^W$-martingales and $(N^{\zeta^*}_t)_{t\in[0,T]}$, $(\widetilde N^{\zeta^*}_t)_{t\in[0,T]}$ become \cadlag $\F^2$-martingales by Proposition \ref{thm:aggr2}. 
Instead, when $\xi$ and $\zeta$ are taken equal to zero, the above processes take the form: for $t \in [0, T]$,
\begin{equation}\label{eq:MNN}
\begin{aligned}
M^{0;i}_{t} &=  \int_{[0, t)} g^i_s \ud \zeta^{*,i}_s + (1-\zeta^{*,i}_{t-}) U^i_{t}, \qquad i=0,1,\\
N^0_t &= (1-\psi_t)\int_{[0,t)}\!\! f_s \ud \tilde\xi^{*,0}_s\!+\!\psi_t\int_{[0,t)}\!\! f_s \ud \tilde\xi^{*,1}_s+\langle\psi_t,1\!-\!\tilde\xi^*_{t-}\rangle V_t,\\
\widetilde N^0_t &= \int_{[0,t)}\!\! (1-\psi_s)f_s \ud \tilde\xi^{*,0}_s\!+\!\int_{[0,t)}\!\!\psi_s f_s \ud \tilde\xi^{*,1}_s+\langle\psi_t,1\!-\!\tilde\xi^*_{t-}\rangle V_t.
\end{aligned}
\end{equation}
Proposition \ref{thm:aggr1} asserts that $(M^{0; i}_t)_{t\in[0,T]}$ is a \cadlag $\F^W$-submartingale and $(N^0_t)_{t\in[0,T]}$ (hence also $(\widetilde N^0_t)_{t\in[0,T]}$) is a \cadlag $\F^2$-supermartingale. These can be shown to be martingales up to $\bar \tau^i_*(z)$ and $\bar \sigma_*(z)$, respectively, for any $z\in[0,1)$ by the same arguments as those employed in Section \ref{sec:partial}. 
\medskip

{\bf Support of optimal strategies}. In this paragraph, some minor changes, compared to the case with partially observed scenarios arise, due to the replacement of the prior $\pi$ with the posterior process $\psi_t$. In this spirit, it is worth noticing that since $g(t,X_t)$ is $\cF^2_t$-measurable
\begin{align*}
\optional{\big(}g_\cdot(1-\xi^*_\cdot)\big)^{\F^2}_t&=\E\big[g(t,X_t)\big(\ind_{\{\cJ=0\}}(1-\xi^{*,0}_t)+\ind_{\{\cJ=1\}}(1-\xi^{*,1}_t)\big)\big|\cF^2_t\big]\\
&=g(t,X_t)\E\big[\big(\ind_{\{\cJ=0\}}(1-\tilde\xi^{*,0}_t)+\ind_{\{\cJ=1\}}(1-\tilde\xi^{*,1}_t)\big)\big|\cF^2_t\big]=g(t,X_t)\langle\psi_t,1-\tilde\xi^*_t\rangle,
\end{align*}
where we used \eqref{eq:projF2}. Analogously,
$\optional{(}h_\cdot\Delta\xi^*_\cdot)^{\F^2}_t=h(t,X_t)\langle\psi_t,\Delta\tilde\xi^*_t\rangle$.
Then, in preparation for a statement of Proposition \ref{prop:support} in this setting, let $g_t=g(t,X_t)$, $h_t=h(t,X_t)$ and 
\begin{align*}
Y^1_t &= \sum_{i=0}^1 \indd{\cJ=i} \big((1-\zeta^{*,i}_{t-}) U^i_t - f^i_t(1-\zeta^{*,i}_t) - h^i_t \Delta \zeta^{*,i}_t \big) =: \indd{\cJ=0} Z^0_t + \indd{\cJ=1} Z^1_t,\\
Y^2_t &=  \langle \psi_t, 1-\tilde\xi^{*}_{t-}\rangle V_t - g_t\langle \psi_t, 1-\tilde\xi^{*}_t\rangle - h_t\langle\psi_t, \Delta\tilde\xi^{*}_t \rangle.
\end{align*}
Corollary \ref{cor:acting} holds in the same form and similar considerations as in the paragraph following it apply to the present case with obvious notational changes. 
\begin{corollary} \label{cor:acting2}
We have $Z^0_t \le 0$, $Z^1_t \le 0$ and $Y^2_t \ge 0$ for all $t \in [0, T]$, $\P$\as Moreover,
\begin{equation*}
\int_{[0,T]} Z^0_t \ud \xi^{*,0}_t+\int_{[0,T]} Z^1_t \ud \xi^{*,1}_t = 0 \quad \text{and}\quad \int_{[0,T]} Y^2_t \ud \zeta^*_t = 0.
\end{equation*}
\end{corollary}
The statement of Corollary \ref{cor:acting2} can be rewritten in a more intuitive way under the ansatz that no simultaneous jump of the generating processes occurs for $t<T$. That is: 
\begin{equation*}
\begin{aligned}
&\int_{[0,T)} (U^i_t - f^i_t)(1-\zeta^{*,i}_t) \ud \xi^{*,i}_t = 0,\\
&(U^i_T - h^i_T)\Delta \zeta^{*,i}_T \Delta \xi^{*, i}_T = 0, 
\end{aligned}
\end{equation*}
for $i=0,1$, and
\begin{equation*}
\begin{aligned}
&\int_{[0,T)} \langle \psi_t, (1-\tilde\xi^{*}_{t-})\rangle \big( V_t - g_t \big) \ud \zeta^*_t = 0,\\
&\langle \psi_T, (1-\tilde\xi^{*}_{T-})\rangle \big(V_T - h_T \big) \Delta \zeta^*_T = 0.
\end{aligned}
\end{equation*}
The formulae above convey the intuitive idea that the $i$-th incarnation of the informed player should only stop (with some probability) when the corresponding value process $U^i$ equals the stopping payoff $f^i$. Instead, the uninformed player should only stop (with some probability) when the value process $V$ equals the stopping payoff $g$.
\medskip

\textbf{Sufficient conditions.} An analogue of Corollary \ref{cor:saddle_mart} holds with non-obvious changes to the notation and assumptions. Recall that $f^j_t=f(t,X^j_t)$, $g^j_t=g(t,X^j_t)$, $h^j_t=h(t,X^j_t)$ and $f_t=f(t,X_t)$, $g_t=g(t,X_t)$, $h_t=h(t,X_t)$ with $X=X^\cJ$. Let $\varUpsilon$ be the class of measurable maps $\Phi:[0,T] \times C([0, T]) \to [0,1]$ such that $(\Phi(t, X_{\cdot \wedge t}))_{t \in [0, T]}$ belongs to $\cAcirc_0(\F^2)$ whereas $(\Phi(t, X^0_{\cdot \wedge t}))_{t \in [0, T]}$ and $(\Phi(t, X^1_{\cdot \wedge t}))_{t \in [0, T]}$ belong to $\cAcirc_0(\F^W)$.
Then $\varUpsilon$ is the family of maps that determine admissible strategies according to \eqref{eqn:tildexi_def} and \eqref{eq:zeta*dec}. 

\begin{corollary}\label{cor:saddle_mart2}
Let $(U^0_t)_{t\in[0,T]}$, $(U^1_t)_{t\in[0,T]}$ be $\F^W$-progressively measurable and $(V_t)_{t\in[0,T]}$ be $\F^2$-progressively measurable. Let $\hat \Xi^0, \hat \Xi^1, \hat\Lambda \in \varUpsilon$. Define $\hat\xi^i_t = \hat \Xi^i(t, X_{\cdot \wedge t})$ and $\hat\zeta^i_t = \hat\Lambda(t, X^i_{\cdot \wedge t})$ for $i=0,1$; hence, $\hat\xi^i \in \cAcirc_0(\F^2)$ and $\hat \zeta^i \in \cAcirc_0(\F^W)$. Set, for $t\in[0,T]$,
\begin{align*}
\hat M^{0;i}_{t} &=  \int_{[0, t)} g(s,X^i_s) \ud \hat\zeta^i_s + (1-\hat\zeta^i_{t-}) U^i_{t},\\
\hat N^0_t&=\int_{[0, t)} f(s,X_s)\langle\psi_s, \ud \hat\xi_s\rangle + \langle \psi_t, 1-\hat\xi_{t-} \rangle V_t .
\end{align*}

Assume that
\begin{enumerate}[(i)]
\item the process $(\hat M^{0;i}_{t})_{t\in[0,T]}$ is an $\F^W$-submartingale for $i=0,1$,
\item the process $(\hat N^0_t)_{t\in[0,T]}$ is an $\F^2$-supermartingale,
\item for $i=0,1$, it holds $\P$-a.s.,
\[
f(t,X^i_t)+\frac{(h-f)(t,X^i_t)\Delta\hat\zeta^i_t}{1-\hat\zeta^i_{t-}} \ge U^i_t,\quad\text{for all $t \in [0, T]$ such that $\hat\zeta^i_{t-}<1$}, 
\]
\item it holds $\P$-a.s., 
\[
g(t,X_t)+\frac{(h-g)(t,X_t)\langle\psi_t,\Delta\hat \xi_t\rangle}{\langle \psi_t, 1-\hat\xi_{t-}\rangle } \le V_t,\quad\text{for all $t \in [0, T]$ such that $\langle \psi_t, 1-\hat\xi_{t-}\rangle>0$,}
\] 
\item $V_0 = \E[\ind_{\{\cJ=0\}}U^0_0+\ind_{\{\cJ=1\}}U^1_0]$.
\end{enumerate}
Then the value of the game equals $V_0$ and a saddle point is given by 
\begin{equation}\label{eqn:partial_equil}
(\xi^{*,0}_t,\xi^{*,1}_t)_{t\in[0,T]}=\big(\hat \Xi^0(t, X^0_{\cdot \wedge t}),\hat \Xi^1(t, X^1_{\cdot \wedge t})\big)_{t \in [0,T]} \quad \text{and}\quad (\zeta^*_t)_{t\in[0,T]}=\big(\hat\Lambda(t, X_{\cdot \wedge t})\big)_{t \in [0, T]}.
 \end{equation}
\end{corollary}

We defined players' strategies via maps in $\varUpsilon$ because each strategy appears in two guises depending on the player inspecting it. The assumptions of the corollary view the strategies through the lens of the opponent: the uninformed player considers $\hat \xi^0$ (resp.\ $\hat \xi^1$) which arises when the player {\em pretends} that $\cJ=0$ (resp.\ $\cJ=1$) but observes only the process $X$ -- these are the counterparts of $\tilde \xi^0, \tilde \xi^1$ in the first part of this subsection discussing the necessary conditions; the informed player instead is able to separate the strategy of the uninformed player based on the observation of $\cJ$, see \eqref{eq:zeta*dec}. The saddle point \eqref{eqn:partial_equil} requires inserting appropriate processes into the maps $\hat\Xi^0$, $\hat\Xi^1$ and $\hat\Lambda$ depending on the players' filtrations. The overall idea, which is common in game theory, is that both players know the equilibrium maps $\hat\Xi^0$, $\hat\Xi^1$ and $\hat\Lambda$ but the uninformed player can only compute the realised trajectories of the increasing processes conditional upon the observed filtration.

Notice that the choice of the process $\hat N^0$ in the statement above is motivated by $\widetilde N^0$ in \eqref{eq:MNN} and the equivalence between $N^0$ and $\widetilde N^0$ therein, up to a martingale process. We avoid repeating also Corollary \ref{cor:saddle_stop}, which holds in analogous fashion. 
\smallskip
 
\subsection{A heuristic derivation of PDE systems}\label{sec:PDE}
Corollaries \ref{cor:saddle_mart} and \ref{cor:saddle_mart2} suggest a practical approach to the actual construction of the equilibrium payoffs. Although it is unclear how to formulate a rigorous statement, we want to discuss here some natural ideas that hopefully can provide useful tools for practical solution of specific problems. We start with the problem presented in Section \ref{sec:partdyn} and conclude with the game from Section \ref{sec:partial}. 

\textbf{Partially observed dynamics.}
Let $\hat p_t$ be defined as in \eqref{eq:pt} but with $\tilde \xi^*$ replaced by $\hat \xi$ from Corollary \ref{cor:saddle_mart2}. If we postulate that $U^j_t=u^j(t,\hat p_t,X^j_t)$ for $j=0,1$ and $V_t=v(t,\hat p_t,X_t)$ for suitable functions $u^j$ and $v$ to be determined, then we can connect conditions (i)--(v) from the corollary above to an analytical problem. With a small loss of generality, let us restrict our attention to a situation where the processes $\hat \xi^{i},\hat \zeta^{i}$, $i=0.1$, are continuous. This is not overly restrictive for a characterisation of the equilibrium payoffs, thanks to approximation arguments as those exploited in, e.g., \cite{TouziVieille2002} and \cite[Sec.\ 5]{de2022value}; however, we cannot expect, in general, to find equilibrium strategies with continuous paths. 

Conditions (iii) and (iv) from the corollary translate into: for $i=0,1$, 
\begin{align*}
u^i(t,\pi,x)\le f(t,x)\quad \text{and}\quad v(t,\pi,x)\ge g(t,x)\quad\text{for all $(t,\pi,x)\in[0,T]\times(0,1)\times\R$}.
\end{align*}
These conditions suggest the players' stopping regions. In particular, the two incarnations of the informed player should only stop in the regions
\[
\cS^i\coloneqq \{(t,\pi,x):u^i(t,\pi,x)=f(t,x)\},\quad i=0,1,
\]
whereas the uninformed player should stop in the region
\[
\cS\coloneqq\{(t,\pi,x):v(t,\pi,x)=g(t,x)\}.
\]
Let $\cL_{X,\psi}$ be the infinitesimal generator of the pair $(X_t,\psi_t)$ defined in \eqref{eq:SDEXpsi}. For $i=0,1$, let $\cL^i_{X,\psi}$ be the infinitesimal generator of the process $(X^i_t,\psi_t)$. The analytical counterpart of (i) can be deduced by the following equations: for $i=0,1$
\begin{align}\label{eq:VIi}
\begin{aligned}
&\partial_t u^i(t,\pi,x)+\cL^i_{X,\psi}u^i(t,\pi,x)\ge 0,\ \ \ \text{on}\ \{v>g\},\\
&\partial_\pi u^i(t,\pi,x)=0,\quad\text{on $\cS^0\cup\cS^1$},
\end{aligned}
\end{align}
which need to be understood in an appropriate sense (e.g., in the viscosity sense) and which we derived using $\ud \hat \zeta^i_t=0$ if $(t,\hat p_t,X^i_t)\in \{v>g\}$, due to Corollary \ref{cor:acting2}. Notice that the first inequality above is the result of the diffusive dynamics of the pair $(\psi_t,X^i_t)_{t\in[0,T]}$, whereas the second condition takes care of the bounded variation component of the optimal dynamics for the belief process (cf.\ \eqref{eq:pt}). The latter, arises in equilibrium  only due to optimal actions of the two incarnations of the informed player. Since there is no action in the set $\{u^0<f\}\cap\{u^1<f\}$, then $\ud \hat p_t$ is proportional to $\ud \psi_t$ (i.e., purely diffusive) whenever $(t,\hat p_t,X_t)\in\{u^0<f\}\cap\{u^1<f\}$. That is why we need the second equation to hold in $\cS^0\cup\cS^1$ only. In particular, the equation says that the equilibrium payoffs of the two incarnations of the informed player are not affected by (optimal) changes in the belief process of the uninformed player.
 
Analogous arguments translate condition (ii) into the inequality
\begin{align}\label{eq:VIv}
\partial_t v(t,\pi,x)+\cL_{X,\psi}v(t,\pi,x)\le 0,\quad \text{on}\ \{u^0<f\}\cap\{u^1<f\}.
\end{align}
Here we do not need a condition on the derivative $\partial_\pi v$ because, as argued above, in equilibrium the belief process $(\hat p_t)_{t\in[0,T]}$ of the uninformed player is purely diffusive whenever $(t,\hat p_t, X_t)\in \{u^0<f\}\cap\{u^1<f\}$. The martingale characterisation also suggests that both inequalities \eqref{eq:VIi} and \eqref{eq:VIv} become strict equalities on the set $\{u^0<f\}\cap\{u^1<f\}\cap \{v>g\}$. That is, 
\begin{align}\label{eq:VIim}
\begin{aligned}
\partial_t u^i(t,\pi,x)+\cL^i_{X,\psi}u^i(t,\pi,x)&= 0,\\
\partial_t v(t,\pi,x)+\cL_{X,\psi}v(t,\pi,x)&= 0,
\end{aligned}
\end{align}
for $(t,\pi,x)\in\{u^0<f\}\cap\{u^1<f\}\cap \{v>g\}$. Finally, condition (v) in Corollary \ref{cor:saddle_mart2} connects the functions $u^0,u^1$ and $v$ via the formula
\[
v(t,\pi,x)=\pi u^1(t,\pi,x)+(1-\pi) u^0(t,\pi,x).
\] 
Precisely, condition (v) only give the above link for $t=0$ and given an initial point $(\pi, x)$ but due to the Markovian structure of the problem, the game can be started at any time $t$ and from any configuration $(\pi, x)$ justifying the above statement.

Players' strategies in Corollary \ref{cor:saddle_mart2} are defined using maps $\Xi^0, \Xi^1$ and $\Lambda$. These maps are evaluated on `wrong' processes in conditions (i)-(v) in order to imply the saddle point assertion of \eqref{eqn:partial_equil}. This feature is clearly visible in the equations above. In \eqref{eq:VIi}, the underlying dynamics is given by $X^i$ and the inequality must hold on $\{v > g\}$, i.e., on the inaction set for the uninformed player pretending to observe $X^i$, hence, when $\hat\zeta^i$ from Corollary \ref{cor:saddle_mart2} does not grow. Similarly, \eqref{eq:VIv} holds when neither $\hat\xi^0$ nor $\hat \xi^1$ act.
 
We conclude by noticing that the system above is precisely the one conjectured in the ve\-ri\-fi\-cation theorem formulated in \cite{DEG2020}, thus providing the theoretical foundation for such theorem. 
\medskip

\textbf{Partially observed scenarios.} Analogous arguments may be developed in the framework of Section \ref{sec:partial}. This is a useful exercise because it leads to a different type of variational problem than the one studied by Gr\"un \cite{Grun2013} in a Markovian formulation of the game with partially observed scenarios. Indeed, Gr\"un obtains a single variational inequality for the value of the uninformed player whereas we obtain a system of variational inequalities which is close in spirit to those found in the PDE literature on nonzero-sum Dynkin games (e.g., \cite{bensoussan1974}). This is in line with our overall approach to the study of the problem.

In the framework of Section \ref{sec:partial} let us now restrict our attention to the case of a diffusive underlying dynamics $(X_t)_{t\in[0,T]}$ in $\R^d$ (fully known to both players) with the infinitesimal generator denoted by $\cL_X$. Assume that there are measurable functions $f^i,g^i,h^i:[0,T]\times\R^d\to \R$ such that $f^i_t=f^i(t,X_t)$, $g^i_t=g^i(t,X_t)$ and $h^i_t=h^i(t,X_t)$ for $i=0,1$. Since we are only interested in heuristics, in order to convey the main ideas we postulate again  that the generating processes $\hat \xi^0$, $\hat \xi^1$ and $\hat \zeta$ from Corollary \ref{cor:saddle_mart} are continuous. Recall that $\hat p_t$ is defined as in \eqref{eq:ppos} but with $\hat \xi$ in place of $\xi^*$.

If we look at equilibrium values as deterministic functions of the state dynamics $(\hat p_t,X_t)_{t\in[0,T]}$, we must determine functions $v,u^0,u^1:[0,T]\times[0,1]\times\R^d\to \R$ such that $V_t=v(t,\hat p_t,X_t)$, $U^j_t=u^j(t,\hat p_t,X_t)$, $j=0,1$. Recall that the belief process $\hat p_t$ in this framework only moves as a result of the informed player's stopping strategy in equilibrium. By the assumed continuity of the generating processes, conditions (iii) and (iv) in Corollary \ref{cor:saddle_mart} translate into
\[
u^i(t,\pi,x)\le f^i(t,x)\quad\text{and}\quad v(t,\pi,x)\ge \pi g^1(t,x)+(1-\pi) g^0(t,x),
\]  
for $(t,\pi,x)\in[0,T]\times[0,1]\times\R^d$ and $i=0,1$. The `flat-off conditions' in Corollary \ref{cor:acting} help us identify the stopping regions for the two players. In particular, we have that the two incarnations of the informed player should stop in the sets
\[
\cS^i\coloneqq\{(t,\pi,x): u^i(t,\pi,x)=f^i(t,x)\},\quad i=0,1.
\]
Instead, the uninformed player should stop in the set
\[
\cS\coloneqq\{(t,\pi,x): v(t,\pi,x)=\pi g^1(t,x)+(1-\pi)g^0(t,x)\}.
\]
Denoting $\cC=([0,T]\times[0,1]\times\R^d)\setminus\cS$, the submartingale condition (i) in Corollary \ref{cor:saddle_mart} translates into
\begin{align*}
&\partial_t u^i(t,\pi,x)+\cL_X u^i(t,\pi,x)\ge 0,\ \ \text{on}\ \ \cC,\\
&\partial_\pi u^i(t,\pi,x)= 0,\ \ \text{on}\ \ \cS^0\cup\cS^1, 
\end{align*}
where the second equation accounts for the fact that a change in the uninformed player's belief should not affect the informed player's equilibrium payoff. Analogously, the supermartingale condition in (ii) of Corollary \ref{cor:saddle_mart} reads in analytical terms as 
\[
\partial_t v(t,\pi,x)+\cL_X v(t,\pi,x)\le 0\ \ \ \ \text{on}\ \ \{u^0<f^0\}\cap\{u^1<f^1\}.
\]
Finally, by the martingale characterisation we get
\[
\partial_t u^i(t,\pi,x)+\cL_{X}u^i(t,\pi,x)= 0\ \ \text{and}\ \ \partial_t v(t,\pi,x)+\cL_{X}v(t,\pi,x)= 0,
\]
for $(t,\pi,x)\in\{u^0<f\}\cap\{u^1<f\}\cap \cC$, $i=0,1$, and condition (v) in Corollary \ref{cor:saddle_mart} reads as
\[
v(t,\pi,x)=\pi u^1(t,\pi,x)+(1-\pi)u^0(t,\pi,x),
\]
where, again, we refer to the Markovianity of the framework to allow for arbitrary $t, \pi, x$ instead of a fixed initial point in (v).

\appendix
\section{Review of aggregation results}\label{app:aggr} 

We recall useful concepts from \cite{elkaroui1981} (see paragraph 2.11). Throughout the section, $\G$ is a filtration of sub-$\sigma$-algebras of $\cF$, unless otherwise specified.
\begin{definition}\label{def:supsubsys}
A family ${\bf X}\coloneqq\{X(\theta),\theta\in\cT_0(\G)\}$ is a $\cT_0(\G)$-system if $X(\theta)$ is $\cG_\theta$-measurable for every $\theta\in\cT_0(\G)$ and $X(\theta_1)=X(\theta_2)$ on $\{\theta_1=\theta_2\}$ for any $\theta_1,\theta_2\in\cT_0(\G)$.

{\bf(a)} A $\cT_0(\G)$-system ${\bf X}$ is a $\cT_0(\G)$-(super/sub)martingale system if
\begin{itemize}
\item[ (i)] $\E[|X(\theta)|]<\infty$ for all $\theta\in\cT_0(\G)$,
\item[(ii)] For any $\theta_1,\theta_2\in\cT_0(\G)$ with $\theta_1\le\theta_2$, it holds
\begin{align*}
&\E[X(\theta_2)|\cG_{\theta_1}]=X(\theta_1),\quad\text{$\P$-a.s {\em (for a martingale system)}}\\
&\E[X(\theta_2)|\cG_{\theta_1}]\le X(\theta_1),\quad\text{$\P$-a.s {\em (for a supermartingale system)}}\\
&\E[X(\theta_2)|\cG_{\theta_1}]\ge X(\theta_1),\quad\text{$\P$-a.s {\em (for a submartingale system)}}
\end{align*}
\end{itemize}

{\bf(b)} A $\cT_0(\G)$-system ${\bf X}$ is right/left-continuous in expectation if for any decreasing/increasing sequence $(\theta_n)_{n\in\N}\subset\cT_0(\G)$ converging to $\theta\in\cT_0(\G)$ we have
\[
\lim_{n\to\infty}\E[X(\theta_n)]=\E[X(\theta)].
\]

{\bf(c)} A $\cT_0(\G)$-system ${\bf X}$ is of class $(D)$ if the family $\{X(\theta),\theta\in\cT_0(\G)\}$ is uniformly integrable.

{\bf(d)} A $\G$-optional process $(X_t)_{t\in[0,T]}$ aggregates the $\cT_0(\G)$-system $\{X(\theta),\theta\in\cT_0(\G)\}$ if 
\[
\P(X(\theta)=X_\theta)=1,\quad\text{for all $\theta\in\cT_0(\G)$}.
\]
\end{definition}

We are also going to need the following aggregation result which combines \cite[Prop.\ 2.14]{elkaroui1981} and arguments from \cite[Thm.\ I.3.13]{karatzas1998brownian}. 
\begin{proposition}\label{prop:aggr}
Let ${\bf X}\coloneqq\{X(\theta),\theta\in\cT_0(\G)\}$ be a $\cT_0(\G)$-(super/sub)martingale system which is also right-continuous in expectation and of class $(D)$. There exists a \cadlag (super/sub)martingale $(X_t)_{t\in[0,T]}$ of class $(D)$ that aggregates ${\bf X}$.
\end{proposition}
\begin{proof}
The proof of \cite[Prop.\ 2.14]{elkaroui1981} can be immediately adapted to the case of a family ${\bf X}$ of class $(D)$, yielding a right-continuous super/sub martingale $(X_t)_{t\in[0,T]}$ of class $(D)$ that aggregates ${\bf X}$. We apply arguments from \cite[Thm.\ I.3.13]{karatzas1998brownian} (which uses Prop.\ 3.14 therein) to show that the process $(X_t)_{t\in[0,T]}$ has \cadlag trajectories with probability $1$.
\end{proof}

Any martingale system ${\bf X}$ of class $(D)$ is trivially continuous in expectation, which leads to a useful corollary. 
\begin{corollary}\label{cor:aggr}
If ${\bf X}\coloneqq\{X(\theta),\theta\in\cT_0(\G)\}$ is a $\cT_0(\G)$-martingale system of class $(D)$, then there exists a c\`adl\`ag martingale $(X_t)_{t\in[0,T]}$ of class $(D)$ that aggregates ${\bf X}$.
\end{corollary}
Finally, we recall a standard result from martingale theory.
\begin{lemma}\label{lem:M}
A $\cT_0(\G)$-system ${\bf X}$ is a supermartingale system if and only if
\begin{align}\label{eq:iffM}
\E[X(\tau)]\le \E[X(\sigma)],\quad\text{for every pair $\tau,\sigma\in\cT_0(\G)$, $\tau \ge \sigma$}.
\end{align}
Moreover, ${\bf X}$ is a martingale system if and only if \eqref{eq:iffM} holds with equality.
\end{lemma}
\begin{proof}
The only if part of the claim is obvious. Now assume \eqref{eq:iffM} holds. Take $\sigma,\tau\in\cT_0(\G)$ with $\sigma\le \tau$ and $A\in\cG_\sigma$. Set $\theta=\sigma \ind_A+\tau \ind_{A^c}$ so that $\theta\in\cT_0(\G)$ with $\sigma\le \theta\le \tau$. By \eqref{eq:iffM} and the fact that ${\bf X}$ is a $\cT_0(\G)$-system we have
\[
\E[X(\tau)]\le \E[X(\theta)]=\E[X(\sigma)\ind_{A}+X(\tau)\ind_{A^c}]\implies \E[\E[X(\tau)|\cG_\sigma]1_A]=\E[X(\tau)1_A]\le \E[X(\sigma)1_A].
\]
By the arbitrariness of $A\in\cG_\sigma$ we conclude that $\E[X(\tau)|\cG_\sigma]\le X(\sigma)$.
\end{proof}

\section{Upward and downward directed families}\label{subsec:updownd}

We recall that a family of non-negative random variables $\Upsilon$ is {closed under pairwise maximisation} if $X,Y\in\Upsilon\implies X\vee Y\in\Upsilon$. This also implies that the family is upward-directed, i.e., $X,Y\in\Upsilon\implies\exists Z\in\Upsilon$ such that $Z\ge X\vee Y$ -- we use the two notions interchangeably. If the family $\Upsilon$ is closed under pairwise maximisation, then $\esssup\{X:X\in\Upsilon\}=\lim_{n\to\infty}X_{n}$, where $(X_n)_{n\in\N}\subset\Upsilon$ is a non-decreasing sequence, see \cite[Thm.~A.3]{Karatzas1998}. Clearly the property extends to families of random variables bounded from below by a real-valued random variable. An analogue definition for downward-directed families holds in the case of random variables bounded from above.

Given a $\cT_0(\G)$-system ${\bf X}=\{X(\theta),\theta\in\cT_0(\G)\}$ satisfying $\E[\esssup_{\theta\in\cT_0(\G)}|X(\theta)|]<\infty$ and a filtration $\H\subset\G$, fix an arbitrary $\sigma\in\cT_0(\H)$. It is a well-known fact in the optimal stopping theory that 
the family
\[
\big\{\E[X(\tau)|\cH_\sigma],\tau\in \cT_\sigma(\G)\big\}
\]
is both upward-directed and downward-directed. Therefore, there are sequences $(\tau_n)_{n\in\N},(\tau^k)_{k\in\N}\subset \cT_\sigma(\G)$ such that 
\begin{align}\label{eq:updownd}
\begin{aligned}
\essinf_{\tau\in \cT_\sigma(\G)}\E[X(\tau)|\cH_\sigma]=\lim_{n\to\infty}\E[X(\tau_n)|\cH_\sigma]\quad\text{and}\quad\esssup_{\tau\in\cT_\sigma(\G)}\E[X(\tau)|\cH_\sigma]=\lim_{k\to\infty}\E[X(\tau^k)|\cH_\sigma],
\end{aligned}
\end{align}
with both limits being monotone, the first one from above and the second one from below. 

As a consequence, for any $\rho\le\sigma$, $\rho,\sigma\in\cT_0(\H)$ we have by the monotone convergence theorem and the tower property of conditional expectation
\begin{align*}
\E\big[\essinf_{\tau\in\cT_\sigma(\G)}\E[X(\tau)|\cH_\sigma]\big|\cH_\rho\big]=\lim_{n\to\infty}\E[X(\tau_n)|\cH_\rho]\ge \essinf_{\tau\in\cT_\sigma(\G)}\E[X(\tau)|\cH_\rho].
\end{align*}
That, combined with the obvious inequality
\[
\E\big[\essinf_{\tau\in\cT_\sigma(\G)}\E[X(\tau)|\cH_\sigma]\big|\cH_\rho\big]\le \essinf_{\tau\in\cT_\sigma(\G)}\E[X(\tau)|\cH_\rho],
\]
yields
\begin{align}\label{eq:commute}
\begin{aligned}
&\E\big[\essinf_{\tau\in\cT_\sigma({\G})}\E[X(\tau)|\cH_\sigma]\big|\cH_\rho\big]= \essinf_{\tau\in\cT_\sigma({\G})}\E[X(\tau)|\cH_\rho],\\
&\E\big[\esssup_{\tau\in\cT_\sigma({\G})}\E[X(\tau)|\cH_\sigma]\big|\cH_\rho\big]= \esssup_{\tau\in\cT_\sigma({\G})}\E[X(\tau)|\cH_\rho],
\end{aligned}
\end{align}
where the second equality is obtained by analogous arguments as the first.


\section{Remaining proofs}

\subsection{Proof of Lemma \ref{lem:ud}}\label{app:lemud}
Given $\xi^1,\xi^2\in\cAcirc_\theta(\F^1)$ we want to show that there is $\xi^3\in\cAcirc_\theta(\F^1)$ such that $J^{\Pi^{*,1}_\theta}(\xi^3,\zeta^{*;\theta}|\cF^1_\theta)=\min\{J^{\Pi^{*,1}_\theta}(\xi^1,\zeta^{*;\theta}|\cF^1_\theta),J^{\Pi^{*,1}_\theta}(\xi^2,\zeta^{*;\theta}|\cF^1_\theta)\}$. From \eqref{eq:J} we know that 
\[
A\coloneqq\{\omega\in\Omega:J^{\Pi^{*,1}_\theta}(\xi^1,\zeta^{*;\theta}|\cF^1_\theta)(\omega)\le J^{\Pi^{*,1}_\theta}(\xi^2,\zeta^{*;\theta}|\cF^1_\theta)(\omega)\}\in\cF^1_\theta.
\]
Define $\xi^3_{\theta-}=0$ and $\xi^3_t\coloneqq\xi^1_t 1_A+\xi^2_t \ind_{A^c}$ for $t\in[\theta,T]$. It is easy to check that $\xi^3\in\cAcirc_\theta(\F^1)$ and using \eqref{eq:J}
\begin{align*}
\begin{aligned}
J^{\Pi^{*,1}_\theta}(\xi^3,\zeta^{*;\theta}|\cF^1_\theta)&=\E^{\Pi^{*,1}_\theta}\Big[\int_{[\theta,T)}f_t(1-\zeta^{*;\theta}_t)\ud \xi^3_t+\int_{[\theta,T)}g_t(1-\xi^3_t)\ud \zeta^{*;\theta}_t+\sum_{t\in[\theta,T]}h_t\Delta\zeta^{*;\theta}_t\Delta\xi^3_t\Big|\cF^1_\theta \Big]\\
&=1_A\E^{\Pi^{*,1}_\theta}\Big[\int_{[\theta,T)}f_t(1-\zeta^{*;\theta}_t)\ud \xi^1_t+\int_{[\theta,T)}g_t(1-\xi^1_t)\ud \zeta^{*;\theta}_t+\sum_{t\in[\theta,T]}h_t\Delta\zeta^{*;\theta}_t\Delta\xi^1_t\Big|\cF^1_\theta \Big]\\
&\quad+\ind_{A^c}\E^{\Pi^{*,1}_\theta}\Big[\int_{[\theta,T)}f_t(1-\zeta^{*;\theta}_t)\ud \xi^2_t+\int_{[\theta,T)}g_t(1-\xi^2_t)\ud \zeta^{*;\theta}_t+\sum_{t\in[\theta,T]}h_t\Delta\zeta^{*;\theta}_t\Delta\xi^2_t\Big|\cF^1_\theta \Big]\\
&=1_AJ^{\Pi^{*,1}_\theta}(\xi^1,\zeta^{*;\theta}|\cF^1_\theta)+\ind_{A^c}J^{\Pi^{*,1}_\theta}(\xi^2,\zeta^{*;\theta}|\cF^1_\theta)\\
&=\min\{J^{\Pi^{*,1}_\theta}(\xi^1,\zeta^{*;\theta}|\cF^1_\theta),J^{\Pi^{*,1}_\theta}(\xi^2,\zeta^{*;\theta}|\cF^1_\theta)\}.
\end{aligned}
\end{align*}
That proves that the family is downward-directed.  Hence, there is a minimising sequence using similar arguments as in Section \ref{subsec:updownd}.

The second part of the lemma is analogous.\hfill$\square$

\subsection{Proof of Proposition \ref{prop:meas}}\label{app:propmeas}
The proof requires the following auxiliary measurability result.
\begin{lemma}\label{lem:meas}
Let $X:[0,1]\times\Omega \to \R$ be $\cB([0,1]) \times \cF$-measurable, either right- or left-continuous in the first variable and satisfying the integrability condition $\E[\sup_{z \in [0,1]} |X_z|] < \infty$. 

For any complete $\sigma$-algebra $\cG\subseteq\cF$, the process $\{\E[X_z|\cG],\ z\in[0,1]\}$ admits a $\cB([0,1]) \times \cG$-me\-a\-su\-ra\-ble modification, in the sense that there is a $\cB([0,1]) \times \cG$-measurable function $Y$ such that, for each $z \in [0,1]$,
\[
Y_z = \E[ X_z | \cG], \quad \P-a.s.
\]
\end{lemma}
\begin{proof}
Assume that $z\mapsto X_z$ is right-continuous. Let us define 
\[
Y^n_z\coloneqq\sum_{k=0}^{2^n-1}\ind_{[\frac{k}{2^n},\frac{k+1}{2^n})}(z)\E[X_{\frac{k+1}{2^n}}|\cG]+\indd{1}(z)\E[X_1|\cG].
\] 
It is clear that $Y^n_z(\omega)$ is uniquely defined for all $(z,\omega)\in[0,1]\times\Omega_n$, for some $\Omega_n \in\cF$ with $\P(\Omega_n)=1$. Letting $\Omega_0\coloneqq\cap_{n\in\N}\Omega_n$, the sequence $\{Y^n_z(\omega), n\in\N\}$ is defined for all $(z,\omega)\in[0,1]\times\Omega_0$ and $\P(\Omega_0)=1$. Moreover, for any $a\in\R$
\begin{align*}
\{(z,\omega):Y^n_z(\omega)>a\}=\bigcup_{k=0}^{2^n}\big[\tfrac{k}{2^n},\tfrac{k+1}{2^n}\big)\times\{\omega: \E[X_{\frac{k+1}{2^n}}|\cG](\omega)>a\}\in\cB([0,1])\times\cG.
\end{align*} 
Defining $Y^+_z(\omega)\coloneqq \limsup_{n\to\infty}Y^n_z(\omega)$ and $Y^-_z(\omega)\coloneqq \liminf_{n\to\infty}Y^n_z(\omega)$ for $(z,\omega)\in[0,1]\times\Omega_0$, it is clear that $Y^\pm$ is $\cB([0,1])\times\cG$-measurable. It remains to show that for each $z\in[0,1]$, $Y^+_z=Y^-_z=\E[X_z|\cG]$, $\P$-a.s.

For every $z\in[0,1]$, 
\begin{align*}
\big|Y^n_z-\E[X_z|\cG]\big|&=\Big|\sum_{k=0}^{2^n-1}\E\big[X_{\frac{k+1}{2^n}}-X_z\big|\cG\big]\ind_{[\frac{k}{2^n},\frac{k+1}{2^n})}(z)\Big|\le \E\big[\sup_{0\le \lambda\le 1/2^n}\big|X_{(z+\lambda)\wedge 1}-X_{z}\big|\big|\cG\big].
\end{align*}
Let $n\to\infty$. By the right continuity of $X_z$ we have
\[
\lim_{n\to\infty}\sup_{0\le \lambda\le 1/2^n}\big|X_{(z+\lambda)\wedge 1}(\omega)-X_{z}(\omega)\big|=0,\quad \text{for all $\omega\in\Omega$}.
\] 
Thanks to boundedness of $X_z$ we can use the conditional version of the Dominated Convergence Theorem to pass the limit inside expectation. Thus, 
\[
\lim_{n\to\infty}\big|Y^n_z-\E[X_z|\cG]\big|=0,\quad\P-a.s.
\]
The latter implies $Y^+_z=Y^-_z=\E[X_z|\cG]$, $\P$-a.s., as needed. Then, setting $Y_z(\omega)\coloneqq Y^+_z(\omega)$ for $(z,\omega)\in[0,1]\times\Omega_0$ and $Y_z(\omega)=0$ for $(z,\omega)\in[0,1]\times(\Omega\setminus\Omega_0)$ concludes the proof because $\P(\Omega\setminus\Omega_0)=0$ and $\cG$ is complete.

If $z\mapsto X_z$ is left-continuous, the same proof as above but with
\[Y^n_z\coloneqq\ind_{0}(z) \E[X_{0}|\cG] + \sum_{k=0}^{2^n-1}\ind_{(\frac{k}{2^n},\frac{k+1}{2^n}]}(z)\E[X_{\frac{k}{2^n}}|\cG]\] yields the desired result.
\end{proof}

\begin{proof}[{\bf Proof of Proposition \ref{prop:meas}}]
We only show the full argument for $M^*_\theta$ and $m^*(\theta; z)$ as the one for $N^*_\gamma$ and $n^*(\gamma; z)$ is analogous. 
In the proof, when we refer to joint measurability in $(z, \omega)$, without further specifying, we mean the measurability with respect to the $\sigma$-algebra $\cB([0,1])\times\cF^1_\theta$ (notice that $\cF^1_\theta$ is complete as required by Lemma \ref{lem:meas}). 

It is clear that the term $\ind_{\{\theta\le \tau_*(z)\}}\hat V^{*,1}_\theta$ is jointly measurable in $(z, \omega)$ by the measurability of $\tau_*(z)$. Observe that $z \mapsto \tau_*(z)$ is non-decreasing and right continuous (cf., \cite[Ch.~0, Lemma~4.8]{revuzyor}). Since $\ind_{\{\tau_*(z)<\theta\}} f_{\tau_*(z)}\ind_{\{\tau_*(z)<\sigma_*\}}$ is $\cB([0,1])\times\cF$-measurable and right-continuous with respect to $z$, Lemma \ref{lem:meas} yields that the conditional expectation
\[
\E\Big[\ind_{\{\tau_*(z)<\theta\}} f_{\tau_*(z)}\ind_{\{\tau_*(z)<\sigma_*\}}\Big|\cF^1_\theta\Big](\omega)
\]
admits a jointly $(z,\omega)$-measurable modification $m^1_\theta(z,\omega)$. The map $(z, \omega) \mapsto \ind_{\{\sigma_*<\theta\}}\ind_{\{\sigma_*<\tau_*(z)\}}g_{\sigma_*}$ is $\cB([0,1])\times\cF$-measurable but neither left nor right-continuous with respect to $z$. However, $\{\sigma_*<\tau_*(z)\} = \bigcap_{\epsilon > 0} \{\sigma_* + \epsilon \le \tau_*(z)\}$ and $\indd{\sigma_* + \epsilon \le \tau_*(z)}$ is right-continuous in $z$, so using Lemma \ref{lem:meas} again, we have that
\[
\E\Big[\ind_{\{\sigma_*<\theta\}}\ind_{\{\sigma_* + \epsilon \le \tau_*(z)\}}g_{\sigma_*}\Big|\cF^1_\theta\Big]
\]
admits a jointly $(z, \omega)$-measurable modification $m^{2,\eps}_\theta(z,\omega)$. The dominated convergence theorem yields
\begin{equation}\label{eq:secondterm}
\E\Big[\ind_{\{\sigma_*<\theta\}}\ind_{\{\sigma_* < \tau_*(z)\}}g_{\sigma_*}\Big|\cF^1_\theta\Big] = \lim_{\epsilon \downarrow 0} \E\Big[\ind_{\{\sigma_*<\theta\}}\ind_{\{\sigma_* + \epsilon \le \tau_*(z)\}}g_{\sigma_*}\Big|\cF^1_\theta\Big].
\end{equation}
Thus, the limit $m^{2}_\theta(z,\omega)\coloneqq\lim_{\eps\to 0}m^{2,\eps}_\theta(z,\omega)$ exists and it is a jointly $(z,\omega)$-measurable modification of the expression on the left-hand side of \eqref{eq:secondterm}, as a pointwise limit of measurable functions. Finally, we notice that 
\[
\ind_{\{\tau_*(z)<\theta\}}\E\Big[h_{\tau_*(z)}\ind_{\{\tau_*(z)=\sigma_*\}}\Big|\cF^1_\theta\Big]
=\ind_{\{\tau_*(z)<\theta\}}\Big(\E\Big[h_{\sigma_*}\ind_{\{\tau_*(z)\ge \sigma_*\}}\Big|\cF^1_\theta\Big]-\E\Big[h_{\sigma_*}\ind_{\{\tau_*(z)>\sigma_*\}}\Big|\cF^1_\theta\Big]\Big),
\] 
and each one of the two terms on the right-hand side admits a jointly $(z, \omega)$-measurable modification by Lemma \ref{lem:meas} and the arguments above.

Combining the results from the paragraph we obtain existence of the jointly measurable modification $(z,\omega)\mapsto m^*(\theta;z)(\omega)$ for 
\[
\E\Big[\ind_{\{\tau_*(z)<\theta\}}\Big(f_{\tau_*(z)}\ind_{\{\tau_*(z)<\sigma_*\}}+h_{\tau_*(z)}\ind_{\{\tau_*(z)=\sigma_*\}}\Big)+\ind_{\{\sigma_*<\theta\}}\ind_{\{\sigma_* < \tau_*(z)\}}g_{\sigma_*}\Big|\cF^1_\theta\Big]+\ind_{\{\theta\le \tau_*(z)\}}\hat V^{*,1}_\theta.
\]
This proves the measurability of functions $m^*(\theta; \cdot)$ and $n^*(\gamma; \cdot)$ in (ii).

In order to justify (i), it is sufficient to show that $\E[\ind_A M^*_\theta] = \E\big[\ind_A \int_0^1 m^*(\theta; z) dz\big]$ for $A\in\cF^1_\theta$. Take an arbitrary $A\in\cF^1_\theta$ and write
\begin{equation*}
\begin{aligned}
&\E\big[\ind_A \int_0^1 m^*(\theta;z)\ud z\big]
=
\int_0^1\E\big[\ind_A m^*(\theta;z)\big]\ud z\\
&=
\int_0^1\!\!\E\Big[\ind_{A}\E\Big[\ind_{\{\tau_*(z)<\theta\}}\big(f_{\tau_*(z)}\ind_{\{\tau_*(z)<\sigma_*\}}\!+\!h_{\tau_*(z)}\ind_{\{\tau_*(z)=\sigma_*\}}\big)\\
&\qquad\qquad\qquad+\!\ind_{\{\sigma_*<\theta\}}\ind_{\{\sigma_* < \tau_*(z)\}}g_{\sigma_*}\Big|\cF^1_\theta\Big]\!+\!\ind_{\{\theta\le \tau_*(z)\}}\hat V^{*,1}_\theta\Big)\Big]\ud z\\
&=
\int_0^1\!\!\E\Big[\ind_{A}\Big(\ind_{\{\tau_*(z)<\theta\}}\big(f_{\tau_*(z)}\ind_{\{\tau_*(z)<\sigma_*\}}\!+\!h_{\tau_*(z)}\ind_{\{\tau_*(z)=\sigma_*\}}\big)\\
&\qquad\qquad\qquad+\!\ind_{\{\sigma_*<\theta\}}\ind_{\{\sigma_* < \tau_*(z)\}}g_{\sigma_*}\!+\!\ind_{\{\theta\le \tau_*(z)\}}\hat V^{*,1}_\theta\Big)\Big]\ud z,
\end{aligned}
\end{equation*}
where the first equality is by Fubini's theorem (which holds by joint measurability of $(z,\omega)\mapsto m^*(\theta;z)(\omega)$), the second equality is by (ii) and the third one is by tower property. On the other hand,
\begin{equation*}
\begin{aligned}
&\E[\ind_A M^*_\theta]\\
&=\E\Big[\ind_{A}\E\Big[\ind_{\{\tau_*<\theta\}}\big(f_{\tau_*}\ind_{\{\tau_*<\sigma_*\}}+h_{\tau_*}\ind_{\{\tau_*=\sigma_*\}}\big)+\ind_{\{\sigma_*<\theta\}}\ind_{\{\sigma_* < \tau_*\}}g_{\sigma_*}\Big|\cF^1_\theta\Big]+\ind_A\ind_{\{\theta\le \tau_*\}}\hat V^{*,1}_\theta\Big]\\
&=\E\Big[\ind_{A}\Big(\ind_{\{\tau_*<\theta\}}\big(f_{\tau_*}\ind_{\{\tau_*<\sigma_*\}}+h_{\tau_*}\ind_{\{\tau_*=\sigma_*\}}\big)+\ind_{\{\sigma_*<\theta\}}\ind_{\{\sigma_* < \tau_*\}}g_{\sigma_*}+\ind_{\{\theta\le \tau_*\}}\hat V^{*,1}_\theta\Big)\Big]\\
&=\int_0^1\E\Big[\ind_{A}\Big(\ind_{\{\tau_*(z)<\theta\}}\big(f_{\tau_*(z)}\ind_{\{\tau_*(z)<\sigma_*\}}\!+\!h_{\tau_*(z)}\ind_{\{\tau_*(z)=\sigma_*\}}\big)\\
&\qquad\qquad\qquad+\!\ind_{\{\sigma_*<\theta\}}\ind_{\{\sigma_* < \tau_*(z)\}}g_{\sigma_*}\!+\!\ind_{\{\theta\le \tau_*(z)\}}\hat V^{*,1}_\theta\Big)\Big] \ud z,
\end{aligned}
\end{equation*}
where the second equality is by tower property and the third one by Fubini's theorem, which is justified by the joint measurability in $(z,\omega)$ of the expression under the expectation. This concludes the proof of (i).
\end{proof}

\section{Some decompositions of processes and stopping times}\label{sec:optionalproj}
In this section we obtain a handy decomposition of stochastic processes for the study of examples from Section \ref{sec:examples}. We believe most of these results to be well-known from the general theory of stochastic processes but we are unable to provide precise references for them.

Let $\G\subset\F$ be a right-continuous filtration completed with $\P$-null sets. Given a process $X\in \cL_b(\P)$, we denote by $\optional{X}^{\G}=(\optional{X}^{\G}_t)_{t\in[0,T]}$ its $\G$-optional projection under the measure $\P$, i.e., the unique $\G$-optional process such that $\optional{X}^{\G}_\tau \ind_{\{\tau<\infty\}}=\E\big[X_\tau \ind_{\{\tau<\infty\}}\big|\cG_\tau\big]$, $\P\as$, for any stopping time $\tau\in\cT_0(\G)$. We recall that the optional projection $\optional{\!A}^\G$ of a {\em non-decreasing} process $A$ is a submartingale because
\[
\optional{\!A}^\G_s=\E[A_s|\cG_s]\le \E[A_t|\cG_s]=\E[\E[A_t|\cG_t]|\cG_s]=\E[\optional{\!A}^\G_t|\cG_s],\quad\text{ for } s\le t.
\] 
The optional projection of a bounded variation process is the difference of two submartingales hence, in particular, a semi-martingale. 

In the next lemma we take $\F=\G\vee\sigma(\Theta)$, where $\Theta$ is a random variable on $(\Omega,\cF,\P)$ taking values in some measurable space $(E,\cE)$. Thats is, $\F$ is the initial enlargement of the filtration $\G$ by $\Theta$. Notice that we do not assume independence of $\Theta$ from $\G$.
\begin{lemma}\label{lem:decomp}
Any $\F$-optional process $(A_t)_{t\in[0,T]}$ can be written as $A_t(\omega)=\tilde A(t,\omega,\Theta(\omega))$, where $\tilde A$ is $\cB([0,T])\times\cF\times\cE$-measurable function and, for every fixed $z \in E$, the process $(t, \omega) \mapsto \tilde A(t, \omega, z)$ is $\G$-optional. 
\end{lemma}
\begin{proof}
By an adaptation of Prop.~3.3 and Cor.~3.4 in \cite{esmaeeli2018} to a general measurable space $(E, \cE)$ instead of $(\R, \cB(\R))$, for any $\F$-stopping time $\tau$, there is a measurable function $\tau':\Omega \times E \to [0, T]$ such that $\tau(\omega) = \tau'(\omega, \Theta(\omega))$ and $\tau'(\omega, z)$ is a $\G$-stopping time for any $z \in E$. This implies that the statement of the lemma holds true for processes $A_t = \ind_{\{t \ge \tau \}}$, where $\tau$ is an $\F$-stopping time. 

Let $\cC$ be the class of $\cF_T$-measurable (not necessarily optional) processes $(A_t)_{t\in[0,T]}$ that satisfy $A_t(\omega)=\tilde A(t,\omega,\Theta(\omega))$ for a measurable $\tilde A$ such that $(t, \omega) \mapsto \tilde A(t, \omega, z)$ is a $\G$-optional process for any $z \in E$. From the previous paragraph, $\cC$ contains constants and processes of the form $A_t = \ind_{\{t \ge \tau \}}$, where $\tau$ is a $\F$-stopping time. Recall that $(\{\tau\le t\},\tau\in\cT(\F),t\ge 0)$ forms a $\pi$-system that generates the $\F$-optional $\sigma$-algebra (\cite[Thm.~IV.64\roundbrackets{c}]{DellacherieMeyerA}). We claim that for any bounded sequence $\{(A^n_t(\omega))_{t\in[0,T]},n\in\N\}\subset \cC$ such that $A^n_t(\omega)\uparrow A_t(\omega)$ for all $(t,\omega)$ as $n\to\infty$ we also have $A_t(\omega)=\tilde A(t,\omega,\Theta(\omega))$ for a measurable function $\tilde A(t,\omega,z)$ such that $(t,\omega)\mapsto \tilde A(t,\omega,z)$ is $\G$-optional for any $z\in E$. Then, a monotone class theorem (\cite[Thm.\ 3.14]{williams1991probability}) guarantees that $\cC$, which is also a vector space, contains every bounded $\F$-optional process. An extension to unbounded processes is immediate by truncation and taking limits.

Let us now verify the claim about the monotone convergence of a sequence $\{(A^n_t)_{t\in[0,T]},n\in\N\} \subset \cC$. For $(A^n_t(\omega))_{t\in[0,T]}\in\cC$, the function $(t,\omega,z)\mapsto\tilde A^n(t,\omega,z)$ is $\cB([0,T])\times\cF\times\cE$-measurable and satisfies: $(t,\omega)\mapsto \tilde A^n(t,\omega,z)$ is $\G$-optional for every $z\in E$. Define $\tilde A(t,\omega,z)\coloneqq \limsup_{n\to\infty}\tilde A^n(t,\omega,z)$ so that $\tilde A$ is $\cB([0,T])\times\cF\times\cE$-measurable and $(t,\omega)\mapsto \tilde A_t(t,\omega,z)$ is $\G$-optional for every $z\in E$.
We have
\[
A_t(\omega)=\lim_{n\to\infty}A^n_t(\omega)=\limsup_{n\to\infty} A^n_t(\omega)=\limsup_{n\to\infty} \tilde A^n(t,\omega,\Theta(\omega))=\tilde A(t,\omega,\Theta(\omega)).
\]
Hence $(A_t(\omega))_{t\in[0,T]}\in\cC$ as needed.
\end{proof}

When we specify a bit more the structure of the probability space $(\Omega,\cF,\P)$ we are able to obtain finer properties of the representation $A_t(\omega)=\tilde A(t,\omega,\Theta(\omega))$ than in the lemma above. Assume that 
\begin{align}\label{eq:prodspace}
(\Omega,\cF,\P)=(\Omega^0\times\Omega^1,\cF^0\times\cF^1,\P^0\times\P^1).
\end{align}
Given a filtration $\G^0=(\cG^0_t)_{t\in[0,T]}$ on $\cF^0$ satisfying usual conditions let $\G$ be the $\P^0\times\P^1$-completion of $(\cG^0_t\times \{\Omega^1,\varnothing\})_{t\in[0,T]}$. Let $\Theta:(\Omega^1,\cF^1)\to (E,\cE)$ be measurable and let $\Sigma(\Theta)$ be the $\sigma$-algebra $(\{\Omega^0,\varnothing\}\times \Theta^{-1}(A), A\in\cE)$ in $\cF^0\times\cF^1$. Set $\F=\G\vee\Sigma(\Theta)$, denote by $(\G^0)^o$ the $\G^0$-optional $\sigma$-algebra on $[0,T]\times\Omega^0$ and let $(\F)^o$ be the $\F$-optional $\sigma$-algebra on $[0,T]\times\Omega$. In this context, $\sigma(\Theta)\coloneqq\{\Theta^{-1}(H),H\in\cE\}$, with $\Theta^{-1}(H)=\{\omega_1:\Theta(\omega_1)\in H\}$, is the $\sigma$-algebra generated by $\Theta$ on $\Omega^1$ (see, e.g., \cite[p.\ 76]{Halmos}). Moreover, sets of the form $\{(t, \omega_0): t \ge \tau(\omega_0)\}$ for $\G^0$-stopping times $\tau$ generate the $\G^0$-optional $\sigma$-algebra $(\G^0)^o$ on $[0,T]\times\Omega^0$ (see \cite[Thm.~IV.64\roundbrackets{c}]{DellacherieMeyerA}). 
This setting will apply to the following 3 lemmas.

\begin{lemma}\label{lem:decomp_tau}
Assume that $E$ is countable. Any $\F$-stopping time $\tau$ has the representation $\tau(\omega) = \tau'(\omega_0, \Theta(\omega_1))$ for a $\G^0 \times \cE$-measurable function $\tau'$ such that $\omega_0 \mapsto \tau'(\omega_0, z)$ is a $\G^0$-stopping time for any $z \in E$.
\end{lemma}
\begin{proof}
We start with an auxiliary result. Let $\varphi:\Omega \to \R$ be of the form $\varphi(\omega_0, \omega_1) = \varphi_0(\omega_0) \ind_{B}(\omega_1)$, where $B \in \sigma(\Theta)$, i.e., $B = \Theta^{-1}(H)$ for some $H \in \cE$ (see, e.g., \cite[p.\ 76]{Halmos}) and $\varphi_0$ is $\cG^0_t$-measurable random variable. Then 
\begin{equation}\label{eqn:simplePhi}
\varphi(\omega) = \varphi'(\omega_0, \Theta(\omega_1))                                                                                                                                                                                                                                                                                                                                                                                                                                                                                                                                                                                      
\end{equation}
for a $\cG^0_t \times \cE$-measurable function $\varphi'$. Functions $\varphi$ as in \eqref{eqn:simplePhi} form a vector space, which is closed under monotone limits and it contains constants and indicator functions of a $\pi$-system that generates $\cG^0\times\sigma(\Theta)$. Hence, the representation \eqref{eqn:simplePhi} extends to any $\cF_t$-measurable function $\varphi$ thanks to the Monotone Class Theorem (\cite[Thm.\ 3.14]{williams1991probability}) and noticing that $\cF_t=\cG^0_t\times\sigma(\Theta)$ in the setting of the lemma. 

Without loss of generality, we assume that $\P^1(\Theta = z) > 0$ for any $z \in E$; this will simplify notation in the arguments below.
Take an $\F$-stopping time $\tau$. It is $\cF_T$-measurable, so, using the first part of the proof, it has the representation $\tau(\omega) = \tau'(\omega_0, \Theta(\omega_1))$ for a $\cG^0_T \times \cE$-measurable function $\tau'$. We will show that $\omega_0 \mapsto \tau'(\omega_0, z)$ is a $\G^0$-stopping time for any $z \in E$. To this end, fix $t \in [0, T]$. The function $\varphi \coloneqq \ind_{\{\tau \le t\}}$ is $\cF_t$-measurable, so by the arguments in the first paragraph of the proof there is a $\cG^0_t \times \cE$-measurable function $\varphi'$ that satisfies \eqref{eqn:simplePhi}. Hence, we have the equality
\[
\{ (\omega_0, \omega_1) \in \Omega: \tau'(\omega_0, \Theta(\omega_1)) \le t \} = \{ (\omega_0, \omega_1) \in \Omega: \varphi'(\omega_0, \Theta(\omega_1)) = 1 \}.
\]
By applying the map $(\omega_0, \omega_1) \mapsto (\omega_0, \Theta(\omega_1))$ and recalling that $\Theta$ is a surjective map, the above equality yields
\[
\{ (\omega_0, z) \in \Omega^0 \times E: \tau'(\omega_0, z) \le t \} = \{ (\omega_0, z) \in \Omega^0 \times E: \varphi'(\omega_0, z) = 1 \}.
\]
The set on the right-hand side is $\cG^0_t \times \cE$-measurable by the construction of $\varphi'$. So is the set on the left-hand side and, by Fubini's theorem, $z$-sections of this set are $\cG^0_t$-measurable, i.e.,
\[
\{\omega_0 \in \Omega^0: \tau'(\omega_0, z) \le t \} \in \cG^0_t
\]
for any $z \in E$. By the arbitrariness of $t$, we conclude that $\tau'(\cdot, z)$ is a $\G^0$-stopping time for any $z$. We finish by commenting why we could exclude from the above analysis those $z \in E$ with $\P^1(\Theta = z) = 0$: for such $z$, we set $\tau'(\cdot, z) = 0$.
\end{proof}

\begin{lemma}\label{cor:decomp}
We have $(\G^0)^o \times \sigma(\Theta)\subseteq(\F)^o$. Any $A: [0, T] \times \Omega \mapsto \R$ which is $(\G^0)^o \times \sigma(\Theta)$-measurable can be written as $A_t(\omega_0,\omega_1)=\tilde A(t,\omega_0,\Theta(\omega_1))$, where $\tilde A$ is $(\G^0)^o\times\cE$-measurable function. Moreover, if $E$ is countable then the representation holds for any $\F$-optional process $(A_t)_{t\in[0,T]}$.
\end{lemma}

\begin{proof}
First we prove the representation of $(\G^0)^o \times \sigma(\Theta)$-measurable process $A$.
For any $\G^0$-stopping time $\tau$ and any $B = \Theta^{-1}(H)$, where $H \in \cE$, let us first consider processes of the form 
\begin{align}\label{eq:simpleA}
A_t(\omega) = \ind_{\{\tau(\omega_0) \le t\}} \ind_{B}(\omega_{1}).
\end{align}
Then we have $A_t(\omega)=\tilde A(t, \omega_0, \Theta(\omega_1))$ with $\tilde A(t,\omega_0,z) = \ind_{\{\tau(\omega_{0}) \le t\}} \ind_{H}(z)$. 
Since processes of the form \eqref{eq:simpleA} generate $(\G^0)^o \times \sigma(\Theta)$, using the Monotone Class Theorem (\cite[Thm.\ 3.14]{williams1991probability}) the representation $A_t(\omega)=\tilde A(t,\omega_0,\Theta(\omega_1))$ extends to all bounded functions $A: [0, T] \times \Omega \mapsto \R$ which are measurable with respect to $(\G^0)^o \times \sigma(\Theta)$. Moreover, each process of the form \eqref{eq:simpleA} is certainly $(\F)^o$-measurable and 
then we also obtain $(\G^0)^o \times \sigma(\Theta)\subseteq(\F)^o$.

For the last statement we need to show that when $E$ is countable, then the $\F$-optional $\sigma$-algebra $(\F)^o$ coincides with $(\G^0)^o \times \sigma(\Theta)$. The inclusion $(\G^0)^o \times \sigma(\Theta)\subseteq(\F)^o$ has already been proved. For the reverse inclusion, 
we recall that the $\F$-optional $\sigma$-algebra is generated by sets of the form $\{(t, \omega) \in [0, T] \times \Omega: \tau(\omega) \le t\}$, where $\tau$ is an $\F$-stopping time (see \cite[Thm.\ IV.64\roundbrackets{c}]{DellacherieMeyerA}). It remains to show that such sets belong to $(\G^0)^o \times \sigma(\Theta)$. To this end, we fix an $\F$-stopping time $\tau$. By Lemma \ref{lem:decomp_tau}, we have
$\tau(\omega)=\tau'(\omega_0,\Theta(\omega_1))$ with $\tau'(\cdot,z)$ a $\G^0$-stopping time for each $z\in E$. Hence,
\begin{align*}
\begin{aligned}
&\big\{(t,\omega)\in[0,T]\times\Omega:\tau(\omega)\le t\big\}=\big\{(t,\omega)\in[0,T]\times\Omega:\tau'(\omega_0,\Theta(\omega_1))\le t\big\}\\
&=\bigcup_{z\in E}\Big(\big\{(t,\omega)\in[0,T]\times\Omega:\tau'(\omega_0,z)\le t\big\}\cap\{(t,\omega)\in[0,T]\times\Omega:\Theta(\omega_1)=z\}\Big)\\
&=\bigcup_{z\in E}\Big[\Big(\big\{(t,\omega_0)\in[0,T]\times\Omega^0:\tau'(\omega_0,z)\le t\big\}\times\Omega^1\Big)\cap\Big([0,T]\times\Omega^0\times\{\omega_1\in\Omega^1:\Theta(\omega_1)=z\}\Big)\Big].
\end{aligned}
\end{align*}
Since 
\begin{align*}
\begin{aligned}
\big\{(t,\omega_0)\in[0,T]\times\Omega^0:\tau'(\omega_0,z)\le t\big\}\times\Omega^1&\in (\G^0)^o \times \sigma(\Theta)\\
[0,T]\times\Omega^0\times\{\omega_1\in\Omega^1:\Theta(\omega_1)=z\}&\in (\G^0)^o \times \sigma(\Theta),
\end{aligned}
\end{align*}
we have
\[
\big\{(t,\omega)\in[0,T]\times\Omega:\tau(\omega)\le t\big\}\in (\G^0)^o \times \sigma(\Theta),
\]
which concludes the proof that $(\F)^o\subseteq(\G^0)^o \times \sigma(\Theta)$.
\end{proof}

\begin{lemma}\label{lem:suffcdual2}
Consider the setting of Lemma \ref{cor:decomp} and assume that $\Theta$ takes at most countably many values $(\theta_i)_{i\in\N}$. Assume further that $(A_t)_{t\in[0,T]}$ is \cadlag. Then, the decomposition from Lemma \ref{cor:decomp} takes the form
\begin{equation}\label{eqn:decomp_countable}
A_t(\omega_0, \omega_1) = \sum_{i=1}^\infty \ind_{\{\Theta(\omega_1) = \theta_i\}} A^i_t(\omega_0), \qquad t \in [0,T],
\end{equation}
for \cadlag $\G^0$-optional processes $(A^i_t)_{t\in[0,T]}$, $i\in\N$. Moreover, the equality 
\begin{align}\label{eq:dualsimple}
\optional{\!A}^\G_{\tau-}=\E\big[A_{\tau-}\big|\cG_\tau\big]
\end{align}
holds for any $\G$-stopping time $\tau$.
\end{lemma}
\begin{proof}
The decomposition \eqref{eqn:decomp_countable} is immediate from Lemma \ref{cor:decomp} with $\G^0$-optional processes $(A^i_t)_{t\in[0,T]}$, $i\in\N$. It remains to prove that such processes are \cadlag. If $\P^1(\Theta = \theta_i) = 0$ for some index $i\in\N$, then the choice of $(A^i_t)_{t\in[0,T]}$ does not play any role and we can set it equal to $0$ (we are using here that there are at most countably many such events). With no loss of generality we assume $\theta_i\neq\theta_j$ for $i\neq j$ so that $(\{\Theta=\theta_j\},j\in\N)$ is a partition on $\Omega_1$. Thus, for any $i\in\N$ such that $\P^1(\Theta = \theta_i) > 0$, the $\P^0 \times \P^1$ measure of the set 
\[
\{ (\omega_0, \omega_1):\ \Theta(\omega_1) = \theta_i\text{ and } t \mapsto A_t(\omega_0, \omega_1) \text{ is not \cadlag\!\!} \}
\]
is zero because $(A_t)_{t\in[0,T]}$ is \cadlag. By the decomposition \eqref{eqn:decomp_countable}, the above set reads equivalently as
\[
\{ (\omega_0, \omega_1):\ \Theta(\omega_1) = \theta_i\text{ and } t \mapsto A^i_t(\omega_0) \text{ is not \cadlag\!\!} \}.
\]
Since $\P^1(\Theta=\theta_i) > 0$, we must have $\P^0(t \mapsto A^i_t(\omega_0) \text{ is not \cadlag\!\!}) = 0$, as claimed.

Let $\tau$ be a $\G$-stopping time. By the construction of $\G$ in the paragraph above Lemma \ref{lem:decomp_tau}, $\tau$ is $\P$\as equal to
$(\omega_0, \omega_1) \mapsto \tau^0(\omega_0)$, where $\tau^0$ is a $\G^0$-stopping time. We apply the decomposition \eqref{eqn:decomp_countable}
\begin{align}\label{eq:cond}
\E[ A_{\tau-} | \cG_T] = \E\Big[ \sum_{i=1}^\infty \ind_{\{\Theta(\omega_1) = \theta_i\}} A^i_{\tau^0-}(\omega_0) \Big| \cG_T\Big] = \sum_{i=1}^\infty \P(\Theta = \theta_i) A^i_{\tau-},
\end{align}
where the last equality follows from the fact that $\cG_T=\cG^0_T\times\{\Omega_1,\varnothing\}$ and by construction $\Theta$ is independent of $\cG^0_T$; therefore, taking conditional expectation with respect to $\cG_T$ is equivalent to integrating out $\omega_1$, see \cite[Lemma 4.1]{baldi2017stochastic}.

Using analogous arguments, for any $\G$-stopping time $\tau$ we have
\begin{equation}\label{eq:Aopt}
\begin{aligned}
\optional{\!A}^\G_{\tau} = \E\Big[\sum_{i=1}^\infty \ind_{\{\Theta= \theta_i\}} A^i_\tau\Big|\cG_\tau\Big]= \sum_{i=1}^\infty \P(\Theta = \theta_i) A^i_{\tau},
\end{aligned}
\end{equation}
where we implicitly extended $A^i_t(\omega_0)$ to the product space $\Omega_0 \times \Omega_1$ in a trivial manner. The identity in \eqref{eq:Aopt} means that the processes 
\[
t\mapsto \optional{\!A}^\G_{t}\quad\text{ and }\quad t\mapsto \sum_{i=1}^\infty \P(\Theta = \theta_i) A^i_{t}
\] 
are indistinguishable. We recall that $(A^i_t)$, $i\in\N$, are \cadlag, so for any $\G$-stopping time $\tau$
\[
\optional{\!A}^\G_{\tau-} = \sum_{i=1}^\infty \P(\Theta = \theta_i) A^i_{\tau-}=\E[ A_{\tau-} | \cG_T],
\]
where the last equality is by \eqref{eq:cond}. Then \eqref{eq:dualsimple} holds by further conditioning with respect to $\cG_{\tau}$.
\end{proof}


\section{Technical results for partially observed scenarios}\label{app:pobs}
In this section we develop useful results for the analysis performed in Section \ref{sec:partial}. We recall that the random variable $\cJ$ takes at most a countable number of values, for simplicity, a subset of $\N$.

\begin{lemma}\label{lem:sigmas}
Let $\cH\coloneqq\cG\vee\sigma(\cJ)$ with $\cJ$ not necessarily independent of $\cG$. Then, for any $\cH$-measurable $X$ there is a function $f:\Omega\times\N\to\R$ such that $\omega\mapsto f(\omega,j)$ is $\cG$-measurable for each $j\in\N$ and $X(\omega)=f(\omega,\cJ(\omega))$. 
\end{lemma}
\begin{proof}
Let $\Lambda$ be the class of functions $f:\Omega\times\N\to\R$ such that $\omega\mapsto f(\omega,j)$ is $\cG$-measurable for each $j\in\N$ and denote
\[
\Sigma\coloneqq\{X: \text{$X$ is $\cH$-measurable, } X(\omega)=f(\omega,\cJ(\omega))\ \text{for some}\ f\in\Lambda\}.
\]
The class $\Sigma$ is a monotone class (cf., the proof of Lemma \ref{lem:decomp}) that contains random variables of the form $X(\omega)= \ind_{G}(\omega)\ind_{\{\cJ=j\}}(\omega)$ for $G\in\cG$ and $j\in\N$. Since sets of the form $G\cap\{\cJ=j\}$ are a $\pi$-system that generates $\cH$, then $\Sigma$ contains all $\cH$-measurable functions. 
\end{proof}
Recall the notation for the symmetric difference of two sets $A\triangle B=(A\setminus B)\cup (B\setminus A)$.
\begin{lemma}\label{cor:sigmas}
With the notation of Lemma \ref{lem:sigmas} we have $H\in \cH$ if and only if there is a collection of sets $(G^H_i)_{i\in\N}\subset\cG$ such that $H\cap\{\cJ=j\}=G^H_j\cap{\{\cJ=j\}}$ for all $j\in\N$. 

If $\cG$ is independent of $\sigma(\cJ)$ and $(L^H_i)_{i\in\N}\subset\cG$ is another collection of sets such that for all $j\in\N$, $H\cap\{\cJ=j\}=L^H_j\cap{\{\cJ=j\}}$, then 
\[
\P\Big(\bigcup_{j\in\N}(G^H_j\triangle L^H_j)\Big)=0.
\]
\end{lemma}
\begin{proof}
The `if' implication in the first statement is trivial. For the `only if' implication we take $X=1_H$, so that by Lemma \ref{lem:sigmas} there is $f_H:\Omega\times\N\to\R$ such that $f_H(\cdot,j)$ is $\cG$-measurable for each $j$ and $X(\omega)=f_H(\omega,\cJ(\omega))$ for all $\omega\in\Omega$. Then, setting 
$G^H_j\coloneqq\{\omega\in\Omega: f_H(\omega,j)=1\}$, $j\in\N$,
we have
\[	
X(\omega)=\sum_{j\in\N}\ind_{H\cap\{\cJ=j\}}(\omega)\quad\text{and}\quad X(\omega)=\sum_{j\in\N}f_H(\omega,j)\ind_{\{\cJ=j\}}(\omega)=\sum_{j\in\N}\ind_{G^H_j\cap\{\cJ=j\}}(\omega),
\]
because $\{\cJ=j\}$, $j\in\N$, is a partition of $\Omega$. This completes the proof of the first statement.

Let us now prove the uniqueness. By assumption, we have $G^H_j\cap\{\cJ=j\}=L^H_j\cap\{\cJ=j\}$, for all $j\in\N$.
Then, $\big(G^H_j\setminus L^H_j\big)\cap\{\cJ=j\}=\varnothing$, for all $j\in\N$. By the independence of $\cG$ and $\cJ$ we deduce
\[
0=\P\Big(\big(G^H_j\setminus L^H_j\big)\cap\{\cJ=j\}\Big)=\P\big(G^H_j\setminus L^H_j\big)\P(\cJ=j)\implies \P\big(G^H_j\setminus L^H_j\big)=0,
\]
and analogously 
\[
0=\P\Big(\big(L^H_j\setminus G^H_j\big)\cap\{\cJ=j\}\Big)=\P\big(L^H_j\setminus G^H_j\big)\P(\cJ=j)\implies \P\big(L^H_j\setminus G^H_j\big)=0,
\]
for all $j\in\N$. Then the second claim holds.
\end{proof}

By an application of the above results we obtain the next two facts which are fundamental for our interpretation of the model in Section \ref{sec:partial}. In what follows we use the notation for filtrations $\F^1$ and $\F^2$ introduced in Section \ref{sec:partial}.
\begin{lemma}\label{lem:set_decomposition}
Let $\theta\in\cT_0(\F^1)$ and recall the decomposition $\theta=\sum_{j\in\N}\theta_j\ind_{\{\cJ=j\}}$, where $\theta_j\in\cT_0(\F^2)$ for each $j\in\N$. Then for any $A\in\cF^1_\theta$, there are $F_j\in\cF^2_{\theta_j}$, $j \in \N$, such that
\[
A\cap\{\cJ=j\}=F_j\cap\{\cJ=j\},\qquad j\in\N.
\]
\end{lemma}
\begin{proof}
Since $A\in\cF^1_\theta$, we have $A\in\cF^1_T$ with $A\cap\{\theta\le t\}\in\cF^1_t$ for all $t\ge 0$. Lemma \ref{cor:sigmas} guarantees that for every $j \in \N$ and every $t \ge 0$, there is $F^t_j\in\cF^2_t$ such that
\[
\big(A\cap\{\theta\le t\}\big)\cap\{\cJ=j\}=\big(A\cap\{\theta_j\le t\}\big)\cap\{\cJ=j\}= F^t_j\cap\{\cJ=j\}.
\]
Since all our stopping times are bounded by $T$, the above equation implies
\begin{align*}
A\cap\{\cJ=j\}&=\big(A\cap\{\theta \le T\}\big)\cap\{\cJ=j\}\\
&=\big(A\cap\{\theta_j \le T\}\big)\cap\{\cJ=j\}= F_j\cap\{\cJ=j\},\quad\forall j\in\N
\end{align*}
with $F_j\in\cF^2_T$. However, the equations above yield for any $t\ge 0$
\begin{align*}
F_j\cap\{\theta_j\le t\}\cap\{\cJ=j\}=A\cap\{\theta_j\le t\}\cap\{\cJ=j\}=F^t_j\cap\{\cJ=j\}.
\end{align*}
The uniqueness result in Lemma \ref{cor:sigmas} guarantees that the symmetric difference $\big(F_j\cap\{\theta_j\le t\}\big)\triangle F^t_j$ is a $\P$-null set. Hence, $F_j\cap\{\theta_j\le t\} \in \cF^2_t$ by the completeness of the filtration $\F^2$ and, since the inclusion holds for any $t \ge 0$, we have $F_j \in \cF^2_{\theta_j}$.
\end{proof}

\begin{lemma}\label{lem:rv_decomposition}
Assume that $Z$ is $\cF^2_T$-measurable and integrable. Then, for any $\theta \in \cT_0(\F^1)$, recalling the decomposition $\theta=\sum_{j\in\N}\theta_j\ind_{\{\cJ=j\}}$ for $\theta_j\in\cT_0(\F^2)$, we have
\begin{equation}\label{eqn:rv_decomp}
\E[Z | \cF^1_\theta] = \sum_{j \in \N} \E[Z|\cF^2_{\theta_j}] \ind_{\{\cJ=j\}}.
\end{equation}
Consequently, for any $\cF^1_T$-measurable and integrable $\hat Z$ with the decomposition $\hat Z = \sum_{j \in \N} \indd{\cJ=j} Z_j$ and $\cF^2_T$-measurable $Z_j$, we have
\[
\E[\hat Z | \cF^1_\theta] = \sum_{j \in \N} \E[Z_j|\cF^2_{\theta_j}] \ind_{\{\cJ=j\}}.
\]
\end{lemma}
\begin{proof}
We need to show that the expectations of the left and right-hand sides of \eqref{eqn:rv_decomp} multiplied by the indicator function of any set $A \in \cF^1_{\theta}$ are identical. Take $A \in \cF^1_{\theta}$. By Lemma \ref{lem:set_decomposition}, it has a representation $A\cap\{\cJ=j\}=F_j\cap\{\cJ=j\}$ for some $F_j\in\cF^2_{\theta_j}$, $j\in\N$. We have
\begin{align*}
\E\Big[\ind_{A} \sum_{j \in \N} \E[Z|\cF^2_{\theta_j}] \ind_{\{\cJ=j\}}\Big]
&=
\E\Big[ \sum_{j \in \N} \E[Z|\cF^2_{\theta_j}] \ind_{F_j} \ind_{\{\cJ=j\}}\Big]
=
\sum_{j \in \N} \E\big[ \E[Z \ind_{F_j}|\cF^2_{\theta_j}] \ind_{\{\cJ=j\}}\big]\\
&=
\sum_{j \in \N} \E[Z \ind_{F_j}] \P(\cJ=j)
=
\E\Big[ \sum_{j \in \N} Z \ind_{F_j} \ind_{\{\cJ=j\}}\Big]\\
&=
\E[Z \ind_{A}]
=
\E\big[\E[Z | \cF^1_\theta] \ind_{A}\big], 
\end{align*}
where in the third and fourth equality we used the independence of $\cF^2_T$ from $\cJ$.

For the second statement, we write
\[
\E[\hat Z | \cF^1_\theta] = \sum_{j \in \N} \ind_{\{\cJ=j\}} \E[Z_j|\cF^1_{\theta}]
= \sum_{j \in \N} \ind_{\{\cJ=j\}} \sum_{i \in \N} \indd{\cJ=i} \E[Z_j|\cF^2_{\theta_i}]
= \sum_{j \in \N} \ind_{\{\cJ=j\}} \E[Z_j|\cF^2_{\theta_j}]
\]
with the second equality justified by the first part of the lemma.
\end{proof}

\bibliographystyle{alpha}
\bibliography{biblio}

\end{document}